\documentclass[10pt,reqno]{amsart}
\usepackage{tipa}
\usepackage{eucal}
\usepackage{mathtools}

\usepackage[T1]{fontenc}

\usepackage[euler]{textgreek}

\let\oldll\ll
\let\oldgg\gg
\usepackage{amssymb,mathabx,mathrsfs}
\usepackage{graphicx}
\renewcommand{\ll}{\oldll}
\renewcommand{\gg}{\oldgg}
\usepackage[margin=1in]{geometry}
\usepackage{enumerate}
\usepackage{bm}

\usepackage[usenames,dvipsnames]{xcolor}
\usepackage[colorlinks,linkcolor=violet,citecolor=SeaGreen]{hyperref}
\usepackage[font=scriptsize,justification=centering]{caption}
\usepackage[font=scriptsize,justification=centering]{subcaption}

\newcommand{\beq}{\begin{equation}}
\newcommand{\eeq}{\end{equation}}
\newtheorem{mainthm}{Theorem}
\newtheorem{thm}{Theorem}[section]
\newtheorem{lem}[thm]{Lemma}
\newtheorem{cor}[thm]{Corollary}
\newtheorem{ppn}[thm]{Proposition}
\theoremstyle{definition}
\newtheorem{emp}[thm]{Example}
\newtheorem{dfn}[thm]{Definition}
\newtheorem{rmk}[thm]{Remark}

\DeclareMathOperator{\image}{im}
\DeclareMathOperator{\argmax}{arg\,max}
\DeclareMathOperator{\supp}{supp}
\DeclareMathOperator{\Hess}{Hess}
\DeclareMathOperator{\diam}{diam}

\renewcommand{\P}{\mathbb{P}}
\renewcommand{\emptyset}{\varnothing}
\renewcommand{\log}{\ln}

\renewcommand{\vec}[1]{\underline{\smash{#1}}}

\newcommand{\E}{\mathbb{E}}
\newcommand{\R}{\mathbb{R}}
\newcommand{\Q}{\mathbb{Q}}
\newcommand{\f}{\frac}
\newcommand{\Hb}{\mathbb{H}}

\newcommand{\err}{\textup{\textup{\textsf{err}}}}
\newcommand{\lone}[1]{\|#1\|}
\newcommand{\loneB}[1]{\Big\|#1\Big\|}
\newcommand{\haus}{d_{\ent}}
\newcommand{\ent}{\mathcal{H}}
\newcommand{\dkl}{\mathcal{D}_\textup{\textsc{kl}}}
\newcommand{\Ind}[1]{\mathbf{1}\{#1\}}
\newcommand{\set}[1]{\{#1\}}

\newcommand{\ueta}{\vec{\eta}}
\newcommand{\deta}{\dot{\eta}}
\newcommand{\heta}{\hat{\eta}}
\newcommand{\urho}{\vec{\rho}}

\newcommand{\dsi}{\dot{\sigma}}
\newcommand{\hsi}{\hat{\sigma}}
\newcommand{\usi}{\vec{\sigma}}
\newcommand{\dta}{\dot{\tau}}
\newcommand{\hta}{\hat{\tau}}
\newcommand{\uta}{\vec{\tau}}

\newcommand{\bx}{{\bm{x}}}
\newcommand{\by}{{\bm{y}}}
\newcommand{\ux}{\vec{x}}

\newcommand{\uxx}{\underline{\underline{x}}}

\newcommand{\FE}{\textup{\textsf{f}}}
\newcommand{\annFE}{\textup{\textsf{f}}^\textsc{rs}}
\newcommand{\onersbFE}{\textup{\textsf{f}}^\textsc{1rsb}}

\newcommand{\COLS}{\Omega}
\newcommand{\tcols}{\COLS_T}
\newcommand{\fcl}{\printcol{f}}
\newcommand{\albet}{\mathcal{X}}
\newcommand{\dCOLS}{\dot{\COLS}}
\newcommand{\hCOLS}{\hat{\COLS}}
\newcommand{\dtcols}{\dot{\COLS}_T}
\newcommand{\htcols}{\hat{\COLS}_T}

\newcommand{\aaa}{\text{\footnotesize\textsc{a}}}
\newcommand{\bbb}{\text{\footnotesize\textsc{b}}}
\newcommand{\uu}{\textup{\textsf{update}}}

\newcommand{\pd}{\partial}
\newcommand{\lit}{\textup{\texttt{L}}}
\newcommand{\ulit}{\vec{\lit}}
\newcommand{\ulitp}{\ulit'}
\newcommand{\GG}{\mathscr{G}}
\newcommand{\glit}{\mathscr{G}}

\newcommand{\graph}{\mathcal{G}}

\newcommand{\dir}{\textup{\textsc{\footnotesize e}}}
\newcommand{\rev}{\textup{\textsc{\footnotesize f}}}

\newcommand{\etree}{\mathscr{T}}
\newcommand{\onetree}{\mathcal{D}}
\newcommand{\nlit}{\mathscr{N}}
\newcommand{\ngraph}{\mathcal{N}}
\newcommand{\nbd}{\mathbf{N}}

\newcommand{\dH}{\dot{H}}
\newcommand{\hH}{\hat{H}}
\newcommand{\eH}{\bar{H}}
\newcommand{\Hsym}{H^\textup{sy}}
\newcommand{\hHsym}{\hH^\textup{sy}}
\newcommand{\Hsamp}{H^\textup{sm}}
\newcommand{\dHsamp}{\dH^\textup{sm}}
\newcommand{\hHsamp}{\hH^\textup{sm}}
\newcommand{\eHsamp}{\eH^\textup{sm}}
\newcommand{\dhtree}{\dot{h}^\textup{tr}}
\newcommand{\Htree}{H^\textup{tr}}

\newcommand{\msg}{\textup{\texttt{m}}}
\newcommand{\dmsg}{\dot{\msg}}
\newcommand{\hmsg}{\hat{\msg}}
\newcommand{\dMM}{\dot{\mathscr{M}}}
\newcommand{\hMM}{\hat{\mathscr{M}}}
\newcommand{\dotY}{\dot{Y}}
\newcommand{\hatY}{\hat{Y}}
\newcommand{\dotT}{\dot{T}}
\newcommand{\hatT}{\hat{T}}
\newcommand{\ST}{\textsf{S}}
\newcommand{\join}{\textup{\textsf{join}}}

\newcommand{\printcol}[1]{\textup{\texttt{#1}}}
\newcommand{\invalid}{\texttt{z}}
\newcommand{\red}{\printcol{r}}
	\newcommand{\redz}{\red_\zro}
	\newcommand{\redo}{\red_\one}
\newcommand{\blu}{\printcol{b}}
	\newcommand{\bluz}{\blu_\zro}
	\newcommand{\bluo}{\blu_\one}
\newcommand{\grn}{\printcol{g}}
\newcommand{\grnz}{\grn_\zro}
\newcommand{\grno}{\grn_\one}

\newcommand{\spc}{\printcol{s}}
\newcommand{\ylw}{\printcol{y}}
\newcommand{\ppl}{\printcol{p}}
\newcommand{\hatOm}{\printcol{a}}
\newcommand{\zro}{\textup{\texttt{0}}}
\newcommand{\one}{\textup{\texttt{1}}}
\newcommand{\fre}{\textup{\texttt{f}}}

\newcommand{\RMB}{\textup{\texttt{X}}}
\newcommand{\BLU}{\mathbb{B}}
\newcommand{\BLUeq}{\BLU_\texttt{=}}
\newcommand{\BLUne}{\BLU_{\texttt{\textdoublebarpipe}}}
\newcommand{\RED}{\mathbb{R}}
\newcommand{\REDeq}{\RED_\texttt{=}}
\newcommand{\REDne}{\RED_{\texttt{\textdoublebarpipe}}}

\newcommand{\lgf}{\fcl_{\ge 1}}

\newcommand{\dBP}{\dot{\textup{\texttt{BP}}}}
\newcommand{\hBP}{\hat{\textup{\texttt{BP}}}}
\newcommand{\vBP}{\textup{\texttt{BP}}}
\newcommand{\dBPu}{\dot{\textup{\texttt{NB}}}}
\newcommand{\hBPu}{\hat{\textup{\texttt{NB}}}}

\newcommand{\dz}{\dot{z}}
\newcommand{\hz}{\hat{z}}
\newcommand{\dZ}{\dot{Z}}

\newcommand{\Zcal}{\mathfrak{Z}}
\newcommand{\ZZ}{\bm{Z}}
\newcommand{\barZZ}{\bar{\bm{Z}}}

\newcommand{\ZH}{\mathfrak{z}}
\newcommand{\dbz}{\dot{\bm{z}}}
\newcommand{\hbz}{\hat{\bm{z}}}

\newcommand{\hI}{\hat{I}^\textup{lit}}
\newcommand{\dphi}{\dot{\varphi}}
\newcommand{\hphi}{\hat{\varphi}^\textup{lit}}
\newcommand{\ephi}{\bar{\varphi}}
\newcommand{\dPhi}{\dot{\Phi}}
\newcommand{\hPhi}{\hat{\Phi}}
\newcommand{\hPl}{\hat{\Phi}^\textup{lit}}

\newcommand{\hF}{\hat{F}}
\newcommand{\wt}{\bm{w}^\textup{lit}}
\newcommand{\avwt}{\bm{w}}
\newcommand{\maxwt}{\bm{W}}
\newcommand{\ePhi}{\bar{\Phi}}

\newcommand{\clust}{\bm{\gamma}}
\newcommand{\size}{\bm{s}}
\newcommand{\treesize}{\bm{s}^\textup{tr}}
\newcommand{\logp}{\bm{v}}
\newcommand{\SIGMA}{\bm{\Sigma}}
\newcommand{\XI}{\bm{\Xi}}
\newcommand{\treeSIGMA}{\bm{\Sigma}^\textup{tr}}
\newcommand{\bF}{\bm{F}}
\newcommand{\LAMBDA}{\bm{\Lambda}}
\newcommand{\optLAMBDA}{\bm{\Lambda}^\textup{op}}
\newcommand{\GAMMA}{\bm{\Gamma}}
\newcommand{\SIZE}{\textup{\textsf{size}}}
\newcommand{\SOL}{\textup{\textsf{SOL}}}
\newcommand{\CLUST}{\textup{\textsf{CL}}}

\newcommand{\marg}{\textup{\small\textsf{MARG}}}
\newcommand{\XX}{\ZZ_\textup{ns}}
\newcommand{\XXge}{\ZZ_\succcurlyeq}
\newcommand{\sepZZ}{\bm{S}}

\newcommand{\COUP}{\textup{\textsf{\footnotesize COUP}}}
\newcommand{\filt}{\mathscr{F}}
\newcommand{\mcur}{{\delta'U_{r,\circ}}}
\newcommand{\mcvr}{{\delta'V_{r,\circ}}}

\newcommand{\flip}{\mathfrak{f}}

\newcommand{\dq}{{\dot{q}}}
\newcommand{\hq}{\hat{q}}
\newcommand{\tq}{\tilde{q}}
\newcommand{\dtp}{\dot{p}}
\newcommand{\htp}{\hat{p}}
\newcommand{\utp}{\dot{p}^\textup{\textrm{u}}}

\newcommand{\dotm}{\dot{m}}
\newcommand{\hatm}{\hat{m}}
\newcommand{\drecH}{\hat{\drec}}
\newcommand{\drecD}{\dot{\drec}}
\newcommand{\drec}{\mathscr{R}}
\newcommand{\dmu}{\dot{\mu}}
\newcommand{\hmu}{\hat{\mu}}
\newcommand{\dpi}{\dot{\pi}}
\newcommand{\hpi}{\hat{\pi}}
\newcommand{\dtu}{\dot{u}}
\newcommand{\htu}{\hat{u}}

\title{The number of solutions for random regular NAE-SAT}

\author[A.~Sly]{Allan Sly$^*$}
\author[N.~Sun]{Nike Sun$^\dagger$}
\author[Y.~Zhang]{Yumeng Zhang}

\date{28 April 2016.
Research supported in part by
$^*$NSF DMS-1208338, DMS-1352013, Sloan Fellowship, and $^\dagger$NSF MSPRF}

\newcommand{\Elit}{\E^\textup{lit}}

\newcommand{\simplex}{\bm{\Delta}}
\newcommand{\spxsamp}{\bm{\Delta}^\textup{sm}}\newcommand{\spxtree}{\bm{\Pi}}

\newcommand{\SEP}{\textup{\textsf{sep}}}

\newcommand{\CUT}{\mathcal{W}}

\newcommand{\asat}{\alpha_\textup{sat}}
\newcommand{\aubd}{\alpha_\textup{ubd}}
\newcommand{\albd}{\alpha_\textup{lbd}}
\newcommand{\aclust}{\alpha_\textup{clust}}
\newcommand{\acond}{\alpha_\textup{cond}}
\newcommand{\ars}{\alpha_\textup{\textsc{rs}}}

\newcommand{\bemph}{\textbf}

\begin{document}

\maketitle

\vspace{-25pt}
\begin{center}
\textit{\footnotesize University of California, Berkeley}
\end{center}

\begin{abstract} Recent work has made substantial progress in understanding the transitions of random constraint satisfaction problems. In particular, for several of these models, the exact satisfiability threshold has been rigorously determined, confirming predictions of statistical physics. Here we revisit one of these models, random regular $k$-\textsc{nae-sat}: knowing the satisfiability threshold, it is natural to study, in the satisfiable regime, the number of solutions in a typical instance. We prove here that these solutions have a well-defined free energy (limiting exponential growth rate), with explicit value matching the one-step replica symmetry breaking prediction. The proof develops new techniques for analyzing a certain ``survey propagation model'' associated to this problem. We believe that these methods may be applicable in a wide class of related problems. \end{abstract}

\section{Introduction}\label{s:intro}

In a \bemph{random constraint satisfaction problem} (\textsc{csp}), we have $n$ variables taking values in a (finite) alphabet $\albet$, subject to a random set of constraints. In previous works on models of this kind, it has emerged that the space of solutions --- a random subset of $\albet^n$ --- can have a complicated structure, posing obstacles to mathematical analysis. Major advances in intuition were achieved by statistical physicists, who developed powerful analytic heuristics to shed light on the behavior of random \textsc{csp}s (\cite{MR2317690} and references therein). Their insights and methods are fundamental to the current understanding of random \textsc{csp}s.
	
One prominent application of the physics heuristic is in giving explicit (non-rigorous) predictions for the locations of satisfiability thresholds in a large class of random \textsc{csp}s (\cite{mertens2006threshold} and others). Recent works have given rigorous proofs for some of these thresholds: in the random regular \textsc{nae-sat} model \cite{dss-naesat-stoc,MR3440193}, in the random $k$-\textsc{sat} model \cite{MR3388183}, and in the independent set model on random regular graphs \cite{MR3689942}. However, the satisfiability threshold is only one aspect of the rich picture that physicists have developed. There are deep conjectures for the behavior of these models inside the satisfiable regime, and it remains an outstanding mathematical challenge to prove them. In this paper we address one part of this challenge, concerning the \bemph{total number of solutions} for a typical instance in the satisfiable regime.

\subsection{Main result}

Given a \textsc{cnf} boolean formula, a \bemph{not-all-equal-\textsc{sat}} (\bemph{\textsc{nae-sat}}) solution is an assignment $\vec{\bx}$ of literals to variables such that both $\vec{\bx}$ and its negation $\neg\vec{\bx}$ evaluate to \textsc{true} --- equivalently, such that no clause gives the same evaluation to all its variables. A $k$-\textsc{nae-sat} problem is one in which each clause has exactly $k$ literals; it is termed \bemph{$d$-regular} if each variable appears in exactly $d$ clauses. Sampling such a formula in a uniformly random manner gives rise to the \bemph{random $d$-regular $k$-\textsc{nae-sat} model}. (The formal definition is given in Section~\ref{s:comb}.) See \cite{MR2263010} for important early work on the closely related model of random (Erd\H{o}s--R\'enyi) \textsc{nae-sat}. The appeal of this model is that it has certain symmetries making the analysis particularly tractable, yet it is expected to share most of the interesting qualitative phenomena exhibited by other commonly studied problems, including random $k$-\textsc{sat} and random graph colorings.

Following convention, we fix $k$ and then parametrize the model by its clause-to-variable ratio, $\alpha\equiv d/k$. The \bemph{partition function}, denoted $Z\equiv Z_n$, is the number of valid \textsc{nae-sat} assignments for an instance on $n$ variables. It is conjectured that for each $k\ge3$, the model has an exact satisfiability threshold $\asat(k)$: for $\alpha<\asat$
it is satisfiable ($Z>0$) with high probability, 
but for $\alpha>\asat$
it is unsatisfiable ($Z=0$) with high probability.
This has been proved \cite[Thm.~1]{MR3440193} for all $k$ exceeding an absolute constant $k_0$, together with an exact formula for $\asat$ which matches the physics prediction. It can be approximated as
	\beq\label{e:alpha.sat.asymptotics}
	\asat 
	= \left(2^{k-1} -\f12
	-\f1{4\log 2}\right)\log 2 
	+ \epsilon_k\eeq
where $\epsilon_k$ denotes an error tending to zero as $k\to\infty$.

We say the model has \bemph{free energy} $\FE(\alpha)$ if $Z^{1/n}$ converges to $\FE(\alpha)$ in probability as $n\to\infty$.
\textit{A~priori}, the limit may not be well-defined. If it exists, however, Markov's inequality and Jensen's inequality imply that it must be upper bounded by the \bemph{replica symmetric free energy}
	\beq\label{e:annealed.free.energy}
	\annFE(\alpha) \equiv (\E Z)^{1/n}
	= 2 \bigg(1-\f{2}{2^k}\bigg)^\alpha\,.\eeq
(In this model and in other random regular models,  the replica symmetry free energy is the same as the annealed free energy.) One of the intriguing predictions 
from the physics analysis \cite{zdeborova2007phase,1742-5468-2008-04-P04004}
 is that there is a critical value $\acond$
strictly below $\asat$, such that $\FE(\alpha)$ and $\annFE(\alpha)$ agree up to $\alpha=\acond$ and diverge thereafter. In particular, this implies that 
the function $\FE(\alpha)$ must be non-analytic at $\alpha=\acond$. This is the \bemph{condensation transition} (or \emph{Kauzmann transition}), and will be further described below in \S\ref{ss:intro.statphys}. For all $0\le\alpha<\asat$, the free energy is predicted to be given by a formula
	\[\FE(\alpha)=\onersbFE(\alpha)
	\begin{cases}
	=\annFE(\alpha) & \textup{for $0\le\alpha\le\acond$,}\\
	<\annFE(\alpha) & \textup{for $\alpha>\acond$.}
	\end{cases}
	\]
The function $\onersbFE(\alpha)$ is quite explicit, although not extremely simple to state;
it is formally presented below in Definition~\ref{d:onersbFE}. The formula for $\onersbFE(\alpha)$ is derived via the \bemph{one-step replica symmetry breaking} (1\textsc{rsb}) heuristic, discussed further below. Our main result is to prove this prediction for large $k$:

\begin{mainthm}\label{t:main} In random regular $k$-\textsc{nae-sat} with $k\ge k_0$, for all $\alpha<\asat(k)$ the free energy $\FE(\alpha)$ exists and equals the predicted value $\onersbFE(\alpha)$.
\end{mainthm}

\begin{rmk}\label{r:rmks}
We allow for $k_0$ to be adjusted as long as it remains an absolute constant (so it need not equal the $k_0$ from \cite{MR3440193}). It is assumed throughout the paper that $k\ge k_0$, even when not explicitly stated.  The following considerations restrict the range of $\alpha=d/k$ that we must consider:
\begin{enumerate}[--]
\item A convenient upper bound on the satisfiable regime is given by 
	\[
	\asat \le \ars\equiv \f{\log2}{-\log(1-2/2^k)}
	< 2^{k-1}\log 2 \equiv \aubd\,.
	\]
This bound is certainly implied by the estimate \eqref{e:alpha.sat.asymptotics} from \cite{MR3440193}, but it follows much more easily and directly from the first moment calculation \eqref{e:annealed.free.energy}. Indeed, we see from
\eqref{e:annealed.free.energy} that the function $\annFE(\alpha)$ is decreasing in $\alpha$ and satisfies $\annFE(\ars)=1$, so $(\E Z)^{1/n}<1$ for all $\alpha>\ars$. Thus, by Markov's inequality, we have that $\P(Z>0) \le \E Z$ tends to zero as $n\to\infty$, i.e., the random problem instance is unsatisfiable with high probability.

\item For $\alpha>\asat$ we must have $\FE(\alpha)=0$. On the other hand, we can see by comparing \eqref{e:alpha.sat.asymptotics} and \eqref{e:annealed.free.energy} that $\asat$ is strictly smaller than $\ars$, and $\annFE(\asat)$ is strictly positive. This suggests  that $\acond$ occurs strictly before $\asat$, since $\acond=\asat$ would mean that $\FE(\alpha)=\annFE(\alpha)$ up to $\asat$, and in this case we would expect to have $\asat=\ars$. Formally, it requires further argument to confirm that $\acond<\asat$ in random regular \textsc{nae-sat}, and we obtain this as a consequence of results in the present paper. However, the phenomenon of $\acond<\asat$ 
was previously confirmed by \cite{MR3205212} and \cite{MR3566764} for random hypergraph bicoloring and random regular \textsc{sat}, both of which are very similar to random regular \textsc{nae-sat}. As for the value of $\FE(\alpha)$ at the threshold $\alpha=\asat$, we point out that $\alpha=\asat$ makes sense in the setting of this paper only if $d_\textup{sat}(k)\equiv k\asat(k)$ is integer-valued for some $k$. We have no reason to think that this ever occurs; however, if it does, then the probability for $Z>0$ is bounded away from both zero and one \cite[Thm.~1]{MR3440193}. In this case,  $Z^{1/n}$ does not concentrate around a single value but rather on two values, 
	\[\bigg\{0,
	\lim_{\alpha \uparrow \asat} \onersbFE(\alpha)
	\bigg\}\,.\]

\item In \cite[Propn.~1.1]{MR3440193} it is shown that for $0\le\alpha \le \albd\equiv (2^{k-1}-2)\log 2$ and  $n$ large enough,
	\[
	\f{\E(Z^2)}{(\E Z)^2} \le C\equiv C(k,\alpha)<\infty
	\]
where $Z\equiv Z_n$ and $C(k,\alpha)$ does not depend on $n$. 
Thus, for any fixed $0<\epsilon<1$ and $n$ large enough, 
	\[\P(Z\ge \epsilon \E Z)
	\stackrel{\odot}{\ge}
	\f{\E(Z \Ind{ Z\ge \epsilon \E Z})^2}{\E(Z^2)}
	\ge \f{(1-\epsilon)^2(\E Z)^2}{\E(Z^2)}
	\ge \f{(1-\epsilon)^2}{C} \equiv \delta\,.
	\]
where the step marked $\odot$ is by the Cauchy--Schwarz inequality. The results of \cite[Sec.~6]{MR3440193} imply the stronger statement that
for any $0\le\alpha\le\albd$,
	\[\adjustlimits
	\lim_{\epsilon\downarrow0}
	\liminf_{n\to\infty} \P(Z \ge \epsilon \E Z)=1\,.
	\]
On the other hand we already noted 
in \eqref{e:annealed.free.energy}
that $\E(Z^{1/n}) \le (\E Z)^{1/n}=\annFE(\alpha)$ for all $\alpha\ge0$ and  $n\ge1$. It follows by combining these facts that 
$Z^{1/n}$ converges in probability to $\annFE(\alpha)$ in probability for any $0\le\alpha\le\albd$. That is to say, the result of Theorem~\ref{t:main}
is already proved for $\alpha\le\albd$, with $\FE(\alpha)=\annFE(\alpha)$. This also implies that the condensation transition $\acond$ must occur above $\albd$.
\end{enumerate}
In summary, we have $\albd < \asat < \ars < \aubd$, 
and it remains to prove Theorem~\ref{t:main} for  $\alpha\in(\albd,\asat)$. Thus, we can assume for the remainder of the paper that 
	\beq\label{e:regime.alpha}
	(2^{k-1}-2)\log2 = \albd \le \alpha
	\le \aubd = 2^{k-1}\log2\,.\eeq 
In the course of proving Theorem~\ref{t:main} we will also identify the condensation threshold $\acond\in(\albd,\asat)$ (characterized in Proposition~\ref{p:onersb.fe.analysis} below).
\end{rmk}

The \textsc{1rsb} heuristic, along with its implications for the condensation and satisfiability thresholds, has been studied in numerous recent works, which we briefly survey here. The existence of a condensation transition was first shown in random hypergraph bicoloring \cite{MR3205212}, which as we mentioned above is a model very similar to random \textsc{nae-sat}. We also point out \cite{MR2961553} which is the first work to successfully analyze solution clusters within the condensation regime, leading to a very good lower bound on satisfiability threshold. This was an important precursor to subsequent works \cite{MR3440193,MR3388183,MR3689942}
on exact satisfiability thresholds
in random regular \textsc{nae-sat}, random \textsc{sat}, and independent sets.
Condensation has been demonstrated to occur even at positive temperature in hypergraph bicoloring  (which is very similar to \textsc{nae-sat}) \cite{MR3513593}. However, determining the precise location of $\acond$ is challenging, and was first achieved for the random graph coloring model \cite{MR3440196} by an impressive and technically challenging analysis. A related paper pinpoints $\acond$ for random regular $k$-\textsc{sat} (which again is very similar to \textsc{nae-sat}) \cite{MR3566764}. Subsequent work \cite{MR3818090} characterizes the condensation threshold in a more general family of models, and shows a correspondence with information-theoretic thresholds in statistical inference problems. The main contribution of this paper is to determine for the first time the free energy throughout the condensation regime $(\acond,\asat)$.

\subsection{Statistical physics predictions}
\label{ss:intro.statphys}

According to the heuristic analysis by statistical physics methods, the random regular \textsc{nae-sat} model has a single level of replica symmetry breaking (\textsc{1rsb}). We summarize here some of the key phenomena that are predicted from the \textsc{1rsb} framework \cite{zdeborova2007phase, MR2317690, 1742-5468-2008-04-P04004}, referring the reader to \cite[Ch.~19]{MR2518205} for a full expository account. While much of the following description remains conjectural, the implications at the free energy level are rigorously established by the present paper. Throughout the following we write $\doteq$ to indicate equality up to subexponential factors ($\exp\{o(n)\}$).

Recall that we consider \textsc{nae-sat} with $k$ fixed, parametrized by the clause density $\alpha\equiv d/k$. Abbreviate $\zro\equiv\textsc{true}$, $\one\equiv\textsc{false}$. For small $\alpha$, almost all of the solutions lie in a single well-connected subset of $\set{\zro,\one}^n$. This holds until a \bemph{clustering transition} (or \emph{dynamical transition}) $\aclust$, above which the solution space becomes broken up into many well-separated pieces, or \bemph{clusters} (see \cite{mezard2005clustering, achlioptas2006solution, achlioptas2008algorithmic,MR2663730}). Informally speaking, clusters are subsets of solutions which are characterized by the property that within-cluster distances are very small relative to between-cluster distances. Conjecturally, $\aclust$ also coincides with the \bemph{reconstruction threshold} \cite{gerschenfeld2007reconstruction, MR2317690, montanari2011reconstruction}, and is small relative to $\asat$ when $k$ is large, with $\aclust/\asat\asymp (\log k)/k$.

For $\alpha$ above $\aclust$ it is expected that the number of clusters of size $\exp\{ n s\}$ has mean value $\exp\{ n\Sigma(s;\alpha) \}$, and is \bemph{concentrated} about this mean.
The function $\Sigma(s)\equiv\Sigma(s;\alpha)$
is referred to as the ``cluster complexity.'' The \textsc{1rsb} framework of statistical physics
gives an \bemph{explicit} conjecture for $\Sigma$, discussed below in \S\ref{ss:intro.tiltedmeasure}. Then, summing over cluster sizes $0\le s\le\log 2$ gives that the total number $Z$ of \textsc{nae-sat} solutions has mean
 	\beq\label{e:intro.cluster.complexity}
	\E Z \doteq
	\sum_s \exp\{ n[s+\Sigma(s)] \}
	\doteq \exp\{ n[s_1+\Sigma(s_1)] \},\eeq
where $s_1=\argmax[s+\Sigma(s)]$. It is expected that $\Sigma$ is continuous and strictly concave in $s$, and also that $s+\Sigma(s)$ has a unique maximizer $s_1$ with $\Sigma'(s_1)=-1$. 
For \textsc{nae-sat} and related models, this explicit calculation
reveals a critical value
$\acond
\in(\aclust,\asat)$,
characterized as
	\[\acond
	=\sup\set{\alpha \ge \aclust
		: \Sigma(s_1(\alpha);\alpha)\ge0 }\,.\]
By contrast, the satisfiability threshold can be characterized as 
	\[\asat
	=\sup\set{\alpha\ge\aclust
		: \max_s\Sigma(s;\alpha) \ge0}\,.\]
For all $\alpha\ge\aclust$, the expected partition function $\E Z$ is dominated by clusters of size $\exp\{n s_1\}$. However, for $\alpha>\acond$, we have $\Sigma(s_1)<0$, so the expected number of clusters of this size is very small: $\exp\{n\Sigma(s_1)\}$ tends to zero exponentially fast as $n\to\infty$. This means that clusters of size $\exp\{ns_1\}$  are highly unlikely to appear in a typical realization of the model. Instead, in a typical realization we only expect to see clusters of size $\exp\{ns\}$ with $\Sigma(s)\ge0$. As a result the solution space should be dominated (with high probability) by clusters of size $s_{\max}$ where
	\[s_{\max}
	\equiv s_{\max}(\alpha)
	\equiv \argmax\set{
		s+\Sigma(s)
		: \Sigma(s)\ge0
		}\,.\]
Since $\Sigma$ is continuous, $s_{\max}$ is the largest root of $\Sigma$, and for $\alpha\in(\acond, \asat)$ we should have
	\[Z\doteq \exp\{n s_{\max}\}\ll \E Z
	= \exp\{n[s_1+\Sigma(s_1)]\}
	\]
(where the approximation for $Z$ holds with high probability).  The \bemph{\textsc{1rsb} free energy}, formally given by Definition~\ref{d:onersbFE} below,
should be interpreted as an expression for
the function $\onersbFE(\alpha)=s_{\max}(\alpha)$.

\subsection{The tilted cluster partition function}
\label{ss:intro.tiltedmeasure}

From the discussion of \S\ref{ss:intro.statphys} we see that once the function $\Sigma(s;\alpha)$ is determined, it is possible to derive $\acond$, $\asat$, and $\FE(\alpha)$. However, previous works have not taken the approach of actually computing $\Sigma$. Indeed, $\asat$ was determined \cite{MR3440193} by an analysis involving only $\max_s\Sigma(s;\alpha)$, which contains less information than the full curve $\Sigma$. Related work on the exact determination (in a range of models) of $\acond$ \cite{MR3440196, MR3566764,MR3818090} also avoids computing $\Sigma$, reasoning instead via the so-called ``planted model.''

In order to compute the free energy, however, we cannot avoid computing (some version of) the function $\Sigma$, which we will do by a physics-inspired approach. First consider the $\lambda$-tilted partition function
	\beq\label{e:intro.Z.lambda}
	\barZZ_\lambda \equiv
	\sum_{\clust\in\CLUST(\glit)}
	|\clust|^\lambda\,,\eeq
where $\CLUST(\glit)$ denotes the set of solution clusters of $\glit$, and $|\clust|$ denotes the number of satisfying assignments inside the cluster $\clust$. According to the conjectural picture described above, we should have 
	\[
	\E\barZZ_\lambda
	\doteq
	\sum_s (\exp\{ns\})^\lambda
	\exp\{n\Sigma(s)\}
	\doteq \exp\{n\mathfrak{F}(\lambda)\}
	\]
where $\mathfrak{F}$ is the Legendre dual of $-\Sigma$:
	\beq\label{e:intro.legendre.dual}
	\mathfrak{F}(\lambda)
	\equiv (-\Sigma)^\star(\lambda)
	\equiv \max_s \bigg\{
		\lambda s + \Sigma(s)
		\bigg\}
	= \lambda s_\lambda + \Sigma(s_\lambda)\,,\eeq
where $s_\lambda\equiv\argmax_s[\lambda s + \Sigma(s)]$. Moreover, if $\Sigma(s_\lambda)\ge0$, then $\ZZ_\lambda$ should  concentrate near $\E\ZZ_\lambda$, and in this regime physicists have an exact prediction for  $\mathfrak{F}(\lambda)$, which will be further discussed below in \S\ref{ss:intro.onersb}. In short, the physics approach to computing $\Sigma$ is to first compute $\mathfrak{F}(\lambda)$ (in the regime where $\Sigma(s_\lambda)\ge0$), and then set $\Sigma = -\mathfrak{F}^\star$. Note that by differentiating $\mathfrak{F}(\lambda) = n^{-1}\log \E\barZZ_\lambda$ we find that $\mathfrak{F}$ is convex in $\lambda$, so the resulting $\Sigma$ will indeed be concave.

At first glance the reduction to computing $\mathfrak{F}(\lambda)$ may not seem to improve matters. It is not immediately clear how ``clusters'' should be defined. It turns out that in the regime we are studying, a reasonable definition is that two \textsc{nae-sat} solutions are connected if they differ by a single bit, and the \bemph{clusters} are the connected components of the solution space. A typical \textsc{nae-sat} solution will have a positive density of variables which are \bemph{free}, meaning their value can be changed without violating any clause; any such solution must belong in a cluster of exponential size. Each cluster may be a complicated subset of $\set{\zro,\one}^n$ --- changing the value at one free variable may affect whether its neighbors are free, so a cluster need not be a simple subcube of $\set{\zro,\one}^n$. Nevertheless, we wish to sum over the cluster sizes raised to non-integer powers.  This computation is made tractable by constructing more explicit combinatorial models for the \textsc{nae-sat} solution clusters, as we next describe.

\subsection{Modeling solution clusters}
\label{ss:intro.clustermodel}

In our regime of interest (i.e., $k\ge k_0$ and $\albd\le\alpha\le\aubd$; see Remark~\ref{r:rmks}), the analysis of \textsc{nae-sat} solution clusters is greatly simplified by the fact that in a typical satisfying assignment the vast majority of variables are \bemph{frozen} rather than free. The result of this, roughly speaking, is that a cluster $\clust\in\CLUST(\glit)$ can be encoded by a configuration $\ux\in\set{\zro,\one,\fre}^n$ (representing its circumscribed subcube, so $x_v=\fre$ indicates a free variable) with no essential loss of information. This is formalized by a combinatorial model of ``frozen configurations'' representing the clusters (Definition~\ref{d:frozen}). These frozen configurations can be viewed as the solutions of a certain \textsc{csp} lifted from the original \textsc{nae-sat} problem --- so the physics heuristics can be applied again to (the randomized version of) the lifted model. Variations on this idea appear in several places in the physics literature; in the specific context of random \textsc{csp}s we refer to \cite{parisi2002local, MR2155706, MR2351840}. Analyzing the number of frozen configurations, corresponding to \eqref{e:intro.Z.lambda} with $\lambda=0$, yields the satisfiability threshold for this model \cite{MR3440193}. 

Analyzing \eqref{e:intro.Z.lambda} for general $\lambda$ requires deeper investigation of the arrangement of free variables in a typical frozen configuration $\ux$. For this purpose it is convenient to view an \textsc{nae-sat} instance as a (bipartite) graph $\glit$, where the vertices are given by variables and clauses, and the edges indicate which variables participate in which clauses (the formal description appears in Section~\ref{s:comb}). A key piece of intuition is that if we consider the subgraph of $\GG$ induced by the free variables, together with the clauses through which they interact, then this subgraph is predominantly comprised of disjoint components $\bm{T}$ of bounded size. (In fact, the majority of free variables are simply isolated vertices; a smaller fraction occur in linked pairs; a yet smaller fraction occur in components of size three or more.) Each free component $\bm{T}$ is surrounded by frozen variables, and we let $z(\bm{T})$ be the number of \textsc{nae-sat} assignments on $\bm{T}$ which are consistent with the frozen boundary. Since disjoint components $\bm{T},\bm{T}'$ do not interact, the size of the cluster represented by $\ux$ is simply the product of $z(\bm{T})$
over all $\bm{T}$.

Another key observation is that the random \textsc{nae-sat} graph has few short cycles, so almost all of the free components will be \bemph{trees}. As a result, their weights $z(\bm{T})$ can be evaluated recursively by \bemph{belief propagation} (\textsc{bp}), a well-known dynamic programming method
(see e.g.\ \cite[Ch.~14]{MR2518205}). In the \textsc{rsb} heuristic framework, a cluster is represented by a vector $\vec{\msg}$ of ``messages,'' indexed by the directed edges of the \textsc{nae-sat} graph $\glit$. Informally, for a given cluster, and for any variable $v$ adjacent to any clause $a$,
	{\setlength{\jot}{0pt}\begin{align}\nonumber
	&\textup{$\msg_{v\to a}$ represents the 
	``within-cluster 
	law of $\bx_v$ in absence of $a$'';}\\
	&\textup{$\msg_{a\to v}$ represents the 
	``within-cluster 
	law of $\bx_v$ in absence of
	$\partial v\setminus a$'',}
	\label{e:interpret.messages}
	\end{align}
where $\partial v$ denotes the neighboring clauses of $v$. Each message is a probability measure on $\set{\zro,\one}$, and the messages are related to one another via local consistency equations, which are known as the \textsc{bp} equations. A configuration $\vec{\msg}$ which satisfies all the local consistency equations is a \bemph{\textsc{bp} solution}. Thus a cluster $\clust$ can be encoded either by a frozen configuration $\ux$ or by a \textsc{bp} solution $\vec{\msg}$; the latter has the key advantage that \bemph{the size of $\clust$ can be easily read off from $\vec{\msg}$, as a certain product of local functions.} For the cluster size raised to power $\lambda$, simply raise each local function to power $\lambda$. Thus the configurations $\vec{\msg}$ with $\lambda$-tilted weights form a \bemph{spin system} (Markov random field), whose partition function is the quantity of interest \eqref{e:intro.Z.lambda}. This new spin system is sometimes termed the ``auxiliary model'' (e.g.\ \cite[Ch.~19]{MR2518205}).

\subsection{One-step replica symmetry breaking}\label{ss:intro.onersb}

In \S\ref{ss:intro.clustermodel} we described informally how a solution cluster $\clust$ can be encoded by a frozen configuration $\ux$, or a \textsc{bp} solution $\vec{\msg}$. An important caveat is that the converse need not hold. In the \textsc{nae-sat} model, for any value of $\alpha$, a trivial \textsc{bp} solution is always given by the ``replica symmetric fixed point'' (also called the ``factorized fixed point''), where every $\msg_{v\to a}$ is the uniform measure on $\set{\zro,\one}$. However, above $\alpha_\text{cond}$, this is a spurious solution.  One way to see this is via the heuristic ``cavity calculation'' of $\annFE(\alpha)$, which we now describe to motivate the more complicated expression for $\onersbFE(\alpha)$.

Given a random regular \textsc{nae-sat} instance $\glit$ on $n$ variables, choose $k$ uniformly random variables $v_1,\ldots,v_k$, and assume for simplicity that no two of these share a clause. Remove the $k$ variables along with their $kd$ incident clauses, producing the ``cavity graph'' $\glit''$. Then add $d(k-1)$ new clauses to $\glit''$, producing the graph $\glit'$. Under this operation (cf.\ \cite{ass2003}), $\glit'$ is distributed as a random regular \textsc{nae-sat} instance on $n-k$ variables. If the free energy $\FE(\alpha)=\lim_{n\to\infty} Z^{1/n}$ exists, then we would expect it to agree asymptotically with
	\beq\label{e:rs.cavity}\bigg(\f{Z(\glit)}{Z(\glit')}\bigg)^{1/k}
	=\bigg(\f{Z(\glit)}{Z(\glit'')}\bigg)^{1/k}
	\bigg/
	\bigg(\f{Z(\glit')}{Z(\glit'')}\bigg)^{1/k}\,.\eeq
Let $U$ denote the set of ``cavity neighbors'': the variables in $\glit''$ of degree $d-1$, which neighbored the clauses that were deleted from $\glit$. Then $\glit$ and $\glit'$ differ from $\glit''$ only in the addition of a few small subgraphs which are attached to $U$. Computing $Z(\glit)/Z(\glit'')$ or $Z(\glit')/Z(\glit'')$ reduces to understanding the joint law of the spins $(\bx_u)_{u\in U}$ under the \textsc{nae-sat} model defined by $\glit''$. Since $\glit$ is unlikely to have many cycles, the vertices of $U$ are typically far apart from one another in $\glit''$. Therefore, one plausible scenario is that their spins are approximately independent under the \textsc{nae-sat} model on $\glit''$, with $\bx_u$ marginally distributed according to $\msg_{u\to a}$ where $a$ is the deleted clause that neighbored $u$ in $\glit$. If this is the case, then each $\msg_{u\to a}$ must be uniform over $\set{\zro,\one}$, by the negation symmetry of \textsc{nae-sat}. Under this assumption, we can calculate
	\beq\label{e:rs.cavity.ratios}
	\bigg(\f{Z(\glit)}{Z(\glit'')}\bigg)^{1/k}
	= 2(1-2/2^k)^d,\quad
	\bigg(\f{Z(\glit')}{Z(\glit'')}\bigg)^{1/k}
	= (1-2/2^k)^{\alpha(k-1)},\eeq
Substituting into \eqref{e:rs.cavity} gives the replica symmetric free energy prediction $\FE(\alpha)\doteq \annFE(\alpha)$, which we know to be false for large $\alpha$ (in particular, it is inconsistent with the known satisfiability threshold). Thus the replica symmetric fixed point, $\msg_{u\to a}= \textup{unif}(\set{0,1})$ for all $u\to a$, is a spurious \textsc{bp} solution. In reality the $\bx_u$ are \bemph{not} approximately independent in $\glit''$, even though the $u$'s are far apart. This phenomenon of \bemph{non-negligible long-range dependence} may be taken as a definition of replica symmetry breaking (\textsc{rsb}) in this setting, and occurs precisely for $\alpha$ larger than $\acond$.

Since above $\acond$ the partition function cannot be estimated by 
\eqref{e:rs.cavity.ratios}  due to replica symmetry breaking, a different approach is needed. To this end, 
the \bemph{one-step \textsc{rsb}} (1\textsc{rsb}) heuristic posits that even when the original \textsc{nae-sat} model exhibits \textsc{rsb}, the (seemingly more complicated) ``auxiliary model'' of $\lambda$-weighted \textsc{bp} solutions $\vec{\msg}$ \bemph{is replica symmetric, for $\lambda$ small enough}: conjecturally, as long as $\Sigma(s_\lambda)\ge0$ for $s_\lambda\equiv\argmax_s \set{ \lambda s+\Sigma(s)}$ (cf.\ the discussion below \eqref{e:intro.legendre.dual}).  That is, for such $\lambda$, the auxiliary model is predicted to have correlation decay, in contrast with the long-range correlations of the original model. This would mean that in the auxiliary model of the cavity graph $\glit''$, the spins $(\msg_{u\to a})_{u\in U}$ are approximately independent, with each $\msg_{u\to a}$ marginally distributed according to some law $\dq_{u\to a}$. The model has a replica symmetric fixed point, $\dq_{u\to a}=\dq_\lambda$ for all $u\to a$ (the analogue of $\msg_{u\to a}=\textup{unif}(\set{\zro,\one})$ for all $u\to a$). If we substitute this assumption into the cavity calculation (the analogues of \eqref{e:rs.cavity} and \eqref{e:rs.cavity.ratios}), we obtain the replica symmetric prediction for the auxiliary model free energy $\mathfrak{F}(\lambda)$, expressed as a function of $\dq_\lambda$. As explained above, from $\mathfrak{F}(\lambda)$ we can derive the complexity function $\Sigma(s)$ and the \textsc{1rsb} \textsc{nae-sat} free energy $\onersbFE(\alpha)$.

\subsection{The 1RSB free energy prediction}

Having described the heuristic reasoning, we now proceed to formally state the \textsc{1rsb} free energy prediction. We first describe $\dq_\lambda$ as a certain discrete probability measure over $\msg$. Since $\msg$ is a probability measure over $\set{\zro,\one}$, we encode it by $x\equiv\msg(\one)\in[0,1]$. A measure $q$ on $\msg$ can thus be encoded by an element $\mu\in\mathscr{P}$ where $\mathscr{P}$ denotes the set of discrete
probability measures on $[0,1]$.
For measurable $B\subseteq[0,1]$, define
	\begin{align}\nonumber
	\drecH_\lambda\mu(B)
	&\equiv 
	\f1{\hat{\mathscr{Z}}(\mu)}
		\int \bigg(2-\prod_{i=1}^{k-1}x_{i}-\prod_{i=1}^{k-1}(1-x_{i})\bigg)^{\lambda}
	\mathbf{1}\bigg\{
			\f{1-\prod_{i=1}^{k-1}x_{i}}
			{2- \prod_{i=1}^{k-1}x_{i} - \prod_{i=1}^{k-1}(1-x_{i})}
		\in B\bigg\}
		\,
		\prod_{i=1}^{k-1}{\mu}(dx_{i})\,,\\
	\drecD_\lambda\mu(B)
	&\equiv
	\f1{\dot{\mathscr{Z}}(\mu)}
		\int\bigg(\prod_{i=1}^{d-1}y_{i}+\prod_{i=1}^{d-1}(1-y_{i})\bigg)^{\lambda}
		\mathbf{1}\bigg\{
			\f{\prod_{i=1}^{d-1}y_{i}}
			{\prod_{i=1}^{d-1}y_{i}+\prod_{i=1}^{d-1}(1-y_{i})}
		\in B
		\bigg\}
		\,
		\prod_{i=1}^{d-1}\mu(dy_{i})\,,
	\label{e:dist.recur}
	\end{align}
where $\hat{\mathscr{Z}}(\mu)$ and $\dot{\mathscr{Z}}(\mu)$ are the normalizing constants such that
$\drecH_\lambda\mu$ and $\drecD_\lambda\mu$
are also probability measures on $[0,1]$.
(In the context of $\lambda=0$ we take the convention that $0^0=0$.)
Denote $\drec_{\lambda} \equiv \drecD_{\lambda}\circ \drecH_{\lambda}$.
The map $\drec_{\lambda}:\mathscr{P}\to\mathscr{P}$ represents the \textsc{bp} recursion for the auxiliary model. The following presents a solution for $\alpha$ in the interval
$(\albd,\aubd)$
which we recall (Remark~\ref{r:rmks}) is a superset of $(\acond,\asat)$.

\begin{ppn}[proved in Appendix~\ref{appx:onersb}]
\label{p:drec_fixpoint}
For $\lambda \in [0,1]$, let $\dmu_{\lambda,0}\equiv \f{1}{2} \delta_0 + \f{1}{2} \delta_1\in\mathscr{P}$, and define recursively $\dmu_{\lambda,l+1} = \drec_\lambda \dmu_{\lambda,l}\in\mathscr{P}$ for all $l\ge0$. Define
	$S_l \equiv 
	(\supp\dmu_{\lambda,l})
	\setminus
	(\supp(
	\dmu_{\lambda,0}
	+\ldots+\dmu_{\lambda,l-1}
	))$; this is a finite subset of $[0,1]$.
Regard $\dmu_{\lambda,l}$ as an infinite sequence indexed by the elements of $S_1$ in increasing order, followed by the elements of $S_2$ in increasing order, and so on.
For $k\ge k_0$ and $\albd \le \alpha \le \aubd$, 
in the limit $l\to\infty$,
$\dmu_{\lambda,l}$ converges in the $\ell^1$ sequence space to a limit
$\dmu_\lambda\in\mathscr{P}$ satisfying the fixed point equation $\dmu_\lambda = \drec_\lambda \dmu_\lambda$,
as well as the estimates
	$\dmu_\lambda((0,1))\le 
	7/2^k$ and
$\dmu_\lambda(dx) = \dmu_\lambda(d(1-x))$.
\end{ppn}

The limit $\dmu_\lambda$ of Proposition~\ref{p:drec_fixpoint} encodes the desired replica symmetric solution $\dq_\lambda$
for the auxiliary model. We can then express 
$\mathfrak{F}(\lambda)$ in terms of $\dmu_\lambda$ as follows. Writing $\hmu_\lambda \equiv  \hat{\mathscr{R}}_\lambda\dmu_\lambda$, let $\dot{w}_\lambda,\hat{w}_\lambda,\bar{w}_\lambda\in \mathscr{P}$ be defined by
	\begin{align}\nonumber
	\dot{w}_{\lambda}(B)	
	&= ({\dot{\Zcal}}_{\lambda})^{-1}
		\int\bigg(\prod_{i=1}^{d}y_{i}+\prod_{i=1}^{d}(1-y_{i})\bigg)^{\lambda}
		 \mathbf{1}\bigg\{
		 \prod_{i=1}^{d}y_{i}
		 +\prod_{i=1}^{d}(1-y_{i})\in B
		 \bigg\}
		 \prod_{i=1}^{d}\hmu_{\lambda}(dy_{i})\,,\\
		 \nonumber
	\hat{w}_{\lambda}(B)	
		&=
		({\hat{\Zcal}}_{\lambda})^{-1}
		\int\bigg(1-\prod_{i=1}^{k}x_{i}-\prod_{i=1}^{k}(1-x_{i})\bigg)^{\lambda}
			\mathbf{1}\bigg\{
			1-\prod_{i=1}^{k}x_{i}-\prod_{i=1}^{k}(1-x_{i})\in B
			\bigg\}
			\prod_{i=1}^{k}\dmu_\lambda(dx_{i})\,,\\
	\bar{w}_{\lambda}(B)	
		&= ({\bar{\Zcal}}_{\lambda})^{-1}
			\iint\bigg(xy+(1-x)(1-y)\bigg)^{\lambda}
			\mathbf{1}\Big\{
			xy+(1-x)(1-y)\in B
			\Big\}
			\dmu_\lambda(dx)\hmu_{\lambda}(dy)
			\,,
		\label{e:drec_marginal}
	\end{align}
with $\dot{\Zcal}_{\lambda},\hat{\Zcal}_{\lambda},\bar{\Zcal}_{\lambda}$ the normalizing constants. The analogue of \eqref{e:rs.cavity.ratios}
for this model is
	\[\bigg(\f{\barZZ_\lambda(\glit)}
	{\barZZ_\lambda(\glit'')}\bigg)^{1/k}
	=\dot{\Zcal}_\lambda
	(\hat{\Zcal}_\lambda
	/\bar{\Zcal}_\lambda)^d,\quad
	\bigg(\f{\barZZ_\lambda(\glit')}
	{\barZZ_\lambda(\glit'')}\bigg)^{1/k}
	= (\hat{\Zcal}_\lambda)^{\alpha(k-1)},\]
and substituting into \eqref{e:rs.cavity} gives
the \textsc{1rsb} prediction
$\barZZ_\lambda \doteq 
\exp\{\mathfrak{F}(\lambda)\}$ where
	\beq\label{e:drec_Sigma}
	\mathfrak{F}(\lambda)\equiv
	\mathfrak{F}(\lambda;\alpha)
	\equiv \log \dot{\Zcal}_{\lambda} 
	+ \alpha \log \hat{\Zcal}_{\lambda} 
	- k\alpha \log \bar{\Zcal}_{\lambda}\,.\eeq
Further, the maximizer of 
$s\mapsto(\lambda s+\Sigma(s))$
 is predicted to be given by
	\beq\label{e:drec_size_lambda}
	s_\lambda
	\equiv s_\lambda(\alpha)\equiv
	\int \log(x) \dot{w}_\lambda(dx)
			+ \alpha \int
			\log(x) \hat{w}_\lambda(dx)
	- k\alpha \int \log(x) \bar{w}_\lambda(dx)\,.\eeq
If $s=s_\lambda$ for $\lambda\in[0,1]$ then we define
	\beq\label{e:formal.definition.Sigma}\Sigma(s)
	\equiv
	\Sigma(s;\alpha) 
	\equiv \mathfrak{F}(\lambda;\alpha)
	-\lambda s_\lambda(\alpha)\,.\eeq
We then use \eqref{e:formal.definition.Sigma} to define the thresholds 
	{\setlength{\jot}{0pt}\begin{align*}
	\acond &\equiv
	\sup\set{\alpha : \Sigma(s_1;\alpha)>0}\,,\\
	\asat &\equiv
	\sup\set{\alpha : \Sigma(s_0;\alpha)>0}\,.
	\end{align*}}%
We can now formally state the
predicted free energy of the original
\textsc{nae-sat} model:
\begin{dfn}\label{d:onersbFE}
For $\alpha\in k^{-1}\mathbb{Z}$,
\textsc{1rsb} free energy prediction $\onersbFE(\alpha)$ is defined as
	\beq\label{e:onersbFE}
	\onersbFE(\alpha) = \begin{cases}
			\annFE(\alpha)
			=2(1-2/2^k)^\alpha & \textup{for $\alpha\le\acond$,}\\
			\exp[\sup \set{ s 
			: \Sigma(s) \ge 0}]
			& \textup{for $\acond \le \alpha < \asat$}\\
			0 &\textup{for $\alpha > \asat$.}
		\end{cases}\eeq
(In regular $k$-\textsc{nae-sat} we must have integer $d=k\alpha$, so we need not consider $\alpha\notin k^{-1}\mathbb{Z}$.)
\end{dfn}

\begin{ppn}[proved in Appendix~\ref{appx:onersb}] \label{p:onersb.fe.analysis} Assume $k\ge k_0$ and write $A\equiv [\albd, \aubd] \cap (k^{-1}\mathbb{Z})$.
\begin{enumerate}[a.]
\item\label{p:strict.decrease} For each $\alpha\in A$, the function $s\mapsto\Sigma(s;\alpha)$ is well-defined, continuous, and strictly decreasing in $s$.
\item\label{p:large.k} For each $0\le\lambda\le1$, the function $\alpha\mapsto \Sigma(s_\lambda;\alpha )= \mathfrak{F}(\lambda) -\lambda s_\lambda$ is strictly decreasing with respect to $\alpha \in A$. There is a unique $\alpha_\lambda\in A$ such that $\Sigma(s_\lambda;\alpha)$ is nonnegative for all $\alpha\le\alpha_\lambda$, and is negative for all $\alpha>\alpha_\lambda$. Taking $\lambda=0$ we recover the estimate \eqref{e:alpha.sat.asymptotics}; and taking $\lambda=1$ we obtain in addition
	\beq\label{e:alpha.cond.asymptotics}
	\acond=\alpha_1
	= (2^{k-1}-1)\log2+ \epsilon_k\,.\eeq
\end{enumerate}
(The main purpose of this proposition is to show that $\Sigma(s_1)<0$ for all $\alpha\in(\acond,\asat)$, i.e., that the ``condensation regime'' is a contiguous range of values of $\alpha$. The expansion of $\alpha_{\textup{cond}}$ matches an earlier result of \cite{MR3205212}, which was obtained for a slightly different but closely related model.)\end{ppn}

\subsection{Proof approach}\label{ss:intro.proof}

Since $\FE=\FE(\alpha)$ is \textit{a~priori} not well-defined, the statement $\FE\le\textsf{g}$ means formally that for all $\epsilon>0$, $\P( Z^{1/n} \ge \textsf{g}+\epsilon)$ tends to zero as $n\to\infty$. With this notation, we will prove separately the upper bound $\FE(\alpha)\le\onersbFE(\alpha)$ and the matching lower bound $\FE(\alpha)\ge\onersbFE(\alpha)$. This implies the main result Theorem~\ref{t:main}: the free energy $\FE(\alpha)$ is indeed well-defined, and equals $\onersbFE(\alpha)$.

The upper bound is proved by an interpolation argument, which we defer to Appendix~\ref{appx:ubd}. This argument builds on similar bounds for spin glasses on Erd\H{o}s--R\'enyi graphs \cite{MR1972121,MR2095932}, together with ideas from \cite{MR3161470,MR3256814} for interpolation in random regular models. Let $Z_n(\beta)$ denote the partition function of \textsc{nae-sat} at inverse temperature $\beta>0$. The interpolation method yields an upper bound on $\E\log Z_n(\beta)$ which is expressed as the infimum of a certain function $\mathcal{P}(\mu;\beta)$, with $\mu$ ranging over probability measures on $[0,1]$. We then choose $\mu$ according to Proposition~\ref{p:drec_fixpoint}, and take $\beta\to\infty$ to obtain the desired bound $\FE(\alpha)\le\onersbFE(\alpha)$.

Most of the paper is devoted to establishing the
matching lower bound. The proof strategy is inspired by the physics picture described above, and at a high level proceeds as follows. Take any $\lambda$ such that $\Sigma(s_\lambda)$ (as defined by \eqref{e:drec_size_lambda} and \eqref{e:formal.definition.Sigma}) is nonnegative, and let $\bm{Y}_\lambda$ be the number of clusters of size roughly $\exp\{ns_\lambda\}$. (As discussed in \S\ref{ss:intro.tiltedmeasure}, we shall think of a cluster as a connected component of the solution space.) The informal statement of what we show is that
	\beq\label{e:intro.lbd}
	\bm{Y}_\lambda \doteq \exp\{ n[ \lambda s_\lambda + \Sigma(s_\lambda) ] \}\,.\eeq
Adjusting $\lambda$ 
as indicated by \eqref{e:onersbFE}
then proves the desired bound
$\FE(\alpha)\ge\onersbFE(\alpha)$.

Proving a formalized version of \eqref{e:intro.lbd} occupies a significant part of the present paper. We introduce a slightly modified version of the messages $\msg$ which record the topologies of the free trees $\bm{T}$. We then restrict to cluster encodings in which every free tree has fewer than $T$ variables, which limits the distance that information can propagate between free variables. We prove a version of \eqref{e:intro.lbd} for every fixed $T$, and show that this yields the sharp lower bound in the limit $T\to\infty$. The proof of \eqref{e:intro.lbd} for fixed $T$ is via the moment method for the auxiliary model, which boils down to a complicated optimization problem over many dimensions. It is known (see e.g.\ \cite[Lem.~3.6]{MR3440193}) that stationary points of the optimization problem correspond to ``generalized'' \textsc{bp} fixed points --- these are measures $Q_{v\to a}(\msg_{v\to a},\msg_{a\to v})$, rather than the simpler ``one-sided'' measures $q_{v\to a}(\msg_{v\to a})$ considered in the \textsc{1rsb} heuristic.

The one-sided property is a crucial simplification in physics calculations (cf.\ \cite[Proposition~19.4]{MR2518205}), but is challenging to prove in general. One contribution of this work that we wish to highlight is a novel resampling argument
which yields a reduction to one-sided messages, and allows us to solve the moment optimization problem.
(We are helped here by the truncation on the sizes of free trees.) Furthermore, the approach allows us to bring in methods from large deviations theory.
With these we can
show that the objective function has negative-definite Hessian at the optimizer,
which is necessary for the second moment method.
This resampling approach is quite general and should apply in a broad range of models.

\subsection{Open problems}

Beyond the free energy, it remains a challenge to establish the full picture predicted by statistical physicists for $\alpha\le \asat$. We refer the reader to several recent works targeted at a broad class of models in the regime $\alpha\le \acond$ \cite{MR3570985, MR3692136, MR3813215},  and to work on the location on $\acond$ in a general family of models \cite{condensation2017stoc}. In the condensation regime $(\acond,\asat)$, an initial step would be to show that most solutions lie within a bounded number of clusters. A much more refined prediction is that the mass distribution among the largest clusters forms a Poisson--Dirichlet process. Another question is to show that on a typical problem instance over $n$ variables, if $\vec{\bx}^1,\vec{\bx}^2$ are sampled independently and uniformly at random from the solutions of that instance, then the normalized overlap $R_{1,2}\equiv n^{-1}\set{v:\bx^1_v=\bx^2_v}$ concentrates on two values (corresponding roughly to the two cases that $\vec{\bx}^1,\vec{\bx}^2$ come from the same cluster, or from different clusters) --- this criterion is sometimes taken as the precise definition of \textsc{1rsb}. During the final revision stage of this paper, some of the above questions were addressed by a new preprint \cite{onersb}.

Beyond the immediate context of random \textsc{csp}s, understanding the condensation transition may deepen our understanding of the stochastic block model, a model for random networks with underlying community structure. Here again ideas from statistical physics have played an important role~\cite{decelle2011asymptotic}. A great deal is now known rigorously for the case of two blocks \cite{MR3238997, MR3383334}, where there is no condensation regime. For models with more than two blocks, however, it is predicted that the condensation can occur, and may define a regime where detection is information-theoretically possible but computationally intractable. A condensation threshold has been established for the anti-ferromagnetic Potts model, corresponding to the disassortative regime of the stochastic block model.
An analogous transition is expected in the ferromagnetic (assortative) case, and this remains open.

\subsection*{Acknowledgements} We are grateful to Amir Dembo, Jian Ding, Andrea Montanari, and Lenka Zdeborov\'a for helpful conversations. We thank the anonymous referee and Youngtak Sohn for pointing out errors and giving many helpful comments on drafts of the paper. We also gratefully acknowledge the hospitality of the Simons Institute at Berkeley, where part of this work was completed during a spring 2016 semester program.

\subsection*{Data availability statement} Data sharing is not applicable to this article as no datasets were generated or analyzed during this study.

\section{Combinatorial model}\label{s:comb}

In this section we formalize a combinatorial model of clusters, which allows us to rigorously lower bound the tilted cluster partition function \eqref{e:intro.Z.lambda}. We begin by reviewing the (standard) graphical depiction of \textsc{nae-sat}. A \bemph{not-all-equal-\textsc{sat}} (\bemph{\textsc{nae-sat}}) problem instance is naturally represented by a \bemph{bipartite factor graph} $\glit$ with signed edges, as follows. The vertex set of $\glit$ is partitioned into a set $V=\set{v_1,\ldots,v_n}$ of variables and a set $F=\set{a_1,\ldots,a_m}$ of clauses; 
we then have a set $E$ of edges joining
 variables to clauses. For each edge $e\in E$ we write $v(e)$ for the incident variable, and $a(e)$ for the incident clause;
 and we assign an edge literal $\lit_e\in\set{\zro,\one}$
 to indicate whether $v(e)$ participates affirmatively ($\lit_e=\zro$) or negatively ($\lit_e=\one$) in $a(e)$.
 We define all edges to have length one-half, so two variables $v\ne v'$ lie at unit distance if and only if they appear in the same clause. Throughout this paper we denote $\graph\equiv(V,F,E)$ for the bipartite graph without edge literals, and $\glit\equiv(V,F,E,\ulit)\equiv(\graph,\ulit)$ for the \textsc{nae-sat} instance.

Formally we regard the edges $E$ as a permutation, as follows. Each variable $v\in V$ has incident half-edges $\delta v$, while each clause $a\in F$ has incident half-edges $\delta a$. Write $\delta V$ for the labelled set of all variable-incident half-edges, and $\delta F$ for the labelled set of all clause-incident half-edges; we require that $|\delta V|=|\delta F|=\ell$. Then any permutation $\mathfrak{m}$ of $[\ell]\equiv\set{1,\ldots,\ell}$ defines $E$ by defining a matching between $\delta V$ and $\delta F$. Note that any permutation of $[\ell]$ is permitted, so multi-edges can occur. In this paper we assume that the graph is $(d,k)$-regular: each variable has $d$ incident edges, and each clause has $k$ incident edges, so $|E|=nd=mk$. A \bemph{random $k$-\textsc{nae-sat} instance} is given by $\glit=(V,F,E,\ulit)$ where $|V|=n$, $|F|=m$, $E$ is given by a uniformly random permutation $\mathfrak{m}$ of $[nd]$, and $\ulit$ is a uniformly random sample from $\set{\zro,\one}^E$. We write $\P$ and $\E$ for probability and expectation over the law of $\glit$.

\begin{dfn}[solutions and clusters]\label{d:nae.cluster} A \bemph{solution} of the \textsc{nae-sat} problem instance $\glit=(V,F,E,\ulit)$ is any assignment $\vec{\bx}\in\set{\zro,\one}^V$ such that for all $a\in F$, $(\lit_e \oplus\bx_{v(e)})_{e\in\delta a}$ is neither identically $\zro$ nor identically $\one$. Let $\SOL(\glit)\subseteq\set{\zro,\one}^V$ denote the set of all solutions of $\glit$, and define a graph on $\SOL(\glit)$ by connecting any pair of solutions at unit Hamming distance. The (maximal) connected components of the $\SOL(\glit)$ graph are the \bemph{solution clusters}, hereafter denoted $\CLUST(\glit)$.\end{dfn}

The aim of this section is to establish that (under a certain restriction) the \textsc{nae-sat} solution clusters can be represented by a combinatorial model of what we will term ``colorings.'' We will describe the correspondence in
a few stages. Informally, the progression is given by
	{\setlength{\jot}{0pt}\begin{align}\nonumber
	&\text{\textsc{nae-sat} solution clusters $\clust$} 
	\leftrightarrow\text{frozen configurations $\ux$}
	\leftrightarrow\\
	&\leftrightarrow\text{warning configurations $\vec{y}$}
	\leftrightarrow\text{message configurations $\uta$}
	\leftrightarrow\text{colorings $\usi$.}
	\label{e:bij}
	\end{align}}%
Each step of \eqref{e:bij} is formalized below. As mentioned previously, the key feature of the last model is that the size of a cluster $\clust$ can be easily read off from its corresponding coloring $\usi$, as a product of local functions. Some steps of the correspondence \eqref{e:bij} appear in existing literature (see~\cite{parisi2002local,MR2155706,MR2351840,MR2518205,MR3440193}) but we present them here in detail for completeness.

\subsection{Frozen and warning configurations}

We introduce a new value $\fre$ (free), and adopt the convention $\zro\oplus\fre\equiv \fre \equiv \one\oplus\fre$. For $l\ge1$ and $\ux\in\set{\zro,\one,\fre}^l$ let $I^\textsc{nae}(\ux)$ be the indicator that $\ux$ is neither identically~$\zro$ nor identically~$\one$. Given an \textsc{nae-sat} instance $\glit=(V,F,E,\ulit)$ and an assignment $\ux\in\set{\zro,\one,\fre}^V$, denote
	\[I^\textsc{nae}(\ux;\glit)
	\equiv
	\prod_{a\in F}
	I^\textsc{nae}(
	(\lit_e\oplus x_{v(e)})_{e\in\delta a}
	)\,.\]
By Definition~\ref{d:nae.cluster}, an \textsc{nae-sat} solution is an assignment $\vec{\bx}\in\set{\zro,\one}^V$ satisfying $I^\textsc{nae}(\vec{\bx};\glit)=1$.

\begin{dfn}[frozen configurations]\label{d:frozen} Given an \textsc{nae-sat} instance $\glit=(V,F,E,\ulit)$, for any $e\in E$ let $\glit\oplus\one_e$ denote the instance obtained by flipping the edge label
$\lit_e$ to $\lit_e\oplus\one$. We say that $\ux\in\set{\zro,\one,\fre}^V$ is a valid \bemph{frozen configuration} on $\glit$ if
(i) no \textsc{nae-sat} constraint is violated, meaning
$I^\textsc{nae}(\ux;\glit)=1$; and
(ii) for all $v\in V$, $x_v$ takes a value in $\set{\zro,\one}$ only when \bemph{forced} to do so,
meaning there is some $e\in\delta v$ such that
	\beq\label{e:forced}
	I^\textsc{nae}(\ux;\glit\oplus\one_e)=0\,.\eeq
If no such $e\in\delta v$ exists then $x_v=\fre$.
\end{dfn}

It is well known that on any given $\glit$, every \textsc{nae-sat} solution $\vec{\bx}$ can be mapped to a frozen configuration $\ux=\ux(\vec{\bx})$ via a ``coarsening'' or ``whitening'' procedure \cite{parisi2002local}, as follows. Initialize $\ux=\vec{\bx}$. Then, whenever $x_v\in\set{\zro,\one}$ but there exists no $e\in\delta v$ such that \eqref{e:forced} holds, update $x_v$ to $\fre$. Iterate until no further updates can be made; the result is then a valid frozen configuration. Two \textsc{nae-sat} solutions $\vec{\bx}$, $\vec{\bx}'$ map to the same frozen configuration $\ux$ if and only if they lie in the same cluster $\clust\in\CLUST(\glit)$. Thus, for any given $\glit$, we have a well-defined mapping from clusters $\clust$ to frozen configurations $\ux$. This map is one-to-one but not necessarily onto: for instance, the all-free assignment $\ux\equiv\fre$ is always trivially a valid frozen configuration, but on many instances $\glit$ there is no solution cluster $\clust\in\CLUST(\glit)$ whose coarsening is $\ux\equiv\fre$. Since the aim is to lower bound the clusters, the lack of surjectivity must be addressed. We will do so momentarily (Definition~\ref{d:free.cycle} below), but first we review an useful alternative representation of frozen configurations:

\begin{dfn}[warning configurations]
\label{d:wp}
For the integers $l\ge1$, define functions $\dotY : \set{\zro,\one,\fre}^l\to\set{\zro,\one,\fre,\invalid}$ and $\hatY : \set{\zro,\one,\fre}^l\to\set{\zro,\one,\fre}$ by
	\[\dotY(\vec{\hat{y}})
	=\begin{cases}
	\zro &\textup{if $\zro \in \set{\hat{y}_i}
	\subseteq \set{\zro,\fre}$,}\\
	\one &\textup{if $\one\in\set{\hat{y}_i}\subseteq \set{\one,\fre}$,}\\
	\fre &\textup{if $\set{\hat{y}_i}=\fre$,}\\
	\invalid &\textup{otherwise;}\end{cases}
	\quad \hatY(\vec{\dot{y}})
	= \begin{cases}
	\zro &\textup{if $\set{\dot{y}_i}=\set{\one}$,}\\
	\one &\textup{if $\set{\dot{y}_i}=\set{\zro}$;}\\
	\fre &\textup{otherwise.}\end{cases}\]
Denote $M\equiv\set{\zro,\one,\fre}^2$. We write $\vec{y}\in M^E$ if $\vec{y}=(y_e)_{e\in E}$ where $y_e\equiv(\dot{y}_e,\hat{y}_e)\in M$. If edge $e$ joins variable $v$ to clause $a$, then $\dot{y}_e$ represents a ``warning'' along $e$ from $v$ to $a$, while $\hat{y}_e$ represents a ``warning'' along $e$ from $a$ to $v$. We say that $\vec{y}\in M^E$ is a valid \bemph{warning configuration} on $\glit$ if it satisfies the local equations
	\beq\label{e:wp.equations}
	y_e=(\dot{y}_e,\hat{y}_e)
	=\Big( \dotY( \vec{\hat{y}}_{\delta v(e)\setminus e}),
	\lit_e \oplus \hatY(
		(\ulit\oplus\vec{\dot{y}})_{\delta a(e)\setminus e}
	) \Big)\eeq
for all $e\in E$
(with no $\dot{y}_e=\invalid$).
\end{dfn}

It is well known that on any given $\glit$ there is a natural bijection
	\beq\label{e:frozen.warning.bij}
	\left\{\hspace{-3pt}\begin{array}{c}
	\text{frozen configurations}\\
	\text{$\ux\in\set{\zro,\one,\fre}^V$}\end{array}
	\hspace{-3pt}\right\}
	\longleftrightarrow
	\left\{\hspace{-3pt}\begin{array}{c}
	\text{warning configurations}\\
	\text{$\vec{y}\in M^E$}\end{array}
	\hspace{-3pt}\right\}\,.\eeq 
In the forward direction,
given a (valid) frozen configuration $\ux$,
for any variable $v$ and any edge $e\in\delta v$
such that \eqref{e:forced} holds, set $\hat{y}_e=x_v \in\set{\zro,\one}$; then in all other cases set $\hat{y}_e=\fre$. Then, having defined all the $\hat{y}_e$, the $\dot{y}_e$ can only be defined by the local equations \eqref{e:wp.equations}. One can check that the resulting assignment $\vec{y}\in M^E$ is a warning configuration. Conversely, given a warning configuration $\vec{y}$, a frozen configuration $\ux$ can be obtained by setting $x_v=\dotY(\hat{y}_{\delta v})$ for all $v$.

\subsection{Message configurations}
\label{s:msg.config}

We return to the question of surjectivity: does a given frozen configuration $\ux$ encode a (nonempty) solution cluster $\clust\in\CLUST(\glit)$? We will now state an easy sufficient condition for this to hold. The condition is not in general necessary, but we will show that it captures enough of the solution space to deliver a sharp lower bound on the free energy.

\begin{dfn}[free cycles] \label{d:free.cycle} Let $\ux\in\set{\zro,\one,\fre}^V$ be a valid frozen configuration on $\glit=(V,F,E,\ulit)$. We say that a clause $a\in F$ is \bemph{separating} (with respect to $\ux$) if there exist $e',e''\in\delta a$ such that $\lit_{e'}\oplus x_{v(e')}=\zro$ while $\lit_{e''}\oplus x_{v(e'')}=\one$. For instance, a forcing clause is also separating. A cycle in $\glit$ is a sequence of edges
	\[e_1 e_2 \ldots e_{2\ell-1} e_{2\ell} e_1,\]
where, taking indices modulo $2\ell$, it holds for each integer $i$ that $e_{2i-1}$ and $e_{2i}$ are distinct but share a clause, while $e_{2i}$ and $e_{2i+1}$ are distinct but share a variable. (In particular, if $v$ is joined to $a$ by two edges $e' \ne e''$, then $e'e''$ forms a cycle.) We say the cycle in $\glit$ is \bemph{free} (with respect to $\ux$) if all its variables are free and all its clauses are non-separating.\end{dfn}

\begin{dfn}[free trees]\label{d:free.trees} 
Let $\ux$ be a frozen configuration on $\glit=(V,F,E,\ulit)$  that has no free cyces. Let $H$ be the subgraph of $\glit$ induced by the free variables and non-separating clauses of $\ux$. Since $\ux$ has no free cycles, $H$ must be a disjoint union of tree components $\bm{t}$, which we term the \bemph{free trees} of $\ux$.  For each $\bm{t}$, let $\bm{T}\equiv\bm{T}(\bm{t})$ be the subgraph of $\glit$ induced by $\bm{t}$  together with its incident variables. The subgraphs $\bm{T}$  (which can contain cycles) will be termed the \bemph{free pieces} of $\ux$. Each free variable is covered by exactly one free piece. In the simplest case, a free piece consists of a single free variable surrounded by $d$ separating clauses.\end{dfn}
 
Let us say that $\vec{\bx}\in\set{\zro,\one}^V$ \bemph{extends} $\ux\in\set{\zro,\one,\fre}^V$ if $\bx_v=x_v$ for all $v$ such that $x_v\in\set{\zro,\one}$. If $\ux$ is a frozen configuration on $\glit$ with no free cycles, it is easy to extend $\ux$ to valid \textsc{nae-sat} solutions $\vec{\bx}\in\set{\zro,\one}^V$ --- we simply extend $\ux$ on each free tree $\bm{t}$, since
\textsc{nae-sat} on a tree is always solvable; the different free trees do not interact. Let $\clust$ denote the set of all valid \textsc{nae-sat} solutions on $\glit$ that extend $\ux$, and denote $\SIZE(\ux)\equiv|\clust|$. Meanwhile, let $\mathfrak{T}(\ux)$ denote the set of all free pieces of $\ux$. For each $\bm{T}\in\mathfrak{T}(\ux)$, let $\SIZE(\ux;\bm{T})$ denote the number of valid \textsc{nae-sat} solutions on $\bm{T}$ that extend $\ux|_{\bm{T}}$. It follows from our discussion that
$\clust\in\CLUST(\glit)$ with
	\beq\label{e:cluster.size.product.free.trees}
	|\clust| = \SIZE(\ux)=\prod_{\bm{T}\in\mathfrak{T}(\ux)}
	\SIZE(\ux;\bm{T})\,.\eeq
That is to say, the absence of free cycles is an easy sufficient condition for a frozen configuration to encode a nonempty cluster; and it further ensures that the cluster has a relatively simple product structure \eqref{e:cluster.size.product.free.trees}. As noted previously, the structure within each free piece $\bm{T}$ can be understood by dynamic programming (\textsc{bp}). This is a well-known calculation (see e.g.\ \cite[Ch.~14]{MR2518205}) but we will review the details for our setting. To this end, we first introduce a combinatorial model of ``message configurations'' which will map directly to the natural \textsc{bp} variables. 

Recall from Definition~\ref{d:wp} that a warning configuration is denoted $\vec{y}\in M^E$
where each $y_e\equiv (\dot{y}_e,\hat{y}_e)\in M$.
We denote a message configuration by
$\vec{\tau}\in\mathscr{M}^E$
where each $\tau_e=(\dta_e,\hta_e)\in\mathscr{M}$
(for $\mathscr{M}$ to be defined below). It will be convenient to let $\dir$ indicate a directed edge, pointing from tail vertex $t(\dir)$ to head vertex $h(\dir)$. If $e$ is the undirected version of $\dir$, then we denote
	\[(y_{\dir},\tau_{\dir})
	=\begin{cases}
	(\dot{y}_e,\dta_e) &\textup{if $t(\dir)$ is a variable,}\\
	(\hat{y}_e,\hta_e) &\textup{if $t(\dir)$ is a clause.}
	\end{cases}\]
We will make a definition such that $\tau_{\dir}$ either takes the value ``$\star$'' or is a bipartite factor tree. The tree is \bemph{unlabelled} except that one vertex is distinguished as the root, and some edges are assigned $\zro$ or $\one$ values as explained below. The root of $\tau_{\dir}$ is required to have degree one, and should be thought of as corresponding to the head vertex $h(\dir)$.

In the context of message configurations $\vec{\tau}$, we use ``$\zro$'' or ``$\one$'' to stand for the tree consisting of a single edge which is labelled $\zro$ or $\one$ and rooted at the endpoint corresponding to the head vertex --- the root is the incident clause in the case of $\dta$, the incident variable in the case of $\hta$. We use $\spc$ to stand for the tree consisting of a single unlabelled edge, rooted at the incident variable; this will be related to the situation of separating clauses from Definition~\ref{d:free.cycle}. Given a collection of rooted trees $t_1,\ldots,t_\ell$ whose roots $o_1,\ldots,o_\ell$ are all of the same type (either all variable or all clauses), we define $t=\join(t_1,\ldots,t_\ell)$ by identifying all the $o_i$ as a single vertex $o$, then adding an edge which joins $o$ to a new vertex $o'$. The vertex $o$ has the same type (variable or clause) as the $o_i$; and the vertex $o'$ is assigned the opposite type and becomes the root of $t$.

\begin{dfn}[message configurations]
\label{d:msg.config} Start with $\dMM_0\equiv\set{\zro,\one,\star}$ and $\hMM_0\equiv\emptyset$, and suppose inductively that $\dMM_t,\hMM_t$ have been defined. For $\vec{\hta}\in (\hMM_t)^{d-1}$ and $\vec{\dta}\in(\dMM_t)^{k-1}$, let us abbreviate $\set{\hta_i}\equiv\set{\hta_1,\ldots,\hta_{k-1}}$
and $\set{\dta_i}\equiv\set{\dta_1,\ldots,\dta_{d-1}}$. Define
	\[\dotT(\vec{\hta}) 
	=\begin{cases}
	\zro&\textup{if $\zro\in\set{\hta_i}\subseteq\hMM_t\setminus\set{\one}$,}\\
	\one&\textup{if $\one\in\set{\hta_i}\subseteq\hMM_t\setminus\set{\zro}$,}\\
	\invalid&\textup{if $\set{\zro,\one}\subseteq\set{\hta_i}$,}\\
	\star&\textup{if $\star\in\set{\hta_i}\subseteq\hMM_t\setminus\set{\zro,\one}$,}\\
	\join\set{\hta_i}
		&\textup{if $\set{\hta_i}\subseteq\hMM_t\setminus\set{\zro,\one,\star}$;}
	\end{cases}\quad
	\hatT(\vec{\dta})
	=\begin{cases}
	\zro&\textup{if $\set{\dta_i}=\set{\one}$,}\\
	\one&\textup{if $\set{\dta_i}=\set{\zro}$,}\\
	\spc&\textup{if $\set{\zro,\one}\subseteq\set{\dta_i}$,}\\
	\star&\textup{if $\set{\zro,\one}\not\subseteq\set{\dta_i}$ and 
		$\star\in\set{\dta_i}$,}\\
	\join\set{\dta_i} &\textup{otherwise.}\end{cases}\]
Then, for $t\ge0$, define recursively the sets
	{\setlength{\jot}{0pt}\begin{align*}
	\hMM_{t+1}
	&\equiv \hMM_t \cup\hatT
	[(\dMM_t)^{k-1}]\,,\\
	\dMM_{t+1}
	&\equiv 
	\dMM_t \cup\dotT
	[(\hMM_{t+1})^{d-1}]\setminus\set{\invalid}
	\end{align*}}%
We then let $\dMM$ be the union of all the $\dMM_t$, let $\hMM$ be the union of all the $\hMM_t$, and let $\mathscr{M}=\dMM\times\hMM$. On $\glit=(V,F,E,\ulit)$, the assignment $\vec{\tau}\in\mathscr{M}^E$ is a valid \bemph{message configuration} if (i) it satisfies the local equations
	\beq\label{e:messageconfig.loceq}
	\tau_e
	=(\dta_e,\hta_e)
	=\Big(
	\dotT(
	\vec{\hta}_{\delta v(e)\setminus e}
	 ),
	\lit_e\oplus
	\hatT(
		(\ulit\oplus\vec{\dta})_{\delta a(e)\setminus e}
	)
	\Big)\eeq
for all $e\in E$ (with no $\dta_e=\invalid$),
and (ii) if one element of $\set{\dta_e,\hta_e}$ equals $\star$ then the other element is in $\set{\zro,\one}$. In \eqref{e:messageconfig.loceq}, we take the convention that $\lit_e\oplus\fre=\fre$ and 
$\lit_e\oplus\star=\star$, and if $\tau$ is a tree with labels then $\lit_e\oplus\tau$ is defined by applying $\lit_e\oplus\cdot$ entrywise to all labels of $\tau$. See Figure~\ref{f:join}.
\end{dfn}

\begin{figure}[h]
\centering
\includegraphics{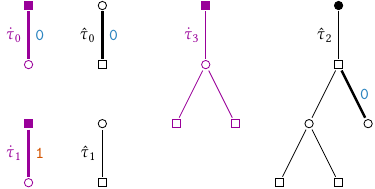}
\caption{Examples of messages (Definition~\ref{d:msg.config}). 
Variables are indicated by circle nodes, clauses by square nodes, and edges by lines.
For simplicity we assume that all edges depicted have literals $\lit_e=\zro$. Each message is shown with its root as a filled node at the top. The variable-to-clause messages $\dta$ are rooted at clauses, while the clause-to-variable messages $\hta$ are rooted at variables. The heavy lines indicates edges inside the message that are labelled $\zro$ or $\one$.
In our notation we have $\dta_0=\zro$ and $\dta_1=\one$.
Next $\hta_0=\hat{T}(\dta_1,\dta_1)=\zro$,
while $\hta_1=\hat{T}(\dta_0,\dta_1)=\spc$.
Finally $\dta_3=\dot{T}(\hta_1,\hta_1)
=\join(\hta_1,\hta_1)$
and $\hta_2=\hat{T}(\dta_3,\dta_0)
=\join(\dta_3,\dta_0)$.}
\label{f:join}
\end{figure}

Suppose $\ux$ is a frozen configuration on $\glit$, and let $\vec{y}$ be its corresponding warning configuration from \eqref{e:frozen.warning.bij}. Given $\vec{y}$, we define $\vec{\tau}$ in four stages:
\begin{enumerate}[1.]
\item If $\dot{y}_e\in\set{\zro,\one}$ then set $\dta_e=\dot{y}_e$; likewise if $\hat{y}_e\in\set{\zro,\one}$ then set $\hta_e=\hat{y}_e$. 
\item If $(\ulit\oplus\vec{\dot{y}})_{\delta a(e)\setminus e}$ has both $\zro$ and $\one$ entries, then set $\hta_e=\spc$. 
\item Apply the local equations \eqref{e:messageconfig.loceq} recursively to define $\dta_e,\hta_e$ wherever possible.
\item Lastly, if any $\dta_e$ or $\hta_e$ remains undefined, then set it to $\star$.
\end{enumerate} An example with $\star$ messages is given in Figure~\ref{f:star.cycle}.

\begin{figure}[h!]\includegraphics{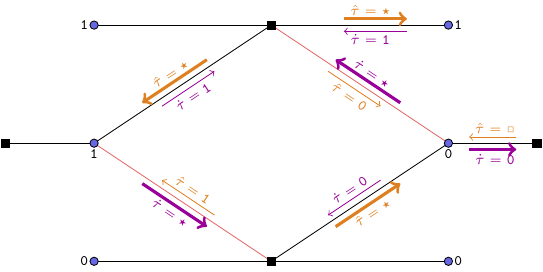}\caption{Example of how $\star$ messages can arise in the mapping from $\vec{y}$ to $\vec{\tau}$ (\S\ref{s:msg.config}). The figure shows a subgraph of $\glit$ with variables indicated by circle nodes, clauses by square nodes, and edges by lines. All edges in the figure are assumed to have label $\lit_e=\zro$. All variables shown are frozen to $\zro$ or $\one$, and all clauses shown are separating. To avoid clutter we did not label the edges with the warnings $y_e$; instead, each variable $v$ is labeled with its frozen configuration spin $x_v$, according to the $\ux\leftrightarrow\vec{y}$ bijection \eqref{e:frozen.warning.bij}. The clauses force the variables along the cycle in the clockwise direction, resulting in $\star$ values in the final $\vec{\tau}$ in the counterclockwise direction of the cycle. (Note also that $\vec{y}$ can be recovered from $\vec{\tau}$ by changing $\star$ to $\fre$; cf.\ Lemma~\ref{l:bij.frozen.message.configs}.)}\label{f:star.cycle}\end{figure}

\begin{lem}\label{l:bij.frozen.message.configs} The mapping described above defines a bijection
	\[\left\{\hspace{-3pt}\begin{array}{c}
	\text{frozen configurations
	$\ux\in\set{\zro,\one,\fre}^V$}\\
	\text{without free cycles}\end{array}
	\hspace{-3pt}\right\}
	\longleftrightarrow
	\left\{\hspace{-3pt}\begin{array}{c}
	\text{message configurations}\\
	\text{$\vec{\tau}\in\mathscr{M}^E$}\end{array}
	\hspace{-3pt}\right\}\,.\]

\begin{proof} Let $\ux\in\set{\zro,\one,\fre}^V$ be a frozen configuration on $\glit=(V,F,E,\ulit)$ without free cycles, and let $\vec{y}\in M^E$ be the warning configuration which corresponds to $\ux$ via \eqref{e:frozen.warning.bij}. We first check that the mapping $\vec{y}\mapsto\vec{\tau}$, as described above, gives a message configuration which is valid, i.e., satisfies conditions~(i)~and~(ii) of Definition~\ref{d:msg.config}. In the first stage, the mapping procedure sets $\dta_e=\dot{y}_e$ whenever $\dot{y}_e\in\set{\zro,\one}$, and $\hta_e=\hat{y}_e$ whenever $\hat{y}_e\in\set{\zro,\one}$. One can argue by induction that the rest of the procedure does not create any additional $\zro$ or $\one$ messages, so that in the final configuration the $\set{\zro,\one}$ values of $\vec{\tau}$ will match those of $\vec{y}$. The second and third stages of the procedure are clearly consistent with the local equations \eqref{e:messageconfig.loceq}. Note in particular that the third stage does not produce any $\dta_e=\invalid$ message, because it would contradict the assumption that $\vec{y}$ is a valid warning configuration; it also does not produce any $\star$ message. All $\star$ messages are created in the fourth stage, and this is clearly consistent with the mapping of $\star$ messages under $\dotT$ and $\hatT$. This concludes the proof that $\vec{\tau}$ satisfies condition~(i) of Definition~\ref{d:msg.config}. To check condition (ii), suppose $\tau_{\dir}=\star$, and let $\rev$ denote the reversal of $\dir$. From the above construction, it must be that $y_{\dir}=\fre$ and $\tau_{\dir'}=\star$ for some $\dir'$ that points to the tail vertex $t(\dir)$ but does not equal $\rev$. Consequently $\dir$ must belong to a directed cycle $\dir_1 \dir_2\ldots \dir_{2k} \dir_1$ with all the $\tau_{\dir_i}$ equal to $\star$. Whenever $\dir$ points from a separating clause $a$ to free variable $v$, we must have $\tau_{\dir}=\spc$. As a result, if all the variables along the cycle are free, then none of the clauses can be separating, contradicting the assumption that $\ux$ has no free cycles. Therefore some variable $v$ on the cycle must take value $x_v\in\set{\zro,\one}$, and by relabelling we may assume $v=t(\dir_1)$. Let $\rev_i$ denote the reversal of $\dir_i$: since $x_v\ne\fre$ but $y_{\dir_1}=\fre$, it must be that $y_{\rev_1}=x_v$. This means that the clause $a=h(\dir_1)=t(\rev_1)$ is forcing to $v$, so in particular $y_{\rev_2}\in\set{\zro,\one}$. Continuing in this way we see that $y_{\rev_i}\in\set{\zro,\one}$ for all $i$, and it follows that $\vec{\tau}$ satisfies condition~(ii), and so is a valid message configuration.

The mapping from $\vec{y}$ to $\vec{\tau}$ is clearly injective. To see that it is surjective, let $\vec{\tau}$ be any valid message configuration. Projecting $\dMM\setminus\set{\zro,\one}\mapsto\fre$ and $\hMM\setminus\set{\zro,\one}\mapsto \fre$ yields a valid warning configuration $\vec{y}$, which in turn maps to a valid frozen configuration $\ux$. It remains then to check that $\ux$ has no free cycles. Indeed, along a free cycle, all the warnings (in either direction) must be $\fre$. This means none of the messages can be in $\set{\zro,\one}$, and as a result none of the messages can be $\star$, by condition~(ii) of Definition~\ref{d:msg.config}. This means all the messages must be in $\dMM\setminus\set{\zro,\one,\star}$ or $\hMM\setminus\set{\zro,\one,\star}$. Suppose in one direction of the cycle we have the directed edges $\dir_1\dir_2\ldots \dir_{2k}\dir_1$. By definition of $\dotT$ and $\hatT$, $\tau_{\dir_i}$ is a proper subtree of $\tau_{\dir_{i+1}}$ for all $i$, with indices modulo $2k$. Going around the cycle we find that $\tau_{\dir_1}$ is a proper subtree of $\tau_{\dir_{2k+1}}=\tau_{\dir_1}$, which gives the contradiction.\end{proof}
\end{lem}

\subsection{Bethe formula} We now describe the dynamic programming (\textsc{bp}) calculation which will ultimately take a message configuration $\vec{\tau}$ and evaluate a product of local functions to compute the size of its associated cluster. The first step is to define the dynamic programming variables; these will formalize the measures $\msg$ which were introduced previously in~\eqref{e:interpret.messages}. Recall that for $l\ge1$ and $\ux\in\set{\zro,\one,\fre}^l$, we write $I^\textsc{nae}(\ux)$ for the indicator that the entries of $\ux$ are not identically $\zro$ or identically $\one$. 

\begin{dfn}\label{d:msg.of.tau} Recall that message configuration spins belong to the space $\mathscr{M}=\dMM\times\hMM$ (Definition~\ref{d:msg.config}). Let $\mathscr{P}(\set{\zro,\one})$ denote the space of probability measures on $\set{\zro,\one}$. Define the mappings $\dmsg:\dMM\to\mathscr{P}(\set{\zro,\one})$ and $\hmsg:\hMM\to\mathscr{P}(\set{\zro,\one})$ as follows. For $\dta\in\set{\zro,\one}$ let $\dmsg(\dta)$ be the unit measure supported on $\dta$. Likewise, for $\hta\in\set{\zro,\one}$ let $\hmsg(\hta)$ be the unit measure supported on $\hta$. For $\dta\in\dMM\setminus\set{\zro,\one,\star}$ or $\hta\in\hMM\setminus\set{\zro,\one,\star}$ we let $\dmsg(\dta)$ and $\hmsg(\hta)$ be recursively defined: if $\dta=\dotT(\hta_1,\ldots,\hta_{d-1})$ where no $\hta_j=\star$, define
	\beq\label{e:def.dm.dta}
	\dz(\dta)
	\equiv
	\sum_{\bx\in\set{\zro,\one}}
	\prod_{i=1}^{d-1} 
	[\hmsg(\hta_i)](\bx)\,,
	\quad [\dmsg(\dta)](\bx)
	\equiv
	\f{1}{\dz(\dta)}
	\prod_{i=1}^{d-1} 
	[\hmsg(\hta_i)](\bx)\,.\eeq
Note that $\hta_1,\ldots,\hta_{d-1}$ can be recovered from $\dta$ modulo permutation of the indices, so these quantities are well-defined. We see inductively that for $\dta\in\dMM\setminus\set{\zro,\one,\star}$, the normalizing factor $\dz(\dta)$ is positive, and $\dmsg(\dta)$ is a nondegenerate probability measure on $\set{\zro,\one}$. Similarly, if $\hta\in\hMM\setminus\set{\zro,\one,\star}$ equals $\hatT(\dta_1,\ldots,\dta_{k-1})$ where none of the $\dta_i$ are $\star$, then set
	\begin{align}\nonumber
	\hz(\hta)
	&\equiv
	\sum_{\bx\in\set{\zro,\one}}
	\sum_{\vec{\dot{\bx}}\in\set{\zro,\one}^{k-1}}
	I^\textsc{nae}(\bx,\vec{\dot{\bx}})
	\prod_{i=1}^{k-1}
	[\dmsg(\dta_i)](\dot{\bx}_i)
	=2
	-\prod_{i=1}^{k-1}[\dmsg(\dta_i)](\zro)
	-\prod_{i=1}^{k-1}[\dmsg(\dta_i)](\one)\,,\\
	[\hmsg(\hta)](\bx)
	&\equiv 
	\f1{\hz(\hta)}
	\sum_{\vec{\dot{\bx}}\in\set{\zro,\one}^{k-1}}
	I^\textsc{nae}(\bx,\vec{\dot{\bx}})
	\prod_{i=1}^{k-1}
	[\dmsg(\dta_i)](\dot{\bx}_i)
	= \f1{\hz(\hta)}
	\bigg(1-\prod_{i=1}^{k-1}[\dmsg(\dta_i)](\bx)\bigg)\,.
	\label{e:def.hm.hta.z}
	\end{align}
Again, we see inductively that for $\hta\in\hMM\setminus\set{\zro,\one,\star}$, the normalizing factor $\hz(\hta)$ is positive, and $\hmsg(\hta)$ is a nondegenerate probability measure on $\set{\zro,\one}$. Finally, we will see below that for our purposes we can take $\dmsg(\star)$ and $\hmsg(\star)$ to be arbitrary nondegenerate probability measures on $\set{\zro,\one}$; we therefore define them both to equal the uniform measure on $\set{\zro,\one}$.\end{dfn}

Given a valid message configuration $\vec{\tau}$ on $\glit$, define $\vec{\msg} = (\msg_e)_{e\in E}$ where $\msg_e\equiv(\dmsg_e,\hmsg_e)$ with $\dmsg_e\equiv\dmsg(\dta_e)$ and $\hmsg_e\equiv\hmsg(\hta_e)$. It follows from Definition~\ref{d:msg.of.tau} that $\vec{\msg}$ satisfies the following local consistency equations, which are inherited from the equations \eqref{e:messageconfig.loceq} satisfied by $\vec{\tau}$, in combination with the above definitions \eqref{e:def.dm.dta} and \eqref{e:def.hm.hta.z}. If $\dta_e\ne\star$, then $\dmsg_e$ is given by the equation
	\beq\label{e:variable.bp}
	\dmsg_e(\bx)
	= \f{1}{\dz(\dta_e)}
	\prod_{e'\in\delta v(e)\setminus e}
	\hmsg_{e'}(\bx)\eeq
for $\bx\in\set{\zro,\one}$. Likewise, if $\hta_e\ne\star$, then $\hmsg_e$ is given by the equation
	\beq\label{e:clause.bp}
	\hmsg_e(\bx)
	= \f{1}{\hz(\hta_e)}
	\sum_{
	\vec{\dot{\bx}}_{\delta a(e)}
	\in\set{\zro,\one}^d
	}
	\Ind{\dot{\bx}_e=\bx } 
	I^\textsc{nae}(
		(\vec{\dot{\bx}}\oplus\ulit
			)_{\delta a(e)}
		)
	\prod_{e'\in\delta a(e)\setminus e}
		\dmsg_{e'}(\bx_{e'})\eeq
for $\bx\in\set{\zro,\one}$. The equations \eqref{e:variable.bp} and \eqref{e:clause.bp} are known as the \bemph{\textsc{bp} equations}. We now proceed to the calculation of the cluster size \eqref{e:cluster.size.product.free.trees}. To this end, we define the local functions
	\begin{align}\nonumber
	\ephi(\dta,\hta)
	&\equiv
	\bigg\{
	\sum_{\bx\in\set{\zro,\one}}
	\dmsg[\dta](\bx) \cdot \hmsg[\hta](\bx)
	\bigg\}^{-1}\,,\\
		\nonumber
	\hphi(\dta_1,\ldots,\dta_k)
	&\equiv 
	\sum_{\vec{\bx}\in\set{\zro,\one}^k}
	I^\textsc{nae}(\vec{\bx})
	 \prod_{i=1}^k \dmsg[\dta_i](\bx_i)
	=1-\sum_{\bx\in\set{\zro,\one}}
		\prod_{i=1}^k \dmsg[\dta_i](\bx)
	= \f{\hz(\hatT( (\dta_j)_{j\ne i} ))) }
		{\ephi(\dta_i,\hatT( (\dta_j)_{j\ne i} ))}
	\,,\\
	\dphi(\hta_1,\ldots,\hta_d)
	&\equiv
	\sum_{\bx\in\set{\zro,\one}}
	\prod_{i=1}^d\hmsg[\hta_i](\bx)
	= \f{\dz(\dotT((\hta_j)_{j\ne i})))}
		{\ephi(\hta_i,\dotT((\hta_j)_{j\ne i}))}\,,
		\label{e:phi.functions}
	\end{align}
where the last identity in the last two lines holds for any choice of $i$. The \textsc{bp} calculation is summarized by the following:

\begin{lem}\label{l:tree.cluster.size} Suppose on $\glit=(V,F,E,\ulit)$ that $\ux$ is a frozen configuration with no free cycles, and let $\vec{\tau}$ be its corresponding message configuration from Lemma~\ref{l:bij.frozen.message.configs}. Let $\bm{T}\in\mathfrak{T}(\ux)$ be a free piece of $\ux$, and let $\bm{t}$ be the free tree inside it. Then the number of \textsc{nae-sat} extensions of $\ux|_{\bm{T}}$ on $\bm{T}$ is given by
	\beq\label{e:free.tree.size}
	\SIZE(\ux;\bm{T})
	=\prod_{v\in V(\bm{t})}
	\bigg\{
	\dphi(\vec{\hat{\tau}}_{\delta v})
	\prod_{e\in\delta v}\ephi(\tau_e)	
	\bigg\}
	\prod_{a\in F(\bm{t})}
	\hphi((\vec{\dta}
	\oplus\ulit)_{\delta a})\eeq
where $V(\bm{t})$ and $F(\bm{t})$ denote respectively the variables and clauses in $\bm{t}$. (An example calculation is worked out in Figure~\ref{f:def.tree.messages}.)

\begin{proof}As we have mentioned before, this calculation is well known (see e.g.\ \cite[Ch.~14]{MR2518205}) but we will review it here, beginning with a minor technical point. As noted in Definition~\ref{d:free.trees}, $\bm{t}$ is a tree but $\bm{T}$ has a cycle wherever a variable $v\in \bm{T}\setminus\bm{t}$ is joined by more than one edge to $\bm{t}$. However, since $\ux|_{\bm{T}\setminus\bm{t}}$ is $\set{\zro,\one}$-valued, these cycles play no role in the question of extending $\ux|_{\bm{T}}$ to a valid \textsc{nae-sat} assignment on $\bm{T}$ --- one can simply duplicate variables in $\bm{T}\setminus\bm{t}$ so that each one joins to $\bm{t}$ by exactly one edge. We may therefore assume for the rest of the proof that all the free pieces $\bm{T}\in\mathfrak{T}(\ux)$ are acyclic.

For any $\bm{T}\in\mathfrak{T}(\ux)$ and any edge $e\in\bm{T}$, delete from $\bm{T}$ the edges $\delta a(e)\setminus e$, and let $\dot{\bm{T}}_e$ denote the component containing $e$ in what remains, rooted at $a(e)$. Likewise, delete from $\bm{T}$ the edges $\delta v(e)\setminus e$, and let $\hat{\bm{T}}_e$ denote the component containing $e$ in what remains, rooted at $v(e)$. For each variable $w\in\dot{\bm{T}}_e\setminus\bm{t}$, let $\acute{x}_w\in\set{\zro,\one}$ be the boolean sum of $x_w$ together with all the edge literals $\lit$ on the path joining $w$ to $a(e)$ in $\dot{\bm{T}}_e$. Note then that $\dta_e$ encodes the isomorphism class of $\dot{\bm{T}}_e$, labelled with boundary data $\acute{x}_w$ (for all the variables $w\in\dot{\bm{T}}_e\setminus\bm{t}$). A similar relation holds between $\hta_e$ and $\hat{\bm{T}}_e$. For each $e\in\bm{T}$, let $\dot{\textsf{s}}_e(\bx;\ux)$ count the number of valid \textsc{nae-sat} assignments that extend $\ux|_{\dot{\bm{T}}_e}$ on $\dot{\bm{T}}_e$ and take value $\bx$ on $v(e)$. Let $\hat{\textsf{s}}_e(\bx;\ux)$ count the number of valid \textsc{nae-sat} assignments that extend $\ux|_{\hat{\bm{T}}_e}$ on $\hat{\bm{T}}_e$ and take value $\bx$ on $v(e)$. Denote
	\[\dot{\textsf{s}}_e(\ux)\equiv
	\sum_{\bx\in\set{\zro,\one}}
	\dot{\textsf{s}}_e(\bx;\ux)\,,\quad
	\hat{\textsf{s}}_e(\ux)
	\equiv
	\sum_{\bx\in\set{\zro,\one}}
	\hat{\textsf{s}}_e(\bx;\ux)\,.\]
There are two boundary cases: if edge $e$ joins a free variable in $\bm{t}$ to a separating clause in $\bm{T}\setminus\bm{t}$, then we have $\hat{\textsf{s}}_e(\zro;\ux) = \hat{\textsf{s}}_e(\one;\ux)=1$. If edge $e$ instead joins a non-separating clause in $\bm{t}$ to a frozen variable in $\bm{T}\setminus\bm{t}$, then we have $\dot{\textsf{s}}_e(\bx;\ux)=\Ind{\bx=x_{v(e)}}$. By induction started from these boundary cases we find that for all $e\in\bm{T}$,
	\[\dmsg_e(\bx)
	=\f{\dot{\textsf{s}}_e(\bx;\ux)}
		{\dot{\textsf{s}}_e(\ux)}\,,\quad
	\hmsg_e(\bx)
	=\f{\hat{\textsf{s}}_e(\bx;\ux)}
		{\hat{\textsf{s}}_e(\ux)}\,.\]
It follows that for any variable $v\in\bm{t}$, any clause $a\in\bm{t}$, and any edge $e\in\bm{T}$, we have the identities
	\[\SIZE(\ux;\bm{T})
	= \dot{\varphi}(\vec{\hta}_{\delta v})
	\prod_{e'\in\delta v}
	\hat{\textsf{s}}_{e'}(\ux)
	= \hphi(\vec{\dta}_{\delta a})
	\prod_{e'\in\delta a}
	\dot{\textsf{s}}_{e'}(\ux)
	= \f{\dot{\textsf{s}}_e(\ux)
	 \hat{\textsf{s}}_e(\ux)}{\ephi(\tau_e)}\,.\]
Combining the identities and rearranging gives
(writing $E(\bm{t})$ for the edges of $\bm{t}$)
	\begin{align*}
	&\prod_{v\in V(\bm{t})}
	\bigg\{
	\dphi(\vec{\hat{\tau}}_{\delta v})
	\prod_{e\in\delta v}\ephi(\tau_e)
	\bigg\}
	\prod_{a\in F(\bm{t})}
	\hphi((\vec{\dta}
	\oplus\ulit)_{\delta a})\\
	&\qquad=\f{\SIZE(\ux;\bm{T})^{
		|V(\bm{t})|+|F(\bm{t})|}}
			{\SIZE(\ux;\bm{T})^{|E(\bm{t})|}}
	\cdot \bigg\{
	\prod_{v\in V(\bm{t})}
	\prod_{e\in\delta v\setminus\bm{t}}
	\f{\dot{\textsf{s}}_e(\ux)}
			{\SIZE(\ux;\bm{T})}
			\bigg\}\bigg/\bigg\{
	\prod_{a\in F(\bm{t})}
	\prod_{e\in\delta a\setminus\bm{t}}
	\dot{\textsf{s}}_e(\ux)
	\bigg\}\,.
	\end{align*}
For $a\in F(\bm{t})$ and $e\in\delta a\setminus\bm{t}$, the variable $v(e)$ is frozen and so we have $\dot{\textsf{s}}_e(\ux)=1$. For any $v\in V(\bm{t})$ and $e\in\delta v\setminus\bm{t}$, we have $\dot{\textsf{s}}_e(\ux)=\SIZE(\ux;\bm{T})$. The tree $\bm{t}$ has Euler characteristic one. The right-hand side of the above equation then simplifies to $\SIZE(\ux;\bm{T})$, thereby proving the claim.\end{proof} \end{lem}

\begin{figure}[h!]\centering\includegraphics{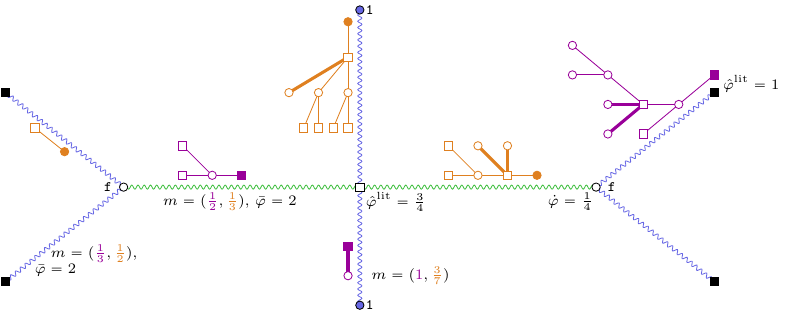}\caption{Example of correspondence (Lemma~\ref{l:bij.frozen.message.configs}) between frozen and message configurations. Variables are indicated by circle nodes, clauses by square nodes, and edges by lines. The graph formed by the wavy lines is a free piece $\bm{T}$, with the free tree $\bm{t}$ in {\color{green!50!gray}green} and $\bm{T}\setminus\bm{t}$ in {\color{blue!70!gray!70!white}blue} (Definition~\ref{d:free.trees}). Each variable $v$ is labelled with its frozen configuration spin value $x_v\in\set{\zro,\one,\fre}$. The four separating clauses are indicated by filled black squares. The message configuration is only partially shown, with the remaining values given by the obvious symmetries. The clause-to-variable messages $\hta$ are shown in black, and the variable-to-clause messages $\dta$ are shown in {\color{violet!80!magenta}purple}. Each message is a tree, with root vertex shown as a filled node. The heavy black and {\color{violet!80!magenta}purple} lines indicate edges inside the messages that are labeled $\one$. For instance, the message coming up out of the bottom variable is a tree consisting of a single edge, labelled $\one$ (indicated in the figure by a heavy {\color{violet!80!magenta}purple} line), rooted at its incident clause. We then calculate on each edge the values $m=( {\color{violet!80!magenta}\dmsg[\dta](\one)}, {\color{black}\hmsg[\hta](\one)})$, and use this to determine the factors $\dphi$, $\hphi$, $\ephi$ from \eqref{e:phi.functions}. In this example, $\bm{t}$ has two free variables each with $\dphi=1/4$, four separating clauses each with $\hphi=1$, and one non-separating clause with $\hphi=3/4$. There are six edges incident to $\bm{t}$, each with $\ephi=2$. Multiplying all these factors together (Lemma~\ref{l:tree.cluster.size}) gives $\SIZE(\ux;\bm{T})=3$. Indeed, in this small example it is easy to see that there are exactly three \textsc{nae-sat} assignments extending the frozen configuration $\ux$ on $\bm{T}$, since the two free variables cannot both take value $\one$, but the remaining three possibilities give valid \textsc{nae-sat} assignments.}\label{f:def.tree.messages}\end{figure}

\begin{cor}\label{c:cancel}Suppose on $\glit=(V,F,E,\ulit)$ that $\ux$ is a frozen configuration with no free cycles, and let $\vec{\tau}$ be its corresponding message configuration from Lemma~\ref{l:bij.frozen.message.configs}. Then the number of valid \textsc{nae-sat} extensions of $\ux$ is given by the product formula
	\[\SIZE(\ux)=\prod_{v\in V}
	\dphi(\vec{\hat{\tau}}_{\delta v})
	\prod_{a \in F}
	\hphi((\vec{\dta}\oplus\ulit)_{\delta a})
	\prod_{e\in E}
	\ephi(\tau_e)\,.\]
This identity holds as long as $\hmsg(\star)$ and $\hmsg(\star)$ are fixed nondegenerate probability measures on $\set{\zro,\one}$.

\begin{proof} Let $V'$ denote the set of free variables, and let $E'$ denote the set of all edges incident to $V'$. Let $F'$ the set of non-separating clauses. From \eqref{e:cluster.size.product.free.trees} and Lemma~\ref{l:tree.cluster.size} we have
	\beq\label{e:product.restr.to.free.trees}
	\SIZE(\ux)
	=\prod_{\bm{T}\in\mathfrak{T}(\ux)}
		\SIZE(\bx;\bm{T})
	=\prod_{v\in V'}
	\dphi(\vec{\hta}_{\delta v})
	\prod_{a\in F'}
	\hphi(
	(\vec{\dta}
	\oplus\ulit)_{\delta a})
	\prod_{e\in E'}\ephi(\tau_e)\,.\eeq
For any edge $e\in E\setminus E'$, the incident variable $v(e)$ must lie in $V\setminus V'$, meaning $x_{v(e)}\in\set{\zro,\one}$. We now partition $E\setminus E'$ into the disjoint union of $E_\red$ and $E_\blu$, as follows. Let $E_\red$ be the set of edges $e\in E\setminus E'$ such that $\hmsg_e$ is fully supported on $x_{v(e)}$. Let $E_\blu$ be the set of edges $e\in E\setminus E'$ such that $\hmsg_e$ is a nondegenerate measure on $\set{\zro,\one}$; note that $\dmsg_e$ must then be fully supported on $x_{v(e)}$. Consider a clause $a\in F\setminus F'$. If $a$ is non-forcing, then $\delta a\cap E_\red=\emptyset$ and $\hphi((\vec{\dta}\oplus\ulit)_{\delta a})=1$. Otherwise, $a$ is forcing in the direction of some edge $e\in\delta a$, in which case $\delta a \cap E_\red= \set{e}$ and $\hphi((\vec{\dta}\oplus\ulit)_{\delta a}) = \dmsg_e(x_{v(e)}) = 1/\ephi(\tau_e)$. We conclude for all $a\in F\setminus F'$ that
	\beq\label{e:separating.clauses.cancellation}
	\hphi((\vec{\dta}\oplus\ulit)_{\delta a})
	\prod_{e\in\delta a \cap E_\red}
	\ephi(\tau_e) = 1\,.\eeq
For $v\in V\setminus V'$, for all $e\in\delta v\cap E_\blu$ we have $\dmsg_e(x_v)=1$ and so $\ephi(\tau_e)=1/\hmsg(x_v)$. Thus, for all $v\in V\setminus V'$,
	\beq\label{e:forced.variables.cancellation}
	\dphi(\vec{\hta}_{\delta v})
	\prod_{e\in\delta v\cap E_\blu}
	\ephi(\tau_e)
	=1\,.\eeq
The identities \eqref{e:separating.clauses.cancellation}~and~\eqref{e:forced.variables.cancellation} remain valid even for vertices incident to $\star$ messages, as long as $\dmsg(\star)$ and $\hmsg(\star)$ are fixed nondegenerate probability measures on $\set{\zro,\one}$. Combining the identities with \eqref{e:product.restr.to.free.trees} proves the claim.\end{proof}\end{cor}

\subsection{Colorings}

We conclude this section by defining the coloring model, building on an encoding introduced by~\cite{MR3436404}. It is a simplification of the message configuration model (Definition~\ref{d:msg.config}) that takes advantage of some of the cancellations (\eqref{e:separating.clauses.cancellation}~and~\eqref{e:forced.variables.cancellation}) seen above. In short, following the notation of Corollary~\ref{c:cancel}, for edges in $E\setminus E'$ it is not necessary to keep all the information of $\tau_e$; instead, it suffices to keep track only of whether $e$ belongs to $E_\red$ or $E_\blu$, along with the value of $x_{v(e)}\in\set{\zro,\one}$. The colorings encode precisely this information. The resulting bijection between colorings and message configurations is the last step of \eqref{e:bij}.

Recall messages take values $\tau\equiv(\dta,\hta)\in\mathscr{M}\equiv\dMM\times\hMM$ (Definition~\ref{d:msg.config}), and let $\set{\fcl}\subseteq\mathscr{M}$ denote the subset of values $\tau\in\mathscr{M}$ where we have both $\dta\in\dMM\setminus\set{\zro,\one,\star}$ and $\hta\in\hMM\setminus\set{\zro,\one,\star}$. Denote $\COLS\equiv \set{\redz,\redo, \bluz,\bluo} \cup \set{\fcl}$. Define a projection $\ST : \mathscr{M} \to \COLS$ by
	\beq\label{e:def.ST.tau}
	\ST(\tau)=\begin{cases}
	\redz & \textup{if $\hta=\zro$,}\\
	\redo & \textup{if $\hta=\one$,}\\
	\bluz & \textup{if $\hta\ne\zro$ and $\dta=\zro$,}\\
	\bluo & \textup{if $\hta\ne\one$ and $\dta=\one$,}\\
	\tau & \textup{otherwise (meaning that $\tau\in\set{\fcl}$).}
	\end{cases}\eeq
(Note that $\ST(\tau)\in\set{\redz,\redo}$ includes the case $\dta=\star$, and $\ST(\tau)\in\set{\bluz,\bluo}$ includes the case $\hta=\star$.) We define a partial inverse to $\ST$ as follows. If $\sigma\in\set{\fcl}$ then define $\tau\equiv\tau(\sigma)\equiv\sigma\equiv(\dsi,\hsi)$. If $\sigma=\red_{\bx}$ for $\bx\in\set{\zro,\one}$ then define $\hta\equiv\hta(\sigma)\equiv\bx$ and leave $\dta(\sigma)$ undefined. If $\sigma=\blu_{\bx}$ for $\bx\in\set{\zro,\one}$ then define $\dta\equiv\dta(\sigma)\equiv\bx$ and leave $\hta(\sigma)$ undefined. For $\sigma\in\set{\redz,\redo,\bluz,\bluo}$ we denote $(\dsi,\hsi)\equiv(\sigma,\sigma)$. The coloring model is the image of the message configuration model under the projection $\ST$, formally given by the following:

\begin{dfn}[colorings] \label{d:colorings}
For $\usi\in\COLS^d$, abbreviate $\set{\sigma_i}\equiv\set{\sigma_1,\ldots,\sigma_d}$, and define
	\[\dot{I}(\usi)
	\equiv\begin{cases}
	1 & \textup{if $\redz\in\set{\sigma_i}\subseteq\set{\redz,\bluz}$,}\\
	1 & \textup{if $\redo\in\set{\sigma_i}\subseteq\set{\redo,\bluo}$,}\\
	1 & \textup{$\set{\sigma_i}\subseteq\set{\fcl}$, and 
		$\dsi_i=\dotT((\hsi_j)_{j\ne i})$ for all $i$,}\\
	0 & \textup{otherwise.}\end{cases}\]
For $\usi\in\COLS^k$, abbreviate $\set{\sigma_i}\equiv\set{\sigma_1,\ldots,\sigma_k}$, and define
	\[\hI(\usi)
	=\begin{cases}
	1 & \textup{if $\exists i : 
		\sigma_i=\redz \text{ and }
		\set{\sigma_j}_{j\ne i}
		=\set{\bluo}$,}\\
	1 & \textup{if $\exists i : 
		\sigma_i=\redo \text{ and }
		\set{\sigma_j}_{j\ne i}
		=\set{\bluz}$,}\\
	1 & \textup{if $\set{\bluz,\bluo}
		\subseteq\set{\sigma_i}
		\subseteq
		\set{\bluz,\bluo}
		\cup\set{\sigma\in\set{\fcl}:
			\hsi=\spc}$,}\\
	1 & \textup{if $\set{\sigma_i}
		\subseteq\set{\bluz}\cup\set{\fcl}$,
		$|\set{\sigma_i}\cap\set{\fcl}|\ge2$,
		and $\hsi_i=\hatT((\dta(\sigma_j))_{j\ne i})$
		for all $i$ where
		$\sigma_i\ne\bluz$;}\\
	1 & \textup{if $\set{\sigma_i}
		\subseteq\set{\bluo}\cup\set{\fcl}$,
		$|\set{\sigma_i}\cap\set{\fcl}|\ge2$,
		and $\hsi_i=\hatT((\dta(\sigma_j))_{j\ne i})$
		for all $i$ where
		$\sigma_i\ne\bluo$;}\\
	0 & \textup{otherwise.}\end{cases}\]
(In the definition of $\hI(\usi)$ we used that if $\set{\sigma_i}\subseteq\set{\bluz,\bluo}\cup\set{\fcl}$, then $\dta(\sigma_i)$ is defined for all $i$.) On an \textsc{nae-sat} instance $\glit=(V,F,E,\ulit)$, a configuration $\usi\in\COLS^E$ is a valid \bemph{coloring} if $\dot{I}(\usi_{\delta v})=1$ for all $v\in V$, and $\hI((\usi\oplus\ulit)_{\delta a})=1$ for all $a\in F$.\end{dfn}

\begin{lem}\label{l:bij.message.cols}
On any given \textsc{nae-sat} instance $\glit=(V,F,E,\ulit)$,
we have a bijection
	\[\left\{\hspace{-3pt}\begin{array}{c}
	\text{message configurations}\\
	\text{$\vec{\tau}\in\mathscr{M}^E$}\end{array}
	\hspace{-3pt}\right\}
	\longleftrightarrow
	\left\{\hspace{-3pt}\begin{array}{c}
	\text{colorings}\\
	\text{$\usi\in\COLS^E$}\end{array}
	\hspace{-3pt}\right\}\,.\]
\begin{proof}Given a valid message configuration $\vec{\tau}$, a valid coloring $\usi$ is obtained by coordinatewise application of the projection map $\ST$ from \eqref{e:def.ST.tau}. In the other direction, given a valid coloring $\usi$, let $x_v=\zro$ if $\usi_{\delta v}$ has any $\redz$ entries, $x_v=\one$ if $\usi_{\delta v}$ has any $\redo$ entries, and $x_v=\fre$ otherwise. The resulting $\ux\in\set{\zro,\one,\fre}^V$ is a valid frozen configuration, and the argument of Lemma~\ref{l:bij.frozen.message.configs} implies that it has no free cycles. It then maps by Lemma~\ref{l:bij.frozen.message.configs} to a valid message configuration $\vec{\tau}$, which completes the correspondence.\end{proof}
\end{lem}

Recall the definitions \eqref{e:phi.functions} of $\ephi$, $\hphi$, and $\dphi$. For $\usi\in\COLS^d$, let
	\[\dPhi(\usi)
	=\begin{cases}
	\dphi(\vec{\hsi})
	&\textup{if $\dot{I}(\usi)=1$
		and 
		$\set{\sigma_i}\subseteq
		\set{\fcl}$};\\
	1 & \textup{if $\dot{I}
		(\usi)=1$
		and $\set{\sigma_i}
		\subseteq
		\set{\redz,\redo,\bluz,\bluo}$;}\\
	0 & \textup{otherwise
		(meaning that $\dot{I}(\usi)=0$).}
	\end{cases}\]
(Note if $\set{\sigma_i}\subseteq\set{\fcl}$
then $\vec{\hsi}=\vec{\hta}$ and $\dphi(\vec{\hsi})$ is well-defined.) For $\usi\in\COLS^k$, let
	\beq\label{e:def.hPhi.literals}
	\hPl(\usi)=\begin{cases}
	\hphi((\dta(\sigma_i))_i) 
		& \textup{if $\hI(\usi)=1$
		and $\set{\sigma_i}
		\subseteq \set{\bluz,\bluo}\cup\set{\fcl}$;}\\
	1 & \textup{if $\hI(\usi)=1$
		and $\set{\sigma_i}
		\cap\set{\redz,\redo}\ne\emptyset$;}\\
	0 &\textup{otherwise
	(meaning that $\hI(\usi)=0$).}\end{cases}\eeq
(Note if $\set{\sigma_i}\subseteq\set{\bluo,\bluo}\cup\set{\fcl}$ then $\dta(\sigma_i)$ is well-defined for all $i$.) Finally, let
	\[\ePhi(\sigma)
	=\begin{cases}
	\ephi(\sigma)
		&\textup{if $\sigma\in\set{\fcl}$,}\\
	1 & \textup{if $\sigma\in
	\set{\redz,\redo,
	\bluz,\bluo}$.}\end{cases}\]
The following is a straightforward consequence of Lemma~\ref{l:tree.cluster.size}:

\begin{lem}\label{l:cluster.size.col} Suppose on $\glit=(V,F,E,\ulit)$ that $\ux$ is a frozen configuration with no free cycles.
Let $\usi$ be the coloring that corresponds to $\ux$ by
 Lemmas~\ref{l:bij.frozen.message.configs}~and~\ref{l:bij.message.cols}. Then the number of valid \textsc{nae-sat} extensions of $\ux$ is given by $\SIZE(\ux)=\wt_{\glit}(\usi)$ where we define
	\beq\label{e:def.col.wt.T}
	\wt_{\glit}(\usi)\equiv
	\prod_{v\in V}
	\dPhi(\usi_{\delta v})
	\prod_{a\in F}
	\hPl((\usi\oplus\ulit)_{\delta a})
	\prod_{e\in E}
	\ePhi(\sigma_e)\,.\eeq
\begin{proof} This is a rewriting of \eqref{e:product.restr.to.free.trees}.\end{proof}
\end{lem}

\begin{dfn}[$T$-colorings]\label{d:T.colorings}
If $\sigma\in\set{\fcl}$, then $\dsi$ is a tree rooted at a clause $a$ incident to a single edge $e(a)$, while $\hsi$ is a tree rooted at a variable $v$ incident to a single edge $e(v)$. Glue $\dsi$ and $\hsi$ together by identifying $e(a)$ with $e(v)$, and let $|\sigma|$ count the number of free variables in the resulting tree. (Note that $|\sigma|$ must be finite because we only consider colorings of finite \textsc{nae-sat} instances $\GG$.) Thus $|\sigma|=|\dsi|+|\hsi|-1$ where $|\dsi|$ is the number of free variables in the tree $\dsi$, and $|\hsi|$ is the number of free variables in the tree $\hsi$. If $\sigma\in\COLS\setminus\set{\fcl}
=\set{\redz,\redo,\bluz,\bluo}$ then define $|\sigma|\equiv0$. If $\usi$ is a valid coloring on $\glit=(V,F,E,\ulit)$, then $|\sigma_e|$ must be finite on every edge $e\in E$, by Definition~\ref{d:colorings}. For $0\le T\le \infty$ define $\tcols\equiv\set{\sigma\in\COLS:|\sigma|\le T}$; we then call $\usi$ a \bemph{$T$-coloring} if $\sigma_e\in\tcols$ for all $e\in E$. Define $\ZZ_{\lambda,T}$ to be the partition function of $\lambda$-tilted $T$-colorings,
	\beq\label{e:Z.lambda}
	\ZZ_{\lambda,T}\equiv \ZZ_{\lambda,T}(\glit)
	\equiv \sum_{\usi\in (\tcols)^E}
	\wt_\glit(\usi)^\lambda\,.\eeq
Denote $\ZZ_\lambda\equiv\ZZ_{\lambda,\infty}$ and note that as $T\uparrow\infty$ we have $\ZZ_{\lambda,T}\uparrow \ZZ_{\lambda,\infty}\equiv\ZZ_\lambda$. \end{dfn}

\begin{ppn}\label{p:cluster.size.col} On an \textsc{nae-sat} instance $\glit=(V,F,E,\ulit)$, recall $\CLUST(\glit)$ denotes the set of solution clusters (connected components of $\SOL(\glit)$), and define
	\beq\label{e:formal.Z.lambda}
	\barZZ_\lambda
	\equiv\barZZ_\lambda(\glit)
	\equiv \sum_{\clust\in\CLUST(\glit)}|\clust|^\lambda\,.\eeq
Then $\barZZ_\lambda$ is lower bounded by $\ZZ_\lambda$, where $\ZZ_\lambda$ is the increasing limit of $\ZZ_{\lambda,T}$ as defined by \eqref{e:Z.lambda}.

\begin{proof} On $\glit$, the colorings $\usi$ are in bijection (Lemma~\ref{l:bij.message.cols}) with the message configurations $\vec{\tau}$, which in turn are in bijection (Lemma~\ref{l:bij.frozen.message.configs}) with the frozen configurations $\ux$ that do not have free cycles. Each such frozen configuration defines a distinct cluster $\clust\in\CLUST(\glit)$, of size $|\clust|=\SIZE(\ux)=\wt_\glit(\usi)$. The claimed inequality directly follows.\end{proof} \end{ppn}

To summarize what we have obtained so far, note that the quantity $\barZZ_\lambda$ of \eqref{e:formal.Z.lambda} is a formal definition of the ``$\lambda$-tilted cluster partition function'' introduced in \eqref{e:intro.Z.lambda}. In a sequence \eqref{e:bij} of combinatorial mappings, we have produced in \eqref{e:Z.lambda} a mathematically well-defined quantity $\ZZ_{\lambda,T}$ which lower bounds $\barZZ_\lambda$ (Proposition~\ref{p:cluster.size.col}), and will be much more tractable thanks to the product formula for cluster sizes (Lemma~\ref{l:cluster.size.col}). The lower bound of Theorem~\ref{t:main} is based on the second moment method applied to $\ZZ_{\lambda,T}$. In preparation for the moment calculation, we conclude the current section by discussing some simplifications obtained by averaging over the literals of the \textsc{nae-sat} instance.

\subsection{Averaging over edge literals}Our eventual purpose is to calculate $\E\ZZ_{\lambda,T}$ and $\E[(\ZZ_{\lambda,T})^2]$, where $\E$ is expectation over the \textsc{nae-sat} instance $\glit$. Recall that $\glit=(\graph,\ulit)$ where $\graph=(V,F,E)$ is the graph without the edge literals $\ulit$. Then $\E\ZZ_{\lambda,T}=\E(\E(\ZZ_{\lambda,T}\,|\,\graph))$ where
	\[\E(\ZZ_{\lambda,T}\,|\,\graph)
	=\sum_{\usi\in(\tcols)^E}
	\E(\wt_\glit(\usi)^\lambda\,|\,\graph)\,.\]
For any $l\ge1$ and any function $g:\set{\zro,\one}^l\to\R$, let $\Elit g$ denote the average value of $g(\ulit)$ over all $\ulit\in\set{\zro,\one}^l$. For any $\usi\in\COLS^E$, we have $\E(\wt_\glit(\usi)^\lambda\,|\,\graph)= \avwt_\graph(\usi)^\lambda$ where (compare~\eqref{e:def.col.wt.T})
	\beq\label{e:first.moment.avg}
	\avwt_\graph(\usi)
	\equiv \prod_{v\in V}
	\dPhi(\usi_{\delta v})
	\prod_{a\in F}\hPhi
	(\usi_{\delta a})
	\prod_{e\in E} \ePhi(\sigma_e)\,,\quad
	\hPhi(\usi)
	\equiv \Big(
	\Elit[\hPl(\usi\oplus\ulit)^\lambda]
	\Big)^{1/\lambda}\eeq
--- that is to say, even after averaging over $\ulit$, the contribution of each $\usi\in\COLS^E$ is still given by a product formula. This means that $\E(\ZZ_{\lambda,T}\,|\,\graph)$ is the partition function of a ``factor model'':

\begin{dfn}[factor model] \label{d:fm} On a bipartite graph $\graph=(V,F,E)$, the \bemph{factor model} specified by $g\equiv(\dot{g},\hat{g},\bar{g})$ is the probability measure $\nu_\graph$ on configurations $\vec{\xi}\in\mathscr{X}^E$ defined by
	\[\nu_\graph(\vec{\xi}) = \f{1}{Z}
	\prod_{v\in V} \dot{g}(\vec{\xi}_{\delta v})
	\prod_{a\in F} \hat{g}(\vec{\xi}_{\delta a})
	\prod_{e\in E}
	\bar{g}(\xi_e),\]
with $Z$ the normalizing constant.\end{dfn}

A further observation is that for $\glit=(\graph,\ulit)$ and $\usi\in\COLS^E$, as we go over all possibilities of $\ulit$ while keeping $\graph$ fixed, the weight $\wt_\glit(\usi)$ (the size of the cluster encoded by $\usi$ on $\glit$, as given by \eqref{e:def.col.wt.T}) does not take more than one positive value. In other words, we can extract the cluster size without referring to the edge literals. The precise statement is as follows:

\begin{lem}\label{l:dont.need.literals} The function $\hPl$ of \eqref{e:def.hPhi.literals} can be factorized as $\hPl(\usi\oplus\ulit)=\hI(\usi\oplus\ulit)\hF(\usi)$ for
	\[\hF(\usi)\equiv\begin{cases}
		1 &\textup{if $\usi\in
		\set{\redz,\redo,\bluz,\bluo}^k$,}\\
		\displaystyle 
		\f{\hz(\hsi_j)}{\ephi(\sigma_j)}
		&\textup{if $\usi\in\COLS^k$ with $\sigma_j\in\set{\fcl}$.}
		\end{cases}\]
As a consequence, the function of \eqref{e:first.moment.avg} satisfies $\hPhi(\usi)^\lambda=\hat{v}(\usi)\hF(\usi)^\lambda$ where $\hat{v}(\usi)\equiv \Elit[\hI(\usi\oplus\ulit)]$.

\begin{proof} For $\usi\in\COLS^k$ abbreviate $\set{\sigma_i}\equiv\set{\sigma_1,\ldots,\sigma_k}$. If $\set{\sigma_i}\subseteq\set{\redz,\redo,\bluz,\bluo}$, then the definition~\eqref{e:def.hPhi.literals} implies $\hPl(\usi\oplus\ulit)=\hI(\usi\oplus\ulit)$ for all $\ulit$, so the factorization holds with $\hF(\usi)\equiv1$. If $\set{\sigma_i}$ nontrivially intersects both $\set{\redz,\redo}$ and $\set{\fcl}$, then $\hPl(\usi\oplus\ulit)=\hI(\usi\oplus\ulit)=0$ for all $\ulit$, so we can set $\hF(\usi)$ arbitrarily. It remains to consider the case where $\set{\sigma_i}$ nontrivially intersects $\set{\fcl}$ but does not intersect $\set{\redz,\redo}$. Recalling the discussion around~\eqref{e:def.ST.tau}, this means that $\dta_i\equiv\dta(\sigma_i)\in\dMM\setminus\set{\star}$ is well-defined for all $i$ --- if $\sigma_i\in\set{\fcl}$ then $\dta_i=\dsi_i\in\dMM\setminus\set{\zro,\one,\star}$, and if $\sigma_i=\blu_{\bx}$ then $\dta_i=\bx\in\set{\zro,\one}$. Following~\eqref{e:messageconfig.loceq}, given any $\ulit\in\set{\zro,\one}^k$, let us define $\hta_{\ulit,i}\equiv \lit_i\oplus\hatT((\dta_j\oplus\lit_j)_{j\ne i} )$. If $\hI(\usi\oplus\ulit)=1$, then it follows from \eqref{e:def.hm.hta.z}, \eqref{e:phi.functions}, and \eqref{e:def.hPhi.literals} that
	\begin{align}\nonumber
	\hPl(\usi\oplus\ulit)
	&=\hphi( \vec{\dta}\oplus\ulit )
	=\sum_{\vec{\bx}\in\set{\zro,\one}^k}
	I^\textsc{nae}(\vec{\bx}\oplus\ulit)
	\prod_{j=1}^k
	[\dmsg(\dta_j)](\bx_j)\\
	&=\hz(\hta_{\ulit,i})
	\sum_{\bx_i} [\dmsg(\dta_i)](\bx_i)
		\cdot
		[\hmsg(\hta_{\ulit,i})](\bx_i)
	=\f{\hz(\hta_{\ulit,i})}
	{\ephi(\dta_i,\hta_{\ulit,i})}\,.
	\label{e:product.factors.z}
	\end{align}
We will have $\hI(\usi\oplus\ulit)=1$ if and only if it holds for all $1\le i\le k$ that $\hta_{\ulit,i}$ is compatible with $\sigma_i$, in the sense that $\ST(\dta_i,\hta_{\ulit,i})=\sigma_i$. In particular, if $\sigma_i\in\set{\fcl}$ (and we assumed $\usi$ has at least one such entry), we must have $(\dta_i,\hta_{\ulit,i})=(\dsi_i,\hsi_i)$. It follows that for any $\usi\in\COLS^k$ having at least one entry in $\set{\fcl}$, we can define $\hF(\usi)\equiv\hz(\hsi_i)/\ephi(\sigma_i)$ for any $i$ where $\sigma_i\in\set{\fcl}$. This completes the proof.\end{proof}
\end{lem}

\begin{cor}\label{c:dont.need.literals} On a bipartite graph $\graph=(V,F,E)$, suppose $\usi\in\COLS^E$ satisfies $\dot{I}(\usi_{\delta v})=1$ for all $v\in V$. Then, for $\glit=(\graph,\ulit)$, it follows from \eqref{e:def.col.wt.T} that 
	\[\wt_\glit(\usi)
	=\bigg\{\prod_{a\in F}
	\hI((\usi\oplus\ulit)_{\delta a})\bigg\}
	\maxwt_\graph(\usi)\,,\quad
	\maxwt_\graph(\usi)\equiv
	\prod_{v\in V}\dPhi(\usi_{\delta v})
	\prod_{a\in F}\hF(\usi_{\delta a})
	\prod_{e\in E}\ePhi(\sigma_e)\,.\]
Combining with \eqref{e:first.moment.avg} gives, with $\hat{v}$ as defined by Lemma~\ref{l:dont.need.literals},
	\[\avwt_\graph(\usi)^\lambda
	=\bm{p}_\graph(\usi)
	\maxwt_\graph(\usi)^\lambda\,,\quad
	\bm{p}_\graph(\usi)
	\equiv\E\bigg[
	\prod_{a\in F}
	\hI((\usi\oplus\ulit)_{\delta a})
	\,\bigg|\,\graph\bigg]
	=\prod_{a\in F}
	\hat{v}(\usi_{\delta a})\,.\]

\begin{proof} Immediate consequence of Lemma~\ref{l:dont.need.literals}.\end{proof}\end{cor}

In the notation of Definition~\ref{d:fm}, the conditional first moment $\E(\ZZ_{\lambda,T}\,|,\graph)$ is the partition function of the factor model with specification $(\dPhi,\hPhi,\ePhi)^\lambda$ restricted to the alphabet $\tcols$. Similarly, the conditional second moment $\E[(\ZZ_{\lambda,T})^2\,|\,\graph]$ is the partition function of the factor model on the alphabet $(\tcols)^2$ with specification $(\dPhi_2,\hPhi_2,\ePhi_2)^\lambda$, where $\dPhi_2\equiv\dPhi\otimes\dPhi$, $\ePhi_2\equiv\ePhi\otimes\ePhi$, and for any $\usi\equiv(\usi^1,\usi^2)\in\COLS^{2k}$ we have
	\[\hPhi_2(\usi)\equiv \bigg(\Elit\Big[
		\hPl(\usi^1\oplus\ulit)^\lambda
		\hPl(\usi^2\oplus\ulit)^\lambda
		\Big]\bigg)^{1/\lambda}
	=\hat{v}_2(\usi)^{1/\lambda}
	(\hF\otimes\hF)(\usi)\,,\]
for $\hat{v}_2(\usi)\equiv\Elit[\hI(\usi^1\oplus\ulit)\hI(\usi^1\oplus\ulit)]$ (by Corollary~\ref{c:dont.need.literals}). We emphasize that $\hPhi,\hPhi_2$ both depend on $\lambda$, although we suppress it from the notation. Moreover, $\hPhi_2\ne\hPhi\otimes\hPhi$ since $\usi^1$ and $\usi^2$ are coupled through their interaction with the same literals $\ulit\in\set{\zro,\one}^k$. Lastly, we have written $\usi$ in the first moment and $\usi\equiv(\usi^1,\usi^2)$ in the second moment --- this is a deliberate abuse of notation, which allows us to treat the two cases in a unified manner. To distinguish the cases we shall refer to the ``first-moment'' or ``single-copy'' model, versus the ``second-moment'' or ``pair'' model. We turn next to the analysis of these models.

\section{Proof outline}\label{s:outline.lbd}

Having formally set up our combinatorial model of \textsc{nae-sat} solution clusters (Section~\ref{s:comb}), we now give a more detailed outline for the (first and second) moment calculation that proves the lower bound of Theorem~\ref{t:main}. (As we mentioned before, the upper bound of Theorem~\ref{t:main} is proved by an interpolation argument which builds on prior results in spin glass theory \cite{MR1972121,MR2095932,MR3161470}. It does not involve the combinatorial model or the moment method, and is deferred to Appendix~\ref{appx:ubd}.) 

\subsection{Empirical measures and moments} We use standard multi-index notations in what follows --- in particular, for any ordered sequence $z=(z_1,\ldots,z_l)$ of nonnegative integers summing to $n$, we denote
	\[\binom{n}{z}
	\equiv n!\bigg/ \prod_{i=1}^l z_i!\,.\]
If $\pi$ is any nonnegative measure on a discrete space, write $\ent(\pi) = -\langle\pi,\log\pi\rangle$ for its Shannon entropy. It follows from Stirling's formula that for any fixed $\pi$, in the limit $n\to\infty$ we have
	\[\binom{n}{n\pi}\asymp 
	\f{\exp\{n\mathcal{H}(\pi)\}}
	{n^{(|\supp\pi|-1)/2}}\,.\]
On a bipartite graph $\graph$, we will summarize colorings $\usi$ according to some ``local statistics,'' as follow: 

\begin{dfn}[empirical measures]\label{d:empirical}Given a bipartite graph $\graph=(V,F,E)$ and $\usi\in\COLS^E$, define
	{\setlength{\jot}{0pt}\begin{alignat*}{2}
	\dH(\vec{\dsi})
	&=|\set{v\in V:\usi_{\delta v}
	=\vec{\dsi}}|/|V|
	&&\quad\text{for }
	\vec{\dsi}\in\COLS^d,\\
	\hH(\vec{\hsi})
	&= |\set{a\in F:
	\usi_{\delta a}
	=\vec{\hsi}}|/|F|
	&&\quad\text{for }
	\vec{\hsi}\in\COLS^k,\\
	\eH(\sigma)&=|\set{e\in E:
		\sigma_e=\sigma}|/|E|
	&&\quad\text{for }
	\sigma\in\COLS;
	\end{alignat*}}%
The triple $H\equiv H(\graph,\usi)\equiv (\dH,\hH,\eH)$ is the \bemph{empirical measure} of $\usi$ on $\graph$.\end{dfn}

Recall from \eqref{e:first.moment.avg} that $\avwt_\graph(\usi)^\lambda$ is the contribution to $\E(\ZZ_\lambda\,|\,\graph)$ from $\usi\in\COLS^E$. We saw in Corollary~\ref{c:dont.need.literals} that $\avwt_\graph(\usi)^\lambda=\bm{p}_\graph(\usi)\maxwt_\graph(\usi)^\lambda$ where $\bm{p}_\graph(\usi)$ is the probability (conditional on $\graph$) that $\usi$ is a valid coloring on $(\graph,\ulit)$; and $\maxwt_\graph(\usi)$ is the size of the cluster encoded if $\usi$ is valid. Now all these quantities can be expressed solely in terms of $H=H(\graph,\usi)$: we have $\bm{p}_\graph(\usi)=\exp(n\logp(H))$ and $\maxwt_\graph(\usi)=\exp(n\size(H))$ where
	\begin{align*}
	\logp(H) &\equiv
	(d/k)
	\langle \log\hat{v},\hH\rangle
	=(d/k)
	\sum_{\usi\in\COLS^k}
	\hH(\usi) \log\hat{v}(\usi) 
	\,,\\
	\size(H)
	&\equiv
	\langle\log\dPhi,\dH\rangle
	+(d/k)\langle
	\log \hF,
	\hH\rangle
	+d\langle\log\ePhi,\eH\rangle\,.\end{align*}
Given $\glit=(\graph,\ulit)$, let $\ZZ_{\lambda,T}(H)$ be the contribution to $\ZZ_{\lambda,T}$ from colorings $\usi\in(\tcols)^E$ such that $H(\graph,\usi)=H$. In what follows we will often suppress the dependence on $\lambda$ and $T$, and write simply $\ZZ\equiv\ZZ_{\lambda,T}$. 

\begin{dfn}[simplex] \label{d:simplex} For $d,k,T$ fixed, the \bemph{simplex of empirical measures} is the space $\simplex\equiv\simplex(T)$ of triples $H\equiv(\dH,\hH,\eH)$ satisfyng the following conditions: $\dH$ is a probability measure supported within the set of $\usi\in(\tcols)^d$ such that $\dot{I}(\usi)=1$; $\hH$ is a probability measure supported within the set of $\usi\in(\tcols)^k$ such that $\hat{v}(\usi)$ is positive; and both $\dH$ and $\hH$ must have marginal $\eH$, that is,
	\beq\label{e:H.edge.marginal}
	\f1d\sum_{\usi\in\COLS^d}
		\dH(\usi)\sum_{i=1}^d \Ind{\sigma_i=\sigma}
	=\eH(\sigma)
	=\f{1}{k}\sum_{\usi\in\COLS^k}
	\hH(\usi)\sum_{j=1}^k
	\Ind{\sigma_j=\sigma}\eeq
for all $\sigma\in\COLS$. It follows that $\eH$ is a probability measure supported on $\tcols$.
\end{dfn}

It follows from Corollary~\ref{c:dont.need.literals} that if $\E$ is expectation over a $(d,k)$-regular \textsc{nae-sat} instance on $n$ variables, then $\E\ZZ(H)$ is positive if and only if $H\in\simplex$ and $(n\dH,m\hH)$ is integer-valued. For such $H$, it follows from the definition of the random regular \textsc{nae-sat} graph that 
	\beq\label{e:graph.expectation}
	\E\ZZ(H)
	=\bigg[\bigg\{\binom{n}{n\dH}
	\binom{m}{m\hH} \bigg/
	\binom{nd}{nd\eH}\bigg\}
	\exp\{n\logp(H)\}\bigg]
	\cdot \exp\{ n\lambda\size(H) \}\,.\eeq
In \eqref{e:graph.expectation}, the first factor (in square brackets) is the expected number $\E\ZZ_{\lambda=0,T}$ of valid colorings with empirical profile $H$. The remaining factor $\exp\{ n\lambda\size(H) \}$ is explained by the fact that any such coloring encodes a cluster of size $\exp\{n\size(H)\}$. By Stirling's formula, in the limit $n\to\infty$ (with $T$ fixed),
	\[\E\ZZ_{\lambda=0,T}
	\asymp \f{\exp\{n
	[\ent(\dH) + (d/k)\ent(\hH)
		-d\ent(\eH)
		+\logp(H)]\}}
		{n^{\wp(H)/2}}
	\equiv \f{\exp\{n\SIGMA(H)\}}
		{n^{\wp(H)/2}}\]
where $\wp(H)\equiv|\supp\dH|+|\supp\hH|-|\supp\eH|-1$, and the exponential rate $\SIGMA(H)$ is a formal analogue of the ``cluster complexity'' function $\Sigma(s)$ appearing in \eqref{e:intro.cluster.complexity}. In analogy with \eqref{e:intro.legendre.dual} we let 
	\beq\label{e:formal.defn.bF}
	\bF\equiv\bF_{\lambda,T}
	\equiv \SIGMA(H)+\lambda\size(H)\,.
	\eeq
Then altogether the first moment can be estimated as 
	\beq\label{e:SIGMA}
	\E\ZZ(H)
	\asymp \bigg(\f{\exp\{n\SIGMA(H)\}}{n^{\wp(H)/2}}\bigg)
	\exp\{n\lambda\size(H)\}
	=\f{\exp\{n\bF(H)\}}{n^{\wp(H)/2}}\,.\eeq
Note that $\asymp$ hides a dependence on $T$, since we keep $T$ fixed throughout our moment analysis. 

\subsection{Outline of first moment}For any subset of empirical measures $\mathbf{H}\subseteq\simplex$, we write $\usi\in\mathbf{H}$ to indicate that $H(\graph,\usi)\in\mathbf{H}$, and write $\ZZ(\mathbf{H})\equiv\ZZ_{\lambda,T}(\mathbf{H})$ for the contribution to $\ZZ_{\lambda,T}$ from colorings $\usi\in \mathbf{H}$. It then follows from \eqref{e:SIGMA} that 
	\[\E\ZZ(\mathbf{H}) 
	=\sum_{H\in\mathbf{H}}\E\ZZ(H) 
	=n^{O(1)}
	\exp\Big\{n\max\set{\bF(H):H\in\mathbf{H}}
	\Big\}\,,\]
for $\bF$ as in \eqref{e:formal.defn.bF}. Thus, calculating the first moment $\E\ZZ$ essentially reduces to the problem of maximizing $\bF$ over $\simplex$. The physics theory suggests that $\bF$ is uniquely maximized at a point $H_\star\in\simplex$ which is given explicitly in terms of a replica symmetric fixed point for the $\lambda$-tilted $T$-coloring model. (Recall from \S\ref{ss:intro.onersb} that in the original \textsc{nae-sat} model, the replica symmetric fixed point was described by the measure $\msg=\textup{unif}(\set{\zro,\one})$. In the coloring model, with spins $\sigma\equiv(\dsi,\hsi)\in\tcols$, the replica symmetric fixed point will be characterized by a measure $\dq$ on the space $\tcols$ of possible values for $\dsi$.)

There are several obstacles to the rigorous moment computation. From a physics perspective, the replica symmetric fixed point of the $\lambda$-tilted coloring model at $T=\infty$ is equivalent to the fixed point described by Proposition~\ref{p:drec_fixpoint}, which was used to define the \textsc{1rsb} prediction \eqref{e:drec_Sigma}.  Mathematically, however, we work with $T$ finite so that $\simplex$ has finite dimension and $\wp(H)$ is defined. Therefore we need to explicitly construct a replica symmetric fixed point at finite $T$, and use it to define $ H_\star=H_{\lambda,T}\in\simplex$ (Definition~\ref{d:Hstar} below). We must then take $T\to\infty$ and show that the limit matches the fixed point of Proposition~\ref{p:drec_fixpoint}, so that  $\bF(H_\star)= \bF_{\lambda,T}(H_{\lambda,T})$ converges as $T\to\infty$ the \textsc{1rsb} prediction \eqref{e:drec_Sigma}. The construction of the fixed point at finite $T$ is stated in Proposition~\ref{p:contract} below, and proved in Appendix~\ref{appx:contract}. The correspondence with \eqref{e:drec_Sigma} in the $T=\infty$ limit is stated in Proposition~\ref{p:drec_equiv.outline} below, and proved in Appendix~\ref{appx:onersb}. 

A more difficult problem is to show that $\bF$ is in fact maximized at $H_\star$. The function $\bF$ is generally not convex, and must be optimized over a space $\simplex$ whose dimension grows with $d,k,T$. Moreover, an analogous but even more difficult optimization must be solved to compute the second moment $\E(\ZZ^2)$. The main part of this analysis is carried out in Sections~\ref{s:reduce.to.tree}~and~\ref{s:tree.opt}. In the remainder of this section we make some preparatory calculations and explain how the pieces will be fit together to prove the main result Theorem~\ref{t:main}. Recall from Remark~\ref{r:rmks} that we can restrict consideration to $\alpha$ satisfying \eqref{e:regime.alpha}. In this regime, we make \textit{a~priori} estimates to show that the optimal $H$ satisfy some basic restrictions. Abbreviate $\set{\red}\equiv\set{\redz,\redo}$, recall $\set{\fcl}\equiv\COLS\setminus\set{\redz,\redo,\bluz,\bluo}$, and let
	\begin{align}\nonumber
	\nbd_\circ&\equiv\bigg\{H\in\simplex:
	\max \set{ \eH(\fcl),
		\eH(\red) } \le \f{7}{2^k}
		\bigg\}\,\,,\\
	\nbd
	&=\bigg\{ H\in\nbd_\circ : 
	\lone{H-H_\star} \le 
	\f1{n^{1/3}}
	\bigg\}
	\subseteq \nbd_\circ\,,
	\label{e:N.zero}
	\end{align}
where $\lone{\cdot}$ denotes the $\ell^1$ norm throughout this paper. We next show that empirical measures $H\notin \nbd_\circ$ give a negligible contribution to the first moment:

\begin{lem}\label{l:restrict.H} For $k\ge k_0$, $\alpha\equiv d/k$ satisfying \eqref{e:regime.alpha}, and $0\le\lambda\le1$, $\E\ZZ(\simplex\setminus\nbd_\circ)$ is exponentially small in $n$.

\begin{proof} Let $Z^\fcl$ count the \textsc{nae-sat} solutions $\vec{\bx}\in\set{\zro,\one}^V$ which map --- via coarsening and the bijection \eqref{e:bij} --- to warning configurations $\vec{y}$ with more than $7/2^k$ fraction of edges $e$ such that $\dot{y}_e=\hat{y}_e=\fre$. Similarly, let $ Z^\red$ count \textsc{nae-sat} solutions $\vec{\bx}$ mapping to warning configurations $\vec{y}$ with more than $7/2^k$ fraction of edges $e$ such that $\hat{y}_e\in\set{\zro,\one}$. It follows from Proposition~\ref{p:cluster.size.col} that for any $0\le \lambda\le1$ we have $\ZZ(\simplex\setminus\nbd_\circ)\le Z^\fcl + Z^\red$. For $\alpha$ satisfying \eqref{e:regime.alpha}, $\E Z^\fcl$ is exponentially small in $n$ by \cite[Propn.~2.2]{MR3440193}. As for $Z^\red$, let us say that an edge $e\in E$ is \bemph{blocked} under $\vec{\bx}\in\set{\zro,\one}^V$ if $\lit_e \oplus x_{v(e)}= \one\oplus \lit_{e'} \oplus x_{v(e')}$ for all $e'\in\delta a(e) \setminus e$. Note that if $\vec{\bx}$ maps to $\vec{y}$, the only possibility for $y_e\in\set{\redz,\redo}$ is that $e$ was blocked under $\vec{\bx}$. (The converse need not hold.) If we condition on $\vec{\bx}$ being a valid \textsc{nae-sat} solution, then each clause contains a blocking edge independently with chance $\theta = 2k/(2^k-2)$; note also that a clause can contain at most one blocking edge. It follows that
	\[\E Z^\red \le (\E Z)
	\P\bigg(\mathrm{Bin}( m, \theta ) \ge 7 nd/2^k\bigg)\,.\]
This is exponentially small in $n$ by a Chernoff bound together with the trivial bound $\E Z \le 2^n$.\end{proof}\end{lem}

We assume throughout what follows that $k\ge k_0$, $\alpha$ satisfies \eqref{e:regime.alpha}, and $0\le \lambda\le1$. In this regime, Lemma~\ref{l:restrict.H} tells us that $\max\set{\bF(H) : H\notin\nbd_\circ}$ is negative. On the other hand, we shall assume that the global maximum of $\bF$ is nonnegative, since otherwise $\E\ZZ$ is exponentially small in $n$ and there is nothing to prove. From this we have that any maximizer $H$ of $\bF$ must lie in $\nbd_\circ$. In Sections~\ref{s:reduce.to.tree}~and~\ref{s:tree.opt} we develop a more refined analysis to solve the optimization problem for $\bF$ restricted to $\nbd_\circ$:

\begin{ppn}[\textit{proved in Section~\ref{s:tree.opt}}] \label{p:first.mmt} Assuming the global maximum of $\bF$ is nonnegative, the unique maximizer of $\bF$ is an explicitly characterized point $H_\star$ in the interior of $\nbd_\circ$. Moreover, there is a positive constant $\epsilon=\epsilon(k,\lambda,T)$ so that for all $\lone{H-H_\star}\le\epsilon$ we have $\bF(H) \le \bF(H_\star)-\epsilon\lone{H-H_\star}^2$.\end{ppn}

A consequence of the above is that we can compute the first moment of $\ZZ$ up to constant factors. In the following, let $\dot{\wp}\equiv\dot{\wp}(T)$ count the number of $d$-tuples $\usi\in(\tcols)^d$ for which $\dot{I}(\usi)>0$. Let $\hat{\wp}\equiv\hat{\wp}(T)$ count the number of $k$-tuples $\usi\in(\tcols)^k$ for which $\hat{v}(\usi)>0$. Let $\bar{\wp}\equiv|\tcols|$, and denote $\wp\equiv\dot{\wp}+\hat{\wp}-\bar{\wp}-1$. Recall from \eqref{e:N.zero} the definition of $\nbd$. 

\begin{cor}\label{c:first.mmt} The coloring partition function $\ZZ\equiv\ZZ_{\lambda,T}$ has first moment $\E\ZZ\asymp \exp\{ n \bF(H_\star) \}$. Moreover the expectation is dominated by $\nbd$ in the sense that $\E\ZZ(\nbd)=(1-o(1)) \E\ZZ$.

\begin{proof} Define the $\bar{\wp}\times\dot{\wp}$ matrix $\dot{M}$ with entries
	\[\dot{M}(\sigma',\usi)
	\equiv
	\sum_{i=1}^d \Ind{\sigma_i=\sigma'}\,,\quad
	\sigma'\in\tcols \textup{ and }
	\usi\in(\tcols)^d\cap(\supp\dot{I})\,.\]
Similarly define the $\bar{\wp}\times\hat{\wp}$ matrix $\hat{M}$ with entries
	\[\hat{M}(\sigma',\usi)
	=\sum_{i=1}^k \Ind{\sigma_i=\sigma'}\,,\quad
	\sigma'\in\tcols\textup{ and }
	\usi\in(\tcols)^k\cap(\supp\hat{v})\,.\]
Lastly define the $\bar{\wp}\times(\dot{\wp}+\hat{\wp})$ matrix $M \equiv \begin{pmatrix}\dot{M} & -\hat{M}\end{pmatrix}$. It follows from the discussion after Definition~\ref{d:simplex} that $\E\ZZ(H)$ is positive if and only if (i) $\dH$ and $\hH$ are nonnegative; (ii) $\langle \mathbf{1},\dH\rangle=1$; (iii) $(k\dH,d\hH)$ lies in the kernel of $M$; and (iv) $(n\dH,m\hH)$ is integer-valued. Conditions (i)--(iii) are equivalent to $H\in\simplex$. One can verify that the matrix $M$ is of full rank, from which it follows that the space of vectors $(\dH,\hH)$ satisfying the conditions (ii) and (iii) has dimension $\wp$. In Lemma~\ref{l:full.rank.M} we will show that $M$ satisfies a stronger condition, which implies that the space of $(\dH,\hH)$ satisfying (ii)--(iv) is an affine transformation of $(n^{-1}\mathbb{Z})^\wp$, where the coefficients of the transformation are uniformly bounded. Then, by substituting the result of Proposition~\ref{p:first.mmt} in to \eqref{e:SIGMA}, we conclude
	\[\E\ZZ\asymp 
	\sum_{z\in (n^{-1}\mathbb{Z})^\wp}
	\f{\exp\{ n[\bF(H_\star)- \Theta(\lone{z}^2)
		]\}}
		{n^{\wp/2}} 
		\asymp \exp\{n\bF(H_\star)\}\]
Empirical measures $H\notin\nbd$ correspond to vectors $z\in (n^{-1}\mathbb{Z})^\wp$ with norm $\lone{z}\ge n^{-1/3}$. These give a negligible contribution to the above sum which proves the second claim $\E\ZZ(\nbd)=(1-o(1))\E\ZZ$.\end{proof}
\end{cor}

\subsection{Outline of second moment}
\label{ss:outline.second.mmt}

By a similar calculation as above, calculating the second moment $\E(\ZZ^2)$ reduces to the problem of maximizing a function $\bF_2\equiv\bF_{2,\lambda,T}$ over a space $\simplex_2$ of \bemph{pair empirical measures}. In fact we will calculate the second moment not of $\ZZ$ itself, but rather of a more restricted random variable $\sepZZ\le\ZZ$, defined below. This leads to a more tractable analysis, as we now explain.

Concretely, $\simplex_2$ is the space of triples $H=(\dH,\hH,\eH)$ satisfying the following conditions: $\dH$ is a probability measure on $(\tcols)^{2d}\cap(\supp\dot{I})^2$, $\hH$ is a probability measure on $(\tcols)^{2d}\cap(\supp\hat{v}_2)$, and $\eH$ is a probability measure on $(\tcols)^2$ which can be obtained from both $\dH$ and $\hH$ as the edge marginal~\eqref{e:H.edge.marginal}. Repeating the derivation of \eqref{e:SIGMA} in the second moment setting, we have the following. For any $H\in\simplex_2$, the expected number of valid coloring pairs $(\usi,\usi')\in(\tcols)^{2E}$ of type $H$ is
	\[\E(\ZZ_{\lambda=0,T})^2(H)
	\asymp \f{\exp\{n[\ent(\dH) + (d/k)\ent(\hH)-d\ent(\eH)
		+\logp_2(H)]\}}{n^{\wp(H)/2}}
	\equiv\f{\exp\{n\SIGMA_2(H)\}}{n^{\wp(H)/2}}\]
Given a pair empirical measure $H\in\simplex_2$, take the marginal on the first element of the pair to define the \bemph{first-copy marginal} $H^1\in\simplex$; define likewise the \bemph{second-copy marginal} $H^2\in\simplex$. The contribution to $\ZZ^2$ from any valid pair is given by $\exp\{n\lambda \size_2(H)\}$ where $\size_2(H)\equiv\size(H^1)+\size(H^2)$. Thus $\bF_2$ is given explicitly by
	\beq\label{e:def.PSI.two}
	\E\ZZ^2(H)
	\asymp\f{\exp\{n\bF_2(H)\}}{n^{\wp(H)/2}}
	\equiv
	\f{\exp\{n
	[\SIGMA_2(H)+\lambda\size_2(H)]
	\}}{n^{\wp(H)/2}}\eeq
(cf.\ \eqref{e:formal.defn.bF} and \eqref{e:SIGMA}). In view of Corollary~\ref{c:first.mmt}, it would suffice for our purposes to calculate the second moment of $\ZZ(\nbd)$ rather than $\ZZ$, which amounts to maximizing $\bF_2$ on the restricted set
	\[\nbd_2
	\equiv\Big\{
	H\in\simplex_2 :
	H^1\in\nbd\textup{ and }
	H^2\in\nbd
	\Big\}\,.\]
Following~\cite{MR3436404} we can simplify the analysis by a further restriction, as follows:

\begin{dfn}[separability] \label{d:sep} If $\usi,\usi'$ are valid colorings on $\glit$, define their \bemph{separation} $\SEP(\usi,\usi')$ to be the fraction of variables where their corresponding frozen configurations (from Lemmas~\ref{l:bij.frozen.message.configs}~and~\ref{l:bij.message.cols}) differ. Write $\usi'\succcurlyeq\usi$ if the frozen configuration of $\usi'$ has more free variables than that of $\usi$. We say that a coloring is \bemph{separable} if $\usi\in\nbd$ (recall this means $H(\graph,\usi)\in\nbd$) and
	\[|\set{\usi'\in\nbd:
	\usi'\succcurlyeq\usi
	\textup{ and }
	\SEP(\usi,\usi')
	\notin I_\textup{se}
	}|
	\le \exp\{ (\log n)^4 \},\]
where $I_\textup{se}\equiv [(1-k^4/2^{k/2})/2, (1+k^4/2^{k/2})/2]$ and it is implicit that both $\usi,\usi'$ must both be valid on $\glit$. Let $\sepZZ\equiv\sepZZ_{\lambda,T}$ be the contribution to $\ZZ(\nbd)$ from separable colorings.
\end{dfn}

\begin{ppn}[\textit{proved in Appendix~\ref{appx:sep}}]
\label{p:sep} The first moment is dominated by separable colorings in the sense that $\E \sepZZ = (1-o(1)) \E\ZZ(\nbd)$.
\end{ppn}

We will apply the second moment method to lower bound $\sepZZ$; the result will follow since $\sepZZ\le\ZZ(\nbd)\le\ZZ$. For $H\in\simplex_2$, all pairs $(\usi,\usi')\in H$ must have the same separation $\SEP(\usi,\usi')$, which we can thus denote as $\SEP(H)$. Partition $\nbd_2$ into $\nbd_\textup{se}\equiv\set{H\in\nbd_2:\SEP(H) \in I_\textup{se}}$ (the near-uncorrelated regime) and $\nbd_\textup{ns}\equiv\nbd_2\setminus\nbd_\textup{se}$ (the correlated regime). Denote the corresponding contributions to $\sepZZ^2$ by $\sepZZ^2(\nbd_\textup{se})$ and $\sepZZ^2(\nbd_\textup{ns})$.

\begin{cor}\label{c:sep}
For separable colorings, the second moment contribution from the correlated regime $\nbd_\textup{ns}$ is bounded by
	$\E[\sepZZ^2(\nbd_\textup{ns})]
	\le\exp\{n \lambda \size(H_\star) + o(n) \}
	\,\E \sepZZ$.

\begin{proof} By the symmetry between the roles of $\usi$ and $\usi'$,
and the definition of separability, we have
	\[\sepZZ^2(\nbd_\textup{ns})
	\le 2\sum_{\substack{(\usi,\usi')\in\nbd_\textup{ns},\\
		\usi\textup{ separable}}}
	\Ind{\usi'\succcurlyeq\usi}
	\wt_{\glit,T}(\usi)^\lambda
	\wt_{\glit,T}(\usi')^\lambda\\
	\le 
	\exp\{ n\size(H_\star)\lambda+o(n) \}
	\sepZZ\,.\]
Taking expectations gives the claim.\end{proof}
\end{cor}

We then conclude the second moment calculation by computing $\E(\ZZ^2(\nbd_\textup{se}))$, which is an upper bound on $\E(\sepZZ^2(\nbd_\textup{se}))$. Therefore we must maximize the function $\bF_2$ on $\nbd_\textup{se}$. As in the first moment, the physics theory suggests that the unique maximizer of $\bF_2$ on $\nbd_\textup{se}$ is given by a specific pair empirical measure $H_\bullet$ which is defined in terms of $H_\star$. In the \textsc{nae-sat} model there is a small complication in this definition: we will have $\dH_\bullet=\dH_\star\otimes\dH_\star$ and $\eH_\bullet=\eH_\star\otimes\eH_\star$, but $\hH_\bullet\ne\hH_\star\otimes\hH_\star$ because the first and second copies interact via the edge literals. It therefore requires a small calculation to argue that $\bF_2(H_\bullet)=2\bF_2(H_\star)$. We address this by giving a simple sufficient condition for $H\in\simplex_2$ to satisfy $\bF_2(H)=\bF(H^1)+\bF(H^2)$ where $H^j\in\simplex$ are its single-copy marginal.

\begin{lem}\label{l:clause.tensorize} Consider $H\in\simplex_2$ with single-copy marginals $H^1,H^2\in\simplex$. Suppose there are functions $g^1,g^2$ which are invariant to literals in the sense that $g^j(\usi^j)=g^j(\usi^j\oplus\ulit)$ for all $\usi^j\in(\tcols)^k$, $\ulit\in\set{\zro,\one}^k$, and
	\[\hH^j(\usi^j)=\hat{v}(\usi^j)g^j(\usi)\,,\quad
	\hH(\usi^1,\usi^2)
	=\hat{v}_2(\usi^1,\usi^2)
		\prod_{j=1,2}g^1(\usi^1)\,.\]
If in addition $\dH=\dH^1\otimes\dH^2$ and $\eH=\eH^1\otimes\eH^2$, then $\bF_2(H)=\bF(H^1)+\bF(H^2)$.

\begin{proof} Let $K^j(\usi^j,\ulit)\equiv\hI(\usi\oplus\ulit)g^j(\usi)/2^k$. This defines a probability measure on $(\tcols)^k\times \set{\zro,\one}^k$ where the marginal on $(\tcols)^k$ is $\hH^j$, and the marginal on $\set{\zro,\one}^k$ is uniform by the assumption on $g^j$. It follows that $K^j(\usi^j\,|\,\ulit) = \hI(\usi\oplus\ulit)g^j(\usi)$ and $\hH^j(\usi^j)=\Elit[K^j(\usi^j\,|\,\ulit)]$. Let $(X^1,X^2,L)$ be a random variable with law 
	\[\P((X^1,X^2,L)=(\usi^1,\usi^2,\ulit))
	=\f1{2^k}\prod_{j=1,2} K^j(\usi^j\,|\,\ulit)\,.\]
The marginal law of $L$ is uniform on $\set{\zro,\one}^k$, and the $X^j$ are conditionally independent given $L$. The marginal law of $(X^1,X^2)$ is $\hH$, and the marginal law of $X^j$ is $\hH^j$. The law of $L$ conditional on $X^j$ is uniform over $2^k\hat{v}(X^j)$ possibilities, whereas the law of $L$ conditional on $(X^1,X^2)$ is uniform over $2^k\hat{v}_2(X^1,X^2)$ possibilities. It follows that
	\begin{align*}
	\ent(\hH)
	&=\ent(X^1,X^2)
	=\ent(L)-\ent(L|X^1,X^2)
		+\sum_{j=1,2}\ent(X^j|L)
	=-\langle\hH,\log\hat{v}_2\rangle
		+\sum_{j=1,2}\ent(X^j|L)\\
	&= -\langle\hH,\log\hat{v}_2\rangle
		+\sum_{j=1,2}
		[\ent(\hH^j)+ \langle\hH^j,\log\hat{v}\rangle]\,.
	\end{align*}
Rearranging gives $\SIGMA_2(H)=\SIGMA(H^1)+\SIGMA(H^2)$, and the result follows.\end{proof}
\end{lem}

It will be clear from the explicit definition that the measure $H_\star$ of Proposition~\ref{p:first.mmt} can be expressed as $H_\star(\usi)=\hat{v}(\usi)g_\star(\usi)$ where $g_\star$ is invariant to literals. Let $H_\bullet=(\dH_\bullet,\hH_\bullet,\eH_\bullet)$ where
	\[\hH_\bullet(\usi^1,\usi^2)
	\equiv \hat{v}_2(\usi^1,\usi^2)
		\prod_{j=1,2} g_\star(\usi^j)\,,\]
$\dH_\bullet=\dH_\star\otimes\dH_\star$, and $\eH_\bullet=\eH_\star\otimes\eH_\star$. The following is the second moment analogue of Proposition~\ref{p:first.mmt}.

\begin{ppn}[\textit{proved in Section~\ref{s:tree.opt}}]\label{p:second.mmt} The unique maximizer of $\bF_2$ in $\nbd_\textup{se}$ is $H_\bullet$. Moreover, there is a positive constant $\epsilon=\epsilon(k,\lambda,T)$ so that for $\lone{H-H_\bullet}\le\epsilon$ we have $\bF_2(H) \le \bF_2(H_\bullet)-\epsilon\lone{H-H_\bullet}^2$. \end{ppn}

\begin{cor}\label{c:second.mmt.sep} For the coloring model, the second moment contribution from the near-uncorrelated regime $\nbd_\textup{se}$ is given by the estimate $\E[\ZZ^2(\nbd_\textup{se})]\asymp\exp\{2n\bF(H_\star)\}$.

\begin{proof} Recall from Corollary~\ref{c:first.mmt} the definition of $(\dot{\wp},\hat{\wp},\bar{\wp})$ for the single-copy model, and define $(\dot{\wp}_2,\hat{\wp}_2,\bar{\wp}_2)$ analogously for the pair model. Let $\wp_2 \equiv \dot{\wp}_2+\hat{\wp}_2-\bar{\wp}_2-1$. For any $H^1,H^2\in\nbd$, let $\nbd_{\textup{se},H^1,H^2}$ denote the set of $H\in \nbd_\textup{se}$ with single-copy marginals $H^1,H^2$. This is a space of dimension $\wp_2-2\wp$, and it follows from Proposition~\ref{p:second.mmt} and Lemma~\ref{l:full.rank.M} that 
	\[\E[\ZZ^2(\nbd_{\textup{se},H^1,H^2})]
	\asymp \f{\exp\{ n[
	\bF_2(H_\bullet)-\Theta(
	\|(H^1,H^2)-(H_\star,H_\star)\|^2)]\}}{n^\wp}\,.\]
Summing over $H^1,H^2\in\nbd$ then gives $\E[\ZZ^2(\nbd_\textup{se})]\asymp\exp\{ n\bF_2(H_\bullet) \}$, which in turn equals $\exp\{2n\bF(H_\star)\}$ by applying Lemma~\ref{l:clause.tensorize}.\end{proof}\end{cor}

\subsection{Conclusion of main result}

We now explain that the main theorem follows. We continue to assume, as we have done throughout the section, that $k\ge k_0$, $\alpha$ satisfies \eqref{e:regime.alpha}, and $0\le\lambda\le1$. The measure $H_\star$ of Proposition~\ref{p:first.mmt} depends on $\lambda$ and $T$, and we now make this explicit by writing $H_\star\equiv H_{\lambda,T}$.

\begin{cor}\label{c:whp} For any $0\le\lambda\le1$ and $T$ finite such that $\SIGMA(H_{\lambda,T})$ is positive, the separable contribution to $\ZZ_{\lambda,T}$ is well-concentrated about its mean:
	\[\lim_{\epsilon\downarrow0}
	\liminf_{n\to\infty}
	\P\bigg(\epsilon (\E\sepZZ) \le\sepZZ
	\le \f{\E\sepZZ}\epsilon\bigg)=1\,.\]

\begin{proof} The upper bound follows trivially from Markov's inequality, so the task is to show the lower bound. In the first moment, we have $\E\sepZZ\asymp\exp\{n\bF(H_\star)\}$ by Corollary~\ref{c:first.mmt} and Proposition~\ref{p:sep}. In the second moment, since $\sepZZ\le\ZZ$, we have $\E(\sepZZ^2) \le\E[\sepZZ^2(\nbd_\textup{ns})] +\E[\ZZ^2(\nbd_\textup{se})]$. Combining with Corollaries~\ref{c:sep}~and~\ref{c:second.mmt.sep} gives
	\[\f{\E(\sepZZ^2)}{(\E\sepZZ)^2}
	\le \f{\exp\{n \lambda \size(H_\star) + o(n) \}}{\E\sepZZ}
		+ O(1)
	\asymp \f{\exp\{o(n)\}}
		{\exp\{n \SIGMA(H_\star)\}}
		+O(1)\,,\]
which immediately implies $\P(\sepZZ\ge\delta(\E\sepZZ))\ge\delta$ for some positive constant $\delta$. This can be strengthened to the asserted concentration result by an easy adaptation of the method described in \cite[Sec.~6]{MR3440193}.\end{proof}
\end{cor}

\begin{ppn}[proved in Appendix~\ref{appx:onersb}]\label{p:drec_equiv.outline} For $0\le\lambda\le1$ and $H_{\lambda,T}=H_\star\in\simplex$ as given by Proposition~\ref{p:first.mmt}, the triple
$(\size(H_{\lambda,T}),\SIGMA(H_{\lambda,T}),\bF(H_{\lambda,T}))$ converges as $T\to\infty$ to $(s_\lambda,\Sigma(s_\lambda),\mathfrak{F}(\lambda))$ from Definition~\ref{d:onersbFE}.
\end{ppn}

\begin{proof}[Proof of Theorem~\ref{t:main}] In Appendix~\ref{appx:ubd} we prove the upper bound, $\FE(\alpha)\le\onersbFE(\alpha)$ for all $0\le \alpha<\asat$. For any $\lambda,T$ such that $\SIGMA(H_{\lambda,T})$ is positive, Corollary~\ref{c:whp} gives
	\[\liminf_{n\to\infty}(\ZZ_{\lambda,T}(\nbd))^{1/n}
	\ge\lim_{n\to\infty}(\sepZZ_{\lambda,T})^{1/n}
	=\exp\{\bF(H_{\lambda,T})\}
	=\exp\{\SIGMA(H_{\lambda,T})
	+\lambda\size(H_{\lambda,T})\}\,.\]
On the other hand, $\ZZ_{\lambda,T}(\nbd)$ consists entirely of clusters of size $\exp\{n\size(H_{\lambda,T}) + o(n)\}$. Therefore, if $\SIGMA(H_{\lambda,T})$ is positive, it must be that $\FE(\alpha)\ge\size(H_{\lambda,T})$. The lower bound $\FE(\alpha)\ge\onersbFE(\alpha)$ then follows
by appealing to Proposition~\ref{p:drec_equiv.outline}, so the theorem is proved.
\end{proof}

The next two sections are devoted to
the optimization of $\bF$ in $\nbd_\circ$,
and of $\bF_2$ in $\nbd_\textup{se}$.
In Section~\ref{s:reduce.to.tree} we show that the optimization of $\bF$ and $\bF_2$ over small regions can be reduced to an optimization problem on trees. In Section~\ref{s:tree.opt} we solve the tree optimization problem by connecting it to the analysis of the \textsc{bp} recursion for the coloring model.
This allows us to prove Propositions~\ref{p:first.mmt}~and~\ref{c:second.mmt.sep},
thereby completing the proof of the main result
Theorem~\ref{t:main}.

\section{Reduction to tree optimization by local updates}\label{s:reduce.to.tree}

In this section we prove the key reduction that ultimately allows us to compute the (first and second) moments of $\ZZ\equiv\ZZ_{\lambda,T}$ (Propositions~\ref{p:first.mmt}~and~\ref{p:second.mmt}). As we have already seen, the calculation reduces to the optimization of functions $\bF$ and $\bF_2$ from \eqref{e:formal.defn.bF} and \eqref{e:def.PSI.two}. These functions are generally not convex over the entirety of their domains $\simplex$ and $\simplex_2$, but we expect them to be convex in neighborhoods around their maximizers $H_\star$ and $H_\bullet$ (as given in Definition~\ref{d:Hstar} below).  With this in mind, we rely on other means (\textit{a~priori} estimates and separability) to restrict the domains --- from $\simplex$ to $\nbd_\circ$ in the first moment (Lemma~\ref{l:restrict.H}), and from $\simplex_2$ to $\nbd_\textup{se}$ in the second moment (Corollary~\ref{c:sep}). Within these restricted regions, we will show that $\bF$ and $\bF_2$ can be optimized by a \bemph{local} update procedure that reduces the (nonconvex) graph optimization to a (convex) tree optimization. 

\begin{figure}[h!]
\centering
\begin{subfigure}[h!]{.8\textwidth}\centering
\includegraphics{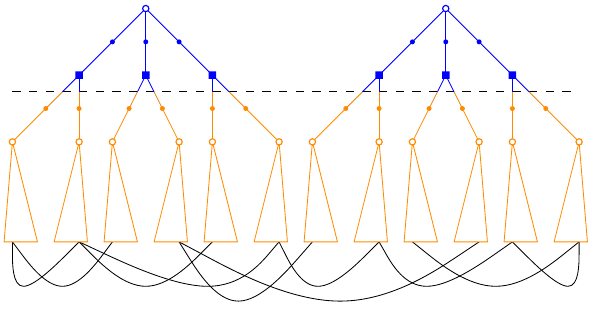}
\caption{$(\glit,Y,\usi)$}
\end{subfigure}\\
\begin{subfigure}[h!]{.8\textwidth}\centering
\includegraphics{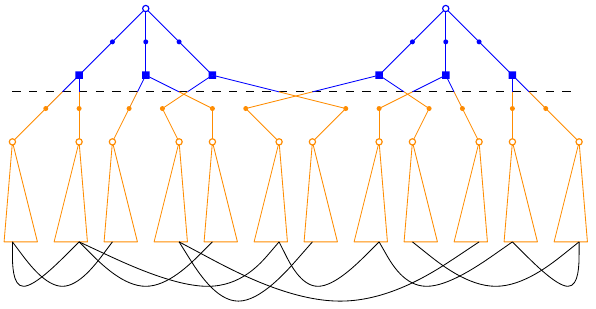}
\caption{$(\glit',Y,\ueta)$}
\end{subfigure}
\caption{One step of the local update procedure. In this figure, open circles indicate variable factors $\dPhi$, solid squares indicate clause factors $\hPl$, and each variable-clause edge is bisected by a small dot indicating the edge factor $\ePhi$. The initial state (top panel) is a triple $(\glit,Y,\usi)$ where $\glit$ is an \textsc{nae-sat} instance, $Y$ is a subset of variables (blue circles), and $\usi$ is a coloring on $\glit$ (not shown). Let $\nlit\equiv\nlit(Y)$ be the subgraph of $\glit$ induced by the variables in $Y$, together with the clausees neighboring $Y$ and their incident half-edges (shown in blue, above the dashed line). The local update procedure resamples the edge literals on $\nlit$, as well as the matching between $\nlit$ and $\glit\setminus\nlit$, to produce a modified instance $\glit'$. The coloring is updated accordingly so that we arrive at the new state $(\glit',Y,\ueta)$ (bottom panel). In both $\glit$ and $\glit'$, the edges cut by the dashed lines will be referred to as \bemph{cut edges}. The half-edges just above the dashed lines will be referred to as the \bemph{leaf edges of $\nlit$}.}\label{f:mc}\end{figure}

\subsection{Local update}\label{s:overview.locupdate}
We begin with an overview. Throughout this section, we assume $1\le T<\infty$. Suppose $\usi$ is a $T$-coloring on $\glit$. Sample from $\glit$ a subset of variables $Y$, and let $\nlit\equiv\nlit(Y)$ be the subgraph of $\glit$ induced by $Y$, together with the clauses neighboring $Y$ and their incident half-edges. The half-edges at the boundary of $\nlit$ will be referred to as the \bemph{leaf edges of $\nlit$.} 

Form a modified instance $\glit'$ (see Figure~\ref{f:mc}) by resampling the edge literals on $\nlit$ as well as the matching between $\nlit$ and $\glit\setminus\nlit$. In both $\glit$ and $\glit'$, we will say \bemph{cut edges} to refer to the edges $e=(av)$ where $a$ is a clause in $\nlit$ and $v$ is a variable in the complement of $\nlit$; these are the edges cut by the dashed lines in Figure~\ref{f:mc}. According to our terminology, the leaf edges of $\nlit$ are the half-edges that lie just above the dashed lines, so each leaf edge of $\nlit$ is half of a cut edge. The coloring is updated accordingly to produce $\ueta$, a $T$-coloring on $\glit'$ which agrees as much as possible with $\usi$ on $\glit\setminus\nlit$: in particular, $\usi$ and $\ueta$ will agree in the variable-to-clause colors on the cut edges. We will define the procedure so that it gives a Markov chain $\pi$ on triples $(\glit,Y,\usi)$ with reversing measure given by $\mu(\glit,Y,\usi)=\P(\glit)\P(Y\,|\,\glit)\wt_\glit(\usi)^\lambda$. (Note that $\mu$ is not normalized to be a probability measure.) 

Reversibility implies that for any subset $A$ of the state space, if $B$ is the set of states reachable in one step from $A$, then
	\begin{align}\nonumber
	\mu(A) 
	&= \sum_{\aaa\in A}
		\sum_{\bbb\in B}\mu(\aaa)\pi(\aaa,\bbb)
	=\sum_{\aaa\in A}\sum_{\bbb\in B}
		\mu(\bbb)\pi(\bbb,\aaa)
	=\sum_{\bbb\in B}
		\mu(\bbb) \sum_{\aaa\in A} \pi(\bbb,\aaa)\\
	&=\sum_{\bbb\in B}
		\mu(\bbb) \pi(\bbb,A)
	\le
	\bigg\{\sum_{\bbb\in B}
		\mu(\bbb)\bigg\}
	\bigg\{ \max_{\bbb\in B}\pi(\bbb,A)
		\bigg\}
	= \mu(B)\max_{\bbb\in B}\pi(\bbb,A)
	\label{e:revers}
	\,.\end{align}
We will design the sampling procedure to ensure that (i) the vertices in $Y$ are far from one another and from any short cycles, and (ii) the empirical measure $\Hsamp$ of $\usi$ on $\nlit$ is close to $\Hsym$, a certain symmetrization of the overall empirical measure $H(\graph,\usi)$. Then $\E\ZZ(H)\approx\mu(A)$ where $A$ is the set of states $(\glit,Y,\usi)$ with $\Hsamp\approx \Hsym$. The update produces a state $(\glit',Y,\ueta)\in B$ with possibly different $\Hsamp$, but with the \bemph{same} empirical measure $\dhtree(\Hsamp)$ of variable-to-clause colors $\dsi$ on the leaf edges of $\nlit$. Bounding $\pi(\bbb,A)$ reduces to calculating the weight of configurations on $\nlit$ with empirical measure $\Hsamp\approx \Hsym$, relative to the weight of all configurations on $\nlit$ with empirical measure $\dhtree(\Hsamp)\approx \dhtree(\Hsym)$ on the leaf edges of $\nlit$. Because $\nlit$ is a disjoint union of \bemph{trees}, this reduces to a convex optimization problem which lends itself much more readily to analysis. The purpose of the current section is to formalize this graphs-to-trees reduction. We begin with the precise definitions of $\Hsym$, $\Hsamp$ and $\dhtree(\Hsamp)$. Recall our notation $\sigma\equiv(\dsi,\hsi)\in\COLS$ from the discussion following \eqref{e:def.ST.tau}.
As $\sigma$ goes over all of $\tcols$,
write $\dCOLS_T$ for the possible values of $\dsi$,
and $\hCOLS_T$ for the possible values of $\hsi$.
Let $\dCOLS\equiv\dCOLS_\infty$ and $\hCOLS\equiv\hCOLS_\infty$, so
	{\setlength{\jot}{0pt}\begin{align*}
	\dCOLS&\equiv\set{\redz,\redo,\bluz,\bluo}
	\cup(\dMM\setminus\set{\zro,\one,\star})\,,\\
	\hCOLS&\equiv\set{\redz,\redo,\bluz,\bluo}
	\cup(\hMM\setminus\set{\zro,\one,\star})
	\end{align*}}%

\begin{dfn}[sample empirical measures] \label{d:hsamp} Given an \textsc{nae-sat} instance $\glit\equiv(\graph,\ulit)$, a $T$-coloring $\usi$ on $\glit$, and a nonempty subset of variables $Y\subseteq V$, we record the local statistics of ``$\usi$ around $Y$'' as follows. Let $\dHsamp$ be the empirical measure of variable-incident colorings in $Y$: for $\ueta\in(\tcols)^d$,
	\[\dHsamp(\ueta)
	\equiv \f{1}{|Y|}\sum_{v\in Y}
		\Ind{\usi_{\delta v}=\ueta}\,.\]
Let $\eHsamp$ be the empirical measure
of colors on the edges incident to $Y$: for $\eta\in\tcols$,
	\[\eHsamp(\eta)
	\equiv\f{1}{|Y|d}
	\sum_{v\in Y}\sum_{e\in\delta v}
	\Ind{\sigma_e=\eta}\,.\]
For $\ueta\in(\tcols)^k$ and $1\le j\le k$ define the rotation $\ueta^{(j)}\equiv (\eta_j,\ldots,\eta_k,\eta_1,\ldots,\eta_{j-1})$. For any $v\in Y$ and $e\in\delta v$, let $j(e)$ be the index of $e$ in $\delta a(e)$.
For $\ueta\in(\tcols)^k$ let
	\[\hHsamp(\ueta)\equiv
	\f1{|Y|d}\sum_{v\in Y}\sum_{e\in\delta v}
	\Ind{(\usi_{\delta a(e)})^{j(e)}=\ueta}\,.\]
Then $\Hsamp\equiv(\dHsamp,\hHsamp,\eHsamp)$ is the \bemph{sample empirical measure} for the state $(\glit,Y,\usi)$; we shall write this hereafter as $\Hsamp=\Hsamp(\graph,Y,\usi)$. Note that $\Hsamp$ lies in the space $\spxsamp$ which is defined similarly to $\simplex$ but with condition~\eqref{e:H.edge.marginal} replaced by
	\beq\label{e:def.treesimplex}
	\f1d\sum_{\usi\in\COLS^d}
		\dHsamp(\usi)\sum_{i=1}^d
		\Ind{\sigma_i=\tau}
	=\eHsamp(\tau)
	=\sum_{\usi\in\COLS^k}
	\hHsamp(\usi)\Ind{\sigma_1=\tau}\,.\eeq
We emphasize that
\eqref{e:H.edge.marginal} and \eqref{e:def.treesimplex}
differ on the right-hand side. However, if we have $H=(\dH,\hH,\eH)\in\simplex$ such that $\hH$ is invariant under rotation of the indices $1\le j\le k$, then $H\in\spxsamp$ as well. With this in mind, for $H\in\simplex$ we define $\Hsym\equiv(\dH,\hHsym,\eH)\in\spxsamp$ where $\hHsym$ is the average over all $k$ rotations of $\hH$. Later we will sample $Y$ such that $\Hsamp$ falls very close to $\Hsym$ with high probability. Lastly, for any $\Hsamp\in\spxsamp$ we let $\dhtree(\Hsamp)$ be the measure on $\dtcols$ given by
	\[
	[\dhtree(\Hsamp)](\deta)
	\equiv
	\f{1}{k-1}
	\sum_{\usi\in(\tcols)^k}
	\sum_{j=2}^k \Ind{\dsi_j=\deta}\hHsamp(\usi)\,.
	\]
Thus $\dhtree(\Hsamp)$ represents the empirical measure of spins $\dsi$ on the leaf edges of $\nlit$, i.e., the edges cut by the dashed lines in Figure~\ref{f:mc}.
\end{dfn}

For $H\in\simplex$, recall from \eqref{e:SIGMA} that $\bF(H)=\SIGMA(H)+\lambda\size(H)$ where $\SIGMA(H)$ is the cluster complexity and $\size(H)$ is (the exponential rate of) the cluster size. The tree analogue of $\SIGMA(H)$ is
$\treeSIGMA(\Hsamp)$ where 
	\beq\label{e:treeSIGMA}
	\treeSIGMA(H)
	\equiv
	\ent(\dH)
	+ d\ent(\hH)
	-d\ent(\eH)
	+\logp(H)\eeq
--- the only difference being that $\SIGMA$
has coefficient $\alpha=d/k$ on the clause entropy term
$\ent(\hH)$, while $\treeSIGMA$ has coefficient $d$. As we see below, this occurs because the ratio of variables to clauses to edges is $1:d:d$ for the disjoint union of trees $\nlit$, versus $1:\alpha:d$ for the full graph $\glit$. We will also see that $\treeSIGMA$ is always concave, though $\SIGMA$ need not be. Likewise, the tree analogue of $\size(H)$ is
$\treesize(\Hsamp)$ where
	\[
	\treesize(H)
	\equiv 
	\langle\log\dPhi,\dH\rangle
	+d\langle
	\log \hF,
	\hH\rangle
	+d\langle\log\ePhi,\eH\rangle\,.\]
The tree analogue of $\bF(H)$ is $\LAMBDA(\Hsamp)$ where
	\beq\label{e:treeLAMBDA}
	\LAMBDA(H) \equiv \treeSIGMA(H)+ \lambda\treesize(H)\,.\eeq Recall Definition~\ref{d:hsamp}: given $(\glit,Y,\usi)$ with sample empirical measure $\Hsamp$, the empirical measure of spins $\dsi$ on the leaf edges of $\nlit(Y)$ 
is given by $\dhtree=\dhtree(\Hsamp)$.
Then, for any probability measure $\dot{h}$ on $\dtcols$, we let
	\[\optLAMBDA(\dot{h})
	\equiv \sup\set{\LAMBDA(H) 
	: H\in\spxsamp\textup{ with }
	\dhtree(H) =\dot{h} }\,,\]
where we emphasize that the supremum is taken over $\spxsamp$ rather than $\simplex$. For $H\in\spxsamp$ we define 
	\beq\label{e:tree.opt.XI}
	\XI(H)
	\equiv \XI_{\lambda,T}(H) 
	\equiv
	\optLAMBDA(\dhtree(H))- \LAMBDA(H)\,.\eeq
The interpretation of $\XI$, formalized below, is that for any $H\in\simplex$, if $A$ is the set of states with $\Hsamp\approx \Hsym$ and $B$ is the set of states reachable in one step of the chain from $A$, then $\max\set{\pi(\bbb,A):\bbb\in B}$ is approximately $\exp\{-|Y|\,\XI(\Hsym)\}$, where we note that $\XI(\Hsym)\ge0$ since $\XI$ is nonnegative on all of $\spxsamp$, and $\Hsym\in\simplex\cap\spxsamp$. Formally, we have the following bound:

\begin{thm}\label{t:mc} For $\epsilon$ small enough (depending only on $d,k,T$), it holds for all $H\in\simplex$ that 
	\[\bF(H)\le \max\Big\{\bF(H'): \lone{H'-H} \le 
		\epsilon(dk)^{2T}\Big\}- \epsilon\,\XI(\Hsym)\,.\]
The analogous statement holds in the second moment with $\bF_2=\bF_{2,\lambda,T}$ and $\XI_2=\XI_{2,\lambda,T}$.
\end{thm}

For the sake of exposition, we will give the proof of Theorem~\ref{t:mc} for $\bF$ only; the assertion for $\bF_2$ follows from the same argument with essentially no modifications. The first task is to define the Markov chain that was informally discussed above. There are a few issues to be addressed: how to sample $Y$ ensuring certain desirable properties; how to resample the matching between $\nlit$ and $\glit\setminus\nlit$; and how to produce a valid coloring $\ueta$ on $\glit'$ without changing the spins $\dsi$ on the cut edges. We address the last issue next. 

\subsection{Tree updates} 

Recall that in the bipartite factor graph $\glit=(V,F,E,\ulit)$, each edge joins a variable to a clause and is defined to have length one-half. For the discussion that follows, it is useful to bisect each edge $e\in E$ with an artificial vertex indicating the edge factor $\ePhi$; these are shown as small dots in Figure~\ref{f:mc}. Thus an edge $e$ joining $a\in F$ to $v\in V$ becomes two quarter-length edges, $(ae)$ and $(ev)$, where $e$ now refers to the artificial vertex. Given a coloring $\usi$ on the original graph, we obtain a coloring on the new graph by simply duplicating the color on each edge, setting $\sigma_{ae}=\sigma_e=\sigma_{ev}$. We then define $\nlit(v)$ as the $(5/8)$-neighborhood of variable $v$, and define $\nlit=\nlit(Y)$ as the union of $\nlit(v)$ for all $v\in Y$: in the top panel of Figure~\ref{f:mc}, $\nlit$ is the subgraph shown in blue, above the dashed line. 

Directly below the same dashed line, the small solid orange dots correspond to the boundary edges, hereafter denoted $\CUT$, of
the cavity graph $\glit_\pd\equiv\glit\setminus\nlit$. For $e\in\CUT$, let $\etree$ be its neighborhood in $\glit\setminus\nlit$ of some radius $\ell >T$ where $2\ell$ is a positive integer. Assuming $e$ is not close to a short cycle, $\etree$ is what we will call a \bemph{directed tree} rooted at $e$. In this case we also call $\etree$ a \bemph{variable-to-clause tree} since the root edge has no incident clause; a \bemph{clause-to-variable tree} is similarly defined. We always visualize a directed tree $\etree$ as in Figure~\ref{f:directed.tree}, with the root edge $e$ at the top, so that paths leaving the root travel downwards. On an edge $e=(av)$, the \bemph{upward color}
is $\dsi_{av}$ if $a$ lies above $v$,
and $\hsi_{av}$ if $v$ lies above $a$.
We let $\delta\etree$ denote the boundary edges of $\etree$, not including the root edge.

\begin{figure}[h!]\centering
\begin{subfigure}[h!]{.3\textwidth}\centering
\includegraphics{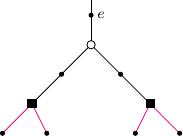}
\caption{A variable-to-clause tree.}
\end{subfigure}
\begin{subfigure}[h!]{.3\textwidth}\centering
\includegraphics{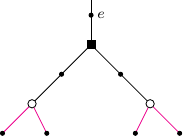}
\caption{A clause-to-variable tree.}
\end{subfigure}
\caption{Two types of directed tree $\etree$.
In each case the root edge $e$ is shown at the top,
and the boundary $\delta\etree$ 
is highlighted in purple. For $e\in\CUT$, the unit-radius neighborhood of $e$ in $\glit\setminus\nlit$ typically looks like the left panel.}
\label{f:directed.tree}
\end{figure}

Suppose $\usi$ is a valid $T$-coloring of a directed tree $\etree$ with root spin $\sigma_e=\sigma$, and consider a new root spin $\eta\in\tcols$. If $\sigma$ and $\eta$ agree on the upward color of the root edge, then there is a \bemph{unique valid coloring}
	\[\ueta=\uu(\usi,\eta;\etree)
	\in(\tcols)^{E(\etree)}\]
which has root spin $\eta$, and agrees with $\usi$ in all the upward colors. Indeed, the only possibility for $\sigma\ne\eta$ is that both $\sigma,\eta\in\set{\fcl}$. Then, recalling \eqref{e:messageconfig.loceq}, the coloring $\uu(\usi,\eta;\etree)$ is uniquely defined by recursively applying the mappings $\dotT$ and $\hatT$, starting from the root and continuing downwards. Since we assumed that $\usi$ was a valid $T$-coloring and $\eta\in\tcols$, it is easy to verify that the resulting $\ueta$ is also a valid $T$-coloring, so the $\uu$ procedure respects the restriction to $\tcols$. From now on we assume all edge colors belong to $\tcols$.

\begin{lem}\label{l:weight.preserving} Suppose $\usi$ is a valid $T$-coloring of the directed tree $\etree$ with root color $\sigma$, and $\eta\in\tcols$ agrees with $\sigma$ on the upward color of the root edge. If $\ueta=\uu(\usi,\eta;\etree)$ agrees with $\usi$ on the boundary $\delta\etree$, then $\wt_{\etree}(\usi)=\wt_{\etree}(\ueta)$.

\begin{proof} It follows from the construction that on any edge $e$ in the tree,
$\sigma_e$ and $\eta_e$ agree on the upward color;
moreover, if $\sigma_e\ne\eta_e$ then we must have
$\sigma_e,\eta_e\in\set{\fcl}$.
For each vertex $x\in\etree$, let $e(x)$ denote the parent edge of $x$, that is, the unique edge of $\etree$ which lies above $x$. We then have
	\[\wt_{\etree}(\usi)
	= \prod_{e\in\delta\etree}
		\ePhi(\sigma_e)
	\prod_{v\in V(\etree)}\bigg\{
	\dPhi(\usi_{\delta v})
	\ePhi(\sigma_{e(v)})\bigg\}
	\prod_{a\in F(\etree)}
	\bigg\{
	\hPl((\usi\oplus\ulit)_{\delta a})
	\ePhi(\sigma_{e(a)})
	\bigg\}\,.\]
For a clause $a$ in $\etree$ with $e(a)=e$, if both $\sigma_e,\eta_e\in\set{\fcl}$ then it follows directly from \eqref{e:product.factors.z} that
	\[\hPl((\usi\oplus\ulit)_{\delta a})\ePhi(\sigma_e)
	= \hphi((\vec{\dta}\oplus\ulit)_{\delta a})\ephi(\sigma_e)
	= \hz(\hsi_e) = \hz(\heta_e)
	= \hPl((\ueta\oplus\ulit)_{\delta a})\ePhi(\eta_e)\,.\]
For a variable $v$ in $\etree$ with $e(v)=e$, if $\sigma_e,\eta_e\in\set{\fcl}$, then a similar calculation as
\eqref{e:product.factors.z} gives
	\[\dPhi(\usi_{\delta v})
	\ePhi(\sigma_e)
	=\dphi(\vec{\hsi}_{\delta v})
	\ephi(\sigma_e)
	=\dz(\dsi_e)
	= \dz(\deta_e)
	=\dPhi(\vec{\eta}_{\delta v})
	\ePhi(\eta_e)\,,\]
where we used the fact that necessarily we have $\usi_{\delta v}\in\set{\fcl}^d$. To conclude, we recall that $\usi$ and $\vec{\eta}$ agree on $\delta\etree$ by assumption, so we have $\wt_{\etree}(\usi)=\wt_{\etree}(\vec{\eta})$ as claimed. \end{proof} \end{lem}

We also use the directed tree as a device to prove the following lemma, which was used in the proofs of Corollaries~\ref{c:first.mmt}~and~\ref{c:second.mmt.sep}.

\begin{lem}\label{l:full.rank.M} Let $\dot{M},\hat{M}$ be as defined in Corollary~\ref{c:first.mmt}, and let $\dot{M}_2,\hat{M}_2$ be their analogues in the pair model. For any $\sigma,\eta\in\COLS$ there exists an integer-valued vector $(\dH,\hH)$ so that
	\[\langle\mathbf{1},\dH\rangle
	=0=\langle\mathbf{1},\hH\rangle
	\quad\text{and}\quad\dot{M}\dH-\hat{M}\hH
	=\mathbf{1}_\sigma-\mathbf{1}_\eta\,,\]
where $\mathbf{1}$ denotes the all-ones vector, and $\mathbf{1}_\sigma$ denotes the vector which is one in the $\sigma$ coordinate and zero elsewhere. The analogous statement holds for $(\dot{M}_2,\hat{M}_2)$.

\begin{proof} We define a graph on $\tcols$ by putting an edge between $\sigma$ and $\eta$ if there exist valid colorings $\usi,\ueta$ on some directed tree $\etree$ which take values $\sigma,\eta$ on the root edge, but agree on the boundary edges $\delta\etree$. If $\sigma,\eta$ are connected in this way, then taking
	\begin{align*}
	\dH(\urho)
	&=\sum_{v\in V(\etree)}
	\Ind{\usi_{\delta v}=\urho}
	-\sum_{v\in V(\etree)}\Ind{\ueta_{\delta v}=\urho}\,,
	\quad\urho\in(\tcols)^d\\
	\hH(\urho)
	&=\sum_{a\in F(\etree)}
	\Ind{\usi_{\delta a}=\urho}
	-\sum_{a\in F(\etree)}
	\Ind{\ueta_{\delta a}=\urho}\,,
	\quad\urho\in(\tcols)^k.
	\end{align*}
gives $\dot{M}\dH-\hat{M}\hH =\mathbf{1}_\sigma-\mathbf{1}_{\eta}$ as required. It therefore suffices to show that the graph we have defined on $\tcols$ is connected (hence complete). First, if $\dsi=\deta$, it is clear that $\sigma$ and $\eta$ can be connected by colorings $\usi,\ueta$ of some variable-to-clause tree $\etree$, with $\ueta=\uu(\usi,\eta;\etree)$. Similarly, if $\hsi=\heta$, then $\sigma$ and $\eta$ can be connected by a clause-to-variable tree. This implies that $\set{\fcl}$ is connected. Next, if $\sigma=\red_{\bx}$ and $\eta=\blu_{\bx}$, then they can be connected by a variable-to-clause tree rooted at edge $e$, containing a single variable factor $v=v(e)$, with $\usi_{\delta v\setminus e}$ identically equal to $\red_{\bx}$.  If $\sigma=\blu_{\bx}$ and $\eta=(\dta,\spc)$ for any $\dta\in\dCOLS\setminus\set{\red,\blu}$, then they can be connected by a clause-to-variable tree rooted at edge $e$, containing a single clause factor $a=a(e)$, with any $\usi_{\delta a\setminus e}$ such that $(\usi\oplus\ulit)_{\delta a\setminus e}$ contains both $\set{\bluz,\bluo}$ entries. It follows that $\tcols$ is indeed connected, which proves the assertion concerning $(\dot{M},\hat{M})$. The proof for $(\dot{M}_2,\hat{M}_2)$ is very similar and we omit the details.\end{proof}\end{lem}
 
\subsection{Markov chain}
We now define a Markov chain on tuples $(\glit,Y,\usi)$ where $\glit$ is an \textsc{nae-sat} instance, $\usi$ is a valid $T$-coloring on $\glit$, and $Y\subseteq V$ is a subset of variables such that
	\beq\label{e:Y.tree.conds}
	\textup{the subgraphs $B_{2T}(v)$,
	for $v\in Y$, are mutually disjoint trees,}\eeq
where $B_{2T}(v)$ is to the $2T$-neighborhood of $v$ in $\glit$. Recall that $\nlit=\nlit(Y)$ is the $\tfrac58$-neighborhood of $Y$;  we write $\nlit = (\ngraph,\ulit_\ngraph)$ where $\ngraph$ is the graph without edge literals. Write $\delta\ngraph$ for the boundary of $\ngraph$, consisting of clause-incident edges that are not incident to $Y$ (just above the dashed line in Figure~\ref{f:mc}). Write $\usi_\ngraph$ for a $T$-coloring on $\ngraph$ (including $\delta\ngraph$), and let
	\beq\label{e:wt.ngraph}
	\wt_\ngraph(\usi_\ngraph|\ulit_\ngraph)
	\equiv \wt_\nlit(\usi_\ngraph)
	\equiv \prod_{v\in Y}\bigg\{
	\dPhi(\usi_{\delta v})
	\prod_{e\in\delta v}
	\Big\{ \hPl(
	(\usi\oplus\ulit)_{\delta a(e)})
	\ePhi(\sigma_e)
	\Big\}\bigg\}\,.\eeq
On the other hand we have $\glit\setminus\nlit \equiv \glit_\pd\equiv(\graph_\pd,\ulit_\pd)$ where $\graph_\pd\equiv(V_\pd,F_\pd,E_\pd)$, and
$\CUT$ denotes the boundary of $\glit_\pd$
(just below the dashed line in Figure~\ref{f:mc}). Write $\usi_\pd$ for a coloring on $\graph_\pd$ (including $\CUT$), and let
	\[\wt_\pd(\usi_\pd)
	\equiv 
	\prod_{v\in V_\pd}
	\dPhi(\usi_{\delta v})
	\prod_{a\in F_\pd}
	\hPl
	((\usi\oplus\ulit)_{\delta a})
	\prod_{e\in E_\pd}
	\ePhi(\sigma_e)\,.\]
By matching $\delta\ngraph$ to $\CUT$
(along the dashed line in Figure~\ref{f:mc}), the graphs $\glit_\pd$ and $\nlit$ combine to form the original instance $\glit$. If $\usi$ is a valid coloring on $\glit$, then $\usi_{\delta\ngraph}$ and $\usi_\CUT$ must agree, and we have
	\beq\label{e:factorization}
	\wt_\glit(\usi)
	=\wt_\pd(\usi_\pd)
	\wt_\ngraph(\usi_\ngraph | \ulit_\ngraph)\,.\eeq
Let $\dhtree(\usi_{\delta \ngraph})=\dhtree$ be the empirical measure of the spins $(\dsi_e)_{e\in\delta \ngraph}$. Given initial state $(\glit,Y,\usi)$, we take one step of the Markov chain as follows: 
\begin{enumerate}[1.]
\item Detach $\nlit$ from $\glit$. On $\nlit$, sample a new assignment $(\ulitp_\ngraph,\ueta_\ngraph)$ from the probability measure
	\beq\label{e:defn.mc}
	p( (\ulitp_\ngraph,\ueta_\ngraph)\,|\,
	(\ulit_\ngraph,\usi_\ngraph))
	=\f{\Ind{
		\dhtree(\ueta_{\delta \ngraph})=\dhtree}
	\wt_\ngraph(\ueta_\ngraph|\ulitp_\ngraph)^\lambda}
	{z(|Y|,\dhtree)}\eeq
where the denominator is the normalizing constant obtained by summing over all possible $(\ulitp_\ngraph,\ueta_\ngraph)$.

\item Form the new graph $\glit'$ by sampling a uniformly random matching of $\delta\ngraph$ with $\CUT$, subject to the constraint that $e\in\CUT$ must be matched to $e'\in\delta\ngraph$ with $\dsi_e=\deta_{e'}$. 
The number of such matchings depends only on $|Y|$ and $\dhtree$, so we denote it as $\mathcal{M}(|Y|,\dhtree)$.
For each matched pair $(e,e')$ where $\hsi_e\ne\heta_{e'}$, let $\etree=\etree(e)$ be the radius-$2T$ neighborhood of $e$ in the graph $\glit_\pd$. Let \[\ueta_{\etree}\equiv\uu(\usi_{\etree},\eta_e;\etree)\] and note that, since $\usi$ is a valid $T$-coloring, $\ueta_{\etree}$ and $\vec{\usi}_{\etree}$ must agree at the boundary of $\etree$. Finally, on the rest of $\graph_\pd$ outside the 
radius-$2T$ neighborhood of $\CUT$, we simply take $\ueta$ and $\usi$ to be the same.
\end{enumerate}
The state of the Markov chain after one step is $(\glit',Y,\ueta)$ where $\ueta$ is a valid $T$-coloring on $\glit'$.

\begin{lem}\label{l:revers} 
Suppose we have a sampling mechanism for a random subset of variables $Y$ in $\glit$ such that,
whenever $(\glit,Y,\usi)$ and $(\glit',Y,\ueta)$
appear in the same orbit of the Markov chain, we have
	\beq\label{e:orbit.eq}
	\P(Y\,|\,\glit)=\P(Y\,|\,\glit')\,.\eeq
A reversing measure for the Markov chain is then given by
	$\mu(\glit,Y,\usi)
	= \P(\glit)
	\P(Y\,|\,\glit)
	\wt_{\glit}(\usi)^\lambda$.

\begin{proof} Given $\aaa=(\glit,Y,\usi)$, let $\bbb=(\glit',Y,\ueta)$ be any state reachable from $\aaa$ in a single step of the chain. By the factorization \eqref{e:factorization},
together with assumption~\eqref{e:orbit.eq}
and the fact that $\P(\glit)=\P(\glit')$,
	\[\f{\mu(\aaa)}{\mu(\bbb)}
	=\f{\P(\glit)\P(Y\,|\,\glit)
		\wt_\pd(\usi_\pd)^\lambda
		\wt_\ngraph(\usi_\ngraph|\ulit_\ngraph)^\lambda}
	{\P(\glit')\P(Y\,|\,\glit')
		\wt_\pd(\ueta_\pd)^\lambda
	 \wt_\ngraph(\ueta_\ngraph|\ulitp_\ngraph)^\lambda}
	=\f{\wt_\pd(\usi_\pd)^\lambda
		\wt_\ngraph(\usi_\ngraph|\ulit_\ngraph)^\lambda}
	 {\wt_\pd(\ueta_\pd)^\lambda
	 \wt_\ngraph(\ueta_\ngraph|\ulitp_\ngraph)^\lambda}
	=\f{\wt_\ngraph(\usi_\ngraph|\ulit_\ngraph)^\lambda}
		{\wt_\ngraph(\ueta_\ngraph|\ulitp_\ngraph)^\lambda}\,,\]
where the last identity is by Lemma~\ref{l:weight.preserving}.
On the other hand, with $\pi$ denoting the transition probabilities for the Markov chain, \eqref{e:defn.mc} implies 
	\[\f{\pi(\aaa,\bbb)}{\pi(\bbb,\aaa)}
	=\f{ \wt_\ngraph(\ueta_\ngraph|\ulitp_\ngraph)^\lambda }
	{\mathcal{M}(|Y|,\dhtree)
		z(|Y|,\dhtree)}
	\f{\mathcal{M}(|Y|,\dhtree)
		z(|Y|,\dhtree)}{\wt_\ngraph(\usi_\ngraph|\ulit_\ngraph)^\lambda}
		=\f{\mu(\bbb)}{\mu(\aaa)}\,.\]
Rearranging proves reversibility,
$\mu(\aaa)\pi(\aaa,\bbb)
=\mu(\bbb)\pi(\bbb,\aaa)$. (We remark that since the Markov chain breaks up into many disjoint orbits, the reversing measure $\mu$ is not unique.)\end{proof}
\end{lem}

\subsection{From graph to tree optimizations}

If $Y$ satisfies condition~\eqref{e:Y.tree.conds} and we define $\nlit=\nlit(Y)=(\ngraph,\ulit)$ as before, then $\ngraph$ consists of $|Y|\equiv s$ disjoint copies of the tree $\onetree$ shown in Figure~\ref{f:onetree}. Recall from Definition~\ref{d:hsamp} the definition of $\Hsamp=\Hsamp(\graph,Y,\usi)$, and note $\Hsamp$ depends only on $\usi_\ngraph$. For any $\Hsamp\in\spxsamp$ we let $\ZZ(\Hsamp;\nlit)$
be the partition function of all colorings on $\nlit$
with empirical measure $\Hsamp$ --- the only randomness comes from the literals $\ulit_\ngraph$. The expected number of valid colorings is
	\[\E\ZZ_{\lambda=0,T}(\Hsamp;\nlit)
	=\exp\{sd\langle\hHsamp,\hat{v}\rangle\}
	\binom{s}{s\dHsamp}
	\binom{ds}{ds\hHsamp}
	\bigg/
	\binom{ds}{ds\eHsamp}\,\]
which by Stirling's formula is
$s^{O(1)} \exp\{s\,\treeSIGMA(\Hsamp)\}$
(see \eqref{e:treeSIGMA}).
Any valid coloring $\usi_\ngraph$
with empirical measure $\Hsamp$
 contributes weight $\wt_\ngraph(\usi_\ngraph|\ulit_\ngraph)^\lambda=\exp\{s\lambda\treesize(\Hsamp)\}$, so altogether
	\beq\label{e:treeLAMBDA.derivation}
	\E\ZZ(\Hsamp;\nlit)
	= s^{O(1)}\exp\{s[\treeSIGMA(\Hsamp)+\lambda\treesize(\Hsamp)]\}
	=s^{O(1)}\exp\{s\,\LAMBDA(\Hsamp)\}\,,\eeq
with $\LAMBDA$ as in \eqref{e:treeLAMBDA}. (This calculation clarifies why we refer to $\treeSIGMA,\LAMBDA$
as the ``tree analogues'' of $\SIGMA,\bF$.)

\begin{figure}[h!]
\includegraphics{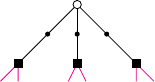}
\caption{If $s\equiv |Y|$
and $\nlit=\nlit(Y)=(\ngraph,\ulit)$,
the graph $\ngraph$ consists of $s$ disjoint copies of the tree $\onetree$ shown here. The boundary $\delta\onetree$ is highlighted in purple.}
\label{f:onetree}
\end{figure}

Fix $\Hsamp\in\spxsamp$ and let $A(\Hsamp)$ be the set of all $(\glit,Y,\usi)$ with $\Hsamp(\graph,Y,\usi)=\Hsamp$. Let $B(\Hsamp)$ be the set of states reachable in one step from $A(\Hsamp)$. Then, for all $\bbb\in B(\Hsamp)$,
	\beq\label{e:reverse.transition.bound}
	\pi(\bbb,A(\Hsamp))
	=\f{\E\ZZ(\Hsamp;\nlit)}
	{\sum_{H'\in\spxsamp}
	\Ind{\dhtree(H')=\dhtree(\Hsamp)}
	\E\ZZ(H';\nlit)}
	=s^{O(1)}\exp\{-s\,\XI(\Hsamp)\}\,.\eeq
This is the key calculation for Theorem~\ref{t:mc}. To complete the proof, what remains is to produce a sampling mechanism $\P(Y\,|\,\glit)$ which satisfies our earlier conditions \eqref{e:Y.tree.conds} and \eqref{e:orbit.eq}, together with some concentration bound to ensure that in most cases $Y$ is large and $\Hsamp\approx\Hsym$. We formalize this as follows:

\begin{dfn}[sampling mechanism]\label{d:samplingmech} Let $\P(Y\,|\,\glit)$ be the probability of sampling the subset of variables $Y$ from the \textsc{nae-sat} instance $\glit$. We call this a \bemph{good sampling mechanism} if the following holds: first, whenever $(\glit,Y,\usi)$ and $(\glit',Y,\ueta)$ appear in the same orbit of the Markov chain, we must have $\P(Y\,|\,\glit)=\P(Y\,|\,\glit')$ (used in Lemma~\ref{l:revers} to show reversibility). Next we require that for every $\glit$, and for every $Y$ with $\P(Y\,|\,\glit)$ positive, the neighborhoods $B_{2T}(v)$ for $v\in Y$ are mutually disjoint trees (condition~\eqref{e:Y.tree.conds}, required for defining the Markov chain). Lastly we require that for all but an exceptional set $\mathscr{B}$ of ``bad'' \textsc{nae-sat} instances, with $\P(\glit\in\mathscr{B}) \le \exp\{-n(\log n)^{1/2}\}$, we have
	\beq\label{e:good.Y}
	\sum_{Y:n\epsilon\le |Y|\le 4n\epsilon}
	\P(Y\,|\,\glit)
	\mathbf{1}\bigg\{ 
	\Big\|\Hsamp(\graph,Y,\usi)
		-\Hsym(\graph,\usi)\Big\|
	\le \f1{(\log\log n)^{1/2}}\bigg\}
	\ge\f12\,.\eeq
for all colorings $\usi$ on $\glit$.
\end{dfn}

\begin{ppn}\label{p:reduce.to.tree}
Assume the existence of a good sampling mechanism  in the sense of Definition~\ref{d:samplingmech}. Then, for $\epsilon$ small enough (depending only on $d,k,T$), it holds for all $H\in\simplex$ that
	\[\f{\E\ZZ(H)}{\E\ZZ(\nbd_{H,\epsilon})}
	\le\f{\exp\{o_n(1)\}}
		{\exp\{n\epsilon\,\XI(\Hsym)\}}\]
where $\nbd_{H,\epsilon}\equiv\set{H'\in\simplex
:\|H-H'\|\le \epsilon (dk)^{2T}}$
and $\Hsym$ is the symmetrization of $H$ from Definition~\ref{d:hsamp}.

\begin{proof}
Abbreviate $\delta\equiv1/(\log\log n)^{1/2}$.
Given $H\in\simplex$ and its symmetrization 
$\Hsym\in\simplex\cap\spxsamp$, let
	\[A=\bigg\{
	(\glit,Y,\usi): |Y|\ge n\epsilon \textup{ and }
		\Big\|\Hsamp(\graph,Y,\usi)-\Hsym\Big\|\le 
		2\delta
	\bigg\}\,.\]
With $\mu$ the reversing measure from Lemma~\ref{l:revers}, we have
	\[\mu(A)
	\ge\sum_{\glit\notin\mathscr{B}}\P(\glit)
	\sum_{\usi} \wt_\glit(\usi)^\lambda
		\sum_{Y:|Y|\ge n\epsilon}
	\P(Y\,|\,\glit)
	\mathbf{1}\bigg\{\hspace{-3pt}\begin{array}{c}
		\|\Hsamp(\graph,Y,\usi)
			-\Hsym(\graph,\usi)\|\le 
		\delta\\ \textup{and }
		\|H(\graph,\usi)-H\|\le \delta
		\end{array}\hspace{-3pt}
	\bigg\}\,,\]
using $\|H(\graph,\usi)-H\|\le
\|\Hsym(\graph,\usi)-\Hsym\|$ together with the triangle inequality. Applying \eqref{e:good.Y} gives
	\[\mu(A)
	\ge \f12
	\sum_{\glit\notin\mathscr{B}}\P(\glit)
	\sum_{\usi} \wt_\glit(\usi)^\lambda
	\mathbf{1}\bigg\{
		\|H(\graph,\usi)-H\|\le \delta\bigg\}
	\ge \f{\E[\ZZ(H);\glit\notin\mathscr{B}]}{2}
	\ge \f{\E\ZZ(H)}{4}\,,\]
where the last step follows from the bound on $\P(\glit\in\mathscr{B})$.
If $B$ is the set of states reachable from $A$ in one step of the Markov chain, then crudely
$\mu(B)\le \E\ZZ(\nbd_{H,\epsilon})$, and
	\[\max_{\bbb\in B}
	\pi(\bbb,A)
	\le \f{\exp\{o_n(1)\}}{\exp\{n\epsilon\,\XI(\Hsym)\}}\]
by our earlier calculation \eqref{e:reverse.transition.bound}.
The result follows by substituting into
\eqref{e:revers}.
\end{proof}
\end{ppn}

\subsection{Sampling mechanism}

To complete the proof of Theorem~\ref{t:mc}, it remains for us to define a sampling mechanism satisfying the conditions of Definition~\ref{d:samplingmech}. To this end, given a $(d,k)$-regular graph $\graph$, let $V_t\subseteq V$ be the subset of variables $v\in V$ such that the $t$-neighborhood $B_t(v)$ around $v$ is a tree. Recall the following form of the Chernoff bound: if $X$ is a binomial random variable with mean $\mu$, then for all $t\ge1$ we have $\P(X\ge t\mu)\le \exp\{ - t\mu \log(t/e) \}$.

\begin{lem}\label{l:loc.tree} Suppose $\graph$ is sampled from the $(d,k)$-regular configuration model on $n$ vertices. For any fixed $t$ we have $\P( |V \setminus V_t| \ge n/(\log\log n) )\le \exp\{ -n (\log n)^{1/2} \}$ for $n$ large enough (depending on $d,k,t$).

\begin{proof} Let $\gamma$ count the total number of cycles in $\graph$ of length at most $2t$. If $v\notin V_t$ then $v$ must certainly lie within distance $t$ of one of these cycles, so crudely we have
	\beq\label{e:near.cycles}
	|V\setminus V_t|
	\le 2t (dk)^t\gamma\,.\eeq
Consider breadth-first search exploration in $\graph$ started from an arbitrary variable, say $v=1$. At each step of the exploration we reveal one edge, so the exploration takes $nd$ steps total. Conditioned on everything revealed in the first $t$ steps, the chance that the edge revealed at step $t+1$ will form a new cycle of length $\le 2t$ is upper bounded by$(dk)^{2t}/(nd-t)$. It follows that the total number of cycles revealed up to time $nd(1-\delta)$ is stochastically dominated by a binomial random variable
	\[\gamma'\sim\mathrm{Bin}\bigg( nd(1-\delta),
	\f{(dk)^{2t}}{nd\delta}\bigg)\,.\]
The final $nd\delta$ exploration steps form at most $nd\delta$ cycles, so $\gamma \le \gamma' + nd\delta$.  Applying the Chernoff bound (as stated above) with $\delta=1/(\log\log n)^2$, we obtain
	\[\P(\gamma \ge 2nd\delta)
	\le\P(\gamma' \ge nd\delta)
	\le\exp\bigg\{ -nd\delta
		\log\bigg(\f{nd\delta^2}{e (dk)^{2t}}\bigg)
		\bigg\}
	\le\exp\{-n(\log n)^{1/2}\}\]
for large enough $n$. Recalling \eqref{e:near.cycles} gives the claimed bound.\end{proof} \end{lem}

Given an instance $\glit=(\graph,\ulit)$, let $V_t$ be as defined above and take $V'\equiv V_{4T}$. We then take i.i.d.\ random variables $I_v\sim\mathrm{Ber}(\epsilon')$ indexed by $v\in V'$ (for $\epsilon'$ a constant to be determined) and let
	\beq\label{e:eps.sample}
	Y_v\equiv\Ind{
	I_v=1, \text{ and }
	I_u=0 \text{ for all }
	u\in B_{4T}(v)\setminus\set{v}}\,.\eeq
We then define $\P(Y\,|\,\glit)$ to be the law of the set $Y=\set{v\in V' : Y_v=1}$. Note that the random variables $Y_v$, for $v\in V'$, all have the same expected value, so we can define $\epsilon\equiv(\E Y_v)/2$.

\begin{lem}\label{l:sample} Define $\mathscr{B}$ to be the set of all $\glit=(\graph,\ulit)$ with $|V\setminus V_{4T}|\le n/(\log\log n)$. For the sampling mechanism described above, condition~\eqref{e:good.Y} holds for any $\glit\notin\mathscr{B}$ and any coloring $\usi$ on $\glit$.

\begin{proof} Fix an instance $\glit\notin\mathscr{B}$ and a coloring $\usi$ on $\glit$. Recalling Definition~\ref{d:hsamp}, for each $v\in V$ denote
	\[X_v\equiv (\dot{X}_v,\hat{X}_v,\bar{X}_v)
	\equiv\Hsamp(\graph,\set{v},\usi)\,.\]
Assume without loss that $V'\equiv V_{4T} = [n']\equiv \set{v_1,\ldots,v_{n'}}$, and for $0\le \ell\le n'$ let $\filt_\ell$ denote the sigma-field generated by $Y_1,\ldots,Y_\ell$. Consider
	\[S \equiv\sum_{v\le n'}A_v Y_v\]
where we can take different choices of $A_v$ to prove various different bounds:
\begin{enumerate}[--]
\item taking $A_v=1$ gives $S=|Y|$ and $\E S = 2n'\epsilon$;
\item taking $A_v=\dot{X}_v(\ueta)$
gives $S=|Y|\dHsamp(\ueta)$
and $|\E S - 2n'\epsilon\dH(\ueta)|\le n-n'$;
\item taking $A_v=\hat{X}_v(\ueta)$
gives $S=|Y|\hHsamp(\ueta)$
and $|\E S - 2n'\epsilon\hH(\ueta)| 
	\le n-n'$;
\item taking $A_v=\bar{X}_v(\eta)$ gives
$S=|Y|\eHsamp(\eta)$ and 
$|\E S-2n'\epsilon\eH(\eta)|\le n-n'$,
\end{enumerate}
where we recall that $n-n'=|V\setminus V_{4T}\le n/(\log\log n)$. Consider the Doob martingale
	\[M_\ell \equiv \E(S\,|\,\filt_\ell)
	\equiv\sum_{v\le n'}
	A_v \, \E(Y_v\,|\,\filt_\ell)\,.\] 
For $\ell\le n'$, if $v$ lies at distance greater than $8T$
from any variable in $[\ell]\equiv\set{v_1,\ldots,v_\ell}$, then
	\[\E(Y_v\,|\,\filt_\ell) = \E Y_v = 2\epsilon\,.\]
Thus, the only possibility for  $\E(Y_v\,|\,\filt_{\ell+1})\ne\E(Y_v\,|\,\filt_\ell)$ is that $v$ lies within distance $8T$ of vertex $\ell+1$. The number of such $v$ is at most $(dk)^{8T}$, so we conclude
	\[|M_{\ell+1}-M_\ell|
	\le (dk)^{8T} \|A\|_\infty
	\le (dk)^{8T}\,.\]
It follows by the Azuma--Hoeffding martingale inequality that
	\[\P(| S - \E S | \ge x)
	\le \exp\bigg\{ -\f{x^2}
		{2 n' (dk)^{16T}}
		\bigg\}\,.\]
The result follows by summing over the choices of $A$ listed above, combined with our above estimates on $\E S$ for each choice of $A$. \end{proof} \end{lem}

\begin{proof}[Proof of Theorem~\ref{t:mc}] It follows from Lemmas~\ref{l:loc.tree}~and~\ref{l:sample} that the sampling mechanism described by \eqref{e:eps.sample} satisfies the conditions of Definition~\ref{d:samplingmech}. The result then follows by taking $n\to\infty$ in Proposition~\ref{p:reduce.to.tree}. \end{proof} 

\section{Solution of tree optimization}\label{s:tree.opt}

From Theorem~\ref{t:mc} we see that if $H\in\simplex$ is a local maximizer for the first moment exponent $\bF=\bF_{\lambda,T}$, then its symmetrization $\Hsym$ must be a zero of the function $\XI=\XI_{\lambda,T}$ defined by \eqref{e:tree.opt.XI}. The analogous statement holds in the second moment with $\bF_2$ and $\XI_2$. The functions $\XI,\XI_2$ correspond to \bemph{tree} optimization problems, which we solve in this section by relating them to the \textsc{bp} recursions for the coloring model. 

\begin{ppn}\label{p:minimizer} For $0\le\lambda\le1$ and $1\le T<\infty$, let $H_\star\in\simplex$ and $H_\bullet\in\simplex_2$ be as in Definition~\ref{d:Hstar} below.
\begin{enumerate}[a.]
\item On $\set{H\in\nbd_\circ:H = \Hsym}$, $\XI$ is uniquely minimized at $H=H_\star$, with $\XI(H_\star)=0$.
\item On $\set{H\in\nbd_\textup{se}:H = \Hsym}$, $\XI_2$ is uniquely minimized at $H=H_\bullet$, with $\XI_2(H_\bullet)=0$. \end{enumerate}
Moreover there is a positive constant $\epsilon=\epsilon(d,k,T)$ such that
\begin{enumerate}[1.]
\item $\XI(H) \ge \epsilon \lone{H-H_\star}^2$ for all $H\in\simplex$ with $H=\Hsym$ and $\lone{H-H_\star}\le\epsilon$, and
\item $\XI_2(H) \ge \epsilon \lone{H-H_\bullet}^2$ for all $H\in\simplex_2$ with $H=\Hsym$ and $\|H-H_\bullet\|\le\epsilon$.
\end{enumerate}\end{ppn}

\subsection{Tree optimization problem}

Recall from the previous section that in the local update procedure, we sample a subset of variables $Y$ and consider its neighborhood $\nlit=(\ngraph,\ulit_\ngraph)$. Writing $s=|Y|$, the graph $\ngraph$ is the disjoint union of $\onetree_1,\ldots,\onetree_s$ where each $\onetree_i$ is a copy of the tree $\onetree$ of Figure~\ref{f:onetree}. Let $\spxtree$ be the space of probability measures on colorings of $\onetree$. Any coloring $\usi_\ngraph$ can be summarized by $\nu\in\spxtree$ where $\nu(\usi_\onetree)$ is the fraction of copies $\onetree_i$ with $\usi_{\onetree_i}=\usi_\onetree$. The sample empirical measure $\Hsamp=\Hsamp(\graph,Y,\usi)$ can be obtained as a linear projection of $\nu$, and we hereafter denote this relation by $\Hsamp=\Htree(\nu)$.  Recalling \eqref{e:wt.ngraph}, we have $\Elit[(\wt_\nlit(\usi_\ngraph))^\lambda] = \avwt_{\onetree}(\usi_{\onetree_1})^\lambda \cdots \avwt_{\onetree}(\usi_{\onetree_s})^\lambda$ where
	\[\avwt_{\onetree}(\usi_\onetree)
	= \dPhi(\usi_{\delta v})
	\prod_{e\in\delta v}\bigg\{
	\ePhi(\sigma_e)
	\hPhi(\usi_{\delta a(e)})
	\bigg\}\,.\]

\begin{lem}\label{l:tree.partition.fn} The function $\LAMBDA$ of \eqref{e:treeLAMBDA} is concave on $\spxsamp$, and can be expressed as
	\beq\label{e:lambda.as.opt}
	\LAMBDA(H)=\sup\Big\{
		\ent(\nu)+\lambda\langle
			\log\avwt_\onetree,\nu \rangle
		:\nu\in\spxtree\textup{ with }
			\Htree(\nu)=H
	\Big\}\,.\eeq

\begin{proof}  The function $\LAMBDA(H)$ is the sum of $\treeSIGMA(H)$ and the linear function $\lambda\treesize(H)$, so it suffices to show that $\treeSIGMA$ is concave on $\spxsamp$.  For $H=(\dH,\hH,\eH)\in\spxsamp$, if $X\in\COLS^k$ is a random variable with law $\hH$, then  the first coordinate $X_1$ has marginal law $\eH$ by \eqref{e:def.treesimplex}. It follows that for any $H\in\spxsamp$ we can express
	\[\treeSIGMA(H)
	=\ent(\dH)+d\ent(X)-d\ent(X_1)+\logp(H)
	=\ent(\dH)+d\ent(X\,|\,X_1)+\logp(H)\,.\]
The entropy function is concave and $\logp$ is linear, so this proves that $\treeSIGMA$ (hence $\LAMBDA$) is indeed concave on $\spxsamp$. In fact this can be argued alternatively, as follows. Recalling \eqref{e:treeLAMBDA.derivation}, note that for $H\in\spxtree$ we have
	\[s^{O(1)}\exp\{s\LAMBDA(H)\}
	=\E\ZZ(H;\nlit)
	=\sum_{\nu\in\spxtree}
	\Ind{\Htree(\nu)=H}
	\binom{s}{s\nu} (\avwt_\onetree)^{\lambda\nu}\,.\]
Expanding with Stirling's formula gives the representation \eqref{e:lambda.as.opt}, which also implies concavity of $\LAMBDA$.\end{proof}\end{lem}

Thus, for $H\in\spxsamp$, we have $\XI(H) = \optLAMBDA(\dhtree(H))- \LAMBDA(H)$ where $\LAMBDA$ is given by \eqref{e:lambda.as.opt}, and
	\beq\label{e:lambda.as.opt.dh}
	\optLAMBDA(\dot{h})
	= \sup\Big\{\ent(\nu)
		+ \lambda
		\langle\log\avwt_{\onetree},\nu
		\rangle: 
		\nu\in\spxtree\textup{ with }
		\dhtree(\Htree(\nu))=\dot{h}
		\Big\}\,.\eeq
Both \eqref{e:lambda.as.opt} and \eqref{e:lambda.as.opt.dh} fall in the general category of entropy maximization problems subject to linear constraints. In Appendix~\ref{appx:entropy.max} we review basic calculations for problems of this type. The discussion there, in particular Remark~\ref{r:match.notation}, implies that for any $\dot{h}$, there is  a unique measure  $\nu=\nu^\textup{op}(\dot{h})$ achieving the maximum in \eqref{e:lambda.as.opt.dh}. Moreover, there exists a probability measure $\dq$ on $\dtcols$ --- serving the role of Lagrange multipliers for the constrained maximization --- such that $\nu^\textup{op}(\dot{h})$ can be expressed as 
	\beq\label{e:multipliers.rep.first}
	\nu(\usi_\onetree)
	=\nu_\dq(\usi_\onetree)
	\equiv
	\f{\avwt_\onetree(\usi)^\lambda}{Z}	
	\prod_{e\in\delta\onetree}
	\dq(\dsi_e)\,,\eeq
where $Z$ is the normalizing constant. The analogous statement holds for the second moment.

\subsection{BP recursions}\label{ss:tree.opt.bp} We now state the \textsc{bp} recursions for the $\lambda$-tilted $T$-coloring model. In the standard formulation (e.g.\ \cite[Ch.~14]{MR2518205}), this is a pair of relations for probability measures $\dot{\bm{q}}$, $\hat{\bm{q}}$ on $\tcols$:
	\begin{align*}
	\dot{\bm{q}}(\sigma)
	&=[\dot{\bm{B}}_{\lambda,T}(\hat{\bm{q}})](\sigma)
	\cong\Ind{\sigma\in\tcols}
	\ePhi(\sigma)^\lambda
	\sum_{\usi\in(\tcols)^d}
	\Ind{\sigma_1=\sigma} \dPhi(\usi)^\lambda
	\prod_{i=2}^d\hat{\bm{q}}(\sigma_i)\\
	\hat{\bm{q}}(\sigma)
	&=[\hat{\bm{B}}_{\lambda,T}(\dot{\bm{q}})](\sigma)
	\cong\Ind{\sigma\in\tcols}
	\ePhi(\sigma)^\lambda
	\sum_{\usi\in(\tcols)^k}
	\Ind{\sigma_1=\sigma} \hPhi(\usi)^\lambda
	\prod_{i=2}^k\dot{\bm{q}}(\sigma_i)
	\end{align*}
where $\cong$ denotes equality up to normalization, so that the mapping always outputs a probability measure. Recall from Definition~\ref{d:hsamp} that for $\dsi\equiv(\dsi,\hsi)\in\tcols$  we have $\dsi\in\dtcols$ and $\hsi\in\htcols$.  For our purposes we can assume a \bemph{one-sided} dependence, meaning there are probability measures $\dq$ on $\dtcols$ and $\hq$ on $\htcols$ such that  $\dot{\bm{q}}(\sigma) \cong \dq(\dsi)\Ind{\sigma\in\tcols}$ and $\hat{\bm{q}}(\sigma)\cong\hq(\hsi)\Ind{\sigma\in\tcols}$. One can then check (e.g.\ \cite[Ch.~19]{MR2518205}) that the \textsc{bp} recursions preserve the one-sided property, so that $\dot{\bm{B}}_{\lambda,T}$ and $\hat{\bm{B}}_{\lambda,T}$ restrict to mappings
	\beq\label{e:defn.vBP}
	\dBP\equiv\dBP_{\lambda,T}
	: \mathscr{P}(\htcols)
	\to\mathscr{P}(\dtcols)\,,\quad
	\hBP\equiv\hBP_{\lambda,T}
	: \mathscr{P}(\dtcols)
	\to\mathscr{P}(\htcols)\,.\eeq
We also denote $\vBP \equiv \vBP_{\lambda,T} \equiv \dBP\circ\hBP$. Given any $\dq\in\mathscr{P}(\dtcols)$, write $\hq\equiv\vBP\dq$, and let $H\equiv H_\dq$ be defined by
	\beq\label{e:def.Hq}
	\dH_\dq(\usi)
	=\f{\dPhi(\usi)^\lambda}
	{\dot{\ZH}}
	\prod_{i=1}^d \hq(\hsi_i), 
	\quad
	\hH_\dq(\usi)
	=\f{\hPhi(\usi)^\lambda}
	{\hat{\ZH}}
	\prod_{i=1}^d \dq(\dsi_i), 
	\quad
	\eH_\dq(\sigma)
	=\f{\ePhi(\sigma)^{-\lambda}}
		{\bar{\ZH}}
	\dq(\dsi)
	\hq(\hsi)\eeq
where $\dot{\ZH}$, $\hat{\ZH}$, and $\bar{\ZH}$ are normalizing constants, all dependent on $\dq$. Clearly, $H_\dq=(H_\dq)^\textup{sy}$. If $\dq$ is a fixed point of $\vBP$, then $H_\dq\in\simplex$.  An entirely similar discussion applies to the pair (second moment) model, where the \textsc{bp} recursion reduces to a pair of mappings between $\mathscr{P}((\dtcols)^2)$ and $\mathscr{P}((\htcols)^2)$. If $\dq\in\mathscr{P}((\dtcols)^2)$ is a fixed point of $\vBP$, then \eqref{e:def.Hq} defines an element $H_\dq\in\simplex_2$.

\begin{lem}\label{l:min.is.zero} For $1\le T<\infty$, if $\dq\in\mathscr{P}(\dtcols)$ is any fixed point of $\vBP_{\lambda,T}$ which has full support on $\dtcols$, then $\XI(H_\dq)=0$. The analogous statement holds for the second moment.

\begin{proof} Consider the optimization problem~\eqref{e:lambda.as.opt.dh} for $\optLAMBDA(\dot{h})$ with $\dot{h}=\dhtree(H_\dq)$. As noted above, $\nu^\textup{op}(\dot{h})$ can be written \eqref{e:multipliers.rep.first} as $\nu_{\tq}$ for some measure $\tq\in\mathscr{P}(\dtcols)$, which may not be unique if the constraint $\dhtree(\Htree(\nu))=\dot{h}$ is rank-deficient. However, if $\dq$ has full support on $\dtcols$, then $\dot{h}=\dhtree(H_\dq)$ does also. In this case it is straightforward to check that the constraints are indeed of full rank, so $\tq$ is unique. Because $\dq$ is a fixed point of $\vBP$, the measure $\nu_\dq$ satisfies $\Htree(\nu_\dq)=H_\dq$, so it also satisfies the weaker constraint $\dhtree(\Htree(\nu_\dq))=\dot{h}$. It follows by the above uniqueness argument that $\dq=\tq$. Therefore, $\nu_\dq$ solves the optimization problem~\eqref{e:lambda.as.opt.dh} for $\optLAMBDA(\dhtree(H_\dq))$, as well as the optimization problem \eqref{e:lambda.as.opt} for $\LAMBDA(H_\dq)$, so we conclude $\XI(H_\dq)=0$ as claimed.\end{proof} \end{lem}

\begin{lem}\label{l:bp.fixed.pt} For $0\le\lambda\le1$ and $1\le T<\infty$, let $\XI=\XI_{\lambda,T}$, $\XI_2=\XI_{2,\lambda,T}$, and $\vBP=\vBP_{\lambda,T}$.
\begin{enumerate}[a.]
\item If $H\in\nbd_\circ$ with $H=\Hsym$ and $\XI(H)=0$, then $H=H_\dq$ where $\dq\in\mathscr{P}(\dtcols)$ is a fixed point of $\vBP$.
\item If $H\in\nbd_\textup{se}$ with $H=\Hsym$ and $\XI_2(H)=0$, then $H=H_\dq$ where $\dq\in\mathscr{P}((\dtcols)^2)$ is a fixed point of $\vBP$.
\end{enumerate}

\begin{proof} Let $\mu=\nu^\textup{op}(H)$ denote the solution of the optimization problem \eqref{e:lambda.as.opt} for $\LAMBDA(H)$, and let $\nu=\nu^\textup{op}(\dhtree(H))$ denote the solution of the optimization problem \eqref{e:lambda.as.opt.dh} for $\optLAMBDA(\dhtree(H))$. Since \eqref{e:lambda.as.opt.dh} has a unique optimizer, we have $\XI(H)=0$ if and only if $\mu=\nu$. This means $\Htree(\nu)=H$, but also $\nu=\nu_\dq$ from \eqref{e:multipliers.rep.first}, which gives
	\beq\label{e:hH.one.BP}\hH(\usi)
	\cong \hPhi(\usi)^\lambda
	((\vBP\dq)(\dsi_1))
	\prod_{i=2}^k \dq(\dsi_i)\,.\eeq
We now claim that in order for $\hH=\hHsym$, we must have $\vBP\dq=\dq$. Note that if $\hPhi$ were fully supported on $(\tcols)^k$, and both $\dq$ and $\vBP\dq$ were fully supported on $\dtcols$, the claim would be obvious. Since $\hPhi$ is certainly not fully supported, and we also do not know \textit{a~priori} whether $\dq$ and $\vBP\dq$ are fully supported, the claim requires some argument, which differs slightly between the first- and second-moment cases:

\begin{enumerate}[a.]
\item In the first moment, Lemma~\ref{l:restrict.H} implies that  $\dq(\dsi)$ is positive for at least one $\dsi\in\set{\bluz,\bluo}$. Assume without loss that  $\dq(\bluz)$ is positive; it follows that $(\vBP\dq)(\dsi)$ is positive for both $\dsi=\bluz,\bluo$. For any $\dsi\in\dCOLS$, there exists $\hsi$ such that
	\beq\label{e:conn}
	\hPhi( (\dsi,\hsi),
	\bluz,\ldots,\bluz)>0\,.\eeq
The symmetry of $\hH$ then gives the relation
	\[\f{(\vBP\dq)(\dsi)}
		{(\vBP\dq)(\bluz)}
	= \f{\dq(\dsi)}
	{\dq(\bluz)},\]
so it follows that $\vBP\dq=\dq$ in the first moment.

\item In the second moment, since we restrict to $H\in\nbd_\textup{se}$, $\dq(\dsi)$ is positive for at least one $\dsi\in\set{\bluz,\bluo}^2$. Assume without loss that  $\dq(\bluz\bluz)$ is positive. For any $\dsi\notin\set{\redz\redo,\redo\redz}$, there exists $\hsi$ such that the second-moment analogue of \eqref{e:conn} holds. The preceding argument gives
	\[\f{(\vBP\dq)(\dsi)}
		{(\vBP\dq)
		(\bluz\bluz)}
	= \f{\dq(\dsi)}
	{\dq(\bluz\bluz)}
	\quad\text{for all }
	\dsi\notin
	\set{\redz\redo,\redo\redz}\,.\]
Since $(\vBP\dq)(\dsi)$ is positive for all  $\dsi\in\set{\bluz,\bluo}^2$, it follows that the same holds for $\dq$, so
	\[\f{(\vBP\dq)(\dsi)}
		{(\vBP\dq)
		(\bluz\bluo)}
	= \f{\dq(\dsi)}
	{\dq(\bluz\bluo)}
	\quad\text{for all }
	\dsi\notin\set{
	\redz\redz,
	\redo\redo}\,.\]
Combining these, we have for $\dsi\in\set{\red_{\zro}\red_{\one},\red_{\one}\red_{\zro}}$ that
	\[\f{(\vBP\dq)(\dsi)}
		{(\vBP\dq)
		(\bluz\bluz)} =\f{(\vBP\dq)(\dsi)}
		{(\vBP\dq)
		(\bluz\bluo)}
	\f{(\vBP\dq)
	(\bluz\bluo)}
		{(\vBP\dq)
		(\bluz\bluz)}
	=\f{\dq(\dsi)}
		{\dq
		(\bluz\bluz)},\]
and this proves $\vBP\dq=\dq$ in the second moment. \end{enumerate}
Altogether, the above proves  in both the first- and second-moment settings that $\dq$ is a \textsc{bp} fixed point.\end{proof} \end{lem}

\subsection{BP contraction and conclusion} \label{ss:bp.contract.conclusion} The next step is to (explicitly) define a subset $\GAMMA$ of measures $\dq$ on which we have a contraction estimate of the form $\lone{\vBP\dq-\dq_\star} \le c\lone{\dq-\dq_\star}$ for a constant $c<1$. A useful feature of \textsc{nae-sat} is that its \textsc{bp} recursions are self-averaging: if $\dq$ is a measure on $\dtcols$, let
	\[\dq^\textup{av}(\dsi)
	\equiv\f{\dq(\dsi)+\dq(\dsi\oplus\one)}{2}\,.\]
Then $\hBP\dq=\hBP\dq^\textup{av}$, and consequently $\vBP\dq=\vBP\dq^\textup{av}$. The analogous statement holds in the second moment. It then suffices to prove contraction on the measures $\dq=\dq^\textup{av}$, since for general $\dq$ it implies
	\[\lone{\vBP\dq-\dq_\star}
	=\lone{\vBP\dq^\textup{av}-\dq_\star}
	\le c\lone{\dq^\textup{av}-\dq_\star}
	\le c\lone{\dq-\dq_\star}\,.\]
Abbreviate $\set{\red}\equiv\set{\redz,\redo}$ and $\set{\blu}\equiv\set{\bluz,\bluo}$. In a mild abuse of notation we now write $\set{\fcl}$ for $(\dCOLS\cup\hCOLS)\setminus\set{\red,\blu}$; so for instance $\dq(\fcl) =\dq(\dCOLS\setminus\set{\red,\blu}) =\dq(\dMM\setminus\set{\zro,\one,\star})$. For the first moment analysis, let $\GAMMA$ be the set of measures $\dq\in\mathscr{P}(\dtcols)$ satisfying $\dq=\dq^\textup{av}$, such that
	\beq\label{e:contract.first}
	\f{\dq(\red) + 2^k\dq(\fcl) }{C}
	\le \dq(\blu)
	\le \f{\dq(\red)}{1-C/2^k}\eeq
for $C$ a large constant (to be determined). For the second moment analysis, let $\GAMMA(c,\kappa)$ be the set of measures $\dq\in\mathscr{P}((\dtcols)^2)$ satisfying $\dq=\dq^\textup{av}$, such that
	\begin{align}
	\tag{1$\GAMMA$}\label{e:contract.second.a}
	&\textup{$|\dq(\bluz\bluz)
		-\dq(\bluz\bluo)| 
		\le(k^9/2^{ck})
		\dq(\blu\blu)$, and
	$\dq(\fcl\fcl)+\dq(\set{\fcl\red,\red\fcl})/2^k
		+\dq(\red\red)/4^k \le (C/2^k) \dq(\blu\blu)$;}\\
	\tag{2$\GAMMA$}\label{e:contract.second.b}
	&\textup{$\dq(\set{\red\fcl,\fcl\red} )
	\le (C/2^{k\kappa})\dq(\blu\blu)$ and
	$\dq(\red\red) \le C 2^{k(1-\kappa) }
		\dq(\blu\blu)$;}\\
	\tag{3$\GAMMA$}\label{e:contract.second.c}
	&\textup{$\dq(\red_\bx\dsi) \ge [1-{C/2^k}]
			\dq(\blu_\bx\dsi)$ and
		$\dq(\dsi\red_\bx) \ge [1-{C/2^k}]
			\dq(\dsi\blu_\bx)$ for all
		$\bx\in\set{\zro,\one}$,
		$\dsi\in\dCOLS$.}
	\end{align}
(To clarify the notation:
since $\blu\equiv\set{\bluz,\bluo}$, in \eqref{e:contract.second.a} the expression
$\dq(\blu\blu)$ refers to $\dq(\set{\bluz,\bluo}^2)$.
Similarly, $\dq(\fcl\red)$ refers to
$\dq(\set{\fcl}\times\set{\redz,\redo})$
where in this context $\set{\fcl}=(\dCOLS\cup\hCOLS)\setminus\set{\red,\blu}$.)

\begin{ppn}[proved in Appendix~\ref{appx:contract}] \label{p:contract} Assume $0\le\lambda\le1$. In the first moment, we have:
\begin{enumerate}[a.]
\item\label{p:contract.first} For any $1\le T\le \infty$, the map $\vBP\equiv\vBP_{\lambda,T}$ has a unique fixed point $\dq_\star\equiv\dq_{\lambda,T}\in\GAMMA$. For any $\dq\in\GAMMA$, we have $\vBP\dq\in\GAMMA$ also, with $\lone{\vBP\dq-\dq_\star}= O(k^2/2^k) \lone{\dq-\dq_\star}$.
\item\label{c:BP_limT} In the limit $T\to\infty$, $\lone{\dq_{\lambda,T} - \dq_{\lambda,\infty}}\to0$.
\end{enumerate}
In the second moment, for any $1\le T\le\infty$, we have the following:
\begin{enumerate}[A.]
\item \label{p:contract.second} The map $\vBP\equiv\vBP_{\lambda,T}$ has a unique fixed point in $\GAMMA(1,1)$, given by $\dq_\star\otimes\dq_\star$ with $\dq_\star$ as in part~\ref{p:contract.first}. Moreover, for $c\in(0,1]$ and $k$ sufficiently large, there is no other fixed point of $\vBP$ in $\GAMMA(c,1)$: if $\dq\in\GAMMA(c,1)$ then $\vBP\dq\in\GAMMA(1,1)$, with $\lone{\vBP\dq-\dq_\star}= O(k^4/2^k) \lone{\dq-\dq_\star}$. \item \label{l:improve.kappa} If for some $c\in(0,1]$ we have $\dq\in\GAMMA(c,0)$ and $\dq = \vBP\dq$, then $\dq\in\GAMMA(c,1)$.
\end{enumerate}\end{ppn}

\begin{dfn}[optimal empirical measures]\label{d:Hstar} For $0\le\lambda\le1$ and $1\le T<\infty$, take the fixed point $\dq_\star=\dq_{\lambda,T}$ as given by Proposition~\ref{p:contract}\ref{p:contract.first}, and use \eqref{e:def.Hq} to define $H_\star=H_{\dq_\star}\in\simplex$ and $H_\bullet=H_{\dq_\star\otimes\dq_\star}\in\simplex_2$. Note that this agrees with our earlier definition of $H_\bullet$, in the discussion below Lemma~\ref{l:clause.tensorize}.\end{dfn}

\begin{lem}\label{l:in.contraction} Let $\dq$ be any fixed point of $\vBP$ that arises from Lemma~\ref{l:bp.fixed.pt}. \begin{enumerate}[a.] \item If $H=H_\dq\in\nbd_\circ$, then $\dq=\dq_\star$ and so $H=H_\star$. \item If $H=H_\dq\in\nbd_\textup{se}$, then $\dq=\dq_\star\otimes\dq_\star$ and so $H=H_\bullet$. \end{enumerate}

\begin{proof} Since $\dq=\vBP\dq$, we must have $\dq=\dq^\textup{av}$. Below we argue separately for the first and second moment. In each case we repeatedly take advantage of the fact that $H=H_\dq$ is symmetric.
\begin{enumerate}[a.]
\item For the first moment, by Proposition~\ref{p:contract}\ref{p:contract.first} it suffices to show that $\dq$ must lie in the set $\GAMMA$ defined by \eqref{e:contract.first}. It follows directly from the relation $\dq=\vBP\dq$ that $\dq(\red)\ge\dq(\blu)$. By definition of $\nbd_\circ$ we must have $\eH(\red)\le 7/2^k$ and $\eH(\fcl)\le 7/2^k$, so the vast majority of clauses must have all incident colors in $\set{\blu}=\set{\bluz,\bluo}$:
	\[1-\f{14k}{2^k} \le \hH(\blu^k) 
	= \f1{\hat{\ZH}}\sum_{\usi\in\blu^k} \hPhi(\usi)^\lambda
		\prod_{i=1}^k
		\dq(\dsi_i)
	\le \f{\dq(\blu)^k}{\hat{\ZH}}\,.\]
Next, if $\usi\in\red\blu^{k-1}$, we have $\hPhi(\usi)^\lambda =\Elit[\hPl(\usi\oplus\ulit)^\lambda] \ge \Elit\hPl(\usi\oplus\ulit)= 2/2^k$, so
	\[\f{7}{2^k}\ge \eH(\red)= \hH(\red\blu^{k-1})
	\ge \f{\dq(\red)\dq(\blu)^{k-1}}{2^{k-1} \hat{\ZH}}
	\ge \f{\dq(\red)}{\dq(\blu)}
	\f{2}{2^k}
	\bigg(1-\f{14k}{2^k} \bigg)\,,\]
which gives $\dq(\red)/\dq(\blu)\le 4$ for large $k$. Similarly, if $\usi\in\fcl\blu^{k-1}$ with $\hsi_1=\spc$ (indicating a separating clause), then $\hPhi(\usi)^\lambda \ge \Elit\hPl(\usi\oplus\ulit) = 1-4/2^k$, so
	\[\f{7}{2^k} \ge \eH(\fcl) 
	\ge \hH(\fcl\blu^{k-1})
	\ge \bigg(1-\f{4}{2^k}\bigg)
	\f{\dq(\fcl)\dq(\blu)^{k-1}}{\hat{\ZH}}
	\ge \f{\dq(\fcl)}{\dq(\blu)}
	\bigg(1-\f{4}{2^k}\bigg)
	\bigg(1-\f{14k}{2^k}\bigg)\,,\]
which gives $\dq(\fcl)/\dq(\blu) \le 8/2^k$ for large $k$. Combining these estimates proves $\dq\in\GAMMA$. 

\item For the second moment, by Proposition~\ref{p:contract}\ref{p:contract.second} it suffices to verify $\dq\in\GAMMA(1,1)$, as defined by \eqref{e:contract.second.a}--\eqref{e:contract.second.c}. Condition~\eqref{e:contract.second.c} is immediate from the relation $\dq=\vBP\dq$. Moreover, by Proposition~\ref{p:contract}\ref{l:improve.kappa} it suffices to show $\dq\in\GAMMA(1,0)$, in which case condition~\eqref{e:contract.second.b} follows from~\eqref{e:contract.second.a}. It remains to verify \eqref{e:contract.second.a}. For this end, we denote $\BLU\equiv\set{\bluz,\bluo}^2$, and partition this into $\BLUeq\equiv\set{\bluz\bluz,\bluo\bluo}$ and $\BLUne\equiv\set{\bluz\bluo,\bluo\bluz}$. By definition, for any $H\in\nbd_\textup{se}$ the single-copy marginals $H^1,H^2$ lie in $\nbd\subseteq\nbd_\circ$, so the total density of $\set{\red,\fcl}$ edges in either copy is very small. As a result the vast majority of clauses have all incident colors in $\BLU$:
	\[1-\f{28k}{2^k}
	\le \hH(\BLU^k)
	\le \f{\dq(\BLU)^k}{\hat{\ZH}}\,.\]
For $H\in\nbd_\textup{se}$, we have $|\eH(\BLUeq)-\eH(\BLUne)|\le k^4/2^{k/2} + O(2^{-k})$. The fraction of edges in not-all-$\BLU$ clauses is $O(k/2^k)$, and for $\usi\in\BLU^k$ we have $1\ge \hPhi(\usi)^\lambda \ge \Elit\hPl(\usi\oplus\ulit) = 1 - O(k/2^k)$, so
	\[\eH(\BLUeq)-\eH(\BLUne) - O(k/2^k)
	= \bigg\{\f{\dq(\BLUeq)-\dq(\BLUne)}{\dq(\BLU)}+ O(k/2^k)\bigg\}
	\f{\dq(\BLU)^k}{\hat{\ZH}}\,.\]
Rearranging gives $|\dq(\BLUeq)-\dq(\BLUne)|/ \dq(\BLU) \le 2k^4/2^{k/2}$, which proves the first part of \eqref{e:contract.second.a} (with $c=1$). It remains to show the second part of \eqref{e:contract.second.a}. If we denote $\REDeq\equiv\set{\redz\redz,\redo\redo}$ and consider $\usi\in\REDeq(\BLUeq)^{k-1}$, then (similarly as above) we have $\hPhi(\usi)^\lambda\ge \Elit\hPl(\usi\oplus\ulit)=2/2^k$, so
	\[\f{7}{2^k}\ge \eH(\REDeq)
	=\hH(\REDeq(\BLUeq)^{k-1})
	\ge \f{2}{2^k}
	\f{\dq(\REDeq)\dq(\BLUeq)^{k-1}}{\hat{\ZH}}
	\ge\f{2}{4^k}\f{\dq(\REDeq)}{\dq(\BLU)}
	\bigg( 1-\f{O(k^5)}{2^{k/2}}\bigg)\,,\]
where the last inequality is by the preceding estimates on $\dq(\BLU)$ and $\dq(\BLUeq)$. The same calculation bounds $\dq(\REDne)$ for $\REDne\equiv\set{\redz\redo,\redo\redz}$. Next consider $\sigma=((\dsi,\spc),\red)\in \set{\fcl\red}$: if $\usi\in\COLS^k$ with $\sigma_1=\sigma$ and the other $k-1$ entries in $\BLU$, then $\hPhi(\usi)^\lambda \ge 4/2^k$ as long as the other entries are not all $\BLUeq$ or all $\BLUne$. Therefore
	\[\f{7}{2^k}\ge \eH(\fcl\red)
	\ge\f{4}{2^k}
	 \f{\dq(\fcl\red)\dq(\BLU)^{k-1}}{\hat{\ZH}} 
	\bigg(1-\f{\dq(\BLUeq)^{k-1}}{\dq(\BLU)^{k-1}}
	-\f{\dq(\BLUne)^{k-1}}{\dq(\BLU)^{k-1}}\bigg)
	\ge\f{\dq(\fcl\red)}{\dq(\BLU)}
	\bigg(1-\f{O(k)}{2^k}\bigg)\,,\]
and the same calculation bounds $\dq(\red\fcl)$. Finally, for $\sigma=((\dsi^1,\spc),(\dsi^2,\spc))\in\set{\fcl\fcl}$, we can consider $\usi\in\COLS^k$ with $\sigma_1=\sigma$ and the other $k-1$ entries in $\BLU$; therefore
	\[\f{7}{2^k}\ge \eH(\fcl\fcl)
	\ge \f{\dq(\fcl\fcl)\dq(\BLU)^{k-1}}{\hat{\ZH}}
		\bigg(1-\f{4}{2^k}\bigg)
	\ge \f{\dq(\fcl\fcl)}{\dq(\BLU)}\bigg(1-\f{O(k)}{2^k}\bigg)\,.\]
Combining these estimate verifies the second part of \eqref{e:contract.second.a}. \end{enumerate} Altogether we have shown that if $H=H_\dq$ lies in $\nbd_\circ$, then $\dq\in\GAMMA$ and so $\dq=\dq_\star$; and if $H=H_\dq$ lies in $\nbd_\textup{se}$ then $\dq\in\GAMMA(1,1)$ and so $\dq=\dq_\star\otimes\dq_\star$. This concludes the proof.
\end{proof}
\end{lem} 

\begin{proof}[Proof of Proposition~\ref{p:minimizer}] We will prove the claim in the first moment; the result for the second moment follows by the same argument. It follows from Lemmas~\ref{l:min.is.zero}, \ref{l:bp.fixed.pt} and \ref{l:in.contraction} that the unique minimizer of $\XI$ on the set $\set{H\in\nbd_\circ:H=\Hsym}$ is $H_\star$ (as given by Definition~\ref{d:Hstar}), with $\XI(H_\star)=0$. It remains to establish that, for $H\in\simplex$ with $H=\Hsym$ and $\lone{H-H_\star}\le\epsilon$, we have $\XI(H)\ge\epsilon\lone{H-H_\star}^2$. As in the proof of Lemma~\ref{l:bp.fixed.pt}, let $\mu=\nu^\textup{op}(H)$ be the solution of the optimization problem \eqref{e:lambda.as.opt} for $\LAMBDA(H)$, and let $\nu=\nu^\textup{op}(\dhtree(H))$ be the solution of the optimization problem \eqref{e:lambda.as.opt.dh} for $\optLAMBDA(\dhtree(H))$. We have from \eqref{e:multipliers.rep.first} that for some $\dq\in\mathscr{P}(\dtcols)$,
	\[\nu(\usi_\onetree)
	=\nu_\dq(\usi_\onetree)
	= \f{\avwt_\onetree(\usi)^\lambda}{Z}	
	\prod_{e\in\delta\onetree}\dq(\dsi_e)\,.\]
For $e\in\delta\onetree$, abbreviate $g_e(\usi_\onetree) \equiv\log\dq(\sigma_e)$. Then, for any probability measure $\varpi$ on colorings $\usi_\onetree$, the quantity $\Lambda(\omega) \equiv \ent(\varpi)+\lambda\langle\log\avwt_\onetree,\varpi\rangle$ can be expressed as
	\[\Lambda(\omega)=\ent(\varpi)
	+\bigg\langle\log\nu+\log Z
		-\sum_{e\in\delta\onetree}g_e,\varpi\bigg\rangle\\
	=-\dkl(\omega|\nu)+\log Z
		-|\delta\onetree|\langle\log\dq,
			\dhtree(\Htree(\varpi))\rangle\,.\]
We have $\dhtree(\Htree(\varpi))=\dhtree(H)$ for both $\varpi=\nu$ and $\varpi=\mu$, so $\XI(H) =\Lambda(\mu)-\Lambda(\nu)=\dkl(\mu|\nu)$. (For further discussion, see Proposition~\ref{p:entropy.max}.) It is well known that $\dkl(\mu|\nu)\gtrsim \lone{\mu-\nu}^2$, so to conclude it remains for us to show that $\lone{\mu-\nu}\gtrsim\lone{H-H_\star}$. To this end, let $\nu_\star\equiv\nu_{\dq_\star}$, and note that $H=\Htree(\mu)$ while $H_\star=\Htree(\nu_\star)$. Recall from the discussion preceding Lemma~\ref{l:tree.partition.fn} that $\varpi\mapsto\Htree(\varpi)$ is a linear projection, so
	\[\lone{H-H_\star}
	\lesssim\lone{\mu-\nu_\star}
	\le\lone{\mu-\nu}+\lone{\nu-\nu_\star}
	\lesssim\lone{\mu-\nu}+\lone{\dq-\dq_\star}\,,\]
where the last bound holds since $\nu=\nu_\dq$ and $\nu_\star=\nu_{\dq_\star}$. Recall from Proposition~\ref{p:contract}\ref{p:contract.first} (or Proposition~\ref{p:contract}\ref{p:contract.second} for the second moment) that we have the contraction estimate $\lone{\vBP\dq-\dq_\star}\le c\lone{\dq-\dq_\star}$ for $c\in(0,1)$, so
	\[(1-c)\lone{\dq-\dq_\star}
	\le\lone{\dq-\dq_\star}-\lone{\vBP\dq-\dq_\star}
	\le\lone{\dq-\vBP\dq}\,.\]
Let $K\equiv(\dot{K},\hat{K},\bar{K})\equiv\Htree(\nu)$, and note that $\hat{K}$ need not be symmetric: if we let $\hat{K}'(\usi)\equiv\hat{K}(\sigma_2,\ldots,\sigma_k,\sigma_1)$ for $\usi\in(\tcols)^k$, then $\hat{K}$ and $\hat{K}'$ need not agree. On the other hand $H=\Htree(\mu)=\Hsym$, so 
	\[\lone{\hat{K}-\hat{K}'}
	\le\lone{\hH-\hat{K}}+\lone{\hH-\hat{K}'}
	=2\lone{\hH-\hat{K}}\le2\lone{H-K}
	\lesssim\lone{\mu-\nu}\,.\]
For any $k$-tuple $\vec{\dot{h}}\equiv(\dot{h}_1,\ldots,\dot{h}_k)$ of probability measures on $\dtcols$, consider 
	\[\hH^\textup{op}(\vec{\dot{h}})
	\equiv\argmax_{\hat{\nu}}
	\bigg\{\ent(\hat{\nu}) + 
	\lambda\langle\log\hPhi,\hat{\nu}\rangle
	: \hat{\nu}(\dsi_1=\cdot)=\dot{h}_i
	\textup{ for all }i
	\bigg\}\]
where $\hat{\nu}$ denotes any probability measure on $\supp\hat{v}\subseteq(\tcols)^k$. The unique optimizer $\hH^\textup{op}(\vec{\dot{h}})$ can be described in terms of another $k$-tuple of probability measures on $\dtcols$, denoted $\vec{\dq}\equiv(\dq_1,\ldots,\dq_k)$, which serve as Lagrange multipliers: $\hH^\textup{op}(\vec{\dot{h}})=\hH(\vec{\dq})$ where
	\[[\hH(\vec{\dq})](\usi)\cong
	\hPhi(\usi)^\lambda\prod_{i=1}^k\dq_i(\dsi_i)\,.\]
In particular, $\hH^\textup{op}(\vec{\dot{h}}_\star)=\hH(\vec{\dq}_\star)$ for $\vec{\dot{h}}_\star\equiv (\dhtree(H_\star),\ldots,\dhtree(H_\star))$ and $\vec{\dq}_\star\equiv(\dq_\star,\ldots,\dq_\star)$. For $\vec{\dot{h}}$ near $\vec{\dot{h}}_\star$, there is a unique $\vec{\dot{q}}$ satisfying $\hH^\textup{op}(\vec{\dot{h}})=\hH(\vec{\dq})$, and we can determine this $\vec{\dot{q}}$ as a smooth function of $\vec{\dot{h}}$. Thus
	\[\lone{\hat{K}-\hat{K}'}
	=\lone{\hat{H}(\vBP\dq,\dq,\ldots,\dq)
	-\hat{H}(\dq,\vBP\dq,\dq,\ldots,\dq)
	}\gtrsim \lone{\dq-\vBP\dq}\,.\]
Combining the above inequalities gives $\lone{H-H_\star}\lesssim\lone{\mu-\nu}$ as desired.\end{proof}

\begin{proof}[Proof of Propositions~\ref{p:first.mmt}~and~\ref{p:second.mmt}] Note that for $\eH$ fixed, $\bF(H)=\bF(\dH,\hH,\eH)$ is a strictly concave function of $\dH,\hH$. It follows that for all $H\in\simplex$ we have $\bF(H)\le \bF(L_H)-\epsilon\lone{H-L_H}^2$ for
	\[L_H\equiv\argmax_L\Big\{\bF(L):
		L\in\simplex\textup{ with }
		\bar{L}=\eH\Big\}\,.\]
Clearly $L_H=(L_H)^\textup{sy}$, so it follows from Theorem~\ref{t:mc} and Proposition~\ref{p:minimizer} that
	\[\bF(L_H)\le \bF(H_\star) - \epsilon\lone{L_H-H_\star}^2.\]
Combining the inequalities (and adjusting $\epsilon$ as needed) gives $\bF(H)\le\bF(H_\star)-\epsilon\lone{H-H_\star}^2$. This concludes the proof of Proposition~\ref{p:first.mmt}, and Proposition~\ref{p:second.mmt} follows by exactly the same argument. \end{proof}

{\bibliographystyle{alphaabbr}
\raggedright
\bibliography{refs}
}

\pagebreak

\appendix

\section{Contraction estimates}\label{appx:contract}

We now prove Proposition~\ref{p:contract}, on the contraction of the \textsc{bp} recursion for the coloring model. In Section~\ref{appx:contract.single} we analyze the recursions for the first moment (single-copy) model and prove Proposition~\ref{p:contract}\ref{p:contract.first}. In Section~\ref{appx:contract.pair} we analyze the the recursions for the first moment (pair) model and prove the remainder of Proposition~\ref{p:contract}. We assume throughout the section that $0\le\lambda\le1$ and $1\le T\le\infty$.

\subsection{Single-copy coloring recursions}
\label{appx:contract.single}

Recall from \S\ref{ss:tree.opt.bp} that the \textsc{bp} recursion is a pair \eqref{e:defn.vBP} of mappings $\dBP: \mathscr{P}(\htcols)\to\mathscr{P}(\dtcols)$ and $\hBP: \mathscr{P}(\dtcols)\to\mathscr{P}(\htcols)$. Recall that for our purposes we can restrict attention to measures satisfying $\dq=\dq^\textup{av}$ and $\hq=\hq^\textup{av}$. Under this restriction, the \textsc{bp} recursion is quite explicit, as we now describe. Recall from Definition~\ref{d:msg.of.tau}, equations \eqref{e:def.dm.dta} and \eqref{e:def.hm.hta.z}, that for $\dta\in\dMM$ and $\hta\in\hMM$ we defined $\dmsg(\dta)$ and $\hmsg(\hta)$ as probability measures on $\set{\zro,\one}$. For convenience, we also define 
	\beq\label{e:dot.m.convention}
	\dmsg(\redo)=\hmsg(\bluo)=\delta_\one,\quad
	\dmsg(\redz)=\hmsg(\bluz)=\delta_\zro\,.\eeq
In what follows we often represent a probability measure on $\set{\zro,\one}$ by the probability assigned to $\one$, writing $\dotm(\dta)\equiv \dmsg[\dta](\one)$ and $\hatm(\hta)\equiv\hmsg[\hta](\one)$. Thus, equations \eqref{e:def.dm.dta}, \eqref{e:def.hm.hta.z}, and \eqref{e:dot.m.convention} together define mappings $\dotm:\dCOLS\to[0,1]$ and $\hatm:\hCOLS\to[0,1]$. Recall that we denote $\set{\red}\equiv\set{\redo,\redz}$, $\set{\blu}\equiv\set{\bluo,\bluz}$, and $\set{\fcl}\equiv\COLS\setminus\set{\red,\blu}$. We also write $\set{\fcl}\equiv(\dCOLS\cup\hCOLS)\setminus\set{\red,\blu}$; the precise meaning of $\set{\fcl}$ will be unambiguous from context. Then, for $\bx\in\set{\zro,\one}$, let us abbreviate 
	\[\grn \equiv \blu\cup\fcl, \
	\grn_{\bx}\equiv\blu_{\bx}\cup\fcl, \
	\ylw \equiv \red \cup \fcl, \
	\ppl_{\bx} \equiv \blu_{\bx} \cup \red_{\bx}\,.\]
The variable recursion $\dBP\equiv\dBP_{\lambda,T}$ is given by
	\[(\dBP\hq)(\dsi)
	\cong
	\begin{cases}
		\hq(\ppl_\one)^{d-1}
		& \textup{if $\dsi \in \set{\redz,\redo}$,}\\
		\hq(\ppl_\one)^{d-1}
		- (\hq(\bluo))^{d-1}
		& \textup{if $\dsi \in\set{\bluz,\bluo}$,}\\
		\displaystyle
		\dz(\dsi)^\lambda
			\sum_{\hsi_2,\ldots,\hsi_d}
			\mathbf{1}\bigg\{ 
			\dsi = \dotT\Big((\hsi_i)_{i\ge2}\Big)
			\bigg\}
			\prod_{i=2}^d \hq(\hsi_i)
		& \textup{if $\dsi \in \dtcols
			\setminus\set{\red,\blu}$,}
	\end{cases}\]
where $\cong$ indicates the normalization which makes $\dBP\hq$ a probability measure on $\dtcols$. For the clause recursion, let us write $\vec{\dsi}\sim\hsi$ if $\vec{\dsi}\equiv(\dsi_2,\ldots,\dsi_k)\in(\dtcols)^{k-1}$ is compatible with $\hsi$, in the sense that
	\beq\label{e:complete.to.clause}
	\Big\{\usi =( (\dsi,\hsi),
	(\dsi_2,\hsi_2),
	\ldots 
	(\dsi_k,\hsi_k))
	\in (\tcols)^k:
	\hI(\usi)=1
	\Big\}
	\ne\emptyset.\eeq
The clause recursion $\hBP\equiv\hBP_{\lambda,T}$ is given by 
	\[(\hBP\dq)(\hsi)
	\cong
	\begin{cases}
	\displaystyle \dq(\blu_{\zro})^{k-1}
		& \textup{if $\hsi \in\set{\redz,\redo}$,}\\
	 \displaystyle \hz(\hsi)^\lambda
			\sum_{\dsi_2,\ldots,\dsi_k}
			\mathbf{1}\bigg\{
			\hsi = \hatT\Big(
			(\dsi_i)_{i\ge2}\Big)
			\bigg\}			
			\prod_{i=2}^k\dq(\dsi_i)
		& \textup{if $\hsi \in \htcols
			\setminus\set{\red,\blu}$,}\\
		\displaystyle
		\sum_{\vec{\dsi}
			\sim\bluo}
	\Big(
	1-\prod_{i=2}^k \dotm(\dsi_i)
	\Big)^\lambda
	\prod_{i=2}^k \dq(\dsi_i)
		& \textup{if $\hsi 
		\in\set{\bluz, \bluo}$,}
	\end{cases}\]
where the last line uses the convention \eqref{e:dot.m.convention}. Recall that $\vBP\equiv\dBP\circ\hBP \equiv\vBP_{\lambda,T}$. We will show the following contraction result (assuming, as always, $0\le\lambda\le1$ and $1\le T\le\infty$).

\begin{ppn}\label{p:BP_contract} If $\dq_1,\dq_2\in\GAMMA$, then $\vBP\dq_1, \vBP\dq_2 \in \GAMMA$ and  $\lone{\vBP\dq_1-\vBP\dq_2} = O(k^2/2^k)\lone{\dq_1-\dq_2}$.
\end{ppn}

Before the proof of Proposition~\ref{p:BP_contract} we deduce the following consequences:

\begin{proof}[Proof of Proposition~\ref{p:contract}\ref{p:contract.first}] Let $\dq^{(0)}$ be the uniform measure on $\set{\bluz,\bluo,\redo,\redz}$, and let $\dq^{(l)} \equiv \vBP\dq^{(l-1)}$. It is clear that $\dq^{(0)}\in\GAMMA$, so Proposition~\ref{p:BP_contract} implies $\dq^{(l)}\in\GAMMA$ for all $l\ge1$, and furthermore that $(\dq^{(l)})_{l\ge1}$ forms an $\ell^1$ Cauchy sequence. By completeness of $\ell^1$ we conclude that there exists $\dq^{(\infty)}=\dq_\star\in\GAMMA$ such that $\vBP\dq_\star = \dq_\star$ and $\lone{\dq^{(l)} - \dq_\star} \to0$ as $l\to\infty$. Applying Proposition~\ref{p:BP_contract} again gives $\lone{\vBP\dq-\dq_\star} = O(k^2/2^k)\lone{\dq-\dq_\star}$ for any $\dq\in\GAMMA$, from which it follows that $\dq_\star$ is the unique fixed point of $\vBP$ in $\GAMMA$.\end{proof}

\begin{proof}[Proof of Proposition~\ref{p:contract}\ref{c:BP_limT}] For each $1\le T\le\infty$, let $(\dq_{\lambda,T})^{(l)}$ ($l\ge0$) be defined in the same way as $\dq^{(l)}$ in the proof of Proposition~\ref{p:BP_contract}. It follows from the definition that $(\dq_{\lambda,T})^{(l)} = (\dq_{\lambda,\infty})^{(l)}$ for all $l \le l_T$, where $l_T\equiv \log T/ \log(dk)$. By the triangle inequality and Proposition~\ref{p:contract}\ref{p:contract.first},
	\[\lone{\dq_{\lambda,T} - \dq_{\lambda,\infty}}
	\le\lone{\dq_{\lambda,T} 
		- (\dq_{\lambda,\infty})^{(l_T)}}
	+\lone{(\dq_{\lambda,\infty})^{(l_T)}
		- \dq_{\lambda,\infty}}\le (C/2^k)^{l_T}\]
for some absolute constant $k$. The result follows assuming $k\ge k_0$.\end{proof}

We now turn to the proof of Proposition~\ref{p:BP_contract}. We work with the non-normalized \textsc{bp} recursions $\dBPu\equiv\dBPu_{\lambda,T}$ and $\hBPu\equiv\hBPu_{\lambda,T}$, defined by substituting ``$\cong$" with ``$=$" in the definitions of $\dBP$ and $\hBP$ respectively. One can then recover $\dBP,\hBP$ from $\dBPu,\hBPu$ via
	\[(\dBP\htp)(\dsi) = \f{(\dBPu\htp)(\dsi)}{\sum_{\dsi' \in \dCOLS}(\dBPu\htp)(\dsi')}\,,\quad
	(\hBP\dtp)(\hsi) = \f{(\hBPu\dtp)(\hsi)}{\sum_{\hsi' \in \hCOLS}(\hBPu\dtp)(\hsi')}\,.\]
Let $\dtp$ be the reweighted measure defined by
	\beq\label{e:reweight}
	\dtp(\dsi) 
	\equiv [\dtp(\dq)](\dsi)
	\equiv \f{\dq(\dsi)}{1 - \dq(\red)}\,.\eeq
In the above we have assumed that the inputs to $\dBP,\hBP,\dBPu,\hBPu$ are probability measures; we now extend them in the obvious manner to nonnegative measures with strictly positive total mass. 
	
Given two measures $r_1,r_2$ defined on any space $\albet$, we denote
	$\Delta r(x) \equiv
	|r_1(x)-r_2(x)|$. We regard $\Delta r$ as a nonnegative measure on $\albet$: for any subset $S\subseteq\albet$,
	\[\Delta r(S)
	=\sum_{x\in S}|r_1(x)-r_2(x)|
	\ge|r_1(S)-r_2(S)|,\]
where the inequality may be strict.
For any nonnegative measure $\hat{r}$ on $\hCOLS$, we abbreviate
	{\setlength{\jot}{0pt}\begin{align*}
	\hatm^\lambda\hat{r}(\hsi) &\equiv
	\hatm(\hsi)^\lambda\hat{r}(\hsi),\\
	(1-\hatm)^\lambda\hat{r}(\hsi) &\equiv(1-\hatm(\hsi))^\lambda
	\hat{r}(\hsi).\end{align*}}%
In what follows we will begin with two measures in $\GAMMA$, and show that they contract under one step of the \textsc{bp} recursion. Let $\hBPu$ and $\dBPu$ be the non-normalized  single-copy \textsc{bp} recursions at parameters $\lambda,T$.  Starting from $\dq_i\in\GAMMA$ ($i=1,2$), denote
	{\setlength{\jot}{0pt}\begin{align*}
	\dtp_i &\equiv\dtp(\dq_i)
	\text{ (as defined by 
	\eqref{e:reweight}),}\\
	\htp_i&\equiv\hBPu(\dtp_i)
	\text{ and }
	\htp_{i,\infty}\equiv\hBPu_{\lambda,\infty}(\dtp_i),\\
	\utp_i &\equiv\dBPu(\htp_i)
	\text{ and }
	\tq_i
	\equiv\dBP\htp_i
		=\vBP\dq_i.\end{align*}}%
With this notation in mind, the proof of Proposition~\ref{p:BP_contract} is divided into four lemmas.

\begin{lem}[effect of reweighting]
\label{l:BP_qtop}
Assuming $\dq_1,\dq_2\in \GAMMA$,
$\lone{\Delta \dtp}
	= O(1) \lone{\dq_1 - \dq_2}$,
where $O(1)$ indicates a constant depending on the constant appearing in~\eqref{e:contract.first}.
\end{lem}

\begin{lem}[clause \textsc{bp}]
\label{l:BP_ptohp} 
Assuming $\dq_1,\dq_2\in \GAMMA$,
	{\setlength{\jot}{0pt}\begin{align}\nonumber
	\hatm^\lambda\htp_i(\spc)
	&= 1 -4/2^k + O(k/4^k),\\ \nonumber
	\hatm^\lambda\htp_i(\fcl)
	&=\hatm^\lambda\htp_i(\spc)
	+ O(k/4^k),\\ \nonumber
	\hatm^\lambda\htp_i(\bluo) 
	&= 1 + O(k/2^k),\\
	\hatm^\lambda\htp_i(\redo)
	&= 
	(2/2^k)[1 + O(k/2^k)].
	\label{e:BP_hatvalue}\end{align}}%
Further, writing $\Delta\hatm^\lambda\htp (\cdot) \equiv \hatm^\lambda(\cdot)|\htp_1 (\cdot) - \htp_2 (\cdot)|$,
	{\setlength{\jot}{0pt}\begin{align}\nonumber
	\Delta\hatm^\lambda
	\htp(\fcl) + \Delta\hatm^\lambda\htp(\red) 
	&= O(k/2^k)\Delta \dtp(\fcl),\\
	\lone{\Delta\hatm^\lambda\htp} 
	&= O(k^2/2^k)
	\lone{\Delta\dtp}.
	\label{e:BP_hatdif}\end{align}}%
(Recall that $\htp(\hsi\oplus\one)=\htp(\hsi)$
and $\hatm(\hsi\oplus\one)=1-\hatm(\hsi)$, so $(1-\hatm)^\lambda\htp(\hsi) = \hatm^\lambda\htp(\hsi\oplus\one)$. As a result, the bounds for $\Delta\hatm^\lambda\htp$ imply analogous
bounds for 
$\Delta(1-\hatm)^\lambda\htp$.)
\end{lem}

\begin{lem}[variable \textsc{bp}, non-normalized] \label{l:BP_hptoup} Assuming $\dq_1,\dq_2\in \GAMMA$, we have
	\beq\label{e:BP_utp}
	\begin{bmatrix}
		\utp_i(\fcl) \\ \utp_i(\red)
	\end{bmatrix}
	= \begin{bmatrix}
	O(2^{-k}) \\ 1+O(2^{-k})
	\end{bmatrix}\utp_i(\blu)
	,\quad
	\begin{bmatrix}
		\Delta\utp(\fcl)\\\Delta\utp(\blu)\\\Delta\utp(\red)
	\end{bmatrix}
	= 
	\begin{bmatrix}
		O(k) \\ O(k2^k) \\ O(k2^k)
	\end{bmatrix} 
	\lone{\Delta\hatm^\lambda\htp}
	\max_{i=1,2}\Big\{
	\utp_i(\blu)\Big\}.
\eeq
\end{lem}

\begin{lem}[variable \textsc{bp}, normalized]
\label{l:BP_uptoq}
Assuming $\dq_1,\dq_2\in \GAMMA$,
we have
$\tq_1,\tq_2\in \GAMMA$ as well, with
\[\lone{\tq_1-\tq_2} 
\lesssim k \lone{\Delta \hatm^\lambda\htp}\,.\]
\end{lem}

\begin{proof}[Proof of Proposition~\ref{p:BP_contract}]
Follows by combining the four preceding lemmas~\ref{l:BP_qtop}--\ref{l:BP_uptoq}.\end{proof}

We now prove the four lemmas.

\begin{proof}[Proof of Lemma~\ref{l:BP_qtop}]
This follows from the elementary identity
	\beq\label{e:BP_id}
		\f{a_1}{b_1} - \f{a_2}{b_2}
		 = \f{1}{b_1} (a_1 - a_2) + \f{b_2 - b_1}{b_1b_2} {a_2}\,.\eeq
 together with
 \eqref{e:contract.first}.\end{proof}

In the proof of the next two lemmas, the following elementary fact will be used repeatedly: suppose for $1\le l\le m$ that we have nonnegative measures  $a^l,b^l$ over a finite set $\albet^l$.  Then, denoting $\vec{\albet}=\albet^1\times\cdots\times\albet^m$, we have
	\begin{align}\nonumber
	\sum_{\ux\in\vec{\albet}}
	\bigg|
	\prod_{l=1}^m a^l(x^l)
	-\prod_{l=1}^m b^l(x^l)
	\bigg|
	&\le\sum_{l=1}^m
	\sum_{\ux\in\vec{\albet}}
	\bigg\{\prod_{1\le j<l}
	b^j(x^j)\bigg\}
	\bigg\{\prod_{l<j\le m}
	a^j(x^j)\bigg\}
	\Big|
	a^l(x^l)-b^l(x^l)
	\Big|\\
	&\le \sum_{l=1}^m \lone{a^l-b^l}
	\prod_{j\ne l}
		\Big( \lone{a^j}
		+\lone{a^j-b^j}\Big).
	\label{e:BP_ineq_generalized}
	\end{align}
If all the $(\albet^l,a^l,b^l)$
are the same $(\albet,a,b)$, 
this reduces to the bound
	\beq\label{e:BP_ineq}
	\sum_{x_1,\dots,x_m\in\albet}
		\bigg| 
		\prod_{i=1}^{m} a(x_i)
		- \prod_{i=1}^{m} b(x_i)
		\bigg|
	\le m\lone{a-b}
	\Big( \lone{a} + \lone{a-b}
	\Big)^{m-1}\,.\eeq
In what follows we will abbreviate
(for $\bx\in\set{\zro,\one}$)
	\beq\label{e:hat.Omega.zero.or.one}
	\hatOm_\bx
	\equiv
	\Big\{\hsi\in\htcols:
	\vec{\dsi}\in
	(\grn_\bx)^{k-1}
	\text{ for all }
	\vec{\dsi}\sim\hsi
	\Big\}\,.\eeq

\begin{proof}[Proof of Lemma~\ref{l:BP_ptohp}]
From the definition, if $\dtp=\dtp(\dq)$ then
	\[\dtp(\blu)
	=\f{\dq(\blu)}{1-\dq(\red)}
	=\f{\dq(\blu)}{\dq(\grn)}
	=1-\dtp(\fcl)\,.\]
It follows that for any $\dq_1,\dq_2 \in \GAMMA$ we have 
	$\Delta\dtp(\blu)
	\le\Delta\dtp(\fcl)
	\le \dtp_1(\fcl)+\dtp_2(\fcl)
	= O(2^{-k})$.
Another consequence of the definition of $\GAMMA$ is that 
$\lone{\Delta\dtp}=O(1)$.
We now control $\Delta\hatm^\lambda\htp(\hsi)$, distinguishing a few cases:
\begin{enumerate}[1.]
\item We first consider $\hsi\in
\hCOLS\setminus\set{\blu,\spc}$. 
For such $\hsi$ we have
	\[\Delta\hatm^\lambda\htp(\hsi)
	=\bigg|
	[\hatm(\hsi)
	\hz(\hsi)]^\lambda
	\sum_{\vec{\dsi}\sim\hsi}
	\bigg(\prod_{j=2}^k\dtp_1(\dsi_j)
	-\prod_{j=2}^k\dtp_2(\dsi_j)\bigg)
	\bigg|,\]
and it is easy to check that
	\[\hatm(\hsi)
	\hz(\hsi)
	=1-\prod_{j=2}^k \dotm(\dsi_j)\in[0,1]\,.\]
Moreover, any such $\hsi$ must belong to $\hatOm_\zro$ or $\hatOm_\one$. By summing over $\hsi\in\hatOm_\zro$ and applying \eqref{e:BP_ineq} we have 
	\[\Delta\hatm^\lambda\htp(\hatOm_\zro)
	\le(k-1)
	\| (\dtp_1-\dtp_2) \|_{\ell^1
		(\grnz)}
	\Big(
	\dtp_1(\grnz)
	+ \Delta\dtp(\fcl)
	\Big)^{k-2}\,.\]
Recalling that $\dtp_1$ and $\dtp_2$ both lie in $\GAMMA$, in the above we have
$\dtp_1(\grnz)
+ \Delta\dtp(\fcl) \le [1 + O(2^{-k})]/2$,
as well as
$\| (\dtp_1-\dtp_2) \|_{\ell^1
(\bluz,\fcl)}
=O(1) \Delta\dtp(\fcl)$. Combining these gives
	\[\Delta\hatm^\lambda\htp(\hatOm_\zro)
	=O(k/2^k) \Delta\dtp(\fcl),\]
and the same bound holds for
$\Delta\hatm^\lambda\htp(\hatOm_\one)$.

\item Next consider $\hsi=\spc$, for which we have $\hatm(\hsi)=1/2$ and $\hz(\hsi)=2$. Thus
	\beq\label{e:sep.clause.unnormalized}
	\hatm^\lambda\htp(\spc)
	= 1 - (\dtp(\grnz))^{k-1} 
	- (\dtp(\grno))^{k-1}
	+ \dtp(\fcl)^{k-1}\,.\eeq
Arguing as above gives
$\Delta\hatm^\lambda\htp(\spc)
= O(k/2^k) \Delta\dtp(\fcl)$, proving the first half of \eqref{e:BP_hatdif}.

\item Lastly consider $\hsi\in\set{\bluz,\bluo}$. Recalling \eqref{e:dot.m.convention} we have $\Delta\hatm^\lambda\htp (\bluz)=0$, so let us take $\hsi=\bluo$,
and consider 
$\vec{\dsi}\sim\bluo$.
Note that if $\vec{\dsi}$
has no entry in $\set{\red}$, then
we also have $\vec{\dsi}\sim\hsi'$ 
for some $\hsi' \in\set{\red,\fcl}$.
Again making use of \eqref{e:dot.m.convention}, this $\vec{\dsi}$ gives the same contribution to
$\hatm^\lambda\htp_\infty(\hsi')$ as to $\hatm^\lambda\htp(\bluo)$. It follows that
	\[\Delta\hatm^\lambda\htp(\bluo)
	\le\Delta\hatm^\lambda\htp_\infty(\ylw)
	+ k\Big| 
	\dtp_1(\red_{\zro})\dtp_1(\bluo)^{k-2}
	 - \dtp_2(\red_{\zro})
	 \dtp_2(\bluo)^{k-2}
	 \Big|\,.\]
The first term on the right-hand side captures the contribution from those $\vec{\dsi}$ with no entry in $\set{\red}$,
 and by the preceding arguments it is $O(k/2^k)\Delta\dtp(\fcl)$. It is easy to check that
the second term is $O(k^2/2^k)\lone{\Delta\dtp}$, which finishes the second part of \eqref{e:BP_hatdif}.
\end{enumerate}
Combining the above estimates proves \eqref{e:BP_hatdif}. We next prove \eqref{e:BP_hatvalue}. For this purpose
we introduce the notation
$\lgf$ to refer to elements of $\dCOLS$ or $\hCOLS$
that contain at least one free variable.
In particular, $\lgf\cap\hCOLS$ is given by
	$\set{\fcl}\setminus\set{\spc}
	\subseteq
	\hatOm_\zro\cup\hatOm_\one
	\subseteq
	\hCOLS$.
Since $\dq_i\in\GAMMA$, we must have
from \eqref{e:contract.first} that
	\beq\label{e:first.bound.lgf}
	\hatm^\lambda\htp_i(\lgf)
	\le	2\sum_{l=1}^{k-1} \binom{k-1}{l} 
		\dtp_i(\fcl)^l 
		\dtp_i(\bluz)^{k-1-l}
	\le 2
		\dtp_i(\bluz)^{k-1}
		\sum_{l=1}^{k-1}
		\bigg( \f{k \dtp_i(\fcl)}
			{\dtp_i(\bluz)}
		 \bigg)^l
		= O(k/4^k)\,.\eeq
On the other hand, we see from \eqref{e:sep.clause.unnormalized} that
	\[\hatm^\lambda\htp_i(\spc)
	= 1-4/2^k+O(k/4^k)\,.\]
If $\vec{\dsi}\sim\bluo$ has no entry in $\set{\red}$, then there must exist some $\hsi\in\set{\fcl}$ such that $\vec{\dsi}\sim\hsi$ as well. Conversely, if $\hsi\in\htcols\setminus\set{\red,\blu}$ and $\vec{\dsi}\sim\hsi$, then $\vec{\dsi}\sim\bluo$, unless $\vec{\dsi}$ has exactly one spin $\dsi_i\in\set{\bluz,\fcl}$ with the remaining $k-2$ spins equal to $\bluo$.\footnote{The converse is not needed for the final bound, but we mention it for the sake of concreteness.} It follows that
	\begin{align}
	\label{e:cancel.red.bf}
	\hatm^\lambda\htp_i(\bluo) 
	&=\htp_i(\bluo)=\hatm^\lambda\htp_{i,\infty} (\fcl)
	+ (k-1)
	\Big[
	\dtp_i(\redz)
	-\dtp_i(\grnz)
	\Big]
	\dtp_i(\bluo)^{k-2} \\
	&\le \hatm^\lambda\htp_{i,\infty} (\fcl)
	+ (k-1)
	\dtp_i(\redz)
	\dtp_i(\bluo)^{k-2}
	= 1 + O(k/2^k).\nonumber
	\end{align}
For a lower bound it suffices to consider the contribution from clauses with all $k$ incident colors in $\set{\blu}$:
	\beq\label{e:lbd.all.blue.clauses}
	\hatm^\lambda\htp_i(\bluo) 
	=\htp_i(\bluo)
	\ge \dtp_i(\blu)^{k-1}
	[1-O(k/2^k)]= 1 - O(k/2^k)\,.\eeq
Lastly, note by symmetry that
	\[\hatm^\lambda\htp_i(\redo) 
	=\htp_i(\redo)
	=\htp_i(\bluz)^{k-1}
	=(2/2^k) \htp_i(\blu)^{k-1}\,.\]
Combining these estimates proves
\eqref{e:BP_hatvalue}.\end{proof}

\begin{proof}[Proof of Lemma~\ref{l:BP_hptoup}] We control $\utp$ and $\Delta \utp$ in two cases.
\begin{enumerate}[1.] \item First consider $\dsi\in\dCOLS\setminus\set{\red,\blu}$. Up to permutation there is a unique $\vec{\hsi}\in \set{\fcl}^{d-1}$ such that $\dsi=\hatT(\vec{\hsi})$. Let $\textup{\textsf{comb}}(\dsi)$ denote the number of distinct tuples $\vec{\hsi}'$ that can be obtained by permuting the coordinates of $\vec{\hsi}$. For this $\vec{\hsi}$ we have
	\beq\label{e:sandwich.dot.z}
	\prod_{j=2}^d\hatm(\hsi_j)^\lambda
	\le\dz(\dsi)^\lambda
	\le\prod_{j=2}^d\hatm(\hsi_j)^\lambda
	+\prod_{j=2}^d(1-\hatm(\hsi_j))^\lambda,\eeq
where the rightmost inequality uses that $(a+b)^\lambda \le a^\lambda + b^\lambda$ for $a,b\ge0$ and $\lambda\in[0,1]$. It follows that for $i=1,2$ we have
	\[ 
	\textup{\textsf{comb}}(\dsi)
	\prod_{j=2}^d
		\hatm\htp_i(\hsi_j)
	\le \utp_i(\dsi) \le\textup{\textsf{comb}}(\dsi) \bigg\{
		\prod_{j=2}^d 
		\hatm^\lambda\htp_i(\hsi_j)
		+ \prod_{j=2}^d
		(1-\hatm)^\lambda\htp_i(\hsi_j)\bigg\}\,.\]
It follows by symmetry that
$\hatm^\lambda\htp_i(\fcl) = (1-\hatm)^\lambda\htp_i(\fcl)$, so
	\beq\label{e:twostep.free.sandwich}
	[\hatm^\lambda\htp_i(\spc)]^{d-1}
	\le \utp_i(\fcl) 
	\le [\hatm^\lambda\htp_i(\fcl)]^{d-1}
	+ [(1-\hatm)^\lambda\htp_i(\fcl)]^{d-1}
	= 2[\hatm^\lambda\htp_i(\fcl)]^{d-1}\,.\eeq
Making use of the symmetry together with \eqref{e:sandwich.dot.z} gives
	\[\Delta\utp(\fcl)
	\le 2\sum_{\vec{\hsi}
		\in(\hCOLS_\fcl)^{d-1}}
	\bigg| \prod_{j=2}^{d-1}
		\hatm^\lambda\htp_1(\hsi_j)
		-\prod_{j=2}^{d-1}
		\hatm^\lambda\htp_2(\hsi_j) \bigg|,\]
and applying \eqref{e:BP_ineq} gives
	\[\Delta\utp(\fcl)
	\lesssim d\lone{\Delta\hatm^\lambda\htp}
	\Big(
	\hatm^\lambda\htp_1(\fcl)
	+\Delta\hatm^\lambda\htp_1(\fcl)
	\Big)^{d-2}\,.\]
Combining \eqref{e:BP_hatvalue} with the lower bound from \eqref{e:sandwich.dot.z} then gives
	\[\Delta\utp(\fcl)
	\lesssim d\lone{\Delta\hatm^\lambda\htp}
	\max_{i=1,2}\Big\{\utp_i(\fcl)\Big\}\,.\]

\item Next consider
$\dsi\in\set{\red,\blu}$: for $\bx\in\set{\zro,\one}$, note that
$\utp_i(\red_\bx)
= \htp_i(\ppl_\bx)^{d-1}$, and
	\beq\label{e:BP_Qublu_lb_init}
	\f{\utp_i(\red_\bx) - \utp_i(\blu_\bx)}
	{\utp_i(\red_\bx)}
	=\f{\htp_i(\blu_\bx)^{d-1}}
	{\htp_i(\ppl_\bx)^{d-1}}
	=\bigg(1 - 
		\f{\htp_i(\red_\bx)}
		{\htp_i(\ppl_\bx)}
		\bigg)^{d-1}=O(2^{-k}),\eeq
where the last estimate uses
\eqref{e:BP_hatvalue}
and $d/k=2^{k-1}\log2+O(1)$.
Applying \eqref{e:BP_ineq} gives
	\[\Delta\utp(\ppl_\one)
	\lesssim d \lone{\hatm^\lambda\htp}
	\Big(\min_{i=1,2}\Big\{
	\hatm^\lambda\htp_i(\ppl_\one)\Big\}
	+\Delta\hatm^\lambda\htp(\ppl_\one)
	\Big)^{d-2}\,.\]
Suppose without loss that
$\hatm^\lambda\htp_1(\blu_1)
\le\hatm^\lambda\htp_2(\blu_1)$: then
	\begin{align*}
	\hatm^\lambda\htp_1(\ppl_\one)
	+\Delta\hatm^\lambda\htp(\ppl_\one)
	&=\hatm^\lambda\htp_2(\bluo)
	+\hatm^\lambda\htp_1(\redo)
	+\Delta\hatm^\lambda\htp(\redo) \\
	&\le\hatm^\lambda\htp_2(\ppl_\one)
	+2\Delta\hatm^\lambda\htp(\redo),
	\end{align*}
and substituting into the above gives
	\[\Delta\utp(\ppl_\one)
	\lesssim d\lone{\hatm^\lambda\htp}
	\Big(\max_{i=1,2}\Big\{
	\hatm^\lambda\htp_i(\ppl_\one)\Big\}
	+\Delta\hatm^\lambda\htp(\redo)
	\Big)^{d-2}\,.\]
From \eqref{e:BP_hatdif} and the definition
\eqref{e:contract.first} of $\GAMMA$ we have $\Delta \hatm^\lambda\htp(\redo) = O(k/2^k)\Delta \dtp(\fcl) = O(k/4^k)$. It follows from \eqref{e:BP_Qublu_lb_init} that
	\beq\label{e:BP_Qublu_lb}
	\Delta\utp(\ppl_\one)
	\lesssim d \lone{\Delta\hatm^\lambda\htp} 
	\max_{i=1,2}\Big\{
	\utp_i(\bluo)\Big\}\,.\eeq
\end{enumerate}
It remains to show
$\utp(\fcl) / \utp(\blu) = O(2^{-k})$.
From \eqref{e:cancel.red.bf},
	\[\hatm^\lambda\htp_i(\fcl)
	-\hatm^\lambda\htp_i(\bluo)
	\le\hatm^\lambda\htp_{i,\infty}(\fcl)
	-\hatm^\lambda\htp_i(\bluo)
	\le (k-1)
	\Big[\dtp_i(\grnz)
	-\dtp_i(\redz)\Big]
	\dtp_i(\bluo)^{k-2},\]
and by definition of $\GAMMA$ the right-hand side is $O(k/4^k) \dtp_i(\blu)^{k-1}$. Now recall from \eqref{e:lbd.all.blue.clauses}
that $\hatm^\lambda\htp_i(\bluo)
\gtrsim \dtp_i(\blu)^{k-1}$.
Combining these gives
	\beq\label{e:BP_Qhfb_dif}
	\hatm^\lambda\htp_i(\fcl)
	\le [1+O(k/4^k)]
	\hatm^\lambda\htp_i(\bluo)\,.\eeq
Recalling \eqref{e:sandwich.dot.z}, it follows that
	\[\f{\utp_i(\fcl)}{\utp_i(\bluo)}
	\lesssim
	\bigg(\f{\hatm^\lambda\htp_i(\fcl)}
	{\hatm^\lambda\htp_i(\ppl_\one)}\bigg)^{d-1}
	\lesssim
	\bigg(\f{\hatm^\lambda\htp_i(\bluo)}
	{\hatm^\lambda\htp_i(\ppl_\one)}\bigg)^{d-1}
	\lesssim 2^{-k},\]
where the last step uses \eqref{e:BP_hatvalue}. This concludes the proof.\end{proof}

\begin{proof}[Proof of Lemma~\ref{l:BP_uptoq}]
Denote $\tq_i \equiv\vBP\dq_i$
and $\Delta\tq\equiv|\tq_1-\tq_2|$. We first check that $\tq_i$ lies in $\GAMMA$: the first condition of \eqref{e:contract.first} follows from \eqref{e:BP_utp}, and the second is automatically satisfied from the definition of $\dBP$. Next we bound $\Delta\tq$. With some abuse of notation, we shall write $\tq_i(\RMB)\equiv
\tq_i(\red)-\tq_i(\blu)$ and 
	\[\Delta \tq (\RMB)
	\equiv
		|(\tq_1(\red)-\tq_1(\blu)) - 
		(\tq_2(\red)-\tq_2(\blu))|\,.\]
Let
$\utp_i(\RMB)$ and
 $\Delta \utp (\RMB)$ be similarly defined.
Arguing similarly as in the derivation of
\eqref{e:BP_Qublu_lb},
	\beq\label{e:BP_tqredblu}
	\Delta \utp (\RMB)
	= 2
	|\htp_1(\bluo)^{d-1}
	-\htp_2(\bluo)^{d-1}|
	\lesssim
	k\lone{\Delta\hatm^\lambda\htp}
	\max_{i=1,2} \Big\{\utp_i(\blu)\Big\}\eeq
Recalling $\lone{\tq_i} =1$, we have 
	{\setlength{\jot}{0pt}\begin{align*}
	2\tq_i(\red) &=[1-\tq_i(\fcl)]
		+[\tq_i(\red)-\tq_i(\blu)]
		\text{ and}\\
	2\tq_i(\blu) &=[1-\tq_i(\fcl)]-[\tq_i(\red)-\tq_i(\blu)],
	\text{ so}\\
	\lone{\Delta \tq} &\lesssim
	\Delta \tq(\fcl) + \Delta\tq(\RMB).\end{align*}}%
If we take $a\in\set{1,2}$ and $b=2-a$, and write $\dot{Z}_i\equiv\lone{\utp_i}$, then
	\[\Delta \tq(\fcl) 
	+ \Delta \tq(\RMB)
	\le\f{\Delta\utp(\fcl)
		+\Delta\utp(\RMB)
		}{\dot{Z}_a}
	+
	\f{|\dot{Z}_a-\dot{Z}_b|}
		{\dot{Z}_a}
	\f{[\utp_b(\fcl)
	+\utp_b(\red)-\utp_b(\blu)]}
	{\dot{Z}_b}\,.\]
If we take $a\in\argmax_i\utp_i(\blu)$,
then, by \eqref{e:BP_utp} and \eqref{e:BP_tqredblu}, 
the first term on the right-hand side is
	\[\lesssim
	\f{
	k\lone{\Delta\hatm^\lambda\htp}
	\utp_a(\blu)
	}
	{\dot{Z}_a}
	\lesssim 
	k\lone{\Delta\hatm^\lambda\htp},\]
where the rightmost inequality uses
$\dot{Z}_i\ge\utp_i(\blu)$.
As for the second term,
 \eqref{e:BP_utp} gives
	\[\f{|\dot{Z}_a-\dot{Z}_b|}{\dot{Z}_a}
	\lesssim
	d\lone{\Delta\hatm^\lambda\htp}
	\quad\text{and}\quad
	\f{[\utp_b(\fcl)
	+\utp_b(\red)-\utp_b(\blu)]}
	{ \dot{Z}_b}
	\lesssim 2^{-k}\,.\]
Combining these estimates yields the claimed bound.\end{proof}

\subsection{Pair coloring recursions} \label{appx:contract.pair}

In this section we analyze the \textsc{bp} recursions for the pair coloring model and prove the remaining assertions of Proposition~\ref{p:contract}. Recall that we assume $\dq=\dq^\textup{av}$ and $\hq=\hq^\textup{av}$, where these are now probability measures on $(\dtcols)^2$ and $(\htcols)^2$ respectively. For any measure $p(x)$ defined on $x\equiv (x^1,x^2)$ in $(\dtcols)^2$ or $(\htcols)^2$, define
	\[(\flip p)(x)
	\equiv p(\flip x)\quad\text{where }
	\flip x
	\equiv x\oplus(\zro,\one)
	\equiv (x^1,x^2\oplus\one)\,.\]
Recall from \S\ref{ss:bp.contract.conclusion} the definition of $\GAMMA(c,\kappa)$. We will prove that

\begin{ppn} \label{p:BPt_contract}
For any $c\in (0,1]$ and any $ \dq_1,\dq_2 \allowbreak\in\GAMMA(c,1)$, we have $\vBP\dq_1, \vBP\dq_2 \in \GAMMA(1,1)$ and 
\beq\label{e:BPt_contract}
	\lone{\vBP\dq_1-\vBP\dq_2} 
	= O(k^4/2^k)\lone{\dq_1-\dq_2} 
	+ O(k^4/2^k)
	\sum_{i=1,2}
	\lone{\dq_i -\flip\dq_i}.
\eeq
\end{ppn}

Assuming this result, it is straightforward to deduce Proposition~\ref{p:contract}\ref{p:contract.second}:

\begin{proof}[Proof of Proposition~\ref{p:contract}\ref{p:contract.second}]
Let $\dq^{(0)}$ be the uniform probability measure on $\set{\bluz,\bluo,\redo,\redz}^2$, and define recursively $\dq^{(l)} = \vBP\dq^{(l-1)}$ for $l\ge 1$. It is clear that $\dq^{(0)}\in\GAMMA(1,1)$ and $\dq^{(0)} =\flip\dq^{(0)}$. Since $\dq^{(l)}=\flip\dq^{(l)}$ for all $l\ge1$, it follows from \eqref{e:BPt_contract}
that $(\dq^{(l)})_{l\ge1}$ forms an $\ell^1$ Cauchy sequence. It follows by completeness of $\ell^1$ that $\dq^{(l)}$ converges to a limit $\dq^{(\infty)}=\dq_\star\in\GAMMA(1,1)$, satisfying
$\dq_\star=\flip\dq_\star
=\vBP\dq_\star$.
This implies
that for any probability measure $\dq$,
	\[\lone{\dq-\flip\dq}
	\le\lone{\dq-\dq_\star}
	+\lone{\dq_\star-\flip\dq}
	=2\lone{\dq-\dq_\star}\,.\]
Applying \eqref{e:BPt_contract} again gives
	\[	\lone{\vBP\dq-\dq_\star}
	= O(k^4/2^k)\lone{\dq-\dq_\star}
	+ O(k^4/2^k)\lone{\dq-\flip\dq}
	= O(k^4/2^k)\lone{\dq-\dq_\star},\]
proving the claimed contraction estimate.
Uniqueness of $\dq_\star$ can be deduced from
this contraction.\end{proof}

We now turn to the proof of Proposition~\ref{p:BPt_contract}; the proof of Proposition~\ref{p:contract}\ref{l:improve.kappa} is given after. Let $\dBPu,\hBPu$ now denote the non-normalized \textsc{bp} recursions for the pair model. Let $\red[\dsi]\in\set{0,1,2}$ count the number of $\red$ spins in $\dsi$, and let $\dtp\equiv\dtp(\dq)$ be the reweighted measure
	\beq\label{e:reweight.second}
	\dtp(\dsi) \equiv
	\f{\dq(\dsi)}{1 - \dq(\red[\dsi] > 0)}\,.\eeq
Recalling convention~\eqref{e:dot.m.convention},
we will denote
	\[\hatm^\lambda\hat{r}(\hsi^1,\hsi^2)
	\equiv
	[\hatm(\hsi^1)
	\hatm(\hsi^2)]^\lambda
	\hat{r}(\hsi^1,\hsi^2)\,.\]
Let $\hBPu$
and $\dBPu$
be the non-normalized 
pair \textsc{bp} recursions
at parameters $\lambda,T$.
 Starting from $\dq_i\in\GAMMA(c,\kappa)$ ($i=1,2$), we denote
	{\setlength{\jot}{0pt}\begin{align*}
	\dtp_i &\equiv\dtp(\dq_i)
	\text{ (as defined by 
	\eqref{e:reweight.second}),}\\
	\htp_i &\equiv\hBPu(\dtp_i)
	\text{ and }
	\htp_{i,\infty}\equiv\hBPu_{\lambda,\infty}(\dtp_i),\\
	\utp_i &\equiv\dBPu(\htp_i)
	\text{ and }
	\tq_i
	\equiv\dBP\htp_i
		=\vBP\dq_i.\end{align*}}%
With this notation in mind, 
the proof of Proposition~\ref{p:BPt_contract} is divided into the following lemmas.

\begin{lem}[effect of reweighting]
\label{l:BPt_qtop}
Suppose $\dq_1,\dq_2\in \GAMMA(c,\kappa)$ 
for $c\in(0,1]$ and $\kappa\in[0,1]$: then
	{\setlength{\jot}{0pt}\begin{align*}
	\lone{\Delta \dtp} &\equiv
	O(2^{2(1-\kappa)k}) \lone{\Delta\dq},\\
	\lone{\dtp_i - \flip\dtp_i} &\equiv
		O(2^{(1-\kappa)k}) \lone{\dq_i - \flip\dq_i}.
	 \end{align*}}%
\end{lem}

\begin{lem}[clause \textsc{bp} contraction]
\label{l:BPt_ptohp_PART_A}
Suppose $\dq_1,\dq_2\in \GAMMA(c,\kappa)$ 
for $c\in(0,1]$ and $\kappa\in[0,1]$: then
	{\setlength{\jot}{0pt}\begin{align}\nonumber
	\Delta\hatm^\lambda\htp(\ylw\ylw) 
	&=O(k^3/2^{k}) \Delta \dtp(\grn\grn)
	=O(k^3/2^{(1+c)k}),\\ \nonumber
	\Delta\hatm^\lambda\htp
		(\set{\blu\red,\blu\lgf})
	&= 
	 O(k^2/2^k) [\Delta\dtp(\grn\grn)
		+2^{-k}\Delta\dtp(
		\dCOLS^2\setminus\set{\red\red})]
	=O(k^3/2^{(1+c)k}),\\
	\lone{\Delta\hatm^\lambda\htp} 
	&= O(k^3/2^{k})\lone{\Delta\dtp}
	=O(k^3 2^{(1-2\kappa)k}),
	\label{e:BPt_hatdif}\end{align}}%
and the same estimates hold with $\flip\htp$ in place of $\htp$. For both $i=1,2$,
	\beq\label{e:BPt_hatoverlap} 
	\lone{\hatm^\lambda\htp_i
	-\hatm^\lambda\flip\htp_i}
	=O(k^3/2^{(1+\kappa)k})
	\lone{\dtp_i-\flip\dtp_i}
	=O(k^3/2^{2\kappa k})
	\lone{\dq_i-\flip\dq_i}\,.\eeq
\end{lem}

\begin{lem}[clause \textsc{bp}
	output values]
\label{l:BPt_ptohp_PART_B}
Suppose $\dq_1,\dq_2\in \GAMMA(c,\kappa)$ 
for $c\in(0,1]$
and $\kappa\in[0,1]$.
For $s,t\subseteq\hCOLS$ let
$st\equiv s\times t$.
Then it holds for all
$s,t \in \set{\redo, \bluo, \fcl, \spc}$ that
	\beq\label{e:BPt_hatvalue}
	\f{\hatm^\lambda\htp_i(s, t)}
	{(2/2^k)^{\red[s]+\red[t]} }
	= \begin{cases} 
		1+O(k^2/2^{k}) & \textup{if $\red[s] + \red[t] \le 1$,}\\
		1+O(k^2/2^{ck}) & \textup{if $\red[s] + \red[t] = 2$.} \\
	\end{cases}\eeq
Furthermore we have the bounds
	{\setlength{\jot}{0pt}\begin{align}\nonumber
	\hatm^\lambda\htp_i(\lgf t) +
	 \hatm^\lambda\htp_i(t\lgf)
		&\le O(k/4^k)
	 \text{ for all }
	 t\in\set{\redo,\bluo,\fcl,\spc},\\
	\hatm^\lambda\htp_i
	(\set{\fcl}\times\hCOLS)
	-\hatm^\lambda\htp_i(\set{\bluo}
		\times\hCOLS) 
		&\le O(k/4^{k}).
	\label{e:BPt_hatlgf} \end{align}}%
The same estimates hold with $\flip\htp_i$ in place of $\htp_i$.
\end{lem}

\begin{lem}[variable \textsc{bp}] \label{l:BPt_hptoq} Suppose $\dq_1,\dq_2\in \GAMMA(c,\kappa)$ for $c\in(0,1]$ and $\kappa\in[0,1]$. Then $\vBP\dq_1,\vBP\dq_2 \in \GAMMA(c',1)$ for $c' = \max\set{0,2\kappa-1}$, and
	\[\loneB{\vBP\dq_1 - \vBP\dq_2} \lesssim k
		 \loneB{\Delta\hatm^\lambda\htp
		 +\Delta\hatm^\lambda\flip\htp} \big)
		+ k2^k\sum_{i=1,2} \loneB{\hatm^\lambda\htp_i - \hatm^\lambda\flip\htp_i}\,.\]
\end{lem}

\begin{proof}[Proof of Proposition~\ref{p:BPt_contract}] Follows by combining the preceding lemmas~\ref{l:BPt_qtop}--\ref{l:BPt_hptoq}.\end{proof}

\begin{proof}[Proof of Proposition~\ref{p:contract}\ref{l:improve.kappa}] If $\dq\in\GAMMA(c,0)$ is a fixed point of $\vBP$, then it follows from Lemmas~\ref{l:BPt_ptohp_PART_A}--\ref{l:BPt_hptoq} that we have $\dq \in \GAMMA(c,0)\cap\GAMMA(0,1) = \GAMMA(c,1)$.\end{proof}

We now prove the three lemmas leading to Proposition~\ref{p:BPt_contract}.

\begin{proof}[Proof of Lemma~\ref{l:BPt_qtop}]
Applying \eqref{e:BP_id} we have
	\[|\dtp_1(\dsi)-\dtp_2(\dsi)| \le\f{|\dq_1(\dsi)-\dq_2(\dsi)|}
		{ \dq_1(\grn\grn) }
	+ \f{|\dq_1(\grn\grn)
		-\dq_2(\grn\grn)|}
		{\dq_1(\grn\grn)\dq_2(\grn\grn)}
		\dq_2(\dsi),\]
and summing over $\dsi\in\dCOLS^2$ gives
	\[\lone{\Delta\dtp}
	\le\f{\lone{\dq_1-\dq_2}}
		{ \dq_1(\grn\grn) }
	+ \f{|\dq_1(\grn\grn)
		-\dq_2(\grn\grn)|}
		{\dq_1(\grn\grn)\dq_2(\grn\grn)}
	\le\f{2\lone{\dq_1-\dq_2}}
	{\dq_1(\grn\grn)\dq_2(\grn\grn)}\,.\]
Since $\dq_i\in\GAMMA$, we have,
using \eqref{e:contract.second.a}~and~\eqref{e:contract.second.b},
	\beq\label{e:BPt_DQh_ub}
	\dtp_i(\dCOLS^2\setminus\set{\red\red})= O(1)\,,\quad
	\dtp_i(\red\red)=O(2^{(1-\kappa)k})\,.\eeq
Consequently $\dq_i(\grn\grn)^{-1}\le O(1) 2^{(1-\kappa)k}$, and the claimed bound on $\lone{\Delta\dtp}$ follows. The bound on $\lone{\dtp_i-\flip\dtp_i}$ follows by noting that if $\dq_2=\flip\dq_1$, then $\dq_1(\grn\grn)=\dq_2(\grn\grn)$.\end{proof}

\begin{proof}[Proof of Lemma~\ref{l:BPt_ptohp_PART_A}] We will prove \eqref{e:BPt_hatdif} for $\htp_i$; the proof for $\flip\htp_i$ is entirely similar. It follows from the symmetry $\dtp_i=(\dtp_i)^\textup{av}$ that for any $\bx,\by\in\set{\zro,\one}$,
	\[\Big|\dtp_i(\blu\blu)
	-4\dtp_i(\blu_\bx\blu_\by)\Big|
	=2\Big| \dtp_i(\blu_\bx\blu_{\by\oplus\one})
	-\dtp_i(\blu_\bx\blu_\by)\Big|
	=2\Big| \dtp_i(\bluz\bluz)
	-\dtp_i(\bluz\bluo)\Big|,\]
from which we obtain that
	\[\Delta\dtp(\blu\blu)
	\lesssim\max_{i=1,2}
	\Big| \dtp_i(\bluz\bluz)
	-\dtp_i(\bluz\bluo)\Big|\,.\]
Recall $\grn=\set{\blu,\fcl}$ and $\dtp_i(\grn\grn)=1$. Combining the above with the definition of $\GAMMA(c,\kappa)$ gives
	\begin{align}\nonumber
	\Delta\dtp(\grn\grn) 
	&\le\Delta\dtp(\blu\blu)+
	\Delta\dtp(\grn\fcl)
	+\Delta\dtp(\fcl\grn)\\
	&\le\sum_{i=1,2}\bigg\{
	\Big| \dtp_i(\bluz\bluz)
	-\dtp_i(\bluz\bluo)\Big|
	+\dtp_i(\grn\fcl)
	+\dtp(\fcl\grn) 
	\bigg\} = O(2^{-ck}).
	\label{e:delta.gg}
	\end{align}

\noindent\textit{Step I.} We first control $\Delta\hatm^\lambda\htp(\hsi)$. By symmetry it suffices to analyze the \textsc{bp} recursion at a clause with all literals $\lit_j=\zro$. We distinguish the following cases of $\hsi\in\hCOLS^2$:
\begin{enumerate}[1.]
\item Recall $\ylw\equiv\red \cup \fcl$, and note  $\set{\ylw}\setminus\set{\spc} \subseteq\hatOm_\zro\cup\hatOm_\one$ (as defined by \eqref{e:hat.Omega.zero.or.one}). Thus
	\beq\label{e:BPt_DQyy}
	\Delta\hatm^\lambda\htp(
	\set{\ylw\ylw}\setminus\set{\spc\spc})
	\le\sum_{\bx\in\set{\zro,\one}} 
	\bigg\{
		\Delta\hatm^\lambda\htp(\hatOm_\bx
		\times\set{\ylw})
		+ \Delta\hatm^\lambda\htp(
		\set{\ylw}\times\hatOm_\bx)
	\bigg\}. \eeq
For $\bx\in\set{\zro,\one}$, consider  $\hsi\in \hatOm_\bx\times\set{\ylw}$: in order for $\vec{\dsi}\in(\dCOLS^2)^{k-1}$ to be compatible with $\hsi$, it is necessary that $\dsi_j\in A \equiv \set{\grn_\bx}\times\set{\grn}$ for all $2\le j\le k$. Combining with \eqref{e:BP_ineq} gives
	\[\Delta\hatm^\lambda\htp(\hatOm_\bx
		\times\set{\ylw})
	\le\sum_{\vec{\dsi}\in A^{k-1}}
	\bigg|
	\prod_{j=2}^k\dtp_1(\dsi_j)
	-\prod_{j=2}^k\dtp_2(\dsi_j)
	\bigg|
	\le k\Delta\dtp(\grn\grn)
	\Big(
	\dtp_1(A)+\Delta\dtp(\grn\grn)
	\Big)^{k-2}\,.\]
It follows from  the definition of $\GAMMA(c,\kappa)$ that  $\dtp_1(A)+\Delta\dtp(\grn\grn) = \tfrac12+O(2^{-ck})$, so we conclude
	\beq\label{e:BPTwo_zroff}
	\Delta\hatm^\lambda\htp(
	\set{\ylw\ylw}\setminus\set{\spc\spc})
	= O(k/2^k)
	\Delta\dtp(\grn\grn)\,.\eeq

\item Now take $\hsi=\spc\spc$: for $\vec{\dsi}\in(\dCOLS^2)^{k-1}$ to be compatible with $\hsi$, it is necessary that $\vec{\dsi}\in\set{\ylw\ylw}^{k-1}$. On the other hand, it is sufficient that $\vec{\dsi}\in\set{\grn\grn}^{k-1}$ does not belong to any of the sets $(A_\zro)^{k-1},(A_\one)^{k-1},(B_\zro)^{k-1},(B_\one)^{k-1}$, where for $\bx\in\set{\zro,\one}$ we define $A_\bx\equiv\set{\blu_\bx\grn}\cup\set{\fcl\grn}$ and $B_\bx\equiv\set{\grn\blu_\bx}\cup
	\set{\grn\fcl}$. Therefore
	\[\Delta\hatm^\lambda\htp(\spc\spc)
	\le\sum_{\bx\in\set{\zro,\one}}
	\sum_{\vec{\dsi} \in (A_\bx)^{k-1}
		\cup (B_\bx)^{k-1}}
	\bigg|
	\prod_{j=2}^k\dtp_1(\dsi_j)
	-\prod_{j=2}^k\dtp_2(\dsi_j)
	\bigg|
	= O(k/2^k) \Delta\dtp(\grn\grn),\]
where the last estimate follows by the same argument that led to \eqref{e:BPTwo_zroff}. This concludes the proof of the first line of \eqref{e:BPt_hatdif}.

\item Now consider $\hsi$ with exactly one coordinate in $\set{\blu}$, meaning the other must be in $\set{\ylw}$. Recalling convention~\eqref{e:dot.m.convention}, we assume without loss that $\hsi\in\set{\bluo\ylw}$ and proceed to bound $\Delta\hatm^\lambda\htp(\hsi)$. Let $\vec{\dsi}\in(\dCOLS^2)^{k-1}$ be  compatible with $\hsi$. There are two cases:
\begin{enumerate}[a.]
\item If $\vec{\dsi}$  has no entry in $\set{\red}$, it must also be compatible with some $\hsi'\in\set{\ylw\ylw}$, as long as we permit the possibility that $|(\hsi')^1|>T$. Such $\vec{\dsi}$ gives the same contribution to $\hatm^\lambda\htp(\hsi)$ as to  $\hatm^\lambda\htp_\infty(\ylw\ylw)$. It follows from the preceding estimates that the contribution to $\Delta\hatm^\lambda\htp(\bluo\ylw)$ from all such $\vec{\dsi}$ is upper bounded by
	\beq\label{e:BPt_DQyy_infty}
	\Delta\hatm^\lambda\htp_\infty(\ylw\ylw)
	=O(k/2^k) \Delta\dtp(\grn\grn)\eeq

\item The only remaining possibility is that  some permutation of $\vec{\dsi}$ belongs to $A\times B^{k-2}$ for $A=\set{\redz\grn}$ and $B=\set{\bluo\grn}$: the contribution to $\Delta\hatm^\lambda\htp(\bluo\ylw)$ from all such $\vec{\dsi}$ is
	\beq\label{e:BPt_DQby_forced}
	\le (k-1)\sum_{\vec{\dsi}\in
	A\times B^{k-2}}\bigg|\prod_{j=2}^k \dtp_1(\dsi_j)
	-\prod_{j=2}^k \dtp_2(\dsi_j)\bigg|
	= O(k^2/2^k)\lone{\Delta\dtp},\eeq
where the last estimate follows using \eqref{e:BP_ineq_generalized}~and~\eqref{e:BPt_DQh_ub}.
\end{enumerate}
Combining the above estimates (and using the symmetry between $\bluo\ylw$ and $\ylw\bluo$) gives
	\beq\label{e:BPt_Dhp_bluylw}
	\Delta\hatm^\lambda\htp(\bluo\ylw)
	+\Delta\hatm^\lambda\htp(\ylw\bluo)
	= O(k^2/2^k)\lone{\Delta\dtp}\,.\eeq
If we further have $\hsi\in\set{\bluo}\times\set{\red,\lgf}$, then, arguing as above, $\vec{\dsi}$ either contributes to $\Delta\hatm^\lambda \htp_\infty(\ylw\times\set{\red,\lgf})$, or else belongs to $A_\bx \times (B_\bx)^{k-2}$ for $A_\bx=\set{\redz\grn_\bx}$, $B_\bx=\set{\blu_1\grn_\bx}$ and $\bx\in\set{\zro,\one}$. The contribution from first case is bounded by \eqref{e:BPTwo_zroff}. The contribution from the second case, using \eqref{e:BP_ineq_generalized}~and~\eqref{e:BPt_DQh_ub}, is
	\[	\lesssim k \Delta\dtp(\dCOLS^2\setminus\set{\red\red}) 
		\Big(\max_{\bx\in\set{\zro,\one}} \dtp_1(B_\bx) 
		+ \Delta\dtp(\grn\grn) 
		\Big)^{k-2}
		=O(k^2/4^k) \Delta\dtp(\dCOLS^2\setminus\set{\red\red})\,.\]
The second claim of \eqref{e:BPt_hatdif} follows by combining these estimates and recalling \eqref{e:delta.gg}.

\item Lastly, consider $\hsi\in\set{\blu\blu}$. Without loss of generality, we take $\hsi=\bluo\bluo$ and proceed to bound $\Delta\hatm^\lambda\htp(\bluo\bluo)$. Let $\vec{\dsi}\in(\dCOLS^2)^{k-1}$ be compatible with $\hsi$. We distinguish three cases: \begin{enumerate}[a.] \item For at least one $i\in\set{1,2}$, $\vec{\dsi}^i$ contains no entry in $\set{\red}$. In this case $\vec{\dsi}$ is also compatible with some $\hsi'\in\set{\bluo\ylw} \cup\set{\ylw\bluo}$, as long as we permit the possibility that $|(\hsi')^i|>T$. The contribution of all such $\vec{\dsi}$ to $\Delta\hatm^\lambda\htp(\bluo\bluo)$ is therefore upper bounded by
	\beq\label{e:BPt_Dhp_bluylw_infty}
	\Delta\hatm^\lambda
	\htp_\infty(\bluo\ylw)
	+\Delta\hatm^\lambda
	\htp_\infty(\ylw\bluo)
	= O(k^2/2^k)\lone{\Delta\dtp},\eeq
where the last step is by the same argument as for \eqref{e:BPt_Dhp_bluylw}.

\item The next case is that  $\vec{\dsi}$  is a permutation of  $(\redz\redz,  (\bluo\bluo)^{k-2})$. The contribution to $\Delta\hatm^\lambda\htp(\bluo\bluo)$ from this case is at most
	\[(k-1)\bigg|
	\dtp_1(\redz\redz)
	\dtp_1(\bluo\bluo)^{k-2}
	-\dtp_2(\redz\redz)
	\dtp_2(\bluo\bluo)^{k-2}
	\bigg|\,.\]
Using \eqref{e:BP_ineq_generalized} and the definition of $\GAMMA(c,\kappa)$, this is at most
	\begin{align}\nonumber
	&O(k^2/4^k)
	\Big(
	\Delta\dtp(\redz\redz)
	+ \dtp(\redz\redz)
	\cdot\Delta\dtp(\bluo\bluo) 
	\Big)\\
	&=O(k^2/4^k)\lone{\dtp}\lone{\Delta\dtp}
	= O(k^2/2^{(1+\kappa)k})\lone{\Delta\dtp}.
	\label{e:subleading.contribution}
	\end{align}

\item The last case is that $\vec{\dsi}$ is a permutation of $(\redz\bluo, \bluo\redz, (\bluo\bluo)^{k-3})$. The contribution to $\Delta\hatm^\lambda\htp(\bluo\bluo)$ from this case is at most
	\[k^2
	\bigg|
	\dtp_1(\redz\bluo)
	\dtp_1(\bluo\redz)
	\dtp_1(\bluo\bluo)^{k-3}
	-\dtp_2(\redz\bluo)
	\dtp_2(\bluo\redz)
	\dtp_2(\bluo\bluo)^{k-3}
	\bigg|\,.\]
This is at most $O(k^2/4^k) \lone{\Delta\dtp}$ by another application of \eqref{e:BP_ineq_generalized} and the definition of $\GAMMA(c,\kappa)$.
\end{enumerate}
The above estimates together give
	\beq\label{e:delta.hm.lambda.blues.one}
	\Delta\hatm^\lambda\htp(\bluo\bluo)
	= O(k^2/2^k)\lone{\Delta\dtp},\eeq
where the main contribution comes from \eqref{e:BPt_Dhp_bluylw_infty}. Combining with the previous bound \eqref{e:BPt_Dhp_bluylw} yields the   last  part of \eqref{e:BPt_hatdif}.
\end{enumerate}
\smallskip\noindent\textit{Step II.} Next we prove \eqref{e:BPt_hatoverlap} by improving the preceding bounds in the special case that $\dtp_1=\dtp$ and $\dtp_2 \equiv \flip\dtp$. Recall $\htp_i \equiv \hBPu(\dtp_i)$; it follows that $\htp_2=\flip\htp_1$. Thus, for any $\hsi\in\hCOLS^2$ with $\hsi^2=\spc$, we have $\hsi=\flip\hsi$, consequently $\htp_2(\hsi) =\htp_1(\flip\hsi)=\htp_1(\hsi)$. For $\hsi\in\hCOLS^2$ with $\hsi^1=\spc$, we have $\hsi=(\flip\hsi)\oplus\one$, so $\htp_2(\hsi)=\htp_1(\flip\hsi)=\htp_1(\hsi)$, where the last step uses that $\htp_1=(\htp_1)^\textup{av}$. It follows that instead of \eqref{e:BPt_DQyy} and \eqref{e:BPt_DQyy_infty} we have the improved bound
	\begin{align*}
	\Delta\hatm^\lambda\htp_\infty(\ylw\ylw)
	&=\Delta\hatm^\lambda\htp_\infty(\set{\ylw\ylw}
		\setminus(\set{\spc\ylw}\cup\set{\ylw\spc}))
	\le\sum_{\bx,\by\in\set{\zro,\one}} 
		\Delta\hatm^\lambda\htp_\infty(\hatOm_\bx
		\times\hatOm_\by)\\
	&= O(k) \lone{\Delta\dtp}
	\sum_{\bx,\by\in\set{\zro,\one}} 
	\Big(
	\dtp_1(\grn_\bx,\grn_\by)
	+\Delta\dtp(\grn\grn)
	\Big)^{k-2}
	= O(k/4^k)\lone{\dtp-\flip\dtp}.
	\end{align*}
Similarly, instead of \eqref{e:BPt_DQby_forced} we would only have a contribution from $\vec{\dsi}$ belonging to either $A_\zro \times(B_\zro)^{k-2}$ or $A_\one \times(B_\one)^{k-2}$, where $A_\bx=\set{\redz\grn_\bx}$ and $B_\bx=\set{\bluo\grn_\bx}$. It follows that instead of \eqref{e:BPt_Dhp_bluylw} and \eqref{e:BPt_Dhp_bluylw_infty} we have the improved bound
	\[\Delta\hatm^\lambda\htp_\infty
		(\bluo\ylw)
	+\Delta\hatm^\lambda\htp_\infty
		(\ylw\bluo)
	=O(k^4/4^k)\lone{\Delta\dtp}\,.\]
Previously the main contribution in \eqref{e:delta.hm.lambda.blues.one} came from \eqref{e:BPt_Dhp_bluylw_infty}, but now it comes instead from \eqref{e:subleading.contribution}. This gives the improved bound $\Delta\hatm^\lambda\htp(\bluo\bluo) =O(k^2/2^{(1+\kappa)k})$, which proves the first part of \eqref{e:BPt_hatoverlap}. The second part follows by applying Lemma~\ref{l:BPt_qtop}.\end{proof}

\begin{proof}[Proof of Lemma~\ref{l:BPt_ptohp_PART_B}] We first prove \eqref{e:BPt_hatvalue}. Assume $s,t\in\set{\bluo,\fcl,\spc}$, and write $st\equiv s\times t \subseteq \hCOLS^2$. Then for a lower bound we have
	\[\hatm^\lambda\htp_i(st)
	\ge[1-O(k/2^k)]
	\dtp_i(\blu\blu)^{k-1}
	=1-O(k/2^k)\,.\]
for an upper bound we have
	\begin{align*}
	\hatm^\lambda\htp_i(st)
	&\le\dtp_i(\grn\grn)^{k-1}
	+ k\dtp_i(\redz\grn)
	\dtp_i(\bluo\grn)^{k-2}
	+ k\dtp_i(\grn\redz)
	\dtp_i(\grn\bluo)^{k-2}\\
	&\qquad + k\dtp_i(\redz\redz)
		\dtp_i(\bluo\bluo)^{k-2}
	+ k^2
	\dtp_i(\redz\bluo)
	\dtp_i(\bluo\redz)
	\dtp_i(\bluo\bluo)^{k-3}
	= 1+O(k^2/2^k).
	\end{align*}
Writing
$\redo t\equiv \redo\times t$
for $t\in\set{\bluo,\fcl,\spc}$,
a similar argument gives
	{\setlength{\jot}{0pt}\begin{align}\nonumber
	\hatm^\lambda\htp_i(\redo t)
	&\ge[1-O(k/2^k)] \dtp_i(\bluz\blu)^{k-1}
	= [1-O(k/2^k)] \cdot (2/2^k)\,,\\
	\hatm^\lambda\htp_i(\redo t)
	&\le\dtp_i(\bluz\grn)^{k-1}
	+ k
	\dtp_i(\bluz\redz)
	\dtp_i(\bluz\bluo)^{k-2}
	= [1-O(k/2^k)] \cdot (2/2^k)\,.\end{align}}%
Lastly, it is easily seen that
	\[\hatm\htp_i(\redo\redo)
	= \dtp_i(\bluz\bluz)^{k-1}
	= [1-O(k/2^{ck})] \cdot (2/2^k)^2\,.\]
This concludes the proof of \eqref{e:BPt_hatvalue}, and we turn next to the proof of \eqref{e:BPt_hatlgf}.
Arguing similarly as for \eqref{e:first.bound.lgf} gives
\[\hatm^\lambda\htp_i(\set{\fcl\fcl}
	\setminus\set{\spc\spc})
	\le \hatm^\lambda\htp_i(\lgf\fcl)
	+\hatm^\lambda\htp_i(\fcl\lgf)
	= O(k/4^{k})\,.\]
Next, suppose $\vec{\dsi}$
is compatible with $\hsi\in\bluo\lgf$:
if $\vec{\dsi}$ has no entry in $\set{\red}$,
then it is also compatible with some
$\hsi'\in\fcl\lgf$, provided we allow
$|(\hsi')^1|>T$. Therefore
	\begin{align*}
	&\hatm^\lambda\htp_i(\bluo\lgf)
	-
	\hatm^\lambda\htp_{i,\infty}
		(\fcl\lgf) \\
	&\qquad \le \sum_{\by\in\set{\zro,\one}}
	\bigg[
	k\dtp_i(\redz\fcl)
	\dtp_i(\bluo\grn_\by)^{k-2}
	+k^2
	\dtp_i(\redz\blu_\by)
	\dtp_i(\bluo\fcl)
	\dtp_i(\bluo\grn_\by)^{k-3}
	\bigg],
	\end{align*}
and by definition of $\GAMMA(c,\kappa)$
this is $O(k/4^k)$. Finally,
 	\[\hatm^\lambda\htp_i(\redo\lgf)
	\le\sum_{\by\in\set{\zro,\one}}
	k\dtp_i(\bluz\fcl)
	\dtp_i(\bluz\grn_\by)^{k-2} =O(k/8^k),\]
which proves the first part of~\eqref{e:BPt_hatlgf}. 
For the second part, arguing as for \eqref{e:BP_Qhfb_dif}, we have for any $\heta\in\hCOLS$ that 
	\[\hatm^\lambda\htp_i(\fcl\heta) 
	-\hatm^\lambda\htp_i(\bluo\heta)
	\le (k-1)
	\sum_{\vec{\dsi}\sim\heta}
	[
	\dtp_i(\grnz\dsi_2)
	-\dtp_i(\redz\dsi_2)
	]
	\prod_{j=3}^k
	\dtp_i(\bluo\dsi_j)\,.\]
Note that $\vec{\dsi}$ has at most one 
entry in $\set{\red}$. If $\dsi_2=\redz$, then $\dsi_j=\bluo$ for all $j\ge3$.
Since $\dq_i\in\GAMMA(c,\kappa)$
(which means also that $\dq_i=(\dq_i)^\textup{av}$), we have
	\[\sum_{\vec{\dsi}\sim\heta}
	\Ind{\dsi_2=\zeta}
	\prod_{j=3}^k
	\dtp_i(\bluo\dsi_j)
	\le \begin{cases}
	\dtp_i(\bluo\bluo)^{k-2}
	\le O(4^{-k})
	& \textup{if $\zeta=\redz$,}\\
	\dtp_i(\bluo\grn)^{k-3} \le O(2^{-k})
	& \textup{if $\zeta\in\dCOLS
	\setminus\set{\redz}$.}
	\end{cases}\]
On the other hand, $\dq_i\in\GAMMA(c,\kappa)$ also implies
	\[\dtp_i(\grnz\zeta)
	-\dtp_i(\redz\zeta)
	\le O(2^{-k})
	\dtp_i(\bluz\zeta)
	+\dtp_i(\fcl\zeta)
	\le \begin{cases} 
	O(1) &\textup{if $\zeta=\redz$,}\\
	O(2^{-k}) &\textup{if $\zeta\in\dCOLS
	\setminus\set{\redz}$.}\end{cases}\]
Combining these estimates and summing over $\heta\in\hCOLS$ proves the second part of \eqref{e:BPt_hatlgf}.\end{proof}

An immediate application of \eqref{e:BPt_hatvalue}, which will be useful in the next proof, is that
	\beq\label{e:BPt_Qhrbratio}
	\f{\hatm^\lambda\htp_i(\red_\bx \heta)}
		{\hatm^\lambda\htp_i(\blu_\bx \heta)}
		\ge [1+O(k^2/2^{k})]
	\cdot (2/2^k)\,.\eeq
for all $\heta\in\set{\bluz,\bluo,\fcl,\spc}$ and all $\bx\in\set{\zro,\one}$.

\begin{proof}[Proof of Lemma~\ref{l:BPt_hptoq}]
We divide the proof in two parts.

\smallskip
\noindent\textit{Step I. Non-normalized messages.}

\begin{enumerate}[1.]
\item First consider $\dsi\in\set{\fcl\fcl}$.
Recalling $(a+b)^\lambda
\le a^\lambda + b^\lambda$
for $a,b\ge0$ and $\lambda\in[0,1]$,
	\[\Delta\utp(\fcl\fcl)
	\le 2
	\sum_{\hat{r}
		\in\set{\htp,\flip\htp} }
	\sum_{\vec{\hsi}
		\in \set{\fcl\fcl}^{k-1}}
	\bigg|\prod_{j=2}^d
	\hatm^\lambda\hat{r}_1(\hsi_j)
	-\prod_{j=2}^d\hatm^\lambda\hat{r}_2(\hsi_j)
	\bigg|\]
where the $\hat{r}=\flip\htp$
term arises from the fact that
	\[\hatm(\hsi^1)^\lambda
	[1-\hatm(\hsi^2)]^\lambda
	\htp(\hsi)
	=\hatm(\hsi^1)^\lambda
	\hatm(\hsi^2\oplus\one)^\lambda
	(\flip\htp)(\flip\hsi)
	=\hatm^\lambda\flip\htp(\flip\hsi)\,.\]
Applying \eqref{e:BP_ineq} gives
	\[\Delta\utp(\fcl\fcl)
	=O(d)
	\sum_{\hat{r}\in\set{\htp,\flip\htp}}
	\Delta\hatm^\lambda\hat{r}(\fcl\fcl)
	\Big(\hatm^\lambda\hat{r}_1(\fcl\fcl) 
		+ \Delta\hatm^\lambda\hat{r}(\fcl\fcl)
	\Big)^{d-2}\,.\]
We have from 
\eqref{e:BPt_hatdif}
and \eqref{e:BPt_hatvalue}
that 
$\hatm^\lambda\htp_1(\fcl\fcl)
\asymp1$
and
$\Delta\hatm^\lambda\htp(\fcl\fcl)
= O(k^3/2^{(1+c)k})$, so 
	\beq\label{e:BPt_DQuff}
	\Delta\utp(\fcl\fcl)
	= O(d)\lone{\Delta\hatm^\lambda\htp 
		+ \Delta\hatm^\lambda\flip\htp}
	\cdot \utp_1(\fcl\fcl)\,.\eeq

\item Next consider $\dsi\in\set{\ppl_\one\fcl}$.
Let $\hat{r}_{\max}(\hsi) \equiv \max_{i=1,2}\hat{r}_i(\hsi)$ --- in this notation,
	\[\hat{r}_{\max}(\hCOLS)
	= \sum_{\hsi\in\hCOLS}
		\max_{i=1,2}
		\hat{r}_i(\hsi)
	\ge \max_{i=1,2}
	\sum_{\hsi\in\hCOLS}
		\hat{r}_i(\hsi)
	= \max_{i=1,2}\hat{r}_i(\hCOLS)\]
where the inequality may be strict. Then
	\[\Delta\utp(\ppl_\one\fcl)
	= O(d)
	\sum_{\hat{r}\in\set{\htp,\flip\htp}}
	\Delta\hatm^\lambda\hat{r}(\ppl_\one\fcl)
	[\hatm^\lambda\hat{r}_{\max}
		(\ppl_\one\fcl)
	]^{d-2}\,.\]
Let $a\in\argmax_i \hat{r}_i(\bluo\spc)$, so that
	\[0\le\hatm^\lambda\hat{r}_{\max}(\ppl_\one\fcl)
	-\hatm^\lambda\hat{r}_a(\ppl_\one\fcl)
	\le\Delta\hatm^\lambda\hat{r}(\redo\fcl)
	+\Delta\hatm^\lambda\hat{r}
	(\bluo\lgf)
	= O(2^{ -(1+ c)k}),\]
where the last estimate is by \eqref{e:BPt_hatdif}
and \eqref{e:BPt_hatlgf}. 
We also have from
\eqref{e:BPt_hatvalue}
that
$\hatm^\lambda\htp(\ppl_\one\fcl)
\ge
\hatm^\lambda\htp(\bluo\fcl)
\asymp1$, and it follows that
\beq\label{e:BPt_kpone}
	[\hatm^\lambda \hat{r}_{\max} 
	(\ppl_\one\fcl)]^{d-2}
	\asymp [\hatm^\lambda\hat{r}_a
	(\ppl_\one\fcl)]^{d-1}.
\eeq
Applying \eqref{e:BPt_hatvalue}
and \eqref{e:BPt_hatlgf} again, we have
(for $i=1,2$)
	\[[\hatm^\lambda\hat{r}_i
	(\ppl_\one\fcl)]^{d-1}
	\asymp [
	\hatm^\lambda\hat{r}_i
	(\ppl_\one\spc)]^{d-1}\,.\]
On the other hand, 
assuming $T\ge1$, we have
	\[\utp_i(\redo\fcl)
	\ge 
	[
	\hatm^\lambda\hat{r}_i(\ppl_\one\spc)
	]^{d-1}
	-[
	\hatm^\lambda\hat{r}_i(\bluo\spc)
	]^{d-1}
	\asymp [
	\hatm^\lambda\hat{r}_i(\ppl_\one\spc)
	]^{d-1}\]
where the last step follows by \eqref{e:BPt_Qhrbratio}. Similarly,
	\begin{align}
	\utp_i(\redo\fcl) - \utp_i(\bluo\fcl)
	&=O(1) \sum_{\hat{r}\in\set{\htp,\flip\htp}}
	\hatm^\lambda\hat{r}_i(\bluo\fcl)^{d-1} 
	= O(2^{-k}) \sum_{\hat{r}\in\set{\htp,\flip\htp}}
	\hatm^\lambda\hat{r}_i(\ppl_\one\fcl)^{d-1}
	\nonumber\\ \label{e:BPt_Qubfrf}
	&= O(2^{-k}) \utp_i(\redo\fcl)
	= O(2^{-k}) \utp_i(\bluo\fcl),
	\end{align}
where the last step follows by rearranging the terms. Combining the above gives
	\beq\label{e:BPt_DQupplf}
	\Delta\utp(\ppl_\one\fcl) 
	\le O(d) 
	\lone{\Delta\hatm^\lambda\htp 
	+ \Delta\hatm^\lambda\flip\htp} 
	\max_{i=1,2} \utp_i(\bluo\fcl).
\eeq
Clearly, similar bounds hold if we replace
$\ppl_\one\fcl$ with
any of $\ppl_\zro\fcl$,
$\fcl\ppl_\one$, or 
$\fcl\ppl_\zro$. 

\item Lastly we bound $\Delta\utp(\ppl_\bx\ppl_\bx )$
for $\bx,\by\in\set{\zro,\one}$. As in the single-copy recursion, we denote
	{\setlength{\jot}{0pt}\begin{align*}
	\dot{r}(\RMB_\bx\dsi)
	&\equiv
		\dot{r}(\red_\bx\dsi)
		-\dot{r}(\blu_\bx\dsi),\\
	\dot{r}(\dsi\RMB_\bx) 
	&\equiv
		\dot{r}(\dsi\red_\bx)
		-\dot{r}(\dsi\blu_\bx),\\
	\dot{r}(\RMB_\bx\RMB_\by)
	&\equiv\dot{r}(\red_\bx\red_\by) 
	-\dot{r}(\red_\bx\blu_\by) 
		- \dot{r}(\blu_\bx\red_\by) 
		+ \dot{r}(\blu_\bx\blu_\by).\end{align*}}%
Applying \eqref{e:BPt_Qhrbratio} gives
	{\setlength{\jot}{0pt}\begin{align*}
	\utp_i(\RMB_\bx\red_\by)
	&=[\htp_i(\blu_\bx\ppl_\by)]^{d-1}
	= O(2^{-k})
	[\htp_i(\ppl_\bx\ppl_\by)]^{d-1}
	= O(2^{-k})\utp_i(\red_\bx\red_\by),\\
	\utp_i(\RMB_\bx\RMB_\by)
	&=[ \htp_i(\blu_\bx\blu_\by)]^{d-1}
	=O(2^{-k})
	\utp_i(\red_\bx\red_\by).\end{align*}}%
Combining the above estimates gives
	\[\utp_i(\red_\bx\red_\by)
	-\utp_i(\blu_\bx\blu_\by)
	=\utp_i (\RMB_\bx\red_\by)
	+ \utp_i (\red_\bx\RMB_\by)
	- \utp_i (\RMB_\bx\RMB_\by)
	=O(2^{-k})
	\utp_i(\red_\bx\red_\by)\,.\]
Further, it follows from the \textsc{bp} equations that
	{\setlength{\jot}{0pt}\begin{align}\nonumber
	\max\set{
	\utp_i(\red_\bx\RMB_\by),
	\utp_i(\blu_\bx\RMB_\by),
	\utp_i(\RMB_\bx\red_\by),
	\utp_i(\RMB_\bx\blu_\by)
	}
	\le\utp_i(\red_\bx\red_\by)
	-\utp_i(\blu_\bx\blu_\by), \\
	\text{so }
	\utp_i(st)
	=[1+O(2^{-k})]\utp_i(\blu_\bx\blu_\by)
	\text{ for all 
	$s\in\set{\red_\bx,\blu_\bx}$, 
	$t\in\set{\red_\by,\blu_\by}$.}
	\label{e:ppl.blu}\end{align}}%
Similarly, we can upper bound
	\begin{align}\nonumber
	\Delta\utp(\ppl_\bx\ppl_\by)
	&\le4[
	\Delta \utp (\red_\bx\red_\by)
	+ \Delta \utp (\RMB_\bx\red_\by)
	+ \Delta \utp (\red_\bx\RMB_\by)
	+ \Delta \utp (\RMB_\bx\RMB_\by)
	].\\
	&\le O(d)
	\sum_{\hat{r}\in\set{
		\htp,\flip\htp}}
	\sum_{\substack{
		s\in\set{\ppl_\bx,\blu_\bx} \\ 
		t\in\set{\ppl_\by,\blu_\by}}}
	\lone{
	\Delta\hatm^\lambda\hat{r}
	}
	[\hatm^\lambda\hat{r}_{\max}(st)]^{d-2}.
	\label{e:delta.utp.pplx.pply}
	\end{align}
For $\hat{r} \in\set{\htp,\flip\htp}$,
let $a=\argmax_{i=1,2}
\hatm^\lambda\hat{r}_i(\bluo\bluo)$: then,
for any 
$s\in\set{\ppl_\bx,\blu_\bx}$,
$t\in\set{\ppl_\by,\blu_\by}$,
	\begin{align*}
	0&\le {\hatm^\lambda\hat{r}_{\max}(st)}-
	{\max_{i=1,2}\hatm^\lambda\hat{r}_i(st)}
	\le 
	{\hatm^\lambda\hat{r}_{\max}(st)}
	-\hatm^\lambda\hat{r}_a(st)\\
	&\le O(1)
	\Delta\hatm^\lambda \hat{r}(\set{\ppl\ppl}\setminus\set{\blu\blu})
	\le O(1/2^{(1+c)k}),\end{align*}
where the last estimate is by
\eqref{e:BPt_hatdif}. Combining with 
\eqref{e:BP_hatvalue}
and \eqref{e:ppl.blu} gives
	\[\sum_{\substack{
		s\in\set{\ppl_\bx,\blu_\bx} \\ 
		t\in\set{\ppl_\by,\blu_\by}}}
	[\hatm^\lambda
	\hat{r}_{\max}(st)]^{d-2}
	= O(1)
	\Big[
	\max_{i=1,2}
	\hat{r}_i(\ppl_\bx\ppl_\by)\Big]^{d-1}
	=O(1)
	\max_{i=1,2}
	\utp_i(\blu\blu)\,.\]
Substituting into \eqref{e:delta.utp.pplx.pply} gives 
	\beq\label{e:BPt_DQubb}
	\Delta\utp(\ppl_\bx\ppl_\by)
	\le 
	O(d)\lone{\Delta\hatm^\lambda\htp + \Delta\hatm^\lambda\flip\htp} 
	\max_{i=1,2}\utp_i
	(\blu\blu)\,.\eeq
Further,
for any $st\in\set{\red_\bx\RMB_\by,
\RMB_\bx\red_\by,\RMB_\bx\RMB_\by}$, we have
	\beq\label{e:BPt_DQurrbb_dif}
	\Delta\utp(st)
	\le 
	O(k)\lone{\Delta\hatm^\lambda\htp + \Delta\hatm^\lambda\flip\htp} 
	\max_{i=1,2}\utp_i
	(\blu\blu)\,.\eeq
Lastly, in the special case $\htp_2 = \flip\htp_1$, \eqref{e:BPt_DQubb} reduces to
	{\setlength{\jot}{0pt}\begin{align}\nonumber
	|\utp_1(\bluz\bluz)
	-\utp_1(\bluz\bluo)|
	&\le O(d) \lone{\hatm^\lambda\htp_1
	-\hatm^\lambda\flip\htp_1} \utp_1(\blu\blu)\\
	&\le k^5 2^{(1-2\kappa)k}
	\lone{\dtp_i-\flip\dtp_i}\,.
	\label{e:BPt_Quoverlap}\end{align}}%
where the last estimate is by \eqref{e:BPt_hatoverlap}.

\end{enumerate}

\smallskip
\noindent\textit{Step II. Normalized messages.}
Recall $\tq_i\equiv\vBP\dq_i$. It remains to verify that
$\tq_i\in\GAMMA(c',1)$
with $c'=\max\set{0,2\kappa-1}$:
recalling
the definition of $\GAMMA$, this means
	\begin{align}
	\label{e:contract.second.PRIME.a}
	\tag{1$\GAMMA'$} 
	&\textup{$|p(\bluz\bluz)
	-p(\bluz\bluo)| 
	\le (k^9/2^{c'k})p(\blu\blu)$ and
	$p(\fcl\fcl)+p(\set{\fcl\red,\red\fcl})/2^k
		+ p(\red\red)/4^k
		= O(2^{-k}) p(\blu\blu)$;}\\
	\label{e:contract.second.PRIME.b}
	\tag{2$\GAMMA'$} 
	& \textup{$p( \fcl\red )
	= O(2^{-k})p(\blu\blu)$ and
	$p(\red\red) = O(1) p(\blu\blu)$;} \\
	\label{e:contract.second.PRIME.c}
	\tag{3$\GAMMA'$} 
	&
	p(\red_\bx\dsi) \ge [1-O(2^{-k})]
		p(\blu_\bx\dsi)
	\text{ for all }
	\bx\in\set{\zro,\one}\text{ and }
	\dsi\in\dCOLS.
	\end{align}
Condition~\eqref{e:contract.second.PRIME.c} is automatically satisfied due to the \textsc{bp} equations.
The second part of~\eqref{e:contract.second.PRIME.b}
follows from \eqref{e:ppl.blu}. 
The first part of~\eqref{e:contract.second.PRIME.a}
holds trivially if $c'=0$,
and otherwise follows from \eqref{e:BPt_Quoverlap}. We claim that
	\beq\label{e:BPt_Qubfbb_ratio}
	\tq_i(\set{\red\fcl,\fcl\red,\fcl\fcl})
	= O(2^{-k})\tq_i(\blu\blu)\,.\eeq
This immediately gives the first part of~\eqref{e:contract.second.PRIME.b}. Further, the \textsc{bp} equations give $\tq_i(\blu\fcl)\le\tq_i(\red\fcl)$ and $\tq_i(\fcl\blu)\le\tq_i(\fcl\red)$, so the second part of~\eqref{e:contract.second.PRIME.a} also follows. To see that \eqref{e:BPt_Qubfbb_ratio} holds, note that the second part of \eqref{e:BPt_hatlgf} gives
	\begin{align*}
	\utp_i(\fcl\fcl)
	&\le O(1)
	\sum_{\hat{r}\in
	\set{\htp,\flip\htp}}
	[\hatm^\lambda\hat{r}_i(\fcl\fcl)]^{d-1}
	\le O(1)
	\sum_{\hat{r}\in
	\set{\htp,\flip\htp}}
	[\hatm^\lambda\hat{r}_i(\bluo\bluo)]^{d-1},\\
	\utp_i(\redo\fcl)
	&\le O(1)
	\sum_{\hat{r}\in
	\set{\htp,\flip\htp}}
	[\hatm^\lambda\hat{r}_i(\ppl_\one\fcl)]^{d-1}
	\le 
	O(1)
	\sum_{\hat{r}\in
	\set{\htp,\flip\htp}}
	[\hatm^\lambda\hat{r}_i
	(\ppl_\one\bluo)]^{d-1}.
	\end{align*}
Combining with \eqref{e:BPt_Qhrbratio} gives
$\utp_i(\set{\redo\fcl,\fcl\fcl})
= O(2^{-k})\utp_i(\redo\redo)$.
Recalling \eqref{e:ppl.blu} 
(and making use of symmetry)
gives \eqref{e:BPt_Qubfbb_ratio}. 
Finally, we conclude the proof of the lemma by bounding the difference $\Delta\tq \equiv |\tq_1-\tq_2|$. Recalling the definition of $\RMB_\bx$, we have
	{\setlength{\jot}{0pt}\begin{align*}
	\Delta\tq(\ppl\ppl)
	&\le O(1)
	\Delta\tq(\set{\blu\blu,
		\red\RMB,\RMB\red,\RMB\RMB}),\\
	\Delta\tq(\dCOLS^2\setminus\set{\ppl\ppl})
	&\le O(1)\Delta\tq(
	\set{\blu\fcl,\fcl\blu,\fcl\fcl,
	\fcl\RMB,\RMB\fcl}).\end{align*}}%
We next bound
$\Delta\tq(\blu\blu)$, which is the sum of
$\Delta\tq(\blu_\bx\blu_\by)$
over $\bx,\by\in\set{\zro,\one}$. By symmetry let us take $\bx=\by=\zro$. Since $\tq_i=(\tq_i)^\textup{av}$,
	$\tq_i(\bluz\bluz)
	=\tfrac14\tq_i(\blu\blu)
	+\tfrac12[
	\tq_i(\bluz\bluz)
	-\tq_i(\bluz\bluo)
	]$, so
	\[\Delta\tq(\bluz\bluz)
	\le \tfrac14
	|\tq_1(\blu\blu)-\tq_2(\blu\blu)|
	+\tfrac12
	\sum_{i=1,2}|\tq_i(\bluz\bluz)
	-\tq_i(\bluz\bluo)|\,.\]
Since the $\tq_i$ are normalized to be probability measures,
	\[1-\tq_i(\dCOLS^2\setminus\set{\ppl\ppl})
	=\tq_i(\ppl\ppl)
	=2\tq_i(\red\RMB)
	+2\tq_i(\RMB\red)
	-3 \tq_i(\RMB\RMB)
	+4\tq_i(\blu\blu),\]
from which it follows that
	\[|\tq_1(\blu\blu)-\tq_2(\blu\blu)|
	\lesssim
	|\tq_1(\dCOLS^2\setminus\set{\ppl\ppl})
		-\tq_2(\dCOLS^2\setminus\set{\ppl\ppl})|
	+ \Delta\tq(\set{
	\red\RMB,\RMB\red,\RMB\RMB
	})\,.\]
Combining the above estimates gives
	\[\lone{\Delta\tq}
	\lesssim
	\Delta\tq(\texttt{A}
	)
	+
	\sum_{i=1,2}
	|\tq_i(\bluz\bluz)
	-\tq_i(\bluz\bluo)|,\quad
	\texttt{A}
	\equiv\set{
	\blu\fcl,\fcl\blu,\fcl\fcl,
	\fcl\RMB,\RMB\fcl,
	\red\RMB,\RMB\red,\RMB\RMB}\,.\]
Write $\dZ_i\equiv\lone{\utp_i}$. 
Taking $a\in\set{1,2}$ and $b=2-a$, 
we find $\lone{\Delta\tq}\le e_1+e_2e_3+e_4$ with
	\[e_1\equiv
	\f{\Delta\utp(\texttt{A})}{\dot{Z}_a}\,,\quad
	e_2\equiv
	\f{|\dot{Z}_1-\dot{Z}_2|}{\dot{Z}_a}
	\le \f{\lone{\Delta\utp}}{\dot{Z}_a}\,, \quad
	e_3\equiv
	\f{\utp_b(\texttt{A})}{\dot{Z}_b}\,, \quad
	e_4\equiv
	\sum_{i=1,2}
	\f{|\utp_i(\bluz\bluz)
	-\utp_i(\bluz\bluo)|}{\dot{Z}_i}\,.\]
It follows from 
\eqref{e:BPt_DQuff}, 
\eqref{e:BPt_DQupplf}, 
\eqref{e:BPt_DQurrbb_dif} and \eqref{e:BPt_Qubfbb_ratio},
and taking
$a=\argmax_i\utp_i(\blu\blu)$, that
	\[e_1
	\lesssim
	\lone{\Delta\hatm^\lambda\htp+\Delta\hatm^\lambda\flip\htp}
	(d/2^k)
	\max_{i=1,2}\utp_i(\blu\blu)/
	\dot{Z}_a
	\lesssim k
	\lone{\Delta\hatm^\lambda\htp
	+\Delta\hatm^\lambda\flip\htp}\,.\]
Further, recalling \eqref{e:BPt_DQubb} gives
	\[e_2
	\lesssim k2^k
	\lone{\Delta\hatm^\lambda\htp+\Delta\hatm^\lambda\flip\htp}\,.\]
Combining \eqref{e:BPt_Qubfrf}, \eqref{e:ppl.blu}, 
and \eqref{e:BPt_Qubfbb_ratio} gives
$e_3=O(2^{-k})$.
Finally, \eqref{e:BPt_Quoverlap} gives
	\[e_4 \lesssim
	k2^{k}
	\lone{\hatm^\lambda\htp_i
	-\hatm^\lambda\flip\htp_i}\,.\]
Combining the 
pieces together finishes the proof.\end{proof}

\section{The 1RSB free energy}\label{appx:onersb}

\subsection{Equivalence of recursions}\label{s:drec}

In this section, we relate the coloring recursion \eqref{e:defn.vBP} to the distributional recursion \eqref{e:dist.recur}, and prove the following:

\begin{ppn}\label{p:drec_equiv}
Let $\dq_\lambda$ be the fixed point given by Proposition~\ref{p:contract}\ref{p:contract.first} 
for parameters $\lambda\in[0,1]$ and $T=\infty$. Let $H_\lambda\equiv(\dH_\lambda,\hH_\lambda,\eH_\lambda)\in\simplex$ be the associated 
triple of measures defined by Proposition~\ref{p:first.mmt}. 
We then have the identity
$(\size(H_\lambda),\SIGMA(H_\lambda),\bF(H_\lambda)) = (s_\lambda,
		\Sigma(s_\lambda),\mathfrak{F}(\lambda))$.
\end{ppn}

In the course of the proof, we will obtain Proposition~\ref{p:drec_fixpoint} as a corollary. Throughout the section we take $T=\infty$ unless explicitly indicated otherwise. We begin with some notations. Recall that $\mathscr{P}(\albet)$ is the space of probability measures on $\albet$.
Given
$\dq\in\mathscr{P}(\dCOLS)$, we define two associated measures
$\dotm^\lambda\dq,(1-\dotm)^\lambda\dq$ on $\dCOLS$ by
	\[(\dotm^\lambda\dq)(\dsi)
	\equiv
	\dotm(\dsi)^\lambda\dq(\dsi),\quad
	((1-\dotm)^\lambda\dq)(\dsi)
	\equiv
	(1-\dotm(\dsi))^\lambda\dq(\dsi),\]
We let $\dpi\equiv\dpi(\dq)$ be the probability measure on $\dMM\setminus\set{\star}$ given by
	\[\dpi(\dta)= \begin{cases}
	[1-\dq(\red)]^{-1}\dq(\dta)
		& \textup{if $\dta \in
			\dMM\setminus\set{\zro,\one,\star}$,} \\
	[1-\dq(\red)]^{-1}\dq(\blu_\bx)
		& \textup{if 
		$\dta=\bx\in\set{\zro,\one}$.}\end{cases}\]
Recall from \S\ref{appx:contract.single} the mappings $\dotm : \dCOLS\to[0,1]$ and $\hatm : \dCOLS\to[0,1]$.
We then denote the pushforward measure
	$\dtu
	\equiv \dtu(\dq)
	 \equiv \dpi \circ \dotm^{-1}$,
so that $\dtu$ belongs to the space
 $\mathscr{P}$
 of discrete probability measures on $[0,1]$.
Analogously, given 
 $\hq\in\mathscr{P}(\hCOLS)$, we define two associated measures
$\hatm^\lambda\hq,(1-\hatm)^\lambda\hq$ on $\hCOLS$.
We let $\hpi\equiv\hpi(\hq)$
be the probability measure on $\hMM\setminus\set{\star}$ given by
	\[\hpi(\hta) \equiv  \begin{cases}
	[1-\hq(\blu)]^{-1}
		\hq(\hta) & 
		\textup{if $\hta\in 
		\hMM\setminus\set{\zro,\one,\star}$,}\\
		[1-\hq(\blu)]^{-1}
		\hq(\red_\bx) & 
		\textup{if $\hta = \bx\in\set{\zro,\one}$,}
		\end{cases}\]
and we then denote $\htu\equiv\htu(\hq)\equiv\hpi\circ\hatm^{-1}$, so that $\htu\in\mathscr{P}$ also. The next two lemmas follow straightforwardly from the above definitions, and we omit their proofs: 

\begin{lem}\label{l:drec_hatmuequiv}
Suppose $\dq\in\mathscr{P}(\dCOLS)$
satisfies $\dq=\dq^\textup{av}$ and
	\beq\label{e:drec_dotsym}
	\dotm^\lambda\dq(\fcl)
	=\dq(\redo) - \dq(\bluo)
	=\dq(\redz) - \dq(\bluz)
	=(1-\dotm)^\lambda\dq(\fcl)\eeq
Then $\hq\equiv\hBP\dq\in\mathscr{P}(\hCOLS)$ must satisfy $\hq=\hq^\textup{av}$ and
	\beq\label{e:drec_hatsym}
	\hatm^\lambda\hq(\fcl)
	= \hq(\bluo)
	= \hq(\bluz) 
	= (1-\hatm)^\lambda\hq(\fcl), \eeq
Let $\hbz\equiv (\hBPu\dq)/(\hBP\dq)$
be the normalizing constant.
Then $\dtu\equiv\dtu(\dq)$
and $\htu\equiv\htu(\hq)$ satisfy
	\beq \htu
	= \drecH_\lambda 
	(\dtu),
	\quad
	\hat{\mathscr{Z}}_\lambda
	(\dtu)
	= \f{\hbz(1-\hq(\blu))}
		{(1-\dq(\red))^{k-1}}.
	\label{e:drec_hatmuequiv}\eeq
\end{lem}

\begin{lem}\label{l:drec_dotmuequiv}
Suppose $\hq\in\mathscr{P}(\hCOLS)$
satisfies $\hq=\hq^\textup{av}$ and \eqref{e:drec_hatsym}.
Then $\dq\equiv\hBP\hq\in\mathscr{P}(\dCOLS)$ must satisfy
$\dq=\dq^\textup{av}$ and \eqref{e:drec_dotsym}.
Let $\dbz
\equiv (\dBPu\hq)/(\dBP\hq)$ be the normalizing constant: then 
\beq\label{e:drec_dotmuequiv}
	\dtu
	= \drecD_\lambda(\htu),
	\quad
	\dot{\mathscr{Z}}_\lambda(\htu)
	= \f{\dbz(1-\dq(\red))}{(1-\hq(\blu))^{d-1}}.
\eeq
\end{lem}

\begin{proof}[Proof of Proposition~\ref{p:drec_fixpoint}] This is simply a rephrasing of the proof of Proposition~\ref{p:contract}\ref{p:contract.first}, using Lemma~\ref{l:drec_hatmuequiv} and Lemma~\ref{l:drec_dotmuequiv}.\end{proof}

We next prove Proposition~\ref{p:drec_equiv}. In the remainder of this section, fix $\lambda\in[0,1]$ and $T=\infty$. Let $\dq\equiv\dq_\lambda$ be the fixed point of $\vBP\equiv\vBP_{\lambda,\infty}$ given by Proposition~\ref{p:contract}\ref{p:contract.first}. Let $\hq\equiv\hq_\lambda$ denote the image of $\dq$ under the mapping
$\hBP\equiv\hBP_{\lambda,\infty}$.
Denote the associated normalizing constants 
\[\hbz\equiv\hbz_\lambda\equiv (\hBPu\dq)/(\hBP\dq),\quad
\dbz\equiv\dbz_\lambda\equiv (\dBPu\hq)/(\dBP\hq)\,.\]
Let $H_\lambda\equiv(\dH_\lambda,\hH_\lambda,\eH_\lambda)$
be the triple of associated measures
defined
as in Proposition~\ref{p:first.mmt},
with normalizing constants 
$(\dot{\ZH}_\lambda,\hat{\ZH}_\lambda,\bar{\ZH}_\lambda)$.
Recall from~\eqref{e:drec_Sigma} that $\mathfrak{F}(\lambda)=\log\dot{\Zcal}_\lambda+\alpha\log\hat{\Zcal}_\lambda- d\log\bar{\Zcal}_\lambda$. We now show that it coincides with $\bF(H_\lambda)$:

\begin{lem}\label{l:drec_PSI}
Under the above notations,
$\bF(H_\lambda) = \log \dot{\ZH}_\lambda
		+ \alpha \log\hat{\ZH}_\lambda 
		- d \log \bar{\ZH}_\lambda$, and
\beq\label{e:drec_dZdZhZhZ}
	\bar{\Zcal}_\lambda
	= \f{\bar{\ZH}_\lambda} {(1-\dq_\lambda(\red))(1-\hq_\lambda(\blu))}
	,\quad
	\dot{\Zcal}_\lambda
	= \f{\dot{\ZH}_\lambda} {(1-\hq_\lambda(\blu))^{d}}
	,\quad
	\hat{\Zcal}_\lambda
	= \f{\hat{\ZH}_\lambda} {(1-\dq_\lambda(\red))^{k}}
	.
\eeq
Consequently $\mathfrak{F}(\lambda) = \bF(H_\lambda)$.

\begin{proof} It follows from 
the definition
\eqref{e:SIGMA}
(and recalling 
from Corollary~\ref{c:dont.need.literals}
that 
$\hPhi(\usi)^\lambda = \hF(\usi)^\lambda \hat{v}(\usi)$) that
	\[\bF(H_\lambda) 
	=	\langle \log (\dPhi^\lambda/\dH)
			, \dH_\lambda \rangle
		+ \alpha \langle 
		\log( \hPhi^\lambda/\hH_\lambda)
			,\hH_\lambda \rangle 
		+ d \langle \log( \ePhi^\lambda\eH_\lambda)
			, \eH_\lambda
			\rangle\,.\]
Substituting in Definition~\ref{d:Hstar} and rearranging gives
	\begin{align*}&
	\bF(H_\lambda) 
	-\Big(\log \dot{\ZH}_\lambda
		+ \alpha \log\hat{\ZH}_\lambda 
		- d \log \bar{\ZH}_\lambda\Big) \\
	&\qquad =- \Big\langle \sum_{i=1}^d\log\hq_\lambda(\hsi_i),
	\dH_\lambda \Big\rangle
	-\alpha\Big\langle
		\sum_{i=1}^k\log\dq_\lambda(\dsi_i),\hH_\lambda
		\Big\rangle
	+ d \langle \log[\dq_\lambda(\dsi)\hq_\lambda(\hsi)],
	\eH_\lambda \rangle.
	\end{align*}
This equals zero since $H_\lambda\in\simplex$. The proof of \eqref{e:drec_dZdZhZhZ} is straightforward from the preceding definitions, and is omitted.
	\end{proof}
\end{lem}

\begin{proof}[Proof of Proposition~\ref{p:drec_equiv}]
By similar calculations as above,
it is straightforward to verify that
$s_\lambda=\size(H_\lambda)$.
Since by definition
$\mathfrak{F}(\lambda) = \lambda s_\lambda + \Sigma(s_\lambda)$ and 
$\bF(H_\lambda) = \lambda \size(H_\lambda) + \SIGMA(H_\lambda)$, it follows that 
$\Sigma(s_\lambda)=
\SIGMA(H_\lambda)$, concluding the proof.\end{proof}

\begin{proof}[Proof of Proposition~\ref{p:drec_equiv.outline}] Immediate consequence of Proposition~\ref{p:drec_equiv} together with Proposition~\ref{p:contract}\ref{c:BP_limT}.\end{proof}

\subsection{Large-$k$ asymptotics}

We now evaluate the large-$k$ asymptotics of
the free energy, beginning with 
\eqref{e:drec_Sigma}. Let $\dmu_\lambda$ be the probability measure on $[0,1]$ given by Proposition~\ref{p:drec_fixpoint}, and write $\hmu_\lambda\equiv \hat{\mathscr{R}}_\lambda(\dmu_\lambda)$. In what follows it will be useful to denote 
$\dmu_\lambda(\fre)
\equiv\dmu_\lambda((0,1))$, as well as
	\[\psi_\lambda
	\equiv
	\int x^\lambda 
	\Ind{x\in(0,1)} \dmu_\lambda(dx),
	\quad
	\rho_\lambda
	\equiv
	\int y^\lambda
	\Ind{ y\in(0,1)
	\setminus\set{\tfrac12} }
	\hmu_\lambda(dy)\,.\]

\begin{ppn}\label{p:Zcal.estimate}
For $k\ge k_0$, $\albd \le\alpha= (2^{k-1}-c)\log2  \le \aubd$, and $\lambda\in[0,1]$,
	\begin{align}
	\label{e:dot.Zcal.estimate}
	\log\dot{\Zcal}_{\lambda}
	&=\log 2 -(1-2^{\lambda-1})/2^k
	+ d\log \Big( 
	2^{-\lambda}\hmu_\lambda(\tfrac12)
	+\hmu_\lambda(1)
	+\rho_\lambda \Big)
	+ \err,\\
	\label{e:bar.Zcal.estimate}
	-d\log\bar{\Zcal}
	&=- d \log\Big( 2^{-\lambda} 
	\hmu_\lambda(\tfrac12)
	+\hmu(1) + \rho_\lambda\Big)
	-
	(k\log 2)
	[- \dmu_\lambda(\fre) + 2 \psi_\lambda]
	 + \err,\\
	 \label{e:hat.Zcal.estimate}
	\alpha
	\log\hat{\Zcal}
	&=\alpha\log(1-2/2^k)
	+ (k\log2)
	( - \dmu_\lambda(\fre)
		+ 2 \psi_\lambda ) + \err,
	\end{align}
where $\err$ denotes any error bounded by
$k^{O(1)}/4^k$. Altogether this yields
	\[\mathfrak{F}(\lambda)
	=\annFE(\alpha)
	- (1-2^{\lambda-1})/2^k
	+ \err
	= [(2c-1)\log2
		-(1-2^{\lambda-1})]/2^k
	+\err\,.\]
On the other hand
$\lambda s_\lambda = \lambda (\log 2) 2^{\lambda-1} / 2^k + \err$.\end{ppn}

\begin{proof}[Proof of Proposition~\ref{p:onersb.fe.analysis}\ref{p:large.k}]
Apply Proposition~\ref{p:Zcal.estimate}:
setting $\mathfrak{F}(\lambda) = \lambda s_\lambda$ gives
	\[\alpha_\lambda=(2^{k-1}-c_\lambda)\log2+\err,\quad
	c_\lambda = \f{1}{2}
		+ \f{ 
		1 - 2^{\lambda-1}(1-\lambda\log2)
		 }{2\log 2}\,.\]
Substituting the special values $\lambda=1$ and $\lambda=0$ gives 
	\[c_\textup{cond}
	=c_1=1,\quad
	c_\textup{sat}
	=c_0
	=\f12 + \f{1}{4\log2},\]
verifying \eqref{e:alpha.sat.asymptotics} and \eqref{e:alpha.cond.asymptotics}.\end{proof}

\begin{proof}[Proof of Proposition~\ref{p:Zcal.estimate}]
Throughout the proof we abbreviate $\epsilon_k$ for a small error term which may change from one occurrence to the next, but is bounded throughout by $k^C/2^k$ for a sufficiently large absolute constant $C$. Note that
	\[\hmu_\lambda(\tfrac12)
	= 1 -2 \cdot \f{2^{1-\lambda}}{2^k}
	+ \epsilon_k, \quad
	\hmu_\lambda(1)
	= \hmu_\lambda(0)
	= \f{2^{1-\lambda}}{2^k} + \epsilon_k,
	\quad
	\hmu_\lambda((0,1)\setminus\set{\tfrac12})
	= \epsilon_k,\]
from which it follows that $\rho_\lambda =\epsilon_k$. Meanwhile, $\psi_\lambda\le\dmu_\lambda(\fre)$, and we will show below that
 	\beq\label{e:dmu.free.estimate}
	\dmu_\lambda(\fre)= \f{2^{\lambda-1}}{2^k}
	+ \epsilon_k\,.\eeq

\smallskip\noindent
\textit{Estimate of $\dot{\Zcal}_{\lambda}$.}
Recall from the definition
\eqref{e:drec_marginal} that
	\[\dot{\Zcal}_{\lambda}
	= \int
	\bigg(\prod_{i=1}^d y_i
	+\prod_{i=1}^d (1-y_i)\bigg)^\lambda 
	\prod_{i=1}^d \hmu_\lambda(dy_i)\,.\]
Let $\dot{\Zcal}_{\lambda}(\fre)$ denote the contribution to $\dot{\Zcal}_{\lambda}$ from free variables, meaning $y_i\in(0,1)$ for all $i$. This can be decomposed further into the contribution $\dot{\Zcal}_{\lambda}(\fre_1)$
from isolated free variables
(meaning $y_i=1/2$ for all $i$)
and the remainder
$\dot{\Zcal}_{\lambda}(\fre_{\ge2})$. We then calculate
	\[\dot{\Zcal}_\lambda(\fre_1)
	= 2^\lambda 
	\Big( 2^{-\lambda} \hmu_\lambda(\tfrac12)\Big)^d\,.\]
This dominates the contribution from non-isolated free variables:
	\begin{align*}
	\dot{\Zcal}_\lambda(\fre_{\ge2})
	&= \sum_{j=1}^d\binom{d}{j}
	\bigg(\int y^\lambda
	\Ind{y\in(0,1)\setminus\set{\tfrac12}}
	\hmu_\lambda(dy)\bigg)^j
	\Big(
	2^{-\lambda} \hmu_\lambda(\tfrac12)
	\Big)^{d-j} \\
	&\le O(1) d
	\hmu_\lambda( 
	(0,1)\setminus\set{\tfrac12})
	\Big(
	2^{-\lambda} \hmu_\lambda(\tfrac12)
	\Big)^d 
	\le \dot{\Zcal}_\lambda(\fre_1)
	k^{O(1)}/2^k. 
	\end{align*}
Next let $\dot{\Zcal}_{\lambda}(\one)$
denote the contribution from variables frozen to $\one$:
	\begin{align*}
	\dot{\Zcal}_{\lambda}(\one)
	&=\Big(
	\int y^\lambda \hmu_\lambda(dy)
	\Big)^d
	-\Big(
	\int y^\lambda
	\Ind{ y\in(0,1)}
	\hmu_\lambda(dy)
	\Big)^d\\
	&= \Big(
	2^{-\lambda}
	\hmu_\lambda(\tfrac12)
	+\hmu_\lambda(1)+
	\rho_\lambda
	\Big)^d
	- 2^{-\lambda}
	\dot{\Zcal}_{\lambda}(\fre_1)
	+ \epsilon_k.
	\end{align*}
The ratio of free to frozen variables is given by
	\[\f{\dot{\Zcal}_{\lambda}(\fre)}
	{2\dot{\Zcal}_{\lambda}(\one)}
	= \f{2^\lambda}{2}
	\bigg(\f{\hmu_\lambda(\tfrac12)}
	{\hmu_\lambda(\tfrac12)
		+ 2^\lambda \hmu_\lambda(1)} \bigg)^d
	+ \epsilon_k
	= \f{2^{\lambda-1}}{2^k} + \epsilon_k. \]
Combining these yields \eqref{e:dot.Zcal.estimate}.
The proof of \eqref{e:dmu.free.estimate} is very similar.

\smallskip\noindent
\textit{Estimate of $\bar{\Zcal}_\lambda$.}
Recall from the definition
\eqref{e:drec_marginal} that
	\[\bar{\Zcal}_\lambda
	= \int \Big(
		xy+(1-x)(1-y)\Big)^\lambda
	\dmu_\lambda(dx)
	\hmu_\lambda(dy)\,.\]
The contribution to $\bar{\Zcal}$ from $x=0$ or $x=1$ is given by
	\[\bar{\Zcal}_\lambda(x=1)
	= \dmu_\lambda(1)
	\Big(
	2^{-\lambda}
	\hmu_\lambda(\tfrac12)
	+\hmu_\lambda(1)+
	\rho_\lambda
	\Big)=\bar{\Zcal}_\lambda(x=0)\,.\]
The contribution from $x\in(0,1)$ and $y=1/2$ is given by
	\[\bar{\Zcal}_\lambda(x\in(0,1),y=1/2)
	=\dmu_\lambda(\fre)
	2^{-\lambda} 
	\hmu_\lambda(\tfrac12)\,.\]
Lastly, the contribution from
$x\in(0,1)$ and $y=1$ is given by
	\[\bar{\Zcal}_\lambda(x\in(0,1),y=1)
	=\hmu_\lambda(1)
	\psi_\lambda,\]
and there is an equal contribution from the case
$x\in(0,1)$ and $y=0$.
The contribution from the case that both $x,y\in(0,1)$ is $\le k^{O(1)}/8^k$.
Combining these estimates gives
	\begin{align*}
	d\log\bar{\Zcal}_\lambda
	&= d \log\Big( 2^{-\lambda} 
	\hmu_\lambda(\tfrac12)
	+ 2 \dmu_\lambda(1) \hmu(1)
	+ 2 \dmu_\lambda(1) \rho_\lambda
	+ 2 \hmu(1)\psi_\lambda \Big)
	+ \epsilon_k \\
	&=d \log\Big( 2^{-\lambda} 
	\hmu_\lambda(\tfrac12)
	+\hmu(1) + \rho_\lambda\Big)
	+d\log\Big(1 +
	\f{\hmu(1)[-\dmu_\lambda(\fre)
	+ 2 \psi_\lambda]
	}{ 2^{-\lambda} 
	\hmu_\lambda(\tfrac12) }
	 \Big)
	+ \epsilon_k.
	\end{align*}
Recalling
$\hmu_\lambda
=\hat{\mathscr{R}}\dmu_\lambda$ gives
	\[d\log\Big(1+\f{
	\hmu(1)[ -\dmu_\lambda(\fre)
	+ 2 \psi_\lambda]
	}{ 2^{-\lambda} 
	\hmu_\lambda(\tfrac12) }\Big)
	=d\dmu_{\lambda}(0)^{k-1}
	(-\dmu_\lambda(\fre) + 2 \psi_\lambda) + \epsilon_k,\]
and \eqref{e:bar.Zcal.estimate} follows.

\smallskip\noindent
\textit{Estimate of $\hat{\Zcal}_\lambda$.}
Recall from the definition
\eqref{e:drec_marginal} that
	\[\hat{\Zcal}_\lambda
	= \int 
	\bigg(1
	-\prod_{i=1}^k x_i
	-\prod_{i=1}^k (1-x_i)\bigg)
	\prod_{i=1}^k\dmu_\lambda(x_i)\,.\]
Writing $\dmu_\lambda(\zro,\fre)\equiv\dmu_\lambda([0,1))$,
the contribution to
$\hat{\Zcal}$
from separating clauses is
	\[1 - 2 \dmu_\lambda(\zro,\fre)^k+\dmu_\lambda(\fre)^k
	= 1 - (2/2^k) (1+k\dmu_\lambda(\fre))
	+ k^{O(1)}/8^k\,.\]
The contribution from clauses which are forcing to some variable that is not forced by any other clause is
$2k \dmu_\lambda(0)^{k-1} \psi_\lambda$.
The contribution from all other clause types is
$\le k^{O(1)}/8^k$, and 
\eqref{e:hat.Zcal.estimate} follows.

\smallskip\noindent
\textit{Estimate of $s_\lambda$.} Recall from \eqref{e:drec_size_lambda}
 the definition 
of $s_\lambda$.
By similar considerations as above, it is straightforward to check that the total contribution from frozen variables,
edges incident to frozen variables,
and separating or forcing clauses is zero. The dominant 
term is the contribution of isolated free variables, and the estimate follows.\end{proof}

\subsection{Properties of the complexity function}\label{s:SIGMA}

We conclude with a few basic properties of the complexity function $\Sigma(s)$, including a proof of Proposition~\ref{p:onersb.fe.analysis}\ref{p:strict.decrease}.

\begin{lem}\label{l:BP_dlambda} For fixed $1\le T < \infty$, the fixed point $\dq_{\lambda,T}$ of Proposition~\ref{p:contract}\ref{p:contract.first} is continuously differentiable as a function of $\lambda \in [0,1]$.

\begin{proof} 	Fix $T<\infty$ and define $f_T[\dq,\lambda] \equiv \vBP_{\lambda,T}[\dq]-\dq$ as the mapping from $\mathscr{P}(\dtcols)\times [0,1]$ to the set of signed measures on $\tcols$. Since function $\dz(\dsi)$ ($\hz(\hsi)$, respectively) can take only finitely many values on $\dtcols$ ($\htcols$, respectively) and therefore must be uniformly bounded away from 0. It is straightforward to check that for any $\lambda \in [0,1]$, 
	\[	f_T[\dq_\star(\lambda,T),\lambda](\dsi) = 0
		,\quad \forall \dsi\in\tcols,\] 
and is uniformly differentiable in a neighborhood of $\set{(\dq_\star(\lambda,T),\lambda):\lambda\in[0,1]}$. 
	
	 For any other $\dq$ in the contraction region \eqref{e:contract.first}, Proposition~\ref{p:BP_contract} guarantees that
	\begin{align*}
	\lone{f_T[\dq,\lambda] - f_T[\dq_\star(\lambda,T),\lambda]}
	&\ge \lone{\dq-\dq_\star(\lambda)} 
	- \lone{\vBP_{\lambda,T}[\dq] 
	- \vBP_{\lambda,T}[\dq_\star(\lambda,T)]}\\
	&\ge (1-O(k^22^{-k})) \lone{\dq-\dq_\star(\lambda,T)}.
	\end{align*}
Therefore the Jacobian matrix 
	\[\Big(\f{\partial f_T(\dsi_i)}
		{\partial\dq(\dsi_j)} \Big)_{\dCOLS\times\dCOLS}
	\]
is invertible at each $(\dq_\star(\lambda,T),\lambda)$. By implicit function theorem, $\dq_\star(\lambda,T)$, as the solution of $f_T[\dq,\lambda] = 0 $, is uniformly differentiable in $\lambda$.\end{proof}\end{lem}

Let us first fix $T<\infty$ and consider
the clusters encoded by $T$-colorings.
We have explicitly defined $\SIGMA(H)$ and $\size(H)$. Let
	$\mathcal{S}(s)
	\equiv
	\sup\set{ \SIGMA(H) : 
	\size(H)=s}$,
with the convention that a supremum over an empty set is $-\infty$. Thus $\mathcal{S}(s)$ is a well-defined function which captures the spirit of the function $\Sigma(s)$ discussed in the introduction. (Note $\mathcal{S}$ implicitly depends on $T$ since the maximum is taken over empirical measures $H$ which are supported on $T$-colorings.) Recall that the physics approach (\cite{MR2317690} and refs.\ therein) takes $\mathcal{S}(s)$ as a conceptual starting point. However, for purposes of explicit calculation the actual starting point is the Legendre dual
	\[\mathfrak{F}(\lambda)
	\equiv (-\mathcal{S})^\star(\lambda)
	= \sup_{s\in\R}
	\Big\{\lambda s + \mathcal{S}(s)\Big\}
	= \sup_H \bF_\lambda(H),\]
where $\bF_\lambda(H)\equiv \lambda\size(H)+\SIGMA(H)$. The replica symmetry breaking heuristic gives an explicit conjecture for $\mathfrak{F}$. One then makes the assumption that $\mathcal{S}(s)$ is \bemph{concave} in $s$: this means it is the same as
	\[ \mathcal{R}(s)
	\equiv
	- \mathfrak{F}^\star(s)
	= -(-\mathcal{S})^{\star\star}(s),\]
so if $\mathcal{S}$ is concave then it can be recovered from $\mathfrak{F}$.

We do not have a proof that $\mathcal{S}(s)$ is concave for all $s$, but we will argue  that this holds on the interval of $s$ corresponding to $\lambda\in[0,1]$. Formally, for $\lambda\in[0,1]$, we proved that $\bF_\lambda(H)$ has a unique maximizer  $H_\star\equiv H_\lambda$. This implies that there is a unique $s_\lambda$ which maximizes $\lambda s +\mathcal{S}(s)$,  given by
	\[s_\lambda=\size(H_\lambda)\,.\]
Recall that $H_\lambda$ and $s_\lambda$ both depend implicitly on $T$. We also have from Lemma~\ref{l:BP_dlambda}  that for any fixed $T<\infty$, $s_\lambda$ is continuous in $\lambda$, so it maps $\lambda\in[0,1]$ onto some compact interval $\mathcal{I} \equiv [s_-,s_+]$.  Define the modified function
	\[\overline{\mathcal{S}}(s)
	\equiv \begin{cases}
	\mathcal{S}(s) & \textup{if $s\in\mathcal{I}$,}\\
	-\infty & \textup{otherwise.}\end{cases}\]
	
\begin{lem}\label{l:cvx} For all $s\in\R$,
	$\overline{\mathcal{S}}(s)
	=-(-\overline{\mathcal{S}})^{\star\star}(s)$. Consequently the function $\overline{\mathcal{S}}$ is concave,
	and $s_\lambda$ is nondecreasing in $\lambda$.
	
\begin{proof} The function $-\mathcal{S}(s)$ has Legendre dual
	\[\overline{\mathfrak{F}}(\lambda)
	= \sup_{s\in\R}\Big\{
		\lambda s
		+ \overline{\mathcal{S}}(s)\Big\}
	= \sup_{s\in \mathcal{I} }\Big\{
		\lambda s
		+ \mathcal{S}(s)\Big\}
	\le \mathfrak{F}(\lambda)\,.\]
For $\lambda\in[0,1]$ it is clear that
$\overline{\mathfrak{F}}(\lambda)
=\mathfrak{F}(\lambda)$. It is straightforward to check that if $\lambda<0$ then
	\[\overline{\mathfrak{F}}(\lambda)
	\le\max_{s\in\mathcal{I}}\lambda s
	+\max_{s\in\mathcal{I}}\mathcal{S}(s)
	= \lambda s_{\min}
	+ \mathcal{S}(s_0),\]
so if $s<s_{\min}$ then
	\[(-\overline{\mathcal{S}})^{\star\star}(s)=(\overline{\mathfrak{F}})^\star(s)
	\ge 
	\sup_{\lambda<0} 
	\Big\{\lambda s
	-\overline{\mathfrak{F}}(\lambda)
	\Big\}
	\ge 
	\sup_{\lambda<0} 
	\Big\{\lambda (s-s_{\min})
	- \mathcal{S}(s_0)
	\Big\}=+\infty\,.\]
A symmetric argument shows that 
$(-\overline{\mathcal{S}})^{\star\star}(s)=+\infty$
also for $s>s_{\max}$.
If $s\in\mathcal{I}$,
we must have $s=s_{\lambda_\circ}$ for some $\lambda_\circ\in[0,1]$, and so
	\[(-\overline{\mathcal{S}})^{\star\star}(s)
	\ge\lambda_\circ s -
	\mathfrak{F}(\lambda_\circ)
	= -\mathcal{S}(s)\,.\]
This proves  $(-\overline{\mathcal{S}})^{\star\star}(s) \ge -\overline{\mathcal{S}}(s)$ for all $s\in\R$. On the other hand, it holds for any function $f$ that  $f^{\star\star} \le f$, so we conclude $(-\overline{\mathcal{S}})^{\star\star}(s) =-\overline{\mathcal{S}}(s)$ for all $s\in\R$. This implies that $\overline{\mathcal{S}}$ is  concave, concluding the proof.\end{proof} \end{lem}

\begin{proof}[Proof of Proposition~\ref{p:onersb.fe.analysis}\ref{p:strict.decrease}] We can obtain $\Sigma(s)$ as the limit of $\overline{\mathcal{S}}(s)$ in the limit $T\to\infty$. It follows from Lemma~\ref{l:cvx} together with Proposition~\ref{p:contract}\ref{c:BP_limT} that it is strictly decreasing in $s$.\end{proof}

\section{Constrained entropy maximization}\label{appx:entropy.max} In this section we review basic calculations for entropy maximization problems under affine constraints.

\subsection{Constraints and continuity} We will optimize a functional over nonnegative measures $\nu$ on a finite space $X$ (with $|X|=s$), subject to some affine constraints $M\nu=b$. We begin by discussing basic continuity properties.  Denote
	\[\Hb(b)\equiv
	\set{\nu\ge0} \cap \set{M\nu=b}
	\subseteq\R^s\,.\]
Let $\Delta \equiv\set{\nu\ge0} \cap \set{\langle\mathbf{1},\nu\rangle=1}$, and let  $\bm{B}$ denote the space of $b\in\R^r$ for which
	\[\emptyset\ne\Hb(b) \subseteq \Delta\,.\]
Then $\bm{B}$ is contained in the image of $\Delta$ under $M$, so $\bm{B}$ is a compact subset of $\R^r$.

\begin{ppn}\label{p:unif.cty} If $\bF$ is any continuous function on $\Delta$ and
	\beq\label{e:general.opt}
	F(b)= \max\set{ \bF(\nu) 
	: \nu\in\Hb(b) },\eeq
then $F$ is (uniformly) continuous over $b\in\bm{B}$. \end{ppn}

Proposition~\ref{p:unif.cty} is a straightforward consequence of the following two lemmas.

\begin{lem}\label{l:polytope} For $b\in\bm{B}$ and any vector $u$ in the unit sphere $\mathbb{S}^{r-1}$, let
	\[d(b,u)\equiv 
	\inf\set{t\ge0: b+tu\notin\bm{B}}\,.\]
There exists $\delta=\delta(b)>0$ such that
	\[d(b,u) \in\set{0} \cup [\delta,\infty)
	\quad\text{for all }b\in\bm{B}\,.\]

\begin{proof} $\bm{B}$ is a polytope, so it can be expressed as the intersection of finitely many closed half-spaces $H_1,\ldots,H_k$, where $H_i = \set{ x\in\R^r : \langle a_i,x\rangle \le c_i }$. Consequently there is at least one index $1\le i\le k$ such that
	\[d(b,u) 
	= \inf\set{t\ge0:b+tu\notin H_i}\,.\]
It follows that $\langle a_i,u\rangle>0$ and
	\[d(b,u)=\f{c_i-\langle a_i,b\rangle}
		{\langle a_i,u\rangle}
	\ge \f{c_i-\langle a_i,b\rangle}{|a_i|}
	= d( b,\pd H_i )\]
where $d( b,\pd H_i )$ is the distance between $b$ and the boundary of $H_i$. In particular, $d(b,u)>0$ if and only if $\langle a_i,b\rangle<c_i$, which in turn holds if and only if $d( b,\pd H_i )>0$. It follows that for all $u\in\mathbb{S}^{r-1}$ we have $d(b,u)\in\set{0} \cup[\delta,\infty)$ with
	\[\delta=\delta(b)=	\min\set{ d( b,\pd H_i ):
			d( b,\pd H_i )>0 };\]
$\delta$ is a minimum over finitely many positive numbers so it is also positive.\end{proof} \end{lem}

\begin{lem}\label{l:hausdorff}
The set-valued function $\Hb$ is continuous on $\bm{B}$ with respect to the Hausdorff metric $\haus$, that is to say, if $b_n\in\bm{B}$ with $\lim_{n\to\infty} b_n=b$ then
	\[\lim_{n\to\infty}\haus(\Hb(b_n),\Hb(b)) = 0\,.\]

\begin{proof} Recall that the Hausdorff distance between two subsets $X$ and $Y$ of a metric space is
	\[\haus(X,Y) = \inf\set{\epsilon\ge0 
	: X \subseteq Y^\epsilon
	\text{ and } Y \subseteq X^\epsilon},\]
where $X^\epsilon,Y^\epsilon$
are the $\epsilon$-thickenings of $X$ and $Y$.
Any sequence $\nu_n\in\Hb(b_n)$ converges along subsequences to limits $\nu\in\Hb(b)$, so
for all $\epsilon>0$ there exists $n_0(\epsilon)$ large enough that
	\[\Hb(b_n)
	\subseteq (\Hb(b))^\epsilon,
	\quad n\ge n_0(\epsilon)\,.\]
In the other direction, we now argue that if $\nu\in \Hb(b)$ and $b'=b+tu$ for $u\in\mathbb{S}^{r-1}$ and $t$ a small positive number,
 then we can find $\nu'\in\Hb(b')$ which is close to $\nu$. For $u\in\mathbb{S}^{r-1}$ let $d(b,u)$ be as in Lemma~\ref{l:polytope}, and take $\nu(b,u)$ to be any fixed element of $\Hb(b+d(b,u)u)$ (which by definition is nonempty). Since we consider $b'=b+tu$ for $t>0$, we can assume that $d(b,u)$ is positive, hence $\ge\delta(b)$ by Lemma~\ref{l:polytope}. We can express $b'=b+tu$ as the convex combination
 	\[b' = (1-\epsilon)b + \epsilon
		[ b+d(b,u)u ],\quad
		\epsilon = \f{t}{d(b,u)}
		= \f{|b'-b|}{d(b,u)}
		\le \f{|b'-b|}{\delta}\,.\]
Then $\nu' = (1-\epsilon)\nu+ \epsilon \nu(b,u)\in \Hb(b')$, so
	\[|\nu'-\nu|
	= \epsilon | \nu(b,u)-\nu |
	\le \f{(\diam\Delta)|b-b'|}{\delta}\]
This implies $H(b) \subseteq (H(b_n))^\epsilon$
for large enough $n$, and the result follows.\end{proof}
\end{lem}

\begin{proof}[Proof of Proposition~\ref{p:unif.cty}]
Take $\nu\in\Hb(b)$ so that $F(b)=\bF(\nu)$. If $b'=b+tu\in\bm{B}$ 
for some $u\in\mathbb{S}^{r-1}$, then Lemma~\ref{l:hausdorff} implies that we can find
$\nu'\in\Hb(b')$ with $|\nu'-\nu| = o_t(1)$,
where $o_t(1)$ indicates a function tending to zero in the limit $t\downarrow 0$, uniformly over $u\in\mathbb{S}^{r-1}$. 
It follows that
$\bF(\nu)
= \bF(\nu')+o_t(1)$,
since $\bF$ is uniformly continuous on $\Delta$ by the Heine--Cantor theorem. Therefore
	\[F(b)
	= \bF(\nu)
	= \bF(\nu') + o_t(1)
	\le F(b') + o_t(1)\,.\]
By the same argument $F(b') \le F(b) + o_t(1)$, concluding the proof.\end{proof}

When solving \eqref{e:general.opt} for a \bemph{fixed} value of $b\in\bm{B}$, it will be convenient to make the following reduction:

\begin{rmk}\label{r:rank.positivity} Suppose $M$ is an $r\times s$ matrix where $s=|X|$. We can assume without loss that $M$ has full rank $r$, since otherwise we can eliminate redundant constraints. We consider only $b\in\bm{B}$, meaning $\emptyset \ne\Hb(b)\subseteq\Delta$. The affine space $\set{M\nu=b}$ has dimension $s-r$; we assume this is positive since otherwise $\Hb(b)$ would be a single point. Then, if $\Hb(b)$ does not contain an interior point of $\set{\nu\ge0}$, it must be that
	\[X_\circ \equiv \set{x\in X
	: \exists \nu\in\set{\nu\ge0}\cap\set{M\nu=b}
	\text{ so that }\nu(x)>0}\]
is a nonempty subset of $X$.  In this case, it is equivalent to solve the optimization problem over measures $\nu_\circ$ on the reduced alphabet $X_\circ$, subject to constraints $M' \nu_\circ=b$ where $M'$ is the submatrix of $M$ formed by the columns indexed by $X_\circ$. Then, by construction, the space
	\[\Hb_\circ(b)
	=\set{\nu_\circ\ge0}
	\cap \set{M' \nu_\circ=b}\]
contains an interior point of  $\set{\nu_\circ\ge0}$. The matrix $M'$ is $r\times s_\circ$ where $s_\circ=|X_\circ|$; and if $M'$ is not of rank $r$ then we can again remove redundant constraints, replacing $M'$ with an $r_\circ \times s_\circ$ submatrix $M_\circ$ which has full rank $r_\circ$. We emphasize that the final matrix $M_\circ$ depends on $b$. In conclusion, when solving \eqref{e:general.opt} for a fixed $b\in\bm{B}$, we may assume with no essential loss of generality that the original matrix $M$ is $r\times s$ with full rank $r$, and that $\Hb(b)=\set{\nu\ge0}\cap\set{M\nu=b}$ contains an interior point of $\set{\nu\ge0}$. It follows that this space has dimension $s-r>0$, and its boundary is contained in the boundary of $\set{\nu\ge0}$. \end{rmk}

\subsection{Entropy maximization}
\label{ss:entropy.max}

We now restrict \eqref{e:general.opt}
to the case of functionals $\bF$ 
which are \bemph{concave} on the domain $\set{\nu\ge0}$. It is straightforward to verify from definitions that the optimal value $F(b)$
is (weakly) concave in $b$. Recall that the convex conjugate of a function $f$ on 
domain $C$ is the function $f^\star$ defined by
	\[ f^\star(x^\star)
	= \sup\set{\langle x^\star,x\rangle
	-f(x) : x\in C}\,.\]
Denote $G(\gamma) \equiv (-\bF)^\star(M^t\gamma)$, and consider the Lagrangian functional
	\[\mathcal{L}(\gamma;b)
	= \sup\set{ \bF(\nu)
		+ \langle \gamma,
		M\nu-b \rangle: \nu\ge0 }
	= -\langle \gamma,b\rangle
	+ G(\gamma)\,.\]
It holds for any $\gamma\in\R^r$ that
$\mathcal{L}(\gamma;b) \ge F(b)$, so
	\beq\label{e:duality.gap}
	F(b)
	\le \inf
	\set{\mathcal{L}(\gamma;b)
	: \gamma\in\R^r}
	= -G^\star(b)\,.\eeq
Now assume $\psi$ is a positive function on $X$, and consider \eqref{e:general.opt} for the special case
	\beq\label{e:entropy.F}
	\bF(\nu)
	= \ent(\nu)
	+ \langle\nu,\log\psi\rangle
	= \sum_{x\in X}
	\nu(x)\log \f{\psi(x)}{\nu(x)}\,.\eeq
We remark that the supremum in
	$(-\ent)^\star(\nu^\star)
	=\sup\set{ \langle\nu^\star,\nu
		\rangle+\ent(\nu): \nu\ge0}$
is uniquely attained by the measure
$\nu^\textup{op}(x)=\exp\{ -1 + \nu^\star(x)\}$, yielding
	\[(-\ent)^\star(\nu^\star)
	=\langle\nu^\textup{op}(\nu^\star)
		,1\rangle
	= \sum_x \exp\{ -1+\nu^\star(x)\}\,.\]
This gives the explicit expression
	\beq\label{e:ent.G}
	G(\gamma)
	=(-\bF)^\star(M^t\gamma)
	=(-\ent)^\star(\log\psi+M^t\gamma)
	=\sum_x \psi(x) \exp\{
		-1 + (M^t\gamma)(x)\}\,.\eeq

\begin{lem}\label{l:g.convex}
Assume $\psi$ is a strictly positive function on a set $X$ of size $s$ and that $M$ is $r\times s$ with rank $r$. Then the function
$G(\gamma)$ of \eqref{e:ent.G} is strictly convex in $\gamma$.

\begin{proof} Let $\nu\equiv\nu(\gamma)$ denote the measure on $X$ defined by
	\[\nu(x)
	= \psi(x)
	\exp\{ -1+(M^t\gamma)(x) \},\]
and write $\langle f(x) \rangle_\nu \equiv \langle f,\nu \rangle$. The Hessian matrix $H \equiv \Hess G(\gamma)$ has entries
	\[H_{i,j}
	= \f{\pd^2\mathcal{L}(\gamma;b)}
	{ \pd \gamma_i\pd \gamma_j }
	=\sum_{x\in X} \nu(x) M_{i,x} M_{j,x}
	=\langle M_{i,x} M_{j,x}\rangle_\nu\,.\]
Let $M_x$ denote the vector-valued function
$(M_{i,x})_{i\le r}$, so
	\[\alpha^t H \alpha
	= \langle (\alpha^t M_x)^2 \rangle_\nu\,.\]
This is zero if and only if $\nu(\set{x\in X:\alpha^t M_x=0})=1$. Since $\nu$ is a positive measure, this can only happen if $\alpha^t M_x=0$ for all $x\in X$, but this contradicts the assumption that $M$ has rank $r$. This proves that $H$ is positive-definite, so $G$ is strictly convex in $\gamma$.\end{proof}
\end{lem}

\begin{ppn}\label{p:entropy.max}
Let $b\in\bm{B}$ such that $\Hb(b)=\set{\nu\ge0}\cap\set{M\nu=b}$ contains an interior point of $\set{\nu\ge0}$, and consider the optimization problem \eqref{e:general.opt}
for $\bF$ as in \eqref{e:entropy.F}. For this problem, the inequality \eqref{e:duality.gap} becomes an equality,
	\[F(b) =\inf\set{
	 \mathcal{L}(\gamma;b):
	 \gamma\in\R^r}=-G^\star(b)\,.\]
Further, $\mathcal{L}(\gamma;b)$ is strictly convex in $\gamma$, and its infimum
is achieved by a unique $\gamma=\gamma(b)$.
The optimum value of \eqref{e:general.opt}
is uniquely attained by the measure
$\nu=\nu^\textup{op}(b)$ defined by
	\beq\label{e:general.nu.opt}
	\nu(x)= \psi(x) 
	\exp\{-1 + (M^t\gamma)(x)\}\,.\eeq
For any $\mu\in\Hb(b)$,
	$\bF(\nu)-\bF(\mu)
	=\dkl(\mu|\nu)
	\gtrsim \|\nu-\mu\|^2$.
Finally, in a neighborhood of $b$ in $\bm{B}$, 
$\gamma'(b)$ is defined and $F(b)$ is strictly concave in $b$.

\begin{proof} Under the assumptions, the boundary of the set $\Hb(b)$ is contained in the boundary of $\set{\nu\ge0}$. The entropy $\ent$ has unbounded gradient at this boundary, so for $\bF$ as in \eqref{e:entropy.F}, the optimization problem \eqref{e:general.opt} must be solved by a strictly positive measure $\nu>0$.
Since $\nu>0$, we can differentiate in the direction of any vector $\delta$ with $M\delta=0$ to find
	\[0=\f{d}{dt}
	\bigg[
	\ent(\nu+t\delta)
	+\langle \log\psi,\nu+t\delta\rangle
	\bigg]\bigg|_{t=0}
	= \langle\delta,-1-\log\nu+\log\psi\rangle\,.\]
Recalling Remark~\ref{r:rank.positivity},
we assume without loss that
 $M$ is $r\times s$ with rank $r$, since otherwise 
 we can eliminate redundant constraints. Then, since $M\delta=0$, for any $\gamma\in\R^r$ we have
	\[0 = 
	\langle\delta,\epsilon\rangle
	\quad\text{where }\epsilon=	-1-\log\nu+\log\psi
		+ M^t\gamma\,.\]
We can then solve for $\gamma$ so that $M\epsilon=0$:\footnote{The matrix $MM^t$ is invertible: if $MM^tx=0$ then $M^t x \in\ker M = (\image M^t)^\perp$. On the other hand clearly $M^t x \in \image M^t$, so
$M^t x \in (\image M^t) \cap (\image M^t)^\perp=\set{0}$. Therefore $x\in\ker M^t$, but $M^t$ is injective by assumption.}
	\[\gamma
	= (M M^t)^{-1} M(\log\nu-\log\psi+1)\,.\]
Setting $\delta=\epsilon$ in the above gives $0=\|\epsilon\|^2$, therefore we must have $\epsilon=0$. This proves the existence of $\gamma =\gamma(b) \in\R^r$ such that \eqref{e:general.opt} is optimized by
$\nu=\nu^\textup{op}(b)$, as given by \eqref{e:general.nu.opt}.
The optimal value of \eqref{e:general.opt} is then
	\begin{align*}
	F(b)
	&= \langle 1,\nu^\textup{op}(b)\rangle
	 - \langle M^t\gamma(b),
	 \nu^\textup{op}(b)\rangle\\
	&= \sum_x \psi(x)
		\exp\{ -1+ (M^t\gamma)(x) \}
		-\langle \gamma,b\rangle
		\bigg|_{\gamma=\gamma(b)}
	= \mathcal{L}(\gamma(b),b).
	\end{align*}
In view of \eqref{e:duality.gap}, this proves
that in fact
	\[-G^\star(b)=\inf\set{
	\mathcal{L}(\gamma,b)
	:\gamma\in\R^r}
	= \min
	\set{ \mathcal{L}(\gamma,b)
	 : \gamma\in\R^r}
	= \mathcal{L}(\gamma(b),b)
	= F(b)\]
as claimed. Now recall from Lemma~\ref{l:g.convex} that $G(\gamma)$ is strictly convex, which implies that $\mathcal{L}(\gamma;b)$ is strictly convex in $\gamma$.
Thus $\gamma=\gamma(b)$ is the unique stationary point of $\mathcal{L}(\gamma;b)$.

These conclusions are valid under the assumption that $\Hb(b)$ contains an interior point of $\set{\nu\ge0}$, which is valid in a neighborhood of $b$ in $\bm{B}$.
Throughout this neighborhood, $\gamma(b)$ is defined by the 
stationarity condition 
$b = G'(\gamma)$. Differentiating again with respect to $\gamma$ gives
	\beq\label{e:gamma.prime}
	b'(\gamma)
	= \Hess G(\gamma),\quad
	\gamma'(b) = [\Hess G(\gamma(b))]^{-1}\,.\eeq
We also find (in this neighborhood) that
	\[F'(b)
	= -\gamma(b),\quad
	F''(b)
	= - \gamma'(b)
	= -[\Hess G(\gamma(b))]^{-1},\]
so $F$ is strictly concave.
It remains to prove
that $\bF(\nu)-\bF(\mu)
=\dkl(\mu|\nu)$.
(The estimate $\dkl(\mu|\nu)\gtrsim
\|\mu-\nu\|^2$
is well known and straightforward to verify.)
For any measure $\mu$,
	\[-\dkl(\mu|\nu)
	= \ent(\mu)
	+ \langle \mu,\log(\psi\exp\{-1+M^t\gamma\}) \rangle\,.\]
Applying this with $\mu=\nu$ gives
	\[0=-\dkl(\nu|\nu)
	= \ent(\nu)
	+ \langle \nu,\log(\psi\exp\{-1+M^t\gamma\}) \rangle\,.\]
Subtracting these two equations gives
	\[-\dkl(\mu|\nu)
	= \ent(\mu)-\ent(\nu)
	+\langle \mu-\nu,\log\psi\rangle
	+ \langle \mu-\nu,\log(\exp\{-1+M^t\gamma\}) \rangle\,.\]
If $M\nu=M\nu=b$ then the last term vanishes, giving $-\dkl(\mu|\nu)
= \bF(\mu)-\bF(\nu)$.\end{proof}
\end{ppn}

\begin{rmk}\label{r:match.notation}
Our main application of Proposition~\ref{p:entropy.max} is for the depth-one tree 
$\onetree$ of Figure~\ref{f:onetree}.
In the notation of the current section,
$X$ is the space of valid $T$-colorings $\usi$ of $\onetree$, and $\psi : X \to (0,\infty)$ is defined by
	\[\psi(\usi)
	= \avwt_{\onetree}(\usi)^\lambda
	=\bigg\{ \dPhi(\usi_{\delta v})
	\prod_{a\in\pd v}
	[\ePhi(\sigma_{av})
	\hPhi(\usi_{\delta a})]
	\bigg\}^\lambda\,.\]
We then wish to solve the optimization problem \eqref{e:general.opt} for
$\bF(\nu)$ as in \eqref{e:entropy.F},
under the constraint that $\nu$ has marginals $\dhtree(\dsi)$ on the boundary edges $\delta\onetree$. This can be expressed as $M\nu=\dot{h}$ where $M$ has rows indexed by the spins $\dsi\in\dCOLS$, columns indexed by valid $T$-colorings 
$\ueta\equiv\ueta_\onetree$ of $\onetree$:
the $(\dsi,\ueta)$ entry of $M$ is given by
	\[M(\dsi,\ueta)
	=|\delta\onetree|^{-1}
	\sum_{e\in\delta\onetree}
	\Ind{\deta_e=\dsi}\,.\]
Recall Remark~\ref{r:rank.positivity}, let $\dCOLS_+=\set{\dsi\in\dCOLS : \dhtree(\dsi)>0}$, and $X_\circ = \set{\ueta\in X:M(\dsi,\ueta)=0\ \forall \dsi\notin\dCOLS}$. Let $M_+$ be the  $\dCOLS_+ \times X_\circ$ submatrix of $M$, and set $\dq(\dsi)=0$ for all $\dsi\notin\dCOLS_+$. Next, in the matrix $M_+$, if the $\deta$ row is a linear combination of other rows, then set $\dq(\deta)=1$ and remove this row. Repeat until we arrive at an $\dCOLS_\circ\times X_\circ$ matrix $M_\circ$ of full rank $r_\circ=|\dCOLS_\circ|$. The original problem reduces to an optimization over $\set{\nu_\circ\ge0}\cap\set{M_\circ\nu_\circ=b_\circ}$ where $b_\circ$ denotes the entries of $b$ indexed by $\dCOLS_\circ$. It follows from Proposition~\ref{p:entropy.max} that the unique maximizer of \eqref{e:general.opt} is the measure $\nu=\nu^\textup{op}(b)$ given by
	\[\nu(\usi)
	= \f{1}{Z} \avwt_{\onetree}(\usi)^\lambda
	= \f{1}{Z}\bigg\{
	\dPhi(\usi_{\delta v})
	\prod_{a\in\pd v}
	[\ePhi(\sigma_{av})
	\hPhi(\usi_{\delta a})]
	\bigg\}^\lambda
	\prod_{e\in\delta\onetree}
	\dq(\sigma_e)\,.\]
Note however that if $M_+$ is not of full rank then $\dq$ need not be unique. 
\end{rmk}

\section{Pairs of intermediate or large overlap}\label{appx:sep} In this section we prove Proposition~\ref{p:sep}, which states that the first moment of $\ZZ=\ZZ_{\lambda,T}$ is dominated by separable colorings provided $0\le\lambda\le1$.

\subsection{Intermediate overlap}We first show that configurations with ``intermediate'' overlap are negligible. This can be done with quite crude estimates, working with \textsc{nae-sat} solutions rather than colorings.

\begin{lem}\label{l:nae.intermediate} Consider random regular \textsc{nae-sat} at clause density $\alpha\ge2^{k-1}\log2 - O(1)$. On $\glit=(V,F,E,\ulit)$, let $Z^2[\rho]$ count the number of pairs  $\ux,\vec{\acute{x}}\in\set{\zro,\one}^V$ of valid \textsc{nae-sat} solutions which agree on $\rho$ fraction of variables. Then
	\[\E Z^2[\rho]
	\le (\E Z)
	\exp\Big\{ n \Big[ H(\rho)
		-(\log2)\pi(\rho)
		+ O(1/2^k) 
		\Big] \Big\},\]
for $\pi(\rho)\equiv1-\rho^k-(1-\rho)^k$.

\begin{proof} For $\vec{u}\in\set{\zro,\one}^V$, let $I^\textsc{nae}(\vec{u};\glit)$ be the indicator that $\vec{u}$ is a valid \textsc{nae-sat} solution on $\glit$. Fix any pair of vectors $\ux,\vec{\acute{x}}\in\set{\zro,\one}^V$ which agree on $\rho$ fraction of variables:
	\[\E Z^2[\rho]
	= 2^n \binom{n}{n\rho}
	\E[ I^\textsc{nae}(\ux;\glit)
	I^\textsc{nae}(\vec{\acute{x}};\glit)]
	= (\E Z)
	\binom{n}{n\rho}
	\E[I^\textsc{nae}(\vec{\acute{x}};\glit)
	\,|\, I^\textsc{nae}(\ux;\glit)=1]\,.\]
Given $\ux,\vec{\acute{x}}$, let  $M\equiv M(\ux,\vec{\acute{x}})$ count the number of clauses $a\in F$ where
	\[|\set{e\in\delta a :
	x_{v(e)}=\acute{x}_{v(e)} }|
	\notin\set{0,k}\,.\]
In each of these clauses, there are $2^k-2$ literal assignments $\ulit_{\delta a}$ which are valid for $\ux$. Out of these, exactly $2^k-4$ are valid also for $\vec{\acute{x}}$. If we define i.i.d.\ binomial random variables $D_a\sim\mathrm{Bin}(k,\rho)$, indexed by $a\in F$, then
	\[\P( M=m\gamma )
	= \P\bigg( \sum_{a\in F}
		\Ind{D_a\notin\set{0,k}}
		\,\bigg|\,
		\sum_{a\in F} D_a = mk\rho \bigg)\,.\]
The $(D_a)_{a\in F}$ sum to $mk\rho$ with probability which is polynomial in $n$, so
	\[\P( M=m\gamma )
	\le n^{O(1)}
	\P(\mathrm{Bin}(m,\pi)=m\gamma)\]
with $\pi=\pi(\rho)$ as in the statement of the lemma. Therefore
	\[\E[I^\textsc{nae}(\vec{\acute{x}};\glit)
	\,|\, I^\textsc{nae}(\ux;\glit)=1]
	\le n^{O(1)} \E\bigg[
	\bigg( \f{2^k-4}{2^k-2} \bigg)^X
	\bigg]\]
for $X\sim\mathrm{Bin}(m,\rho)$. It is easily seen that the above is $\le \exp\{ -m\pi/2^{k-1} \}$, and the claimed bound follows, using the lower bound on $\alpha=m/n$.\end{proof} \end{lem}

\begin{cor}\label{c:int.small} Let $\psi(\rho) = H(\rho) - (\log 2)\pi(\rho)$. Then $\psi(\rho) \le -2k/2^k$ for all $\rho$ in
	\[[\exp\{-k/(\log k)\},
	\tfrac12(1 - k/2^{k/2})] \cup 
	[\tfrac12(1 + k/2^{k/2}),
	1-\exp\{-k/(\log k)\}]\,.\]
Assuming $\alpha=m/n\ge 2^{k-1}\log2-O(1)$, $\E Z^2[\rho] \le \exp\{ -nk/2^k \}$ for all such $\rho$.

\begin{proof} Note that $H( \tfrac{1+\epsilon}{2})
 \le \log2-\epsilon^2/2$.
If $(k\log k)/2^k \le \epsilon
	\le 1/k$, then
	\[\psi( \tfrac{1+\epsilon}{2} )
	\le -\epsilon^2/2 + 
		O(k\epsilon/2^k)
	\le -\epsilon^2/3\,.\]
Both $H( \tfrac{1+\epsilon}{2})$ and $\pi( \tfrac{1+\epsilon}{2})$ are
symmetric about $\epsilon=0$, and 
decreasing on the interval $0\le \epsilon\le 1$. 
It follows that for any $0\le a \le b\le 1$,
	\[\max_{a\le \epsilon\le b}
	\psi( \tfrac{1+\epsilon}{2} )
	\le H(\tfrac{1+a}{2})
		-(\log 2)\pi(\tfrac{1+b}{2})\,.\]
With this in mind, if $1/k \le \epsilon \le 1-5(\log k)/k$,
	\[\psi( \tfrac{1+\epsilon}{2} )
	\le -(2k^2)^{-1} + O(k^{-5/2})
	\le -(4k^2)^{-1}\,.\]
If $1-5(\log k)/k \le \epsilon \le 1-(\log k)^3/k^2$,
	\[\psi( \tfrac{1+\epsilon}{2} )
	\le O(1) (\log k)^2/k
	-\COLS(1) 
	(\log k)^3/k
	\le -\COLS(1) (\log k)^3/k\,.\]
Finally, if $1-(\log k)^3/k^2 \le 
	\epsilon
	\le 1-\exp\{-2k/(\log k)\}$, then
	\[\psi( \tfrac{1+\epsilon}{2} )
	\le O(1) \epsilon k/(\log k) - \COLS(1) \epsilon k
	\le - \COLS(1) \epsilon k\,.\]
Combining these estimates proves the claimed bound on $\psi(\rho)$.
The assertion for $\E[Z^2(\rho)]$ then follows by substituting into Lemma~\ref{l:nae.intermediate}, and noting that $\E Z \le \exp\{O(n/2^k)\}$.\end{proof}
\end{cor}

\subsection{Large overlap} In what follows, we restrict consideration to a small neighborhood $\nbd$ of $H_\star$.   We abbreviate $\usi\in H$ if $H(\graph,\usi)=H$, and $\usi\in\nbd$ if $H(\graph,\usi)\in\nbd$. Recall that we write $\usi'\succcurlyeq\usi$ if the number of free variables in $\ux(\usi')$ upper bounds the number in $\ux(\usi)$. We also write $H' \succcurlyeq H$ if $\usi'\succcurlyeq\usi$ for any (all) $\usi\in H$ and $\usi'\in H'$.  Let $\XX(H,H')$ count the colorings $\usi\in H$ such that
	\[\Big|\Big\{
	\usi'\in H' :
	\SEP(\usi,\usi')
	\le\exp\{-k/(\log k) \}
	\Big\}\Big|
	\ge \omega(n),\]
for $\omega(n) = \exp\{ (\log n)^4 \}$. (Although we will not write it explicitly, it should be understood that  $\XX(H,H')$ depends on $\glit$, since both $\usi,\usi'$ are required to be valid colorings of $\glit$.) Let $\XX(\nbd)$ denote the sum of $\XX(H;H')$ over all pairs $H,H'\in\nbd$ with $H'\succcurlyeq H$. Let $\ZZ(\nbd)$ denote the sum of $\ZZ(H)$ over all $H\in\nbd$. 

\begin{ppn}\label{p:large.overlap}
There exists a small enough positive constant $\epsilon_{\max}(k)$ such that,
if $\nbd$ is the $\epsilon$-neighborhood of $H_\star$ for any $\epsilon \le\epsilon_{\max}$, then
	\[\E\XX(\nbd)\le\E\ZZ(\nbd)
	\exp\{-(\log n)^2\}\,.\]

\begin{proof} By definition, 
	\[\XX(\nbd) = 
	\sum_{H\in\nbd} \XXge(H),\quad
	\XXge(H)\equiv \sum_{H'\in\nbd} \Ind{ H' \succcurlyeq H } 
		\XX(H,H')\,.\]
It suffices to show that for every $H\in\nbd$, $\E\XXge(H)\le \E\ZZ(H)\exp\{ -2 (\log n)^2 \}$. Note that the total number of empirical measures $H'$ is at most $n^c$ for some constant $c(k,T)$. Let $\bm{E}$ denote the set of pairs $(\glit,\usi)$ for which
	\[\Big|\Big\{ \usi'\in \nbd:
		\usi'\succcurlyeq\usi 
		\text{ and }
		\SEP(\usi,\usi')
		\le \exp\{-k/(\log k)\}
		\Big\}\Big|\ge \omega(n)\,.\]
(Again, it is understood that both
$\usi,\usi'$ must be valid colorings of $\glit$.) Then
	\[\XXge(H)
	\le n^{c} \sum_{\usi\in H}
	\Ind{(\glit,\usi)\in \bm{E}}\,.\]
Consequently, in order to show
the required bound on $\E\XXge(H)$,
it suffices to show 
\beq\label{e:sep.planted.nts}
\P^H(\bm{E})\le n^{-c}
\exp\{-2(\log n)^2\},\eeq
where
$\P^H$ is a ``planted'' measure on pairs $(\glit,\usi)$: to sample from $\P^H$,
we start with a set $V$ of $n$ isolated variables each with $d$ incident half-edges,
and a set $F$ of $m$ isolated clauses
each with $k$ incident half-edges.
Assign colorings
of the half-edges,
	\[\usi_\delta\equiv
	(\usi_{\delta V},\usi_{\delta F})
	\quad\text{where }
	\usi_{\delta V}
	\equiv(\usi_{\delta v})_{v\in V}, \
	\usi_{\delta F}
	\equiv(\usi_{\delta a})_{a\in F},\]
which are uniformly random subject to the empirical measure $H$. Then $\usi_\delta$ is the ``planted'' coloring: conditioned on it,
we sample uniformly at random a graph $\glit$
such that $\usi_\delta$ becomes a valid coloring $\usi$ on $\glit$. The resulting pair $(\glit,\usi)$ is a sample from $\P^H$.

Suppose $(\glit,\usi)\in\bm{E}$. The total number of configurations $\usi'$ with $\SEP(\usi,\usi') \le \delta$ is at most $(cn)^{n\delta}$, which is $\ll \omega(n)$ if $\delta \le n^{-1} (\log n)^2$. This implies that there must exist $\usi'\in\nbd$ such that
$\usi'\succcurlyeq\usi$ and
\[n^{-1} (\log n)^2\le \SEP(\usi,\usi') \le \exp\{-k/(\log k)\}\,.\] 
It follows that
	\[S\equiv\set{v\in V:
	x_v(\usi)\in\set{\zro,\one}
	\text{ and }
	x_v(\usi') \ne x_v(\usi)}\]
has size $|S| \equiv ns$ for
$s \in[(2n)^{-1}(\log n)^2,\exp\{-k/(\log k)\}]$.
The set $S$ is \bemph{internally forced} in $\usi$: for every $v\in S$, any clause forcing to $v$ must have another edge connecting to $S$. Formally, let $\texttt{R}_U$
(resp.\ $\texttt{B}_U$) count the number of 
$\set{\red}$-colored
(resp.\ $\set{\blu}$-colored) edges incident to a subset of vertices $U$.
Let $I_S$ be the indicator that all variables in $S$ are forced.
For any fixed $S\subseteq V$,
	\[\P^H(S\text{ internally forced})
	\le \E_{\P^H}\bigg[
		I_S k^{\texttt{R}_S}
		\f{(\texttt{B}_S)_{\texttt{R}_S}}
		{(\texttt{B}_F)_{\texttt{R}_S}}
		\bigg]
	\le \E_{\P^H}[
		I_S
		(4ks)^{\texttt{R}_S}]\,.\]
In the first inequality, the factor $k^{\texttt{R}_S}$ accounts for the choice, for each $S$-incident $\set{\red}$-colored  edge $e$, of another edge $e'$ sharing the same clause. The factor  $(\texttt{B}_S)_{\texttt{R}_S}/(\texttt{B}_F)_{\texttt{R}_S}$ then accounts for the chance that the chosen edge $e'$ (which must have color in $\set{\blu}$) will also be $S$-incident. The second inequality follows by noting that we certainly have  $\texttt{B}_S \le nsd$, and for $H$ near $H_\star$ we also clearly have $\texttt{B}_F \ge nd/4$.

To bound the above, we can work with a slightly different measure $\Q^H$: instead of sampling $\usi_\delta$ subject to $H$, we can simply sample  variable-incident colorings $\usi_{\delta v}$ i.i.d.\ from $\dH$, and clause-incident colorings $\usi_{\delta a}$ i.i.d.\ from $\hH$. On the event $\marg$ that the resulting $\usi_\delta$ has empirical measure $H$, we sample the graph $\glit$ according to $\P^H(\glit|\usi_\delta)$, and otherwise we set $\glit=\emptyset$. Then, since $\Q^H(\marg) \ge n^{-c}$ (adjusting $c$ as needed), we have
	\[\P^H((\glit,\usi))=\Q^H((\glit,\usi) \,|\,\marg)
	\le n^c \, \Q^H((\glit,\usi) ; \marg)\,.\]
Let us abbreviate $\dH(\ell)$ for the probability under $\dH$ that $\usi$ has $\ell$ entries in $\set{\red}$:  then
	\beq\label{e:sep.bound}
	\E_{\P^H}[ I_S (4ks)^{\texttt{R}_S}]
	\le n^c \,
	\E_{\Q^H}[ I_S (4ks)^{\texttt{R}_S};
	\marg]
	\le n^c \,
	\, \bigg(\sum_{\ell\ge1}
	\dH(\ell)
	(4ks)^{\ell}\bigg)^{ns}\,.\eeq
For $H$ sufficiently close to $H_\star$, we will have
	\[\dH(\ell)
	\le 2\dH_\star(\ell)
	\le 2 \binom{d}{\ell}
	\f{\hq_\star(\redo)^\ell
	\hq_\star(\bluo)^{d-\ell}}
	{ [\hq_\star(\redo)
	+\hq_\star(\bluo)]^d
	-\hq_\star(\bluo)^d }\,.\]
It follows that the right-hand side of \eqref{e:sep.bound} is (for some absolute constant $\delta$)
	\[\le n^c \, 2^{ns}
	\bigg(\f{[\hq_\star(\redo) \cdot 4ks
	+\hq_\star(\bluo)]^d
	-\hq_\star(\bluo)^d}
	{[\hq_\star(\redo)
	+\hq_\star(\bluo)]^d
	-\hq_\star(\bluo)^d }\bigg)^{ns}
	\le n^c s^{ns} 
	2^{-\delta k n s} ,\]
where the last inequality uses that $s\le \exp\{ -k/(\log k)\}$. Summing over $S$ gives
	\[\P^H(\bm{E})
	\le\max_{s \ge (2n)^{-1}(\log n)^2}
	n^c 2^{-\delta k n s/2}
	\le \exp\{ -\COLS(1) k(\log n)^2\}\,.\]
This implies \eqref{e:sep.planted.nts}; and the claimed result follows as previously explained.\end{proof}\end{ppn}

\begin{proof}[Proof of Proposition~\ref{p:sep}]
Follows by combining Corollary~\ref{c:int.small} and Proposition~\ref{p:large.overlap}.\end{proof}

\section{Free energy upper bound}\label{appx:ubd}

For a family of spin systems that includes \textsc{nae-sat}, an interpolative calculation gives an upper bound for the free energy on Erd\H{o}s-R\'{e}nyi graphs (\cite{MR1972121,MR2095932}, cf.~\cite{MR1957729}). These bounds build on earlier work \cite{MR2037569} concerning the subadditivity of the free energy in the Sherrington--Kirkpatrick model, which was later generalized to a broad class of models \cite{MR3161470, abbe2010concentration, MR3256814, MR3252921}. (Although these results are closely related, we remark that interpolation gives quantitative bounds whereas subadditivity does not.) To prove the upper bound in Theorem~\ref{t:main}, we establish the analogue of \cite{MR1972121,MR2095932} for random \bemph{regular} graphs. Although the main concern of this paper is the \textsc{nae-sat} model, we give the bound for a more general class of models, which may be of independent interest.

\subsection{Basic interpolation bound}\label{s:ubd_itp} 

Recall $\graph=(V,F,E)$ denotes a $(d,k)$-regular bipartite graph (without edge literals). We consider measures defined on vectors $\ux\in\albet^V$ where $\albet$ is some fixed alphabet of finite size.  Fix also a finite index set $S$. Suppose we have  (random) vectors $b\in\R^S$ and $f\in\mathcal{F}(\albet)^S$, where $\mathcal{F}(\albet)$ denotes the space of functions $\albet\to\R_{\ge0}$. Independently of $b$, let $f_1,\ldots,f_k$ be i.i.d.\ copies of $f$, and define the random function
	\begin{equation}
	\label{e:theta.product}
	\theta(\ux)
	\equiv \sum_{s\in S} b_s
		\prod_{j=1}^k
		f_{s,j}(x_j).
	\end{equation}
Let $h$ be another (random)  element of $\mathcal{F}(\albet)$. Assume there is a constant $\epsilon>0$ so that
	\begin{equation}
	\label{e:bounded.potentials}
	\epsilon \le \set{h,1-\theta}
	 \le \f1\epsilon
	\quad\text{almost surely.}
	\end{equation}
Note we do not require the $b_s$ to be nonnegative; however, we assume that
	\begin{equation}
	\label{e:UB.cond2}
	b^p(\vec{s})
	\equiv
	\E\Big[
	\prod_{\ell=1}^p b_{s_\ell}\Big]\ge0
	\quad
	\text{for any $p\ge1$,
	$\vec{s}
	\equiv(s_1,\ldots,s_p)\in S^p$.}
	\end{equation}
Let $\glit$ denote the graph $\graph$ labelled by a vector $((h_v)_{v\in V},(\theta_a)_{a\in F})$ of independent functions, where the $h_v$ are i.i.d.\ copies of $h$ and the $\theta_a$ are i.i.d.\ copies of $\theta$. For $a\in F$ we abbreviate $\ux_{\delta a}\equiv( x_{v(e)} )_{e\in\delta a}\in\albet^k$, and we consider the (random) Gibbs measure
	\begin{equation}\label{e:gibbs.ubd}
	\mu_{\glit}(\ux)
	\equiv\f{1}{Z(\glit)}
	\prod_{v\in V} h_v(x_v)
	\prod_{a\in F}
	[1-\theta_a( \ux_{\delta a})]\end{equation}
where $Z(\glit)$ is the normalizing constant. Now let $\glit$ be the random $(d,k)$-regular graph on $n$ variables, together with the random function labels. We write $\E_n$ for expectation over the law of $\glit$, and define the (logarithmic) free energy of the model to be
	\[F_n \equiv 
	\f1n\E_n\log Z(\glit)\,.\]

\begin{emp}[positive temperature \textsc{nae-sat}]\label{ex:naesat} Let $\albet=\set{\zro,\one}$, and let $\vec \lit \equiv (\lit_i)_{i\le k}$ be a sequence of i.i.d.\ $\textup{Bernoulli}(1/2)$ random variables. The positive-temperature \textsc{nae-sat} model corresponds to
taking $h\equiv1$ and
	\[\theta(\ux)
	\equiv (1-e^{-\beta})
	\bigg(	 \prod_{i=1}^k \f{\lit_{1}\oplus x_{i}}{2}
		+\prod_{i=1}^k \f{\one \oplus \lit_{i}\oplus x_{i}}{2}
		\bigg)\]
where $\beta\in(0,\infty)$ is the inverse temperature. In this model, each violated clause incurs a multiplicative penalty $e^{-\beta}$.
\end{emp}

\begin{emp}[positive-temperature coloring] Let $\albet=[q]$. The positive-temperature coloring model (i.e., anti-ferromagnetic Potts model) on a $k$-uniform hypergraph  corresponds to $h\equiv1$ and
	\[\theta(\ux) \equiv 
	(1-e^{-\beta})
	\sum_{s=1}^q
	\Ind{x_1=\cdots=x_k=s}\]
where $\beta\in(0,\infty)$ is the inverse temperature. In this model, each  monochromatic (hyper)edge incurs a multiplicative penalty $e^{-\beta}$. \end{emp}

The following theorem is a random regular graph analog of \cite[Thm.~3]{MR2095932}. (We stated our result for a slightly more general class of models than considered in \cite{MR2095932}; however the main result of \cite{MR2095932} extends to these models with only minor modifications.)

\begin{thm}\label{t:free.energy}  Consider a (random) Gibbs measure \eqref{e:gibbs.ubd} satisfying assumptions \eqref{e:theta.product}-\eqref{e:UB.cond2}, and consider the (nonasymptotic) free energy  $F_n \equiv n^{-1}\E_n\log Z(\glit)$. Let
	{\setlength{\jot}{0pt}\begin{align*}
	\mathcal{M}_0
	&\equiv
	\text{space of probability measures over
	$\albet$},\\
	\mathcal{M}_1
	&\equiv
	\text{space of probability measures over
	$\mathcal{M}_0$},\\
	\mathcal{M}_2
	&\equiv
	\text{space of probability measures over
	$\mathcal{M}_1$}.\end{align*}}%
For $\smash{\zeta\in\mathcal{M}_2}$, let $\vec{\eta}\equiv(\eta_{a,j})_{a\ge0,j\ge0}$ be an array of i.i.d.\ samples from $\zeta$.  For each index $(a,j)$ let $\rho_{a,j}$ be a conditionally independent sample from $\eta_{a,j}$, and denote $\vec{\rho}\equiv(\rho_{a,j})_{a\ge0,j\ge0}$. Let $(h\rho)_{a,j}(x) \equiv h_{a,j}(x)\rho_{a,j}(x)$, define random variables
	\begin{align*}
	\bm{u}_a(x)
	&\equiv
	\sum_{\ux\in\albet^k}
		\Ind{x_1=x}
		[1-\theta_a(\ux)]
		\prod_{j=2}^k
		(h\rho)_{a,j}(x_j) ,\\
	\bm{u}_a
	&\equiv \sum_{\ux \in\albet^k}
		[1-\theta_a(\ux)]
		\prod_{j=1}^k
		(h\rho)_{a,j}(x_j).
	\end{align*}
For any $\lambda\in(0,1)$ and any $\smash{\zeta\in\mathcal{M}_2}$,
	\[F_n \le \lambda^{-1} \E\log
		\E'\bigg[
		\Big(
		\sum_{x\in\albet}
		h(x)
		\prod_{a=1}^d \bm{u}_a(x)
		\Big)^\lambda
		\bigg]
	-(k-1)\alpha \lambda^{-1}
		\E\log\E'[(\bm{u}_0)^\lambda]
 +O_{\epsilon}(n^{-1/3}) \]
where $\E'$ denotes the expectation over $\vec\rho$ conditioned on all else,
and $\E$ denotes the overall expectation.
\end{thm}

\begin{rmk}\label{r:dirac} In the statistical physics framework, elements $\rho\in\mathcal{M}_0$ correspond to belief propagation messages for the underlying model, which has state space $\albet$. Elements $\eta\in\mathcal{M}_1$ correspond to belief propagation messages for the 1\textsc{rsb} model (termed ``auxiliary model'' in \cite[Ch.~19]{MR2518205}), which has state space $\mathcal{M}_0$. The informal picture is that the $\eta$ associated to variable $x$ is determined by the geometry of the local neighborhood of $x$ --- that is to say, the randomness of $\zeta$ reflects the randomness in the geometry of the $R$-neighborhood of a uniformly randomly variable in the graph. In random regular graphs this randomness is degenerate --- the $R$-neighborhood of (almost) every vertex is simply a regular tree. It is therefore expected that  the best upper bound in Theorem~\ref{t:free.energy} can be achieved with $\zeta$ a point mass. \end{rmk}

\subsection{Replica symmetric bound} \label{s:ubd_rs}

Along the lines of \cite{MR2095932}, we first prove a weaker ``replica symmetric'' version of Theorem~\ref{t:free.energy}. Afterwards we will apply it to obtain the full result.

\begin{thm}\label{t:free_rs}In the setting of Theorem~\ref{t:free.energy}, define
	\[\Phi_V\equiv \E\log \Big(
		\sum_{x\in\albet}
		h(x)\prod_{a=1}^d \bm{u}_a (x) \Big),\quad
	\Phi_F \equiv
	(k-1)\alpha\E\log(\bm{u}_0)\,.\]Then $F_n \le \Phi_V-\Phi_F-O_{\epsilon}(n^{-1/3})$.\end{thm}

Inspired by the proof of \cite{MR3161470}, we prove Theorem~\ref{t:free_rs} by a combinatorial interpolation between two graphs, $\glit_{-1}$ and $\glit_{nd+1}$. The initial graph $\glit_{-1}$ will have free energy $\Phi_V$, and the final graph $\glit_{nd+1}$  will have free energy $F_n + \Phi_F$. We will show that, up to $O_\epsilon (n^{1/3})$ error, the free energy of $\glit_{-1}$ will be larger than that of $\glit_{nd+1}$, from which the bound of Theorem~\ref{t:free_rs} follows.

To begin, we take $\glit_{-1}$ to be a factor graph consisting of $n$ disjoint trees (Figure~\ref{f:intp_start}). Each tree is rooted at a variable $v$ which joins to $d$ clauses. Each of these clauses then joins to $k-1$ more variables, which form the leaves of the tree. We write $V$ for the root variables, $A$ for the clauses, and $U$ for the leaf variables. Note $|V|=n$, $|A|=nd$, and $|U|=nd(k-1)$.

Independently of all else, take a vector of i.i.d.\ samples $(\eta_u,\rho_u)_{u\in U}$ where $\eta_u$ is a sample from $\zeta$, and $\rho_u$ is a sample from $\eta_u$.\footnote{For the proof of Theorem~\ref{t:free_rs} it is equivalent to sample $\rho$ from $\eta^\textup{av} \equiv \int\eta\,d\zeta$.} As before, the variables and clauses in $\glit_{-1}$ are labelled independently with functions $h_v$ and $\theta_a$.  We now additionally assign to each $u\in U$ the label $(\eta_u,\rho_u)$. Let $(h\rho)_u(x)\equiv h_u(x)\rho_u(x)$. We consider the factor model on $\glit_{-1}$ defined by
	\[\mu_{\glit_{-1}}(\ux)
	= \f{1}{Z(\glit_{-1})}
	\prod_{v\in V} h_v(x_v)
	\prod_{a\in A} 
	[1-\theta_a(\ux_{\delta a})]
	\prod_{u\in U} (h\rho)_u(x_u)\,.\]
We now define the interpolating sequence of graphs $\glit_{-1},\glit_0,\ldots,\glit_{nd+1}$. Fix $m'\equiv 2n^{2/3}$. The construction proceeds by adding and removing clauses. Whenever we remove a clause $a$, the edges $\delta a$ are left behind as $k$ unmatched edges in the remaining graph. Whenever we add a new clause $b$, we label it with a fresh sample $\theta_b$ of $\theta$. The graph $\glit_r$ has clauses $F_r$ which can be partitioned into  $A_{U,r}$ (clauses involving $U$ only), $A_{V,r}$ (clauses involving $V$ only), and $A_r$ (clauses involving both $U$ and $V$). We will define below a certain sequence of events $\COUP_r$. Let $\COUP_{\le r}$ be the event that $\COUP_s$ occurs for all $0\le s \le r$. The event $\COUP_{\le -1}$ occurs vacuously, so $\P(\COUP_{\le -1})=1$. With this notation in mind, the construction goes as follows:
\begin{enumerate}[1.]
\item Starting from $\glit_{-1}$,
choose a uniformly random subset of $m'$ clauses from $F_{-1}=A_{-1}=A$, and remove them to form the new graph $\glit_0$.
\item For $0\le r\le nd-m'-1$, we start from $\glit_r$ and form $\glit_{r+1}$ as follows.
\begin{enumerate}[a.]
\item If $\COUP_{\le r-1}$ succeeds, choose a uniformly random clause $a$ from $A_r$, and remove it to form the new graph $\glit_{r,\circ}$. Let $\delta'U_{r,\circ}$ and $\delta'V_{r,\circ}$ denote the unmatched half-edges incident to $U$ and $V$ respectively in $\glit_{r,\circ}$, and define the event
	\[\COUP_r\equiv\set{
	\min\set{\mcur,\mcvr}\ge k}\,.\]
If instead $\COUP_{\le r-1}$ fails, then $\COUP_{\le r}$ fails by definition.

\item If $\COUP_{\le r}$ fails, let $\glit_{r+1}=\glit_r$. If $\COUP_{\le r}$ succeeds, then with probability $1/k$ take $k$ half-edges from $\delta' V_{r,\circ}$ and join them into a new clause $c$. With the remaining probability $(k-1)/k$ take $k$ half-edges from  $\delta' U_{r,\circ}$ and join them into a new clause $c$. \end{enumerate}
\item For $nd-m' \le r\le nd-1$ let $\glit_{r+1}=\glit_r$. Starting from $\glit_{nd}$, remove all the clauses in $A_{nd}$. Then connect (uniformly at random) all remaining unmatched $V$-incident edges into clauses. Likewise, connect all remaining unmatched $U$-incident edges into clauses.  Denote the resulting graph $\glit_{nd+1}$.
\end{enumerate}
By construction, $\glit_{nd+1}$ consists of two disjoint subgraphs, which are the induced subgraphs $\glit_U,\glit_V$ of $U,V$ respectively. Note that $\glit_V$ is distributed as the random graph $\glit$ of interest, while $\glit_U$ consists of a collection of $nd(k-1)/k = n\alpha(k-1)$ disjoint trees.

\begin{figure}[h!]
\centering
\begin{subfigure}[h!]{.95\textwidth}\centering
	\includegraphics[page=1]{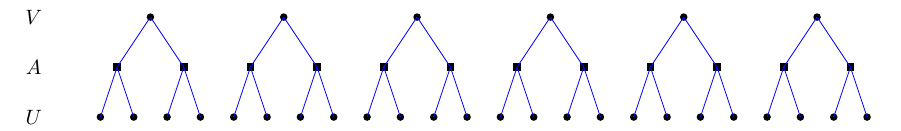}
	\caption{$\glit_{-1}$}
	\label{f:intp_start}
\end{subfigure}\\
\begin{subfigure}[h!]{.95\textwidth}\centering
	\includegraphics[page=2]{interpolation}
	\caption{$\glit_r$ with $m'=1,r=3$}
\end{subfigure}
\begin{subfigure}[h!]{.95\textwidth}\centering
	\includegraphics[page = 3]{interpolation}
	\caption{$\glit_{nd+1}$}
	\label{f:intp_end}
\end{subfigure}
\caption{Interpolation with $d=2,k=3$, $n = 6$.}
\label{f:intp}
\end{figure}

\begin{lem}\label{l:UB.intp}
	 Under the construction above,
\begin{equation}\label{e:UB.intp}
	\E\log Z(\glit_0)\ge\E \log Z(\glit_{nd})-O_\epsilon (n^{1/3})\,,\end{equation}
where the expectation $\E$ is over
the sequence of random graphs
$(\glit_r)_{-1\le r\le nd+1}$.

\begin{proof} Let $\mathscr{F}_{r,\circ}$ be the $\sigma$-field generated by $\glit_{r,\circ}$, and write $\E_{r,\circ}$ for expectation conditioned on $\mathscr{F}_{r,\circ}$. One can rewrite \eqref{e:UB.intp} as 
	\[	\E\log \f{Z(\glit_0)}{Z(\glit_{nd})}
		= \sum_{r=0}^{nd-1} \E \Delta_r,\quad
		\Delta_r\equiv 
		 \E_{r,\circ}
			\log \f{Z(\glit_r)}
				{Z(\glit_{r,\circ})}
			-\E_{r,\circ}\log \f{Z(\glit_{r+1})}
				{Z(\glit_{r,\circ})}\,.\]
In particular, $\Delta_r=0$ if the coupling fails. Therefore it suffices  to show that $\Delta_r$ is positive conditioned on $\COUP_{\le r}$.\footnote{The event $\COUP_{\le r}$ is measurable with respect to $\mathscr{F}_{r,\circ}$, since $\mcvr,\mcur$ would remain less than $k$ if the coupling fails at an earlier iteration.} First we compare $\glit_r$ and $\glit_{r,\circ}$. Conditioned on $\mathscr{F}_{r,\circ}$,  we know $\glit_{r,\circ}$. From $\glit_{r,\circ}$ we can obtain $\glit_r$ by adding a single clause $a\equiv a_r$, together with a random label $\theta_a$ which is a fresh copy of $\theta$. To choose the unmatched edges  $\delta a=(e_1,\ldots,e_k)$  which are combined into the clause $a$, we take $e_1$ uniformly at random from $\delta'V_{r,\circ}$,  then take  $\set{e_2,\ldots,e_k}$  a uniformly random subset of $\delta'U_{r,\circ}$. Let $\mu_{r,\circ}$ be the Gibbs measure on $\glit_{r,\circ}$ (ignoring unmatched half-edges). Let $\uxx\equiv(\ux,\ux^1,\ux^2,\ldots)$ be an infinite sequence of i.i.d.\ samples from $\mu_{r,\circ}$, and write $\langle\cdot\rangle_{r,\circ}$ for the expectation with respect to their joint law. Then
	\[\E_{r,\circ}
	\log \f{Z(\glit_r)}{Z(\glit_{r,\circ})}
	= \E_{r,\circ} \log
		(1- \langle
	\theta(\ux_{\delta a})
	 \rangle_{r,\circ})
	= \sum_{p\ge 1} \f1p
	\mathscr{A}_p,\quad
	 \mathscr{A}_p
	 \equiv
	 \E_{r,\circ}\bigg[\Big \langle
	\prod_{\ell=1}^p \theta(\ux^\ell_{\delta a})
	\Big\rangle_{r,\circ}
	\bigg]\,.\]
We have $\E_{r,\circ}=\E_a\E_\theta$ where $\E_a$ is expectation over the choice of $\delta a$, and $\E_\theta$ is expectation over the choice of $\theta$. Under $\E_a$, the edges $(e_2,\ldots,e_k)$ are weakly dependent, since they are required to be distinct elements of $\delta'U_{r,\circ}$. We can consider instead sampling $e_2,\ldots,e_k$ uniformly \bemph{with replacement} from $\delta'U_{r,\circ}$, so that $e_1,\ldots,e_k$ are independent conditional on $\mathscr{F}_{r,\circ}$; let $\E_{a,\textup{ind}}$ denote expectation with respect to this choice of $\delta a$. Under $\E_{a,\textup{ind}}$ the chance of a collision $e_i=e_j$ ($i\le j$) is  $O(k^2/|\delta'U_{r,\circ}|)$. Recalling $1-\theta\ge\epsilon$ almost surely, we have
	\[\mathscr{A}_{p,\textup{ind}}
	\equiv
	\E_{a,\textup{ind}}
	\E_\theta
	\bigg[
 \Big \langle
	\prod_{\ell=1}^p \theta(\ux^\ell_{\delta a})
	\Big\rangle_{r,\circ}
	\bigg]
	= \mathscr{A}_p
	+ O(1) (1-\epsilon)^p 
	\min\bigg\{ \f{k^2}{|\delta'U_{r,\circ}|},
	1 \bigg\}\,.\]
Recall from \eqref{e:theta.product} the product form of $\theta$, and let $\E_f$ denote expectation over the law of $f\equiv (f_s)_{s\in S}$. Then,
with $b^p(\vec{s})$ as defined in \eqref{e:UB.cond2}, we have
	\begin{align*} \mathscr{A}_{p,\textup{ind}}
	&=\sum_{\vec{s}\in S^p}
	b^p(\vec{s})
	\bigg\langle
	\E_{a,\textup{ind}}\bigg\{
	\prod_{j=1}^k
	\E_f\bigg[
	\prod_{\ell=1}^p
	f_{s_\ell}(x^\ell_{e_j})\bigg]
	\bigg\}
	\bigg\rangle_{r,\circ}\\
	&= \sum_{\vec{s}\in S^p}
	 b^p(\vec{s})
	 \langle
	 I_{V,\vec{s}}(\uxx)
	 I_{U,\vec{s}}(\uxx)^{k-1}
	 \rangle_{r,\circ},
	\end{align*}
where, for $W=U$ or $W=V$, we define
	\[I_{W,\vec{s}}(\uxx)
	\equiv
	\f1{|\delta'W_{r,\circ}|}
	\sum_{e\in \delta'W_{r,\circ}}
	\E_f\bigg[
	\prod_{\ell=1}^p
	f_{s_\ell}(x^\ell_e)
	\bigg]\,.\]
Summing over $p\ge1$ gives that, on the event $\COUP_{\le r}$,
	\begin{align*}
	&\E_{r,\circ} \log\f{Z(\glit_r)}
		{Z(\glit_{r,\circ})}
	= \sum_{p\ge1} \f1p
	\sum_{\vec{s}\in S^p}
	b^p(\vec{s})
	\E_{r,\circ}
	\langle
	I_{V,\vec{s}}(\uxx)
	I_{U,\vec{s}}(\uxx)^{k-1}
	\rangle_{r,\circ}
	+ \textsf{err}_{r,1},\\
	&\qquad\text{where }
	|\textsf{err}_{r,1}|
	\le O_\epsilon(1) 
	\min\bigg\{
	 \f{k^2}{|\delta'U_{r,\circ}|},
	 1 
	 \bigg\}.
	\end{align*}
A similar comparison between $\glit_{r+1}$ and $\glit_{r,\circ}$ gives
	\begin{align*}
	\E_{r,\circ} 
	\log\f{Z(\glit_r)}{Z(\glit_{r,\circ})}
	&= \sum_{p\ge 1} \f1p
	\E_{r,\circ} \bigg[
		 \sum_{\vec s\in S^p}
		b^p(\vec{s})
		\bigg\langle
		\f{k-1}{k} I_{U,\vec s}(\uxx)^{k}
			+ \f{1}{k} I_{V,\vec s}(\uxx)^{k}
		\bigg\rangle_{r,\circ}
	\bigg]
	+\textsf{err}_{r,2},\\
	&\qquad |\textsf{err}_{r,2}|
	\le O_\epsilon(1)
	\min\bigg\{
	 \f{k^2}{
	 \min\set{|\delta'U_{r,\circ}|,
	 	|\delta'V_{r,\circ}|}},
	 1 
	 \bigg\}.
	\end{align*}
We now argue that the sum of the error terms $\textsf{err}_{r,1},\textsf{err}_{r,2}$, over $0\le r\le nd-1$, is small in expectation. First note that for a constant $C=C(k,\epsilon)$,
	\[\sum_{r=0}^{nd-1}
	\E[\textsf{err}_{r,1}
	+\textsf{err}_{r,2}]
	\le C n
	\bigg[ n^{-2/3}
	+ \P\Big( \min
	\set{|\delta'V_{r,\circ}|,
	|\delta'V_{r,\circ}|}
	\le n^{2/3}
	\text{ for some }
	r \le nd \Big) \bigg]\,.\]
The process $(|\delta'V_{r,\circ}|)_{r\ge0}$ is an unbiased random walk started from  $m'+1 = 2n^{2/3}+1$. In each step it goes up by $1$ with chance $(k-1)/k$, and down by $k-1$ with chance $1/k$; it is absorbed if it hits $k$ before time $nd-m'$. Similarly, $(|\delta'U|_{r,\circ})_{r\ge0}$ is an unbiased random walk started from  $(m'+1)(k-1)$ with an absorbing barrier at $k$. By the Azuma--Hoeffding bound, there is a constant $c=c(k)$ such that
	\[\P(|\delta'V_{r,\circ}| 
	\le |\delta'V_{0,\circ}|-n^{2/3})
	+	\P(|\delta'U_{r,\circ}| 
	\le |\delta'U_{0,\circ}|-n^{2/3})
	\le \exp\{ -c n^{1/3} \}\]
Taking a union bound over $r$ shows that  with very high probability, neither of the walks $|\delta'V_{r,\circ}|,|\delta'U_{r,\circ}|$ is absorbed before time $nd-m'$, and (adjusting the constant $C$ as needed)
	\[\sum_{r=0}^{nd-1}
	\E[\textsf{err}_{r,1}
	+\textsf{err}_{r,2}]
	\le C n^{1/3}\,.\]
Altogether this gives
	\begin{align*}
	&\E \log\f{Z(\glit_0)}{Z(\glit_{nd})}
	- O_\epsilon(n^{1/3})\\
	&=\sum_{r=0}^{nd-1}
	\sum_{p\ge1} \f1p
	\sum_{\vec{s}}
	b^p(\vec{s})
	\E_{r,\circ}
	\bigg\langle
	I_{V,\vec{s}}(\uxx)
	I_{U,\vec{s}}(\uxx)^{k-1}
	-\f{k-1}{k} I_{U,\vec{s}}(\uxx)^{k-1}
	-\f1k I_{V,\vec{s}}(\uxx)^{k-1}
	\bigg\rangle_{r,\circ}.
	\end{align*}
Using the fact that $x^k-kxy^{k-1}+(k-1)y^k\ge 0$ for all $x,y\in \R$ and even $k\ge 2$, or $x,y\ge 0$ and odd $k\ge 3$ finishes the proof.\end{proof}
\end{lem}

\begin{cor}\label{c:UB.intp2} In the setting of Lemma~\ref{l:UB.intp}, 
	\[\E\log Z(\glit_{-1})\ge\E \log Z(\glit_{nd+1})
	-O_{\epsilon}(n^{2/3}),\]
where the expectation $\E$ is over the sequence of random graphs $(\glit_r)_{-1\le r\le nd+1}$.

\begin{proof} Adding or removing a clause can change the partition function by at most a multiplicative constant (depending on $\epsilon$). On the event that the coupling succeeds for all $r$,
	\[\bigg|
	\log \f{Z(\glit_0)}{Z(\glit_{-1})}\bigg|
	+\bigg|
	\log \f{Z(\glit_{nd+1})}
		{Z(\glit_{nd})}\bigg|
	= O_\epsilon(m')
	= O_\epsilon(n^{2/3})\,.\]
On the event that the coupling fails,
the difference is crudely $O_\epsilon(n)$.
We saw in the proof of Lemma~\ref{l:UB.intp} that the coupling fails with probability exponentially small in $n$, so altogether we conclude
	\[\E\bigg|
	\log \f{Z(\glit_0)}{Z(\glit_{-1})}\bigg|
	+\E\bigg|
	\log \f{Z(\glit_{nd+1})}
		{Z(\glit_{nd})}\bigg|
	= O_\epsilon(n^{2/3})\,.\]
Combining with the result of Lemma~\ref{l:UB.intp} proves the claim.\end{proof}
\end{cor}

\begin{proof}[Proof of Theorem~\ref{t:free_rs}]
In the interpolation, the initial graph $\glit_{-1}$ consists of $n$ disjoint trees 
$T_v$, each rooted at a variable $v\in V$. Thus
	\[n^{-1}\E \log Z(\glit_{-1})
	=\E\log Z(T_v)
	=\E\log
	\bigg( \sum_{x\in\albet} 
	h_v(x)
	\prod_{a=1}^d \bm{u}_a(x) \bigg)\,.\]
The final graph $\glit_{nd+1}$ is comprised of two disjoint subgraphs --- one subgraph $\glit_V$  has the same law as the graph $\glit$ of interest, while the other subgraph $\glit_U=(U,F_U,E_U)$ consists of $n\alpha(k-1)$ disjoint trees $S_c$, each rooted at a  clause $c\in A_U$. Thus
	\[n^{-1}\E \log Z(\glit_{nd+1})
	= \alpha(k-1)\E\log Z(S_c)
	 + n^{-1}\E\log Z(\glit) 
	= \alpha(k-1)\E \log \bm{u}_0 + F_n\,.\]
The theorem follows by substituting these into the bound of Corollary~\ref{c:UB.intp2}.\end{proof}

\subsection{Proof of 1RSB bound}\label{s:ubd_rsb} For the proof of Theorem~\ref{t:free.energy},  we take $\glit_{-1}$ as before and modify it as follows. Where previously each $u\in U$ had spin value $x_u\in\albet$, it now has the augmented spin $(x_u,\gamma_u)$ where  $\gamma$ goes over the positive integers. Let $\vec{\gamma}\equiv(\gamma_u)_u$. Next, instead of labeling $u$ with $(h_u,\eta_u,\rho_u)$ as before, we now label it with $(h_u,\eta_u,(\rho^\gamma_u)_{\gamma\ge1})$ where $(\rho^\gamma_u)_{\gamma\ge1}$ is an infinite sequence of i.i.d.\ samples from $\eta_u$. Lastly, we join all variables in $U$ to a new clause $a_*$ (Figure~\ref{f:intp1RSB}), which is labelled with the function
	\[\varphi_{a_*}(\vec{\gamma})
	=\sum_{\gamma\ge1}
	z_\gamma
	\prod_{u\in U}\Ind{ \gamma_u=\gamma }\]
for some sequence of (random) weights $(z_\gamma)_{\gamma\ge1}$. Let $\mathscr{H}_{-1}$ denote the resulting graph.

\begin{figure}[h!]
	\centering
	\includegraphics[width=.9\textwidth,page = 4]{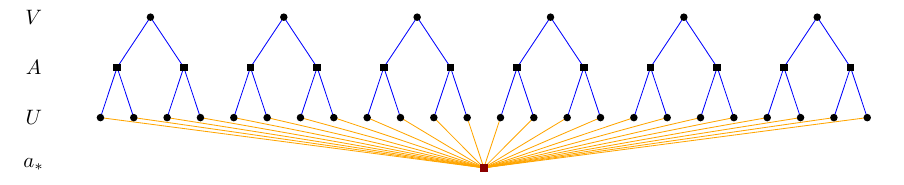}
	\caption{$\mathscr{H}_{-1}$}
	\label{f:intp1RSB}
\end{figure}

Given $\mathscr{H}_{-1}$, let $\mu_{\mathscr{H}_{-1}}$ be the associated Gibbs measure on configurations $(\vec{\gamma},\ux)$. Due to the definition of $\varphi_{a_*}$, the support of $\mu_{\mathscr{H}_{-1}}$ contains only those configurations where all the $\gamma_u$ share a common value $\gamma$, in which case we denote $(\vec{\gamma},\ux)\equiv(\gamma,\ux)$. Explicitly,
	\[\mu_{\mathscr{H}_{-1}}(\gamma,\ux)
	=\f{1}{Z(\mathscr{H}_{-1})}
	z_\gamma
	\prod_{v\in V}
	h_v(x_v)
	\prod_{a\in A}
	[1-\theta_a(\ux_{\delta a})]
	\prod_{u\in U} (\rho^\gamma h)_u(x_u)\,.\]
We can then define an interpolating sequence $\mathscr{H}_{-1},\ldots, \mathscr{H}_{nd+1}$  precisely as in the proof of Theorem~\ref{t:free_rs}, leaving $a_*$ untouched.  Let $\glit_r$ denote the graph $\mathscr{H}_r$ without the clause $a_*$, and let $Z_\gamma(\glit_r)$ denote the partition function on $\glit_r$ restricted to configurations where $\gamma_u=\gamma$ for all $u$. Then, for each $0\le r\le nd+1$,
	\[Z(\mathscr{H}_r)
	=\sum_\gamma 
	z_\gamma
	Z_\gamma(\glit_r)\,.\]
The proofs of Lemma~\ref{l:UB.intp} and Corollary~\ref{c:UB.intp2} carry over to this setting with essentially no changes, giving

\begin{cor}\label{c:UB.rsbintp} Under the assumptions above, 
	\[\E\log Z(\mathscr{H}_{-1})
	\ge\E \log Z(\mathscr{H}_{nd+1})
	-O_{\epsilon}(n^{2/3}),\]
where the expectation $\E$ is  over the sequence of random graphs $(\mathscr{H}_r)_{-1 \le r\le nd+1}$. \end{cor}

The result of Corollary~\ref{c:UB.rsbintp} applies for any $(z_\gamma)_{\gamma\ge1}$. Now take $(z_\gamma)_{\gamma\ge1}$ to be a Poisson--Dirichlet process with parameter $\lambda\in(0,1)$.\footnote{That is to say, let $(w_\gamma)_{\gamma\ge1}$ be a Poisson point process on $\R_{>0}$ with intensity measure $w^{-(1+\lambda)}\,dw$. Let $W$ denote their sum, which is finite almost surely. Assume the points of $w_\gamma$ are arranged in decreasing order, and write $z_\gamma \equiv w_\gamma/W$. Then $(z_\gamma)_{\gamma\ge1}$ is distributed as a Poisson--Dirichlet process with parameter $\lambda$.} The process has the following invariance property (see e.g.\ \cite[Ch.~2]{MR3052333}):

\begin{ppn}\label{p:UB.PoisDirch} Let $(z_\gamma)_{\gamma\ge1}$ be a Poisson--Dirichlet process with parameter $\lambda\in(0,1)$. Independently, let $(\xi_\gamma)_{\gamma\ge 1}$ be a sequence of i.i.d.\ positive random variables with finite second moment. Then the two sequences $(z_\gamma \xi_\gamma)_{\gamma\ge 1}$ and $(z_\gamma (\E \xi_1^\lambda)^{1/\lambda})_{\gamma\ge 1}$ have the same distribution, and consequently
	\[\E\log \sum_{\gamma \ge 1} 
	z_\gamma\xi_\gamma =\f{1}{\lambda}\log \E\xi^\lambda\,.\]
\end{ppn}

\begin{proof}[Proof of Theorem~\ref{t:free.energy}] Consider 
	$\vec{Z}(\gamma)
	\equiv (Z_\gamma(\glit_r)
	)_{-1 \le r \le nd+1}$.
If we condition on everything else except for the $\rho$'s, then $(\vec{Z}(\gamma))_{\gamma\ge1}$ is an i.i.d.\ sequence indexed by $\gamma$.  Let $\E_{z,\rho}$ denote expectation over  the $z$'s and $\rho$'s, conditioned on all else: then applying Proposition~\ref{p:UB.PoisDirch} gives 
	\begin{align*}
	n^{-1}\E \log Z(\mathscr{H}_{-1}) 
	&= (n\lambda)^{-1}
	\E\log
	\E_{z,\rho} [ Z(\mathscr{G}_{-1})^\lambda ]
	= \lambda^{-1}
	\E \log \E_{z,\rho}\bigg[
	\bigg(\sum_{x\in\albet} h(x) 
	\prod_{a=1}^d \bm{u}_a(x)\bigg)^\lambda\bigg],\\
	n^{-1}\E \log Z(\mathscr{H}_{nd+1}) 
	&=F_n
	+ \lambda^{-1}\E
	\log\E_{z,\rho}[ (\bm{u}_0)^\lambda ].
	\end{align*}
Combining with Corollary~\ref{c:UB.rsbintp} proves the result.\end{proof}

\subsection{Extension to higher levels of RSB}
\newcommand{\NN}{\mathbb{N}}
\newcommand{\cA}{\mathcal{A}}
\newcommand{\cF}{\mathcal{F}}
\newcommand{\sfp}{\mathsf{p}}
\newcommand{\sfT}{\mathsf{T}}
\newcommand{\sm}{\setminus}
\newcommand{\MM}{\mathcal{M}}

We finally explain that Theorem~\ref{t:free.energy} can be extended relatively easily to cover the scenario of $r$-step replica symmetry breaking. Before stating the result, we define some notations (mainly following notation of
\cite[\S2.3]{MR3052333}). Let $\NN$ be the set of positive integers and $\NN^r$ be its $r$-fold product; in particular, $\NN^0\equiv\set{\varnothing}$. We consider arrays indexed by the set
	\[
	\cA
	\equiv \bigcup_{p=0}^r \NN^p\,.
	\]
We view $\cA$ as a depth-$r$ infinitary tree rooted at $\varnothing$. For $0\le p\le r-1$, each vertex $\gamma=(\gamma_1,\ldots,\gamma_p)\in\NN^p$ has children $\gamma n \equiv (\gamma_1,\dots,\gamma_s,n)\in \NN^{s+1}$. The leaves of the tree are in the last level $\NN^r$. For $\gamma\in\NN^p$ write $|\gamma|\equiv p$, and let $\sfp(\gamma)$ be the path between the root and $\gamma$
(not inclusive):
	\[
    \sfp( \gamma) \equiv \bigg\{
    \gamma_1,(\gamma_1,\gamma_2),
    \dots,(\gamma_1,\dots,\gamma_{p-1})
    \bigg\}\,.
	\]
Fix a sequence of parameters $ \vec m = (m_1,\dots,m_r)$ satisfying
\begin{equation}
    0< m_0< \dots < m_{r-1}< 1. 
    \label{e:fop}
\end{equation}
For each $\gamma \in \cA\setminus\NN^r$, let $\Pi_\gamma$ be
 (independently of all else)
  a Poisson--Dirichlet point process with parameter $m_{|\gamma|}$. Let $(u_{\gamma n})_{n\in\NN}$ be the points of $\Pi_\gamma$ arranged in decreasing order. As $\gamma$ goes over all of $\cA\sm\NN^r$, we obtain an array $(u_\beta)_{\beta\in\cA\sm\NN^0}$. Let
  	\[w_\gamma \equiv \prod_{\beta\in\sfp(\gamma)}u_\beta\,.\]
The \bemph{Ruelle probability cascade of parameter $\vec m$}
(hereafter $\textup{RPC}(\vec{m})$) is defined as the $\NN^r$-indexed array 
	\[\nu_\gamma
	\equiv \frac{w_\gamma}{\sum_{\beta\in\NN^r}w_\beta}\,.\]
For the validity of the definition, see for instance
\cite[Lem.~2.4]{MR3052333}.
As in the 1\textsc{rsb} setting, we plan to apply Theorem~\ref{t:free_rs} to the modified graph $\mathscr{H}_{-1}$, where we ``glue'' multiple weighted copies of $\glit_{-1}$'s together via the extra clause $a_\star$.  The only difference is that now the copies of $\glit_{-1}$ are indexed by $\gamma\in\NN^r$ instead of $\NN$. More precisely, the extra spin at each vertex $u\in U$ will take a value $\gamma\in\NN^r$; the label at each vertex $u\in U$ will be $(h_u,\eta_u,(\rho^\gamma_u)_{\gamma\in\cA})$; and the function at $a_\star$ will be
\begin{equation}\label{e:phiastar}
    \phi_{a_\star} (\vec\gamma) 
    = \sum_{\gamma\in\NN^r} z_\gamma
    \prod_{u\in U} \Ind{\gamma_u = \gamma}
    ,
\end{equation}
where $(z_\gamma)_{\gamma\in\NN^r}$ is a $\NN^r$-indexed random array representing the weight of copy $\gamma\in\NN^r$. In the proof, we will choose $(z_\gamma)_{\gamma\in\NN^r}$ according to the $\textup{RPC}(\vec{m})$ law.

We now specify the labels $(\rho^\gamma)_{\gamma\in\cA}$ that will be used in the proof. Recall that $\MM_0$ is the space of probability measures on the alphabet $\mathcal{X}$. We recursively define $\MM_r$, for $1\le r\le p$, to be the space of probability measures on $\MM_{r-1}$. Now fix an element $\rho^\varnothing = \zeta \in\MM_r$. For each $0\le p \le r-1$ and $\gamma\in \NN^p$, suppose inductively that we have constructed $\rho^\gamma\in\MM_{r-p}$. We then take
 $\rho^{\gamma n} \in\MM_{r-p-1}$ for $n\in\NN$
 as  i.i.d.\ samples $\rho^{\gamma}\in\MM_{r-p}$ in $\MM_{r-p-1}$. The process terminates with the construction of $\rho^\gamma\in\mathcal{X}$ for each $\gamma\in\NN^r$. Define the $\sigma$-field
 	\[
	\cF_p
	\equiv\sigma\bigg(
	\Big(
	(\rho^\gamma)_{\gamma\in\NN^s}
	\Big)_{s\le p}
	\bigg)\,,
	\]
and write $\E_p$ for expectation conditional on $\cF_p$.
For any deterministic function $V(u,\rho)$, any random variable $U$ independent of $(\rho^\gamma)_{\gamma\in\NN^r}$, and any sequence of parameters $\vec m$ satisfying \eqref{e:fop}, consider the random array
 $(V^\gamma)_{\gamma\in\NN^r} \equiv (V(U,\rho^\gamma))_{\gamma\in\NN^r}$.
Let $\sfT_r(V) = V(U,\rho^{\vec 1})$, and for $0\le p\le r-1$ let
	\[
    \sfT_p(V) = \bigg\{
    	\E_p\Big( \sfT_{p+1}(V) \Big)^{m_p}
	\bigg\}^{1/m_p}
	\]
The resulting operator $\sfT_0$ depends implicitly on the distribution of $U$,  measure $\rho^\varnothing\in\MM_r$ and parameter $\vec m$. The following lemma is a well-known property of the RPC.

 \begin{lem}[{\cite[Prop.~2]{MR2095932}}]\label{l:RPC}
Let $(z_\gamma)_{\gamma\in\NN^r}$ be the RPC with parameter $\vec m$. Under the notations above,
    \[
        \E \log \sfT_0(V) = \E\log\sum_{\gamma\in\NN^r} z_\gamma V^\gamma.
    \]
\end{lem}

\noindent The next result generalizes Theorem~\ref{t:free.energy}.

\begin{thm}\label{t:rRSBbound}  Consider a (random) Gibbs measure \eqref{e:gibbs.ubd} satisfying assumptions \eqref{e:theta.product}--\eqref{e:UB.cond2}. Write $(h\rho)_{a,j}(x) \equiv h_{a,j}(x)\rho_{a,j}(x)$. For each $a\in F$ we define
	\begin{align*}
	\bm{u}_a(x)
	&\equiv
	\sum_{\ux\in\albet^k}
		\Ind{x_1=x}
		[1-\theta_a(\ux)]
		\prod_{j=2}^k
		(h\rho)_{a,j}(x_j)
    ,\\
	\bm{u}_a
	&\equiv \sum_{\ux \in\albet^k}
		[1-\theta_a(\ux)]
		\prod_{j=1}^k
		(h\rho)_{a,j}(x_j)\,.
	\end{align*}
Note that $\bm{u}_a(x)$ and $\bm{u}_a$
are deterministic functions of the variables
 $\big(\theta_a, (h_{a,j})_{j\in[k]}, (\rho_{a,j})_{j\in[k]}\big)$. Let
	\[
    \bm{v}\equiv
    \sum_{x\in\albet}
    h(x)
    \prod_{a=1}^d \bm{u}_a(x)
    .
\]
For any $\smash{\zeta\in\mathcal{M}_r}$ and sequence $\vec m$ satisfying \eqref{e:fop}, let 
$(\rho^\gamma)_{\gamma\in\NN^r}$ be constructed as above,
and let
$(\rho^\gamma_{a,j})_{\gamma\in\NN^r}$ be i.i.d.\ copies indexed by $(a,j)$. Define $\sfT_0$ similarly as above, using the $\sigma$-fields
	\[
    \cF_p = \sigma\bigg(
    (\rho^\gamma_{a,j})_{\gamma\in\NN^s}:
    a\in F, j\in [k], s\le p
    \bigg)\,.
	\] 
Then the nonasymptotic free energy 
$F_n \equiv n^{-1}\E_n\log Z(\glit)$
 satisfies the bound
	\[
        F_n \le 
        \E\log
        \sfT_0
        (\bm{v})
	-(k-1)\alpha 
    \E\log\sfT_0(\bm{u}_0)
 +O_{\epsilon}\bigg(\f1{n^{1/3}}\bigg) 
 \]
where $\E$ denotes the expectation over $(\theta_a)_{a\in F}$ and $(h_{a,j})_{a\in F,j\in[k]}$. 
\end{thm}

\begin{proof} As outlined above, we consider the modified graph $\mathscr{H}_{-1}$ where each vertex $u\in U$ is independently labeled with $(h_u,\eta_u,(\rho^\gamma_u)_{\gamma\in\cA})$ and the extra clause $a_\star$ is labeled with the function defined in \eqref{e:phiastar}. In this setting, each $u\in U$ has spin value $(\gamma,x)\in \NN^r\times \albet$. Since we are interested only in configurations $(\vec\gamma,\vec x)$ such $\gamma_u\equiv \gamma$ for all $u\in U$, we write $(\gamma,\vec x)$ instead of $(\vec\gamma,\vec x)$ and define the Gibbs measure as
	\[
        \mu_{\mathscr{H}_{-1}}(\gamma,\vec x)
        =\f{1}{Z(\mathscr{H}_{-1})}
        z_\gamma
        \prod_{v\in V}
        h_v(x_v)
        \prod_{a\in A}
        [1-\theta_a(\ux_{\delta a})]
        \prod_{u\in U} (\rho^\gamma h)_u(x_u)
        \,.
    \]
Sample the weights $(z_\gamma)_{\gamma\in\NN^r}$ according the law 
$\textup{RPC}(\vec m)$. The result then follows by the proof of
Theorem~\ref{t:free.energy}, with Lemma~\ref{l:RPC} replacing the role of Proposition~\ref{p:UB.PoisDirch}.
\end{proof}

\pagebreak

\newcommand{\uL}{\underline{\smash{\textup{\texttt{L}}}}}
\newcommand{\mm}{\mathfrak{m}}
\newcommand{\free}{\textup{\texttt{f}}}
\newcommand{\Bin}{\textup{\textsf{Bin}}}

\newcommand{\bN}{\mathbf{N}}
\newcommand{\bZ}{\bm{Z}}
\newcommand{\bXi}{\bm{\Xi}}

\newcommand{\bDelta}{\bm{\Delta}}
\newcommand{\se}{\textup{se}}
\newcommand{\sep}{\textup{sep}}
\newcommand{\sy}{\textup{sy}}
\newcommand{\bdy}{\textup{bdy}}
\newcommand{\norm}[1]{\|#1\|}
\newcommand{\ff}{\printcol{f}}
\newcommand{\rr}{\printcol{r}}

\noindent \textbf{CORRECTION (2023).}
In the next two sections, we give corrections for two propositions in the above. In Section~\ref{sec:apriori} we obtain \textit{a~priori} estimates on the occurrence of free variables and forcing clauses, in both the first and second moment. In Section~\ref{sec:correct} we use these estimates to prove Propositions~\ref{corrected:prop:1} and \ref{corrected:prop}, which are the corrected versions of 
Propositions \ref{p:first.mmt}
and \ref{p:second.mmt}
respectively. We gratefully acknowledge an anonymous referee who pointed out the gap in the proof, as well as Youngtak Sohn who helped us write down the argument for closing the gap.

\section{A priori estimates}
\label{sec:apriori}

In this section, we obtain \textit{a~priori} estimates that are necessary to prove suitable modifications of Propositions \ref{p:first.mmt}
and \ref{p:second.mmt}. In \S\ref{subsec:apriori:first}, we present a simpler proof of \cite[Prop.~2.2]{MR3440193}, which gives a first moment bound on free variables; we also prove a first moment bound on forcing clauses. In \S\ref{subsec:apriori:second}, we extend the methods from \S\ref{subsec:apriori:first} to prove second moment bounds on free variables and forcing clauses, which we will use
in Section~\ref{sec:correct} to rule out the possibility of certain boundary maximizers of the function $\bF_2\equiv \bF_{2,\lambda,T}$.

Recall that $G=(V,F,E)$ denotes a random $(d,k)$-regular NAE-SAT instance, with $|V|=n$, $|F|=m$, and $|E|=nd=mk$. We let $\SOL(G)\subseteq\{\zro,\one\}^n$ denote the set of NAE-SAT solutions of $G$, and $Z\equiv |\SOL(G)|$. We first review the law of $G$ conditioned on $\ux$ being a solution, where $\ux\in\{\zro,\one\}^n$ is any given assignment.
To this end, note that we can decompose $G=(\mm,\uL)$ where $\mm$ denotes the random set of edges (i.e., the random matching between variable-incident and clause-incident half-edges), and $\uL$ denotes the random assignment of edge literals. Then the law of $G=(\mm,\uL)$ conditional on $\ux\equiv\zro$ being a solution can be written as
	\[ \P\Big((\mm,\uL)\,\Big|\,\ux\in\SOL(\mm,\uL)\Big)
	=\frac{\Ind{\ux\in\SOL(\mm,\uL)}}{
	\sum_{\mm',\uL'} \Ind{\ux\in\SOL(\mm',\uL')}}
	=\frac{\Ind{\ux\in\SOL(\mm,\uL)}}{(nd)! (2^k-2)^m}\,.
	\]
Summing over $\uL$ gives
	\beq\label{e:random.matching}
	\P\Big(\mm\,\Big|\,\ux\in\SOL(\mm,\uL)\Big)
	= \sum_{\uL}
	\frac{\Ind{\ux\in\SOL(\mm,\uL)}}{(nd)! (2^k-2)^m}
	= \frac{1}{(nd)!}\,,
	\eeq
that is, if we forget the edge literals $\uL$, the conditional law of $\mm$ given $\ux$ is simply uniformly random among all matchings. Taking the ratio of the above two relations gives
	\[
	\P\Big(\uL\,\Big|\,\mm,\ux\in\SOL(\mm,\uL)\Big)
	=\frac{\Ind{\ux\in\SOL(\mm,\uL)}}{(2^k-2)^m}\,,
	\]
that is, conditional on $\ux$ and the matching $\mm$, the edge literals $\uL$ are uniformly random among all $\uL'$ such that $\ux$ is a solution of $(\mm,\uL')$.

\subsection{First moment bound on free variables and forcing clauses}
\label{subsec:apriori:first}
We first present a simpler proof of a result from \cite{MR3440193}. 
Recalling that $Z=|\SOL(G)|$, we let $Z_{\ge nt}$ denote the contribution to $Z$ from solutions whose coarsening $\ueta$ has more than $nt$ free variables. We then have the following:

\begin{ppn}[{\cite[Prop.~2.2]{MR3440193}}] 
\label{p:first.mmt.bound.frees}
We have the first moment bound
	\[
	\E Z_{\ge n7/2^k}
	\le \frac{\E Z}{\exp(6n/2^k)}
	\]
as long as $\alpha \ge 2^{k-1}\log 2(1-O(1/k^2))$.

\begin{proof}
By symmetry, we have $\E Z_{\ge nt}= (\E Z) \mathbf{f}(nt)$,
where $\mathbf{f}(nt)$ denotes the probability, conditional on $\ux\equiv\zro$ being a solution, that the coarsening $\ueta$ of $\ux$ has more than $nt$ free variables. Let $\ueta(nt)\in\{\zro,\one,\free\}^n$ be the configuration that results after $nt$ steps of the coarsening procedure, started from $\ux\equiv\zro$. For any subset $A\subseteq V$ with $|A|=nt$, we will consider the quantity
	\[
	\mathbf{f}(A)
	\equiv
	\P\bigg(
	\textup{$\ueta(nt)$ has free variables $A$}
	\,\bigg|\,\textup{$\ux\equiv\zro$ is a solution}
	\bigg)\,.
	\]
Then note that $\mathbf{f}(nt)$ can be bounded by simply summing $\mathbf{f}(A)$
over the choices of $A$. For any fixed $A\subseteq V$, let
$F_\bullet \equiv F_\bullet(A,\mm)$ denote the set of clauses $a\in F$ that have exactly one edge going to $A$ (while the remaining $k-1$ edges go to $V\setminus A$). Conditional on $\ux\equiv\zro$ being a valid solution, each $a\in F$ has $2^k-2$ valid literal assignments (the all-$\zro$ and all-$\one$ assignments are forbidden). In order for $A$ to be the set of free variables in the coarsened configuration, each $a\in F_\bullet$ has two additional forbidden literal assignments (namely, the ones that would be forcing to the variable in $A$). It follows that
	\[\mathbf{f}(A)
	\le
	\E\bigg[ \bigg(1-\frac{2}{2^k-2}\bigg)^{|F_\bullet|}\bigg]\,.
	\]
Heuristically, we expect $|F_\bullet|$ to be roughly $mkt(1-t)^{k-1}$; we can make this precise as follows. Let $K_a$ be the number of edges between clause $a$ and variables $A$, and let $\bar{K}_a$ be i.i.d.\ $\Bin(k,t)$ random variables. Recall from \eqref{e:random.matching} that, conditional on $\ux$ being an NAE-SAT solution, the law of the edge matching $\mm$ remains uniformly random. Therefore the joint distribution of the $K_a$ under $\P(\mm\,|\,\ux)$ is the same as the joint law of the $\bar{K}_a$ conditioned to sum to $|A|d=ndt=mkt$:
	\beq\label{e:binom.conditioning}
	(K_a)_{a\in F}
	\stackrel{d}{=} 
	\bigg( (\bar{K}_a)_{a\in F}
	\,\bigg|\, \sum_{a\le m} \bar{K}_a= mkt
	\bigg)\,.\eeq
By the local central limit theorem, the $\bar{K}_a$ sum to $mkt$ with probability polynomial in $n$, so
	\begin{align*}
	\mathbf{f}(A)
	&\le n^{O(1)}
	\E \bigg[ 
	\prod_{a\le m}
	\bigg(1-\frac{2}{2^k-2}\bigg)^{\Ind{\bar{K}_a=1}}\bigg]
	=  n^{O(1)}
	\E \bigg[ 
	\bigg(1-\frac{2}{2^k-2}\bigg)^{\Bin(m,kt(1-t)^{k-1})}\bigg] \\
	&=  n^{O(1)}
	\bigg(1 - \frac{2 kt(1-t)^{k-1}}{2^k-2}\bigg)^m
	\le n^{O(1)}
	\exp\bigg\{
	-\frac{m \cdot 2kt(1-t)^{k-1}}{2^k-2}
	\bigg\}
	\,.
	\end{align*}
For $\alpha \ge 2^{k-1}\log 2(1-O(1/k^2))$ and $t=C/2^k$, we have
	\[
	\mathbf{f}(nt)
	\le \sum_{A\subseteq V,|A|=nt} \mathbf{f}(A)
	\le n^{O(1)}
	\exp\bigg\{n\bigg[ H(t)
	-t k\log 2  \bigg(1-\frac{O(1)}{k^2}\bigg)\bigg]\bigg\}\,,
	\]
where $H(t)$ denotes the binary entropy function. Applying the standard bound $H(t)\le t\log (e/t)$ gives
	\[
	\mathbf{f}(nt)
	\le n^{O(1)}
	\exp\bigg\{nt\bigg[ \Big( k\log 2 + 1 - \log C\Big)
	- k\log 2 \bigg(1-\frac{O(1)}{k^2}\bigg)\bigg]\bigg\}
	\le \exp\bigg\{ -\frac{6n}{2^k}\bigg\}\,,
	\]
where the last bound holds by taking $C=7$.
\end{proof}
\end{ppn}

The next result is almost identical to
Lemma~\ref{l:restrict.H}, but we present it here for completeness. 
Let $Z_{\red\ge n t}$  denote the contribution to $Z$ from solutions whose coarsening $\ueta$ has more than $ndt$ red (forcing) edges. We then have the following:

\begin{ppn} \label{prop:red:1}
For $t=7/2^k$ we have
	\[
	\E Z_{\red \ge ndt}
	\le \frac{\E Z}{\exp(\Omega(nk))}
	\]
as long as $\alpha \ge 2^{k-1}\log 2(1-O(1/k^2))$.

\begin{proof} By symmetry, we have $\E Z_{\red \ge ndt}=\E Z \mathbf{g}(ndt)$, where $\mathbf{g}(ndt)$ denotes the probability, conditional on $\ux\equiv\zro$ being a solution, that the coarsening $\ueta$ of $\ux$ has more than $ndt$ (red) forcing edges. Note that the forcing edges for the coarsened configuration $\ueta$ are a subset of the forcing edges for the initial configuration $\ux$, so it suffices to bound the latter.  Conditioned on $\ux\equiv \zro$ being a solution, each clause contains a forcing edge independently with chance $\theta= 2k/(2^k-2)$. Each clause contains at most one forcing edge, so the number of forcing edges is stochastically dominated by a $\Bin(m,\theta)$ random variable. Therefore
	\[
	\mathbf{g}(ndt)
	\le \P\Big(\Bin(m,\theta) \ge ndt = mkt\Big)
	\le \frac{1}{ \exp\{ m H( kt |\theta)\}}
	\le \frac{1}{\exp(\Omega(nk))}\,,
	\]
as claimed.
\end{proof}
\end{ppn}

\subsection{Second moment bound on free variables and forcing clauses}
\label{subsec:apriori:second}
In this subsection we bound the contribution to the second moment from pairs of solutions where the coarsening has too many free variables (Proposition~\ref{p:second.mmt.bound.frees}) or too many forcing clauses (Proposition~\ref{prop:red}).

Recall the calculation \eqref{e:random.matching}, which told us that conditional on $\ux$ being a solution, the law of $\mm$ is still uniform among all matchings. A similar calculation tells us that if we condition on $\ux^1$ \emph{and} $\ux^2$ \emph{both} being solutions, then $\mm$ is no longer uniform among all matchings, but should be weighted proportionally to the number of literal assignments $\uL$ such that $(\ux^1,\ux^2)\in\SOL(\mm,\uL)$. It follows that
	\beq\label{e:cond.matching.second.mmt}
	\P\Big(\mm\,\Big|\,\ux^1,\ux^2 \in\SOL(\mm,\uL)\Big)
	\cong \prod_{a\in F}
	\bigg( \frac{2^k-2}{2^k-4}
	\bigg)^{I_a(\ux^1,\ux^2)}\,,
	\eeq
where $I_a(\ux^1,\ux^2)$ is the indicator that $\ux^1$ and $\ux^2$ either fully agree or fully disagree on $\partial a$, and $\cong$ indicates proportionality up to a normalizing constant. This observation will be used in the proof below. Recall that $Z=|\SOL(G)|$,
and $Z^2[n\rho]$ denotes the contribution
to $Z$ from solution pairs
$(\ux^1,\ux^2)$ that agree on $n\rho$ variables. Let $Z^2[n\rho]_{\ge nt}$ denote the contribution to $Z^2[n\rho]$ from pairs where either of the coarsenings $\ueta^1,\ueta^2$ has more than $nt$ free variables. We then have the following:

\begin{ppn}[second moment version of
Proposition~\ref{p:first.mmt.bound.frees}]
\label{p:second.mmt.bound.frees} For $n$ large enough, we have 
	\[
	\E[Z^2[n\rho]_{\ge n 7/2^k}]
	\le \frac{\E[Z^2[n\rho]]}
	{\exp( n 6 / 2^k)}
	\]
as long as
$\alpha\ge 2^{k-1}\log 2[1-O(1/k^2)]$ and 
$\rho=1/2(1+O(1/k^2))$.

\begin{proof}
Similarly to the proof of Proposition~\ref{p:first.mmt.bound.frees},
fix $\ux^1\equiv\zro$, and let $\ux^2=\zro^{n\rho}\one^{n(1-\rho)}$ --- i.e., $\ux^2$ has its first $n\rho$ coordinates equal to $\zro$, and its last $n(1-\rho)$ coordinates equal to one. Let $\mathbf{f}(nt;n\rho)$ denote the probability, conditional on 
$\ux^1$ and $\ux^2$ being solutions, that the coarsening $\ueta^1$ of $\ux^1$ has more than $nt$ free variables.
By symmetry,
	\beq\label{e:second.mmt.symm.free.bound}
	\frac{\E[Z^2[n\rho]_{\ge n t}]}
		{\E[Z^2[n\rho]]}
	= 
	\mathbf{f}(nt;n\rho)
	\,.\eeq
Let $\ueta^1(nt)\in\{\zro,\one,\free\}^n$ be the configuration that results after $nt$ steps of the coarsening procedure, started from $\ux^1\equiv\zro$. 
For any two subsets $A_{\zro}\subseteq[n\rho]$ and $A_{\one}\subseteq[n]\setminus[n\rho]$ with $|A_{\zro}|=nt_{\zro}$ and $|A_{\one}|=n(t-t_\zro)=nt_{\one}$, we let
	\[\mathbf{f}(A_{\zro},A_{\one};n\rho)
	=\P\bigg(
	\textup{$\ueta^1(nt)$ has free variables 
	$A=A_{\zro}\sqcup A_{\one}$}\,\bigg|\,
	\ux^1,\ux^2\in\SOL(G)
	\bigg)\,.\]
Then note that $\mathbf{f}(nt;n\rho)$ can be bounded by summing $\mathbf{f}(A_{\zro},A_{\one};n\rho)$ over the choices of $A_{\zro}$ and $A_{\one}$.
Let $F_\bullet(A_\zro,\mm)$
denote the set of clauses $a\in F$ that have exactly one edge going to $A_\zro$, and the remaining $k-1$ edges going to $V\setminus A$. For $a\in F_\bullet(A_\zro,\mm)$ there are three main cases, depending on the configuration of the remaining $k-1$ edges:
	\begin{enumerate}[(i)]
	\item The remaining $k-1$ edges all go to $[n\rho]\setminus A_\zro$. In this case $\ux^1$ and $\ux^2$ agree on all of $\partial a$, so we are in the same situation as in the proof of Proposition~\ref{p:first.mmt.bound.frees}: the clause $a$ has two literal assignments (all-$\zro$ and all-$\one$) that are inconsistent with both $\ux^1,\ux^2$ being valid solutions. For $A$ to be free, there are two additional forbidden literal assignments.
	\item The remaining $k-1$ edges all go to $([n]\setminus[n\rho])\setminus A_\one$. In this case, the clause $a$ has two literal assignments that are inconsistent with $\ux^1$ being a valid solution,
and two different literal assignments that are inconsistent with $\ux^2$ being a valid solution. However, the latter two assignments are exactly the ones that would be forcing to the variable in $A_\zro$ in $\ux^1$.
	\item In all remaining cases, the clause $a$ has four literal assignments that are inconsistent with $\ux^1,\ux^2$ being valid solutions. Outside of these four, it has two additional literal assignments that would be forcing to the variable in $A_\zro$ in $\ux^1$.
	\end{enumerate}
Denote $F_\bullet\equiv F_\bullet(A_\zro,\mm)
\sqcup F_\bullet(A_\one,\mm)$;
 this is the set of clauses with exactly one edge going to $A$, and the remaining $k-1$ going to $V\setminus A$.
The only problematic case above is (ii), so we let $F_\zro\subseteq F_\bullet(A_\zro,\mm)$ denote the set of clauses with one edge going to $A_\zro$, and the remaining $k-1$ edges going to $([n]\setminus[n\rho])\setminus A_\one$. Similarly we let $F_\one\subseteq F_\bullet(A_\one,\mm)$ denote the set of clauses with one edge going to $A_\one$, and the remaining $k-1$ edges going to $[n\rho]\setminus A_\zro$. Then, recalling $t_\zro+t_\one=t$, we have
	\begin{align}\nonumber
	&\mathbf{f}(A_{\zro},A_{\one};n\rho)
	\le
	\E\bigg[ \max\bigg\{
	1-\frac{2}{2^k-2},
	1-\frac{2}{2^k-4}
	\bigg\}^{|F_\bullet|-|F_\zro|-|F_\one|}\bigg]
	= 
	\E\bigg[ \bigg(
	1-\frac{2}{2^k-2}
	\bigg)^{|F_\bullet|-|F_\zro|-|F_\one|}\bigg] \\
	&\qquad\le
	\E\bigg[\bigg(
	1-\frac{2}{2^k-2}
	\bigg)^{|F_\bullet|-2 mk^3 t/2^k}
	\bigg]
	+ \P\bigg(|F_\zro| \ge 
		\frac{mk^3 t}{2^k} \bigg)
	+ \P\bigg(|F_\one| \ge 
		\frac{mk^3 t}{2^k}\bigg)
	\,. \label{e:f.Azro.Aone}
	\end{align}
The rest of the proof is a series of straightforward bounds on the quantities appearing above. Let $F_*$ denote the set of clauses $a\in F$ where the number of edges between $a$ and $[n\rho]$ is in the set $\{0,1,k-1,k\}$.

\medskip

\noindent\textbf{Bound on $F_*$.}
Let $Z_a$ be the number of edges between clause $a$ and variables $[n\rho]$. For comparison, let $\theta\in(0,1)$, and let $\bar{Z}_a$ be i.i.d.\ random variables with law
	\[
	\P(\bar{Z}_a=\ell)
	= \frac{
	\binom{k}{\ell} \theta^\ell(1-\theta)^{k-\ell} w^{\Ind{\ell\in\{0,k\}}}
	}{ 1 + [\theta^k + (1-\theta)^k ](w-1)}\,,\quad
	w \equiv \frac{2^k-2}{2^k-4}\,.
	\]
We abbreviate this as $\bar{Z}_a\sim\textup{B}_w(k,\theta)$; this is a slight reweighting of the binomial distribution. It follows from
\eqref{e:cond.matching.second.mmt} that for any fixed $\theta\in(0,1)$, we have
	\[
	(Z_a)_{a\in F}
	\stackrel{d}{=}
	\bigg(
	(\bar{Z}_a)_{a\in F} \,\bigg|\, \sum_{a=1}^m \bar{Z}_a= nd\rho=mk\rho\bigg)\,.
	\]
We can choose $\theta=\rho[1+O(1/4^k)]$ such that
	\[
	\E B_w(k,\theta)
	= \frac{k\theta + (w-1)k\theta^k}{1 + (w-1)[\theta^k + (1-\theta)^k]}
	=k\rho\,,
	\]
in which case it follows from the local central limit theorem that
the event being conditioned on has probability polynomial in $n$. Moreover, for this choice of $\theta$, we have
	\[
	\P\Big(\bar{Z}_a\in\{0,1,k-1,k\}\Big)
	= \frac{k[\theta(1-\theta)^{k-1}
		+ (1-\theta)\theta^{k-1}]
		+[\theta^k+(1-\theta)^k]w}
	{1 + [\theta^k+(1-\theta)^k](w-1)}
	= \frac{k[1+O(1/k)]}{2^{k-1}}\,.
	\]
Write $F(\ell)=\{a\in F: Z_a=\ell\}$, so that $F_*\equiv F(0)\sqcup F(1) \sqcup F(k-1) \sqcup F(k)$. Then it follows from the above discussion that we have the bound
	\beq\label{e:Fstar.bound}
	\P\bigg(|F_*| \ge  \frac{mk^2}{2^k} \bigg)
	\le n^{O(1)}
	\P\bigg( \Big| \Big\{a\in F: \bar{Z}_a\in\{0,1,k-1,k\}\Big\}\Big|
		\ge  \frac{mk^2}{2^k}
		\bigg)
	\le 
	\frac{1}{\exp(\Omega(nk^2\log k))}\,.\eeq
In the bounds below we will use this fact that $F_*$ is small with good probability.
\medskip

\noindent\textbf{Bound on $F_\zro$.}
Recall from above that $F_\zro$ is the set of clauses with exactly one edge going to $A_\zro$, and the remaining $k-1$ edges going to $[n]\setminus[n\rho]$. Recall also that $F(\ell)$ is the set of clauses with exactly $\ell$ edges going to the variables $[n\rho]$. Conditional on $F(1)$, the edges between clauses $F(1)$ and variables $[n\rho]$ form a uniformly random subset (of size $|F(1)|$) of the set of $nd\rho$ edges incident to the variables $[n\rho]$. Let $D_v$ denote the number of edges between variable $v$ and $F(1)$, and let $\bar{D}_v$ be i.i.d.\ $\Bin(d,\psi)$ random variables. Then
	\[
	(D_v)_{v\in[n\rho]}
	\stackrel{d}{=}
	\bigg((\bar{D}_v)_{v\in[n\rho]}
	\,\bigg|\,\sum_{v\in [n\rho]} \bar{D}_v
	=|F(1)|
	\bigg)\,.
	\]
For the event being conditioned on to have probability polynomial in $n$, we choose
	\[
	\psi
	= \frac{|F(1)|}{nd\rho} \stackrel{*}{=} O\bigg( \frac{k}{2^k}\bigg)\,,
	\]
where the bound marked $*$ holds if $|F_*|$ satisfies the bound in \eqref{e:Fstar.bound}. It follows that
	\begin{align*}
	&\P\bigg(|F_\zro| 
		\ge \frac{mk^3 t}{2^k}
		\bigg)
	\le\P\bigg(|F_*| \ge 
	\frac{mk^2}{2^k} \bigg)
	+n^{O(1)} \P\bigg(
	\sum_{v\in A_\zro} \bar{D}_v
		\ge \frac{mk^3 t}{2^k}
	\bigg)\\
	&\qquad\le 
	\frac{1}{\exp(\Omega(nk^2\log k))}
	+\frac{1}{\exp(\Omega(nt k^3\log k))}
	\le
	\frac{1}{\exp(\Omega(nk^3(\log k)/2^k))}\,.\end{align*}
The same bound holds for $F_\one$. \medskip

\noindent\textbf{Bound on $F_\bullet$.} Recall that $F_\bullet$ is the set of clauses with exactly one edge going to $A$, and the remaining $k-1$ edges going to $V\setminus A$. Let $L_a$ denote the number of edges between clause $a$ and variables $A_\zro$, and let $R_a$ denote the number of edges between clause $a$ and variables $A_\one$. Consider the $mk$ clause-adjacent half-edges $\partial F$, and suppose we condition on the bipartition $\partial F = (\partial F)_\zro \sqcup (\partial F)_\one$ where $(\partial F)_\zro$ is the set of half-edges going to $[n\rho]$. It follows from \eqref{e:cond.matching.second.mmt} that, \emph{conditional} on $(\partial F)_\zro$ and $(\partial F)_\one$, we have a uniformly random matching between $(\partial F)_\zro$ and the half-edges incident to $[n\rho]$, and an independent uniformly random matching between $(\partial F)_\one$ and the half-edges incident to $[n]\setminus[n\rho]$. Note that the bipartition also determines the sets $F(\ell)$ for all $0\le\ell\le k$; and they must satisfy the constraint
	\beq\label{e:F.ell.constraint}
	\sum_\ell \ell |F(\ell)|
	= nd\rho\,.
	\eeq
We now bound $F_\bullet$ as follows. For each $a\in F(\ell)$, define random variables
	\[\bar{L}_a\sim\Bin\bigg(
		\ell,\frac{t_\zro}{\rho}\bigg)
		\,,\quad
	\bar{R}_a\sim\Bin\bigg(
		k-\ell,\frac{t_\one}{1-\rho}\bigg)\,,
	\]
where all the $\bar{L}_a$, $\bar{R}_a$ are independent. Then
	\[
	(L_a,R_a)_{a\in F}
	\stackrel{d}{=} \bigg(
	(\bar{L}_a,\bar{R}_a)_{a\in F}
	\,\bigg|\,
	\sum_{a\in F}
	\begin{pmatrix}\bar{L}_a \\
		\bar{R}_a\end{pmatrix}
	= \begin{pmatrix} ndt_\zro \\
	nd t_\one
	\end{pmatrix}\bigg)\,.
	\]
It follows using \eqref{e:F.ell.constraint}
and the local central limit theorem
 that the event being conditioned on has probability polynomial in $n$. Moreover, we note that for $a\in F(\ell)$, we have
	\begin{align*}
	p_a&\equiv \P( \bar{L}_a + \bar{R}_a=1)
	= \ell \frac{t_\zro}{\rho}
	\bigg(1-\frac{t_\zro}{\rho}\bigg)^{\ell-1}
	\bigg(1-\frac{t_\one}{1-\rho}\bigg)^{k-\ell}
	+ (k-\ell)\frac{t_\one}{1-\rho}
	\bigg(1-\frac{t_\one}{1-\rho}\bigg)^{k-\ell-1}
	\bigg(1-\frac{t_\zro}{\rho}\bigg)^{\ell} \\
	&=\bigg(\ell \frac{t_\zro}{\rho} 
		+ (k-\ell) \frac{t_\one}{1-\rho}\bigg)
		\bigg(1 + \frac{O(k)}{2^k}\bigg)\,.
	\end{align*}
Using \eqref{e:F.ell.constraint} again, we see that the $p_a$ sum to $ndt[1+O(k/2^k)]$. It follows that
	\begin{align*}
	&\E\bigg[ \bigg(
	1-\frac{2}{2^k-2}\bigg)^{|F_\bullet|}\bigg]
	\le n^{O(1)}
	\prod_{a\in F}
	\E\bigg[ \bigg(
	1-\frac{2}{2^k-2}\bigg)^{\Ind{\bar{L}_a
		+\bar{R}_a=1}}\bigg]
	\le n^{O(1)}
	\exp\bigg\{ -\sum_{a\in F} 
	\frac{2 p_a}{2^k-2}
		\bigg\} \\
	&\qquad\le 
	\exp\bigg\{ -\frac{ndt}{2^{k-1}}
		\bigg( 1+\frac{O(k)}{2^k}\bigg)
		\bigg\}
	= 
	\exp\bigg\{ -nkt\log 2
	\bigg( 1-\frac{O(1)}{k^2} \bigg)
		\bigg\}\,,
	\end{align*}
having used the assumption that
$\alpha\ge 2^{k-1}\log 2 (1-O(1/k^2))$.\medskip

\noindent\textbf{Entropy bound and conclusion.} Substituting the estimates obtained so far into \eqref{e:f.Azro.Aone} gives	
	\[\mathbf{f}(A_{\zro},A_{\one};n\rho)
	\le
	\exp\bigg\{ -nkt\log 2
	\bigg( 1-\frac{O(1)}{k^2} \bigg)
		\bigg\}\,.
	\]
Let $t_\zro=C_\zro/2^k$ and $t_\one=C_\one/2^k$ with
 $t_\zro+t_\one=t=C/2^k$. Then, recalling the standard bound $H(t)\le t\log (e/t)$,
the number of choices for $A_\zro,A_\one$ is
	\begin{align*}
	&\binom{n\rho}{nt_\zro}
	\binom{n(1-\rho)}{nt_\one}
	\le\exp\bigg\{
	n t_\zro \log \frac{e\rho}{t_\zro}
	+n t_\one\log \frac{e(1-\rho)}{t_\one}
	\bigg\} \\
	&\qquad=\exp\bigg\{ nt\Big(
		k\log2 + \log(e/2) 
		-\frac{C_\zro}{C} \log C_\zro
		-\frac{C_\one}{C} \log C_\one
		\Big)\bigg\}
	\le \exp\bigg\{
	 n t \Big( k\log 2+1-\log C\Big)
	 \bigg\}\,,
	\end{align*}
where the last bound follows by optimizing over $C_\zro+C_\one=C$. Combining these estimates bounds $\mathbf{f}(nt;n\rho)$, and the claim follows from \eqref{e:second.mmt.symm.free.bound}, taking $C=7$.
\end{proof}
\end{ppn}

Recall that $Z^2[n\rho]$ denotes the contribution to $Z^2$ from pairs of assignments $(\ux^1,\ux^2)$ having overlap $n\rho$. Let $Z^2[n\rho]_{\red\ge n t}$  denote the contribution to $Z^2[n\rho]$ from 
pairs $(\ux^1,\ux^2)$  where the coarsening $\ueta^1$ of $\ux^1$ has more than $ndt$ red (forcing) edges.

\begin{ppn}[second moment version of Proposition \ref{prop:red:1}] \label{prop:red}
For $t=7/2^k$ we have
	\[\frac{\E[Z^2[n\rho]_{\red\ge ndt}]}
		{\E[Z^2[n\rho]]}
	\le \frac{1}{\exp(\Omega(nk))}\]
for all $\rho\in[0,1]$, as long as $\alpha \ge 2^{k-1}\log 2(1-O(1/k^2))$.

\begin{proof} This is similar to, but simpler than, the proof of Proposition~\ref{p:second.mmt.bound.frees}. We again fix
$\ux^1\equiv\zro$ and $\ux^2=\zro^{n\rho}\one^{n(1-\rho)}$. Let $\mathbf{g}(ndt;n\rho)$ denote the probability, conditioned on $\ux^1$ and $\ux^2$ being solutions, that the coarsening $\ueta^1$ of $\ux^1$ has more than $ndt$ red (forcing) edges. By symmetry (cf.\ \eqref{e:second.mmt.symm.free.bound}),
	\[
	\frac{\E[Z^2[n\rho]_{\red\ge ndt}]}
		{\E[Z^2[n\rho]]}
	= \mathbf{g}(ndt;n\rho)\,.
	\]
Note that an edge that is forcing for $\ueta^1$ must be forcing for the initial configuration $\ux^1$, so it suffices to bound the probability that there are more than $ndt$ forcing edges for $\ux^1$. Recall from the proof of Proposition~\ref{p:second.mmt.bound.frees} that $F(\ell)$ denotes the set of clauses with exactly $\ell$ edges going to the variables $[n\rho]$, and suppose we condition on $F(\ell)$ for all $0\le \ell\le k$. Note that:
\begin{itemize}
\item A clause in $F(0)\sqcup F(k)$ has two literal assignments (the all-$\zro$ and all-$\one$ assignments) that will be violated by both $\ux^1$ and $\ux^2$. Outside of these, it has $2k$ literal assignments that are forcing to both $\ux^1$ and $\ux^2$. It follows that the clause is forcing to $\ux^1$ with probability $2k/(2^k-2)$.
\item A clause in $F(1)\sqcup F(k-1)$ has four literal assignments that will be violated by either $\ux^1$ or $\ux^2$. Outside of these, it has $2(k-1)$ literal assignments that are forcing to $\ux^1$. (There is one variable incident to the clause that cannot be forced.) It follows that the clause is forcing to $\ux^1$ with probability $2(k-1)/(2^k-4)$.

\item Let $F_*=F(0)\sqcup F(1)\sqcup F(k-1)\sqcup F(k)$ as before. A clause in $F\setminus F_*$ has four literal assignments that will be violated by either $\ux^1$ or $\ux^2$. Outside of these, it has $2k$ literal assignments that are forcing to $\ux^1$. It follows that the clause is forcing to $\ux^1$ with probability $2k/(2^k-4)$.
\end{itemize}
It follows that the number of forcing clauses is stochastically dominated by a $\Bin(m,\varphi)$ random variable with $\varphi=2k/(2^k-4)$, and this holds regardless of $\rho$. Therefore
	\[
	\mathbf{g}(ndt;n\rho)
	\le \P\Big(\Bin(m,\varphi) \ge ndt = mkt\Big)
	\le \frac{1}{ \exp\{ m H( kt |\varphi)\}}
	\le \frac{1}{\exp(\Omega(nk))}\,,
	\]
as claimed.
\end{proof}
\end{ppn}

As a corollary of Propositions \ref{p:second.mmt.bound.frees} and \ref{prop:red}, we have the following: 

\begin{cor}[{second moment version of
\cite[Prop.~2.2]{MR3440193}}]
\label{cor:sec} Assume $\alpha$ lies in the interval $[\albd,\aubd]\equiv [(2^{k-1}-2)\log 2, 2^{k-1} \log 2]$.
For $t=7/2^k$, it holds for $n$ large enough that 
	\[
    \sum_{\rho\in I}
    \bigg(\E Z^2[n\rho]_{\geq nt}+\E Z^2[n\rho]_{\red\geq ndt}\bigg)
    \le \exp\bigg(-\frac{n}{2^k}\bigg)
    \,,\]
where $I\equiv [1/2-1/k^2,1/2+1/k^2]$.

\begin{proof}
Propositions \ref{p:second.mmt.bound.frees} and \ref{prop:red} imply
	\[
        \sum_{\rho \in I}\bigg(\E Z^2[n\rho]_{\geq nt}+\E Z^2[n\rho]_{\red\geq ndt}\bigg)
    \le \frac{2}{\exp(6n/2^k)}
    \sum_{\rho \in I}\E Z^2[n\rho]\,.
	\]
For $\alpha \leq 2^{k-1} \log 2$, it follows straightforwardly from the proof of \cite[Prop.~1.1]{MR3440193} that the contribution to $\E(Z^2)$ from overlaps near $1/2$ is in fact bounded by a constant times the first moment squared:
	\beq\label{eq:bound:secmo}
    \sum_{\rho\in I}
	\E Z^2[n\rho]
	\le C_k (\E Z)^2
	= C_k \exp\bigg\{ 2n 
		\bigg[
		\log 2 + \alpha\log\bigg(1-\frac{2}{2^k}\bigg) \bigg]
		\bigg\}
	\equiv C_k \exp(2n\Phi)
	\,.
     \eeq
Note that if $\alpha=(2^{k-1}-c)\log 2$, then (cf.\ \cite[eq.~(9)]{MR3440193})
	\beq\label{e:Phi}
	\Phi = \frac{(2c-1)\log 2}{2^k}
	+\frac{O(1)}{4^k}\,.
	\eeq
The claimed bound follows, using that $c\le2$ and $6\log2 < 5$.
\end{proof}
\end{cor}

\section{Corrections to second moment analysis}
\label{sec:correct}

In this section, we apply the \textit{a~priori} estimates from
Section~\ref{sec:apriori} to obtain suitable modifications of two propositions from the main text.
In \S\ref{subsec:apriori:first} we prove
Proposition~\ref{corrected:prop:1}, which is a corrected version of 
Proposition~\ref{p:first.mmt}.
In \S\ref{subsec:correct:second}
we prove
Proposition~\ref{corrected:prop},
which is a corrected version of
Proposition~\ref{p:second.mmt}.
We assume throughout that
$\alpha \in [\alpha_{\textsf{lbd}},\alpha_{\textsf{ubd}}]\equiv [(2^{k-1}-2)\log 2, 2^{k-1}\log 2]$; see 
Remark~\ref{r:rmks} which explains that this restriction on $\alpha$ is justified by direct first and second moment calculations on $Z$. 

Recall that most of this work was concerned with bounding the second moment of the tilted cluster partition function (Definition~\ref{d:T.colorings})
	\beq\label{e:Z.lambda.T}
	\bZ\equiv\bZ_{\lambda,T}
	\equiv \sum_{\usi} \bm{w}(\usi)^\lambda
	\eeq
where the sum goes over valid $T$-colorings, and 
$\bm{w}(\usi)$ is the size of the cluster corresponding to $\usi$. Recall also from
Definition~\ref{d:simplex}  that $\bDelta$ denotes the space of empirical measures $H=(\dot{H},\hat{H},\bar{H})$ in the single-copy model. In \eqref{e:N.zero} we defined
	\begin{align*}
	\bN_\circ
	&\equiv \bigg\{ H\in\bDelta : \max\Big\{ 
		\bar{H}(\red),\bar{H}(\free)\Big\} \le \frac{7}{2^k}\bigg\}\,,\\
	\bN
	&\equiv \bigg\{H\in \bN_\circ
	: \|H-H_\star\| \le \frac{1}{n^{1/3}}\bigg\}\,,
	\end{align*}
where $\|\cdot\|$ denotes the $\ell_1$ norm throughout, and $H_\star$ is the optimal empirical measure calculated from the \textit{belief propagation} fixed point of Definition~\ref{d:Hstar}.
As discussed in \S\ref{ss:outline.second.mmt},
the contribution from $H\in \bDelta$ to the first moment $\bZ\equiv \bZ_{\lambda,T}$ is (cf.\ \eqref{e:SIGMA})
\begin{equation}\label{eq:Fone}
    \E \bZ(H)
	=\frac{ \exp\{n\bF_{\lambda,T}(H)\}}{ n^{O(1)}}\,.
\end{equation}
We use $\bDelta_2$ denote the set of empirical measures in the pair model. Similarly to \eqref{eq:Fone}, the contribution from $H\in\bDelta_2$ to the second moment of $\bZ\equiv \bZ_{\lambda,T}$ is (cf.\ \eqref{e:def.PSI.two})
	\beq\label{e:Ftwo.lambda.T}
	\E \bZ^2(H)
	=\frac{ \exp\{n\bF_{2,\lambda,T}(H)\}}{ n^{O(1)}}\,.
	\eeq
We recall also that $H_\bullet\in\bDelta_2$ denotes the product measure
with single-copy marginals $H_\star$
(see the discussion leading to Proposition~\ref{p:second.mmt}).

In the pair model, recall from 
Definition~\ref{d:sep} that for two valid colorings $\usi$ and $\usi'$, we use
$\sep(\usi,\usi')$ to denote the fraction of variables where the corresponding frozen configurations differ. For $H\in\bDelta_2$, we let $\sep(H)$ denote the common value of $\sep(\usi,\usi')$ over all pairs $(\usi,\usi')$ with empirical measure $H$. Now recall
from the discussion around 
Proposition~\ref{p:sep}
that we defined
 	\begin{align}\nonumber
	I_{\se}
	&\equiv \bigg[\frac{1-k^4/2^{k/2}}{2},
		\frac{1+k^4/2^{k/2}}{2}\bigg]\,,\\
	\bN_{\se}
	&\equiv \bigg\{ H\in \bDelta_2
	:\sep(H)\in I_{\se}, H^1\in\bN, H^2\in\bN
	\bigg\}\,.\label{e:old.defn.N.se}
	\end{align}
In fact, in the definition \eqref{e:old.defn.N.se} of $\bN_{\se}$, the constraints $H^i\in\bN$ are extraneous, since the resampling procedure may not respect these constraints. We therefore replace \eqref{e:old.defn.N.se} with the corrected definition
	\beq\label{e:new.defn.N.se}
	\bN_{\se}
	\equiv \bigg\{ H\in \bDelta_2
	: \sep(H)\in I_{\se}, 
	H^1 \in \bN_\circ, H^2 \in \bN_\circ
	\bigg\}\,.\eeq
We then note that Proposition~\ref{p:minimizer} still holds
(by the same proof) if we replace the old definition \eqref{e:old.defn.N.se} of $\bN_{\se}$ by the new definition \eqref{e:new.defn.N.se}:

\begin{ppn}[{correction to Proposition~\ref{p:minimizer}}]
\label{p:correction.to.resampling}
For $0\le\lambda\le1$ and $1\le T<\infty$, 
\begin{enumerate}[a.]
\item On $\{H\in\bN_\circ : H=H^\sy\}$, the function $\bXi$ is uniquely minimized at $H=H_\star$ with $\bXi(H_\star)=0$.
\item On $\{H\in\bN_{\se} : H=H^\sy\}$ \textbf{\emph{with the corrected definition of $\bN_{\se}$ from \eqref{e:new.defn.N.se}}}, the function $\bXi_2$ is uniquely minimized at $H=H_\bullet$ with $\bXi_2(H_\bullet)=0$.
\end{enumerate}
Moreover there is a positive constant $\epsilon=\epsilon(d,k,T,\lambda)$ such that
\begin{enumerate}[1.]
\item $\bXi(H) \ge \epsilon\|H-H_\star\|^2$ for all $H\in\bDelta$ with $H=H^\sy$ and $\|H-H_\star\|\le\epsilon$, and
\item $\bXi_2(H) \ge \epsilon\|H-H_\bullet\|^2$ for all $H\in\bDelta_2$ with $H=H^\sy$ and $\|H-H_\bullet\|\le\epsilon$.
\end{enumerate}
\end{ppn}

Moreover, let us recall the main resampling result of this work, Theorem~\ref{t:mc}. The main results of this section are the following:
\begin{ppn}[{correction to 
Proposition~\ref{p:first.mmt}}]
\label{corrected:prop:1}
Fix $k\geq k_0$ and $\lambda \in [0,1]$. There exists a constant $T_0= T_0(k,\lambda)$
such that for $T\geq T_0$, the unique maximizer of $\bF\equiv \bF_{\lambda,T}$ in $\bN_{\circ}$ is $H_\star\equiv H_{\star,\lambda, T}$, which lies in the interior of $\bN_{\circ}$. Moreover, for $T\geq T_0$, there is a positive constant $\epsilon=\epsilon(k,\lambda,T)$ such that
	\[ \bF(H)
	\le \bF(H_{\star})-\epsilon \norm{H-H_{\star}}_2^2\]
for all $H\in \bN_{\circ}$ with
$\norm{H-H_{\bullet}}\leq \epsilon$.
\end{ppn}

\begin{ppn}[{correction to
Proposition~\ref{p:second.mmt}}]
\label{corrected:prop}
Fix $k\geq k_0$ and $\lambda \in [0,1]$. There exists a constant $T_0= T_0(k,\lambda)$
such that for $T\geq T_0$, the unique maximizer of $\bF_{2,\lambda,T}$ in $\bN_{\se}$ is $H_\bullet=H_{\bullet,\lambda, T}$, which lies in the interior of $\bN_{\se}$. Moreover, for $T\geq T_0$, there is a positive constant $\epsilon=\epsilon(k,\lambda,T)$ such that
	\[ \bF_2(H)
	\le \bF_2(H_{\bullet})-\epsilon \norm{H-H_{\bullet}}_2^2\]
for all $H\in\bN_{\se}$ with
$\norm{H-H_{\bullet}}\leq \epsilon$.
\end{ppn}

\subsection{Boundary for first moment analysis}
\label{subsec:correct:first}

In this subsection, we prove Proposition~\ref{corrected:prop:1}.  We first use Propositions \ref{p:first.mmt.bound.frees} and \ref{prop:red:1} to prove Lemma~\ref{lem:boundary:first}, which says that the first moment of
$\bZ_{\lambda,T}$ has a negligible contribution from configurations where the fraction of free variables or red edges is more than $7/2^k$.

\begin{lem}\label{lem:boundary:first}
    For $k\ge k_0$ and $\alpha \in [(2^{k-1}-2)\log 2, 2^{k-1}\log 2]$, the following holds.
There exists a constant $T_0=T_0(k,\lambda)$ such that
for all $T\geq T_0$, we have
	\[
    \sup \bigg\{ \bF_{\lambda,T}(H)
    :H\in \bDelta \setminus \bN_{\circ} \bigg\}
    \leq 
    \bF_{\lambda,T}(H_{\star,\lambda,T})
    -\frac{2}{2^k}\]
for all $0\le\lambda\le1$.
\end{lem}
\begin{proof} 
    For $\lambda\in[0,1]$,
it follows from
\eqref{e:Z.lambda.T} and \eqref{eq:Fone} that
\begin{align*}
	&\frac{1}{n^{O(1)}}
	\exp \bigg\{ \sup\Big\{
	\bF_{\lambda,T}(H) : H\notin \bN_{\circ}\}\bigg\}
	\stackrel{\eqref{e:Ftwo.lambda.T}}{=}
	\sum_{H\notin \bN_{\circ}} 
	\E\bZ_{\lambda,T}(H)
	\stackrel{\eqref{e:Z.lambda.T}}{\le}
	\sum_{H\notin \bN_{\circ}} 
	\E\bZ_{1,T}(H)\\
	&\quad\quad\le
	\E Z_{\ge 7n/2^k}
	+\E Z_{\red \ge 7nd/2^k}
	\stackrel{\textup{(a)}}\le
	\frac{\E Z}{\exp(n6/2^k)}
	= \frac{\exp\{ n (\log 2 + \alpha\log(1-2/2^k) )\}}{\exp(n6/2^k)}
	\equiv 
	\frac{\exp(n\Phi)}{\exp(n6/2^k)}
	\,,
	\end{align*}
    where (a) holds by Propositions \ref{p:first.mmt.bound.frees} and \ref{prop:red:1}, and $\Phi$ is as in \eqref{e:Phi}. Thus, it follows that
    \begin{equation}\label{eq:boundary:bound:firstmo}
    \sup\bigg\{
	\bF_{\lambda,T}(H) : H\notin \bN_{\circ}
	\bigg\}
	\le \Phi - \frac{6}{2^k}\,.
    \end{equation}    
It now remains only to compare the right-hand side of \eqref{eq:boundary:bound:firstmo} to $\bF_{\lambda,T}(H_{\star,\lambda,T})$.
Let $\alpha = (2^{k-1}-c)\log 2$. By 
Propositions \ref{p:drec_equiv.outline} and \ref{p:Zcal.estimate}, we have
	\begin{equation}\label{eq:F:T:infty}
       \lim_{T\to\infty} \bF_{\lambda,T}(H_{\star,\lambda,T})
        = \mathfrak{F}(\lambda)
	= \frac{(2c-1)\log 2 - (1-2^{\lambda-1})}{2^k} +
		\frac{k^{O(1)}}{4^k}
        \,.
    \end{equation}
For comparison, recalling \eqref{e:Phi}, we have
	\begin{equation}\label{eq:comparison:Phi}
    \Phi= \log 2 +\alpha \log (1-2/2^k)
	=
	\frac{(2c-1)\log 2}{2^k} +\frac{O(1)}{4^k}
	\le
	\mathfrak{F}(\lambda)
	+\frac{1}{2^k}
	\,.
    \end{equation}
Combining with \eqref{eq:boundary:bound:firstmo} gives 
the claim.
\end{proof}

\newcommand{\Hmax}{H_{\textup{m}}}
\begin{proof}[Proof of Proposition \ref{corrected:prop:1}]
Let $\Hmax \in \argmax \{\bF_{\lambda,T}(H): H\in \bN_{\circ}\}$.
Then, taking $\epsilon>0$ small enough in Theorem~\ref{t:mc}, we have
	\begin{align*}
	\bF(\Hmax) +\epsilon \bXi_2((\Hmax)^\sy)
	\le\max\Big\{ \bF(H') : H'\in \bDelta \Big\}= \bF(\Hmax)
	\,,
	\end{align*}
where the last step crucially uses Lemma~\ref{lem:boundary:first}, which guarantees that $\bF$ cannot achieve a higher value on $\bDelta \setminus \bN_{\circ}$ than on $\bN_{\circ}$. This implies $\bXi((\Hmax)^\sy)=0$. Next note that, by concavity of the entropy function, we have
	\beq\label{eq:H:sy:one}
	\bF(H)
	\le \bF(H^{\sy})-\frac{\norm{H-H^{\sy}}_2^2}{C_k}\,.
	\eeq
for all $H\in\bDelta$. For any $H\in \bN_{\circ}$, we must have $H^{\sy}\in \bN_{\circ}$ also. It then follows using \eqref{eq:H:sy:one} that
$\Hmax=(\Hmax)^{\sy}$, and consequently $\bXi(\Hmax)=0$.
By Proposition~\ref{p:correction.to.resampling} this implies $\Hmax=H_{\star}$, which proves the first assertion of this proposition. For $H\in\bDelta$ with $\|H-H_\star\|\le \epsilon$ small enough, it follows using Theorem~\ref{t:mc} that
	\[
	\bF(H)
	\stackrel{\eqref{eq:H:sy:one}}{\le} \bF(H^{\sy})
		-\frac{\norm{H-H^{\sy}}_2^2}{C_k}
	\le \bF(H_\star)-\epsilon\bXi(H^\sy)
		-\frac{\norm{H-H^{\sy}}_2^2}{C_k}\,.
	\]
Combining with Proposition~\ref{p:correction.to.resampling} gives
	\[
	\bF(H)
	\le
	\bF(H_\star)-\epsilon^2\|H^{\sy}-H_\star\|^2
		-\frac{\norm{H-H^{\sy}}_2^2}{C_k}\,,
	\]
which implies the second assertion of this proposition since $2(\|H-H^{\sy}\|^2+\|H^{\sy}-H_{\star}\|^2)\geq \|H-H_{\star}\|^2$ holds by Cauchy-Schwartz and a triangular inequality.
\end{proof}

\subsection{Boundary for the second moment analysis}
\label{subsec:correct:second}
In this subsection we supply a missing estimate, Lemma~\ref{lem:bdry} below, which says that the second moment of $\bZ_{\lambda,T}$ has a negligible contribution from a region just outside the ``boundary region'' considered in the second moment analysis in the main text of this work. The proof of Lemma~\ref{lem:bdry} is based on Propositions \ref{p:second.mmt.bound.frees} and \ref{prop:red}. The proof of Proposition~\ref{corrected:prop} is provided at the end of this section. 

We define a ``boundary region'' just outside $\bN_{\se}$ as follows. Define
	\[
	I_{\bdy}
	\equiv\bigg[
	\frac{1-2k^4 2^{-k/2}}{2},\frac{1-k^4 2^{-k/2}}{2} \bigg]
	\cup 
	\bigg[\frac{1+k^4 2^{-k/2}}{2},\frac{1+2k^4 2^{-k/2}}{2}\bigg]
	\]
Define the following three regions: 
	\begin{align*}
	\bN_{\bdy,\sep}
	&= \bigg\{H\in \bDelta_2: 
	\sep(H) \in I_{\bdy},
	H^1\in\bN_\circ, H^2\in\bN_\circ\bigg\}\,,\\
	\bN_{\bdy,\ff/\red}
	&= \bigg\{H\in \bDelta_2: 
	\sep(H) \in I_{\se}
	\cup I_{\bdy},
	\max\Big\{\bar{H}^1(\free),\bar{H}^2(\free),
	\bar{H}^1(\red),\bar{H}^2(\red)\Big\}
	\in\bigg[ \frac{7}{2^k}, \frac{10}{2^k}\bigg]
	\bigg\}\,.
	\end{align*}
Let $\bN_{\bdy}$ denote the union of
$\bN_{\bdy,\sep}$ and $\bN_{\bdy,\ff/\red}$.
Note that $\bN_{\se}\cup \bN_{\bdy}$ contains an $\epsilon$-neighborhood of $\bN_{\se}$, for $\epsilon$ small enough. The main result of this subsection is the following:

\begin{lem}\label{lem:bdry}
For $k\ge k_0$ and $\alpha \in [(2^{k-1}-2)\log 2, 2^{k-1}\log 2]$, the following holds. There exists a constant $T_0=T_0(k,\lambda)$ such that
for all $T\geq T_0$, we have
	\[
    \sup \bigg\{ \bF_{2,\lambda,T}(H)
    :H\in \bN_{\bdy} \bigg\}
    \leq 
    \bF_{2,\lambda,T}(H_{\bullet,\lambda,T})
    -\frac{2}{2^k}\]
for all $0\le\lambda\le1$.

\begin{proof}
We argue separately for the two regions of $\bN_{\bdy}$ defined above:
\begin{itemize}
\item We start with $\bN_{\bdy,\sep}$. For $\lambda\in[0,1]$,
it follows from
\eqref{e:Z.lambda.T} and \eqref{e:Ftwo.lambda.T} that
	\begin{align}\nonumber
	&\frac{1}{n^{O(1)}}
	\exp \bigg\{ \sup\Big\{
	\bF_{2,\lambda,T}(H) : H\in\bN_{\bdy,\sep}
	\Big\}\bigg\}
	\stackrel{\eqref{e:Ftwo.lambda.T}}{=}
	\sum_{H\in\bN_{\bdy,\sep}} 
	\E(\bZ_{\lambda,T})^2(H) \\
	&\qquad\stackrel{\eqref{e:Z.lambda.T}}{\le}
	\sum_{H\in\bN_{\bdy,\sep}} 
	\E(\bZ_{1,T})^2(H)
	\stackrel{\textup{(a)}}{\le}
	\sum_{\rho\in J_{\bdy}}
	\E Z^2[n\rho]
	\stackrel{\textup{(b)}}{\le}
	\frac{1}{\exp(\Omega(nk/2^k))}
	\,,\label{e:analysis.of.sep.bdy}
	\end{align}
where $J_{\bdy}$ is a slight expansion of $I_{\bdy}$, defined by
	\[
	J_{\bdy}
	\equiv\bigg[
	\frac{1-3k^4 2^{-k/2}}{2},\frac{1-k^3 2^{-k/2}}{2} \bigg]
	\cup 
	\bigg[\frac{1+k^3 2^{-k/2}}{2},
		\frac{1+3k^4 2^{-k/2}}{2}\bigg]\,.
	\]
The step marked (a) in \eqref{e:analysis.of.sep.bdy} is justified as follows: by definition, $\bN_{\bdy,\sep}$ refers to pairs $(\usi^1,\usi^2)$ of valid $T$-colorings whose corresponding frozen configurations agree on $ns$ variables with $s\in I_{\bdy}$, and have no more than $n7/2^k$ free variables each. This means that for any pair $(\ux^1,\ux^2)$, where
$\ux^i\in\{\zro,\one\}^n$ lies in the cluster defined by $\usi^i$,
the fraction $\rho$ of variables where the $\ux^i$ agree must be in $[s-O(1/2^k),s]$, which lies in $J_{\bdy}$ for all $s\in I_{\bdy}$.
The step marked (b) in \eqref{e:analysis.of.sep.bdy} follows from Corollary~\ref{c:int.small}.
\item
We next deal with the region $\bN_{\bdy,\ff/\rr}$. For $\lambda\in[0,1]$, similarly to \eqref{e:analysis.of.sep.bdy} we have
	\begin{align*}
	&\frac{1}{n^{O(1)}}
	\exp \bigg\{ \sup\Big\{
	\bF_{2,\lambda,T}(H) : H\in\bN_{\bdy,\ff/\rr}
	\Big\}\bigg\}
	\stackrel{\eqref{e:Ftwo.lambda.T}}{=}
	\sum_{H\in\bN_{\bdy,\ff/\rr}} 
	\E(\bZ_{\lambda,T})^2(H) \\
	&\qquad
	\stackrel{\eqref{e:Z.lambda.T}}{\le}
	\sum_{H\in\bN_{\bdy,\sep}} 
	\E(\bZ_{1,T})^2(H)
	\le
	\sum_{\rho\in J}\Big(
	\E Z^2[n\rho]_{\ge 7n/2^k}
	+\E Z^2[n\rho]_{\red \ge 7nd/2^k}\Big)
	\le
	\sum_{\rho\in J}
	\frac{2\E Z^2[n\rho]}{\exp(n6/2^k)}
	\,,
	\end{align*}
where $J\equiv I_{\se}\cup J_{\bdy}$, and the last bound follows from Propositions~\ref{p:second.mmt.bound.frees} and \ref{prop:red}. Now, recalling \eqref{eq:bound:secmo}, we have 
	\[\sum_{\rho\in J}
	\E Z^2[n\rho]
	\le C_k (\E Z)^2
	= C_k \exp\bigg\{ 2n \Big(\log 2 + \alpha\log(1-2/2^k) \Big)
		\bigg\}
	\equiv C_k \exp(2n\Phi)
	\,.\]
\end{itemize}
Combining the above bounds gives
	\beq\label{e:F.bound.bdy}
	\sup\bigg\{
	\bF_{2,\lambda,T}(H) : H\in 
	\bN_{\bdy}
	= \bN_{\bdy,\sep} \cup \bN_{\bdy,\free/\rr}
	\bigg\}
	\le 2\Phi - \frac{5}{2^k}\,.
	\eeq
(The bound for $\bN_{\bdy,\free/\rr}$ follows
from the estimates directly above, while the bound for $\bN_{\bdy,\sep}$ follows from our earlier bound \eqref{e:analysis.of.sep.bdy} together with the observation that $\Phi=O(1/2^k)$ in the regime being considered.)
We finally compare $2\Phi$ with $\bF_{2,\lambda,T}(H_{\star,\lambda,T})$. By Lemma~\ref{l:clause.tensorize}, we have $\bF_{2,\lambda,T}(H_{\bullet,\lambda,T})=2\bF_{\lambda,T}(H_{\star,\lambda,T})$, so by \eqref{eq:comparison:Phi}, we have
\begin{equation*}
     2\Phi-\frac{2}{2^k}\leq \lim_{T\to\infty} \bF_{2,\lambda,T}(H_{\bullet,\lambda,T})=\lim_{T\to\infty} 2\bF_{\lambda,T}(H_{\star,\lambda,T})
\end{equation*}
Combining with \eqref{e:F.bound.bdy} gives the claim. 
\end{proof}
\end{lem}

\begin{proof}[Proof of Proposition \ref{corrected:prop}] 
Having Lemma~\ref{lem:bdry} in hand, the proof is identical to that of Proposition~\ref{corrected:prop:1}, but we give it here for completeness.  Let $\Hmax \in \argmax \{\bF_{2,\lambda,T}(H): H\in \bN_{\se}\}$.
Then, taking $\epsilon>0$ small enough in Theorem~\ref{t:mc}, we have
	\begin{align*}
	&\bF_2(\Hmax) +\epsilon \bXi_2((\Hmax)^\sy)
	\le\max\Big\{ \bF_2(H') : \|H'-\Hmax\| \le \epsilon(dk)^{2T}\Big\} \\
	&\qquad\le
	\max\Big\{ \bF_2(H') : 
		H\in \bN_{\se} \cup \bN_{\bdy}\Big\}
	= \bF_2(\Hmax)
	\,,
	\end{align*}
where the last step crucially uses Lemma~\ref{lem:bdry}, which guarantees that $\bF_2$ cannot achieve a higher value on $\bN_{\bdy}$ than on $\bN_{\se}$. This implies $\bXi_2((\Hmax)^\sy)=0$. Next note that, by concavity of the entropy function, we have
	\beq\label{eq:H:sy}
	\bF_2(H)
	\le \bF_2(H^{\sy})-\frac{\norm{H-H^{\sy}}_2^2}{C_k}\,.
	\eeq
for all $H\in\bDelta_2$. For any $H\in \bN_{\se}$, we must have $H^{\sy}\in \bN_{\se}$ also. It then follows using \eqref{eq:H:sy} that
$\Hmax=(\Hmax)^{\sy}$, and consequently $\bXi_2(\Hmax)=0$.
By Proposition~\ref{p:correction.to.resampling} this implies $\Hmax=H_{\bullet}$, which proves the first assertion of this proposition. For $H\in\bDelta_2$ with $\|H-H_\bullet\|\le \epsilon$ small enough, it follows using Theorem~\ref{t:mc} that
	\[
	\bF_2(H)
	\stackrel{\eqref{eq:H:sy}}{\le} \bF_2(H^{\sy})
		-\frac{\norm{H-H^{\sy}}_2^2}{C_k}
	\le \bF_2(H_\bullet)-\epsilon\bXi_2(H^\sy)
		-\frac{\norm{H-H^{\sy}}_2^2}{C_k}\,.
	\]
Combining with Proposition~\ref{p:correction.to.resampling} gives
	\[
	\bF_2(H)
	\le
	\bF_2(H_\bullet)-\epsilon^2\|H-H_\bullet\|^2
		-\frac{\norm{H-H^{\sy}}_2^2}{C_k}\,,
	\]
which implies the second assertion of this proposition (by relabelling $\epsilon$ appropriately).
\end{proof}

\end{document}